\documentclass[12pt,letterpaper,titlepage]{amsart}
\usepackage{amsmath, amssymb, amsthm, amsfonts,amscd,amsaddr}%,wasysym}
\usepackage[
    paper=a4paper,
    portrait=true,
    textwidth=425pt,
    textheight=650pt,
         tmargin=3cm,
    marginratio=1:1
            ]{geometry}

\usepackage[cmtip,matrix,arrow]{xy} \SelectTips{cm}{10}
\newcommand{\cxymatrix}[1]{\vcenter{\xymatrix@=15pt{#1}}}
\theoremstyle{plain}
\newtheorem{theorem}{Theorem}[section]

\newtheorem{cor}[theorem]{Corollary}

\newtheorem{prop}[theorem]{Proposition}
\newtheorem{lemma}[theorem]{Lemma}
\newtheorem{defprop}[theorem]{Definition/Proposition}

\theoremstyle{definition}
\newtheorem{definition}[theorem]{Definition}
\newtheorem{ex}[theorem]{Example}
\newtheorem{rmk}[theorem]{Remark}
\numberwithin{equation}{section}
\newtheorem*{theoremA*}{Theorem A}
\newtheorem*{theoremB*}{Theorem B}
\newtheorem*{theorem1*}{Theorem A'}
\newtheorem*{theoremC*}{Theorem C}
\newtheorem*{theoremD*}{Theorem D}
\newtheorem*{theoremE*}{Theorem E}
\newtheorem*{theoremF*}{Theorem F}
\newtheorem*{theoremE2*}{Theorem E2}
\newtheorem*{theoremE3*}{Theorem E3}
\newcommand{\bs}{\backslash}
\newcommand{\Cc}{\mathcal{C}}
\newcommand{\C}{\mathbb{C}}
\newcommand{\G}{\mathbb{G}}
\newcommand{\A}{\mathcal{A}}
\newcommand{\Nc}{\mathcal{N}}
\newcommand{\Lc}{\mathcal{L}}
\newcommand{\E}{\mathcal{E}}
\newcommand{\Hc}{\mathcal{H}}
\newcommand{\Kc}{\mathcal{K}}

\newcommand{\Q}{\mathbb{Q}}
\newcommand{\Sb}{\mathbb{S}}

\newcommand{\Z}{\mathbb{Z}}
\newcommand{\Zc}{\mathcal{Z}}

\newcommand{\Sc}{\mathcal{S}}
\newcommand{\Rc}{\mathcal{R}}
\newcommand{\Fc}{\mathcal{F}}
\newcommand{\R}{\mathbb{R}}
\newcommand{\N}{\mathbb{N}}
\newcommand{\Pb}{\mathbb{P}}

\newcommand{\Gr}{\operatorname{Gr}}
\newcommand{\Sl}{\operatorname{SL}}

\newcommand{\Ind}{\operatorname{Ind}}
\newcommand{\SO}{\operatorname{SO}}
\newcommand{\Hom}{\operatorname{Hom}}
\newcommand{\End}{\operatorname{End}}

\newcommand{\OO}{\operatorname{O}}
\newcommand{\GL}{\operatorname{GL}}

\newcommand{\SU}{\operatorname{SU}}
\newcommand{\tr}{\operatorname{tr}}

\newcommand{\Sp}{\operatorname{Sp}}
\newcommand{\Lie}{\operatorname{Lie}}
\newcommand{\Ad}{\operatorname{Ad}}

\newcommand{\ad}{\operatorname{ad}}
\newcommand{\diag}{\operatorname{diag}}

\newcommand{\pr}{\operatorname{pr}}

\newcommand{\cl}{\operatorname{cl}}
\newcommand{\ind}{\operatorname{ind}}
\newcommand{\id}{\operatorname{id}}

\newcommand{\vol}{\operatorname{vol}}

\newcommand{\Spec}{\operatorname{spec}}
\newcommand{\supp}{\operatorname{supp}}
\newcommand{\Span}{\operatorname{span}}

\newcommand{\Sym}{\operatorname{Sym}}

\newcommand{\rank}{\operatorname{rank}}

\newcommand{\re}{\operatorname{Re}}

\newcommand{\Mat}{\operatorname{Mat}}

\newcommand{\algebraicgroup}[1]{{\underline{#1}}}

\newcommand{\uG}{\algebraicgroup{G}}
\newcommand{\uH}{\algebraicgroup{H}}
\newcommand{\uZ}{\algebraicgroup{ Z}}
\newcommand{\uQ}{\algebraicgroup{ Q}}
\newcommand{\uL}{\algebraicgroup{ L}}
\newcommand{\uU}{\algebraicgroup{ U}}
\newcommand{\uP}{\algebraicgroup{ P}}
\newcommand{\uA}{\algebraicgroup{ A}}
\newcommand{\uR}{\algebraicgroup{ R}}
\newcommand{\uM}{\algebraicgroup{ M}}
\newcommand{\uMH}{\algebraicgroup{ M}_{\algebraicgroup H}}
\newcommand{\uK}{\algebraicgroup{K}}
\newcommand{\uLH}{\algebraicgroup{ L}_{\algebraicgroup H}}

\newcommand{\uAZ}{\algebraicgroup{  A}_{\algebraicgroup Z}}

\def\hat{\widehat}
\def\af{\mathfrak{a}}
\def\bfrak{\mathfrak{b}}
\def\e{\epsilon}
\def\gf{\mathfrak{g}}

\def\cf{\mathfrak{c}}

\def\hf{\mathfrak{h}}
\def\jf{\mathfrak{j}}
\def\kf{\mathfrak{k}}
\def\lf{\mathfrak{l}}
\def\mf{\mathfrak{m}}
\def\nf{\mathfrak{n}}

\def\pf{\mathfrak{p}}

\def\qf{\mathfrak{q}}
\def\rf{\mathfrak{r}}

\def\so{\mathfrak{so}}
\def\sp{\mathfrak{sp}}
\def\su{\mathfrak{su}}
\def\tf{\mathfrak{t}}
\def\uf{\mathfrak{u}}

\def\zf{\mathfrak{z}}
\def\la{\langle}
\def\ra{\rangle}
\def\1{{\bf1}}

\def\U{\mathcal{U}}
\def\B{\mathcal{B}}
\def\Cc{\mathcal{C}}

\def\Ic{\mathcal {I}}
\def\G{\mathcal{G}}
\def\Oc{\mathcal{O}}
\def\Tc{\mathcal{T}}
\def\Wc{\mathcal{W}}

\def\cR{\mathcal{R}}
\def\M{\mathcal{M}}
\def\oline{\overline}
\def\F{\mathcal{F}}
\def\V{\mathcal{V}}
\def\W{\mathsf{W}}

\def\v{\mathbf{v}}
\def\w{\mathbf {w}}

\def\st{\mathsf{t}}

\def\sP{\mathsf {P}}
\def\sH{\mathsf {H}}

\def\sc{\mathsf{c}}

\def\WF{\operatorname{WF}}

\newcommand{\sA}{\mathsf{A}} \newcommand{\sB}{\mathsf{B}}
\newcommand{\sC}{\mathsf{C}} \newcommand{\sD}{\mathsf{D}}
 \newcommand{\sF}{\mathsf{F}}
 \newcommand{\sX}{\mathsf{X}}
 \newcommand{\sW}{\mathsf{W}}
\newcommand{\sY}{\mathsf{Y}}

\newcommand{\sI}{\mathsf{I}}

\title[Plancherel Formula]
{Plancherel theory for real spherical spaces:  Construction of the Bernstein morphisms}

\begin{document}
\date{October 29, 2020}

\begin{abstract} This paper lays the foundation for Plancherel theory on
real spherical spaces $Z=G/H$, namely it provides the decomposition of $L^2(Z)$ into
different series of representations via Bernstein morphisms.
 These series are parametrized by subsets of spherical
roots which determine  the fine geometry of $Z$ at infinity. In particular, we obtain a generalization of the Maass-Selberg relations.
As a corollary we obtain a partial geometric characterization of the discrete spectrum:
$L^2(Z)_{\rm disc }\neq \emptyset$
if $\hf^\perp$ contains elliptic elements in its interior.
\par In case $Z$ is a real reductive group
or, more generally, a symmetric space our results retrieve the Plancherel formula
of Harish-Chandra (for the group) as well as that of Delorme and van den Ban-Schlichtkrull
(for symmetric spaces) up to the explicit determination of the discrete series for the
inducing datum.
\end{abstract}

\author[Delorme]{Patrick Delorme}
\email{patrick.delorme@univ-amu.fr}
\address{Institut de Math\'ematiques de Marseille,
UMR 7373 du CNRS, \\
Campus de Luminy, Case 907 - 13288 MARSEILLE Cedex 9}

\author[Knop]{Friedrich Knop}
\email{friedrich.knop@fau.de}
\address{Department Mathematik, Emmy-Noether-Zentrum\\
FAU Erlangen-N\"urnberg, Cauerstr. 11, 91058 Erlangen}

\author[Kr\"otz]{Bernhard Kr\"{o}tz}
\email{bkroetz@gmx.de}
\address{Institut f\"ur Mathematik, Universit\"at Paderborn,\\ Warburger Stra\ss e 100,
33098 Paderborn}

\author[Schlichtkrull]{Henrik Schlichtkrull}
\email{schlicht@math.ku.dk}
\address{University of Copenhagen, Department of Mathematics\\Universitetsparken 5,
DK-2100 Copenhagen \O}

\maketitle

\section{Introduction}
Our concern is with a homogeneous real spherical space $Z=G/H$. We assume that $Z$ is algebraic, i.e. there exists
a connected reductive group $\uG$, defined over $\R$, and an algebraic subgroup $\uH\subset \uG$, defined over
$\R$ as well, such that $G=\uG(\R)$ and $H=\uH(\R)$.  Then $Z$ is a $G$-orbit of the variety $\uZ(\R)$ where $\uZ=\uG/ \uH$.
We denote by $z_0=eH\in Z\subset \uZ(\R)$ the standard base point and recall that $Z$ is called real spherical if there is
a minimal parabolic subgroup $P\subset G$ such that $P\cdot z_0$ is open in $Z$.

\par  The goal of this paper is to develop the basic Plancherel theory for $L^2(Z)$, i.e. to establish the foundational
Bernstein-decomposition of $L^2(Z)$ into different series of representations.
Although the main body of the text is written in terms of $Z$, we  focus in this introduction
on $\uZ(\R)$ and the Bernstein decomposition for $L^2(\uZ(\R))$, for which our results are easier to state.
On a technical level we obtain the information for $\uZ(\R)$ by collecting the data of all $G$-orbits in $\uZ(\R)$.

\par Real spherical varieties $\uZ(\R)$ have a well understood $G$-equivariant compactification theory, which is constructed
out of the combinatorial data of $\uZ$ originating from the local structure theorem. We recall from \cite{KKS} that
attached to $\uZ$ there is a torus $\uA_\uZ=\uA/ \uA\cap \uH$, homogeneous for a maximal split torus $\uA$ of $\uG$ contained in $\uP$.
Let $A_Z$ be the identity component of $\uA_\uZ(\R)$, and $\af_Z$ its Lie algebra.
Inside $\af_Z$ one finds a co-simplicial cone $\af_Z^-$, called the compression cone, which is a fundamental domain
for a finite reflection group $W_Z$ \cite{KK}.  In particular there is a set $S\subset \af_Z^*$, of the so-called
spherical roots, such that
the faces of $\af_Z^-$ are given by $\af_I^-:=\af_Z\cap \af_I$ with $I\subset S$ and $\af_I:= I^\perp\subset \af_Z$.
For the simplicity of exposition we assume in this introduction that $S$ is a basis of the character group
$\Xi_Z\simeq \Z^n$ of the torus $\uA_\uZ$, the so-called wonderful case.

\par Now there exists a (wonderful) smooth $G$-equivariant compactification $\hat \uZ(\R)$ of $\uZ(\R)$ featuring
a stratification in $G$-manifolds,
$$\hat \uZ(\R)=\coprod_{I\subset S} \hat \uZ_I(\R),$$
parametrized by subsets $I\subset S$ of spherical roots \cite{KK}
and with $\uZ(\R)=\hat \uZ_S(\R)$.
The strata $\hat Z_I(\R)$ for $I\subset S$ arise as follows.
For every element $X$ in the relative interior $\af_I^{--}$ of the face $\af_I^-$ of $\af_Z^-$, the radial
limit
$$\hat z_{0,I}:=\lim_{t\to \infty} \exp(tX)\cdot z_0\in \hat \uZ(\R)$$
exists and is independent of  $X$.
Then $\hat H_I$, the $G$-stabilizer of  $\hat z_{0,I}$,  is real algebraic, i.e.  $\hat H_I = \hat \uH_I(\R)$,
and $\hat Z_I(\R) := [\uG\cdot \hat z_{0,I}](\R)$ is the set of real points in  the boundary orbit $\uG \cdot \hat  z_{0,I}$.  The group $\hat H_I$ acts on the normal space to
the stratum $\hat Z_I(\R)$ at $\hat z_{0, I}$.
The kernel of this isotropy action defines an algebraic  normal subgroup
$\uH_I\triangleleft \hat \uH_I$ with torus quotient $\uA_I=\hat \uH_I/ \uH_I$. The real spherical space
$\uZ_I(\R):=(\uG/ \uH_I)(\R)$ is in fact canonically attached to $\uZ(\R)$, i.e. it does not depend on the
particular compactification. Geometrically $\uZ_I(\R)$ is a deformation of $\uZ(\R)$ which approximates
$\uZ(\R)$ asymptotically near the vertex $\hat z_{0,I}$.
We denote by $A_I$ the identity component of $\uA_I(\R)$  and note that its Lie algebra is $\af_I$ defined above.

\par We assume now that $Z$ and hence also $\uZ(\R)$ is unimodular, i.e.~it carries a $G$-invariant positive Radon measure.
As $\uZ_I(\R)$ is a deformation of $\uZ(\R)$ for each $I\subset S$, it follows that $\uZ_I(\R)$ carries a natural $G$-invariant
measure as well.  On $\uZ_I(\R)$ the group $G\times A_I$ acts from left times right.  The left $G$-action defines a unitary
representation $L$ of $G$ on $L^2(\uZ_I(\R))$ given by $(L(g)f)(z)= f(g^{-1}\cdot z)$ for $g\in G$, $z\in \uZ_I(\R)$ and $f\in L^2(\uZ_I(\R))$.
The right action of $A_I$ on $\uZ_I(\R)$ defines a normalized unitary representation $\cR(a_I) f(z)= a_I^{-\rho} f(z\cdot a_I)$ for $a_I\in A_I$ and $f,z$ as before. The decomposition of $L^2(\uZ_I(\R))$ with respect to $\cR$ yields the disintegration in unitary $G$-modules

$$ L^2(\uZ_I(\R))=\int_{\hat A_I} L^2(\uZ_I(\R), \chi)\ d \chi$$
with $\hat A_I$ the unitary character group of the non-compact torus $A_I$.  The space $L^2(\uZ_I(\R), \chi)$ is the space of square
integrable densities with respect to $\chi$ and we denote by $L^2(\uZ_I(\R),\chi)_{\rm d}$ the discrete spectrum of this
unitary $G$-module.  We define the twisted discrete spectrum of $L^2(\uZ_I(\R))$ by
$$L^2(\uZ_I(\R))_{\rm td} := \int_{\hat A_I} L^2(\uZ_I(\R), \chi)_{\rm d}\ d \chi\, .$$

 \par The main result of this work (see Theorem \ref{thm planch refined real points} where $\sB$ of \eqref{B} is denoted by
 $B_{\R,{\rm res}}$)
is the construction  of a  $G$-equivariant surjective
map
\begin{equation}\label{B} \sB:  \bigoplus_{I\subset S}  L^2(\uZ_I(\R))_{\rm td} \to L^2(\uZ(\R))\end{equation}
such that source and image have equivalent Plancherel measures, i.e. belong to the same measure class.
Further each $\sB_I:=\sB\big|_{L^2(\uZ_I(\R))_{\rm td}}$ is a sum of partial isometries.
The latter property translates into
the Maass-Selberg relations, see Theorem \ref{eta-I continuous}, and will be explained in more detail below.
The existence of such a map originates from ideas of
J.~Bernstein, and accordingly we call $\sB$ the Bernstein morphism.
Let us remark that in the main text we derive a more general (but more complicated to state) result,
namely a Bernstein decomposition for $L^2(Z)$  (see Theorem \ref{thm planch} and Theorem \ref{thm planch refined})
from which we derive \eqref{B} by collecting the data for the various $G$-orbits in $\uZ(\R)$.

\par
For absolutely spherical spaces of wavefront type over a p-adic field $k$ a Bernstein map for $L^2(\uZ(k))$ with the same properties as above was constructed
by Sakellaridis and Venkatesh in \cite{SV}  under the assumption of certain properties of the discrete series, see \cite[Conjecture 9.4.6]{SV}.
A novel point of view in \cite{SV},  which we have adopted,  is the observation that the decomposition
of $L^2(\uZ(k))$ into  the various series of representations is reflected in the boundary geometry
of a  smooth compactification $\hat \uZ(k)$ of $\uZ(k)$.  Another new insight of
\cite{SV} is that no explicit knowledge of the discrete series is needed to derive the Bernstein
decomposition: the bottom line is the existence of a spectral gap for the discrete series. Since
a spectral gap theorem is established in full generality for real spherical spaces in \cite{KKOS},
we do not have to make any assumptions on the discrete spectrum as in \cite{SV}.

\par With the implementation of the Bernstein decomposition
the Plancherel theorem for $L^2(\uZ(\R))$ essentially reduces to
the understanding of the twisted discrete spectrum for each $\uZ_I(\R)$,
and the determination of $\ker \sB$.  Since the Bernstein map
is isospectral and surjective, it follows that the measure class of the Plancherel measure of $L^2(\uZ(\R))$
is given by countably many copies of the Haar measures on the tori $A_I$.

\par Let us consider the example
$Z=\uZ(\R)=G \times G /\diag G \simeq G$ of a real
semisimple algebraic Lie group.  Here the spherical roots $S$ are  identified with the simple roots with respect to
$\af$, the Lie algebra of $A$ of a maximal split torus of $G$.  Recall that subsets $I\subset S$ parametrize the  parabolic subgroups $P_I =L_I U_I$ of $G$.   Then we have $H_I = \diag(L_I) (U_I \times \overline{ U_I})$ with $\overline{P_I} = L_I \overline{U_I}$
the parabolic opposed to $P_I$ and in particular
$$\uZ_I(\R)= [G/ U_I \times G/\overline{U_I}]/ \diag(L_I)\, .$$
Write $L_I = M_I A_I$ as usual.
Now,  via induction by stages,  we readily obtain %as $G\times G$-module

\begin{equation} \label{intro1}L^2(\uZ_I(\R))_{\rm td} \underset{G\times G}{\simeq}  \sum_{\sigma \in \hat M_{I,{\rm disc}}}  \int_{i\af_I^*}
\pi_{\sigma, \lambda}\otimes \pi_{\sigma, \lambda}^* \ d\lambda\,,\end{equation}
where $\pi_{\sigma, \lambda}=\operatorname{Ind}_{P_I}^G (\lambda \otimes \sigma)$ is the unitarily induced representation
of $G$ with respect to the unitary character of $A_I$ defined by $\lambda$, and $\sigma$ is a discrete series representation
of  $M_I$.   Via basic intertwining theory we then group the occurring representations in \eqref{intro1} into  equivalence classes and obtain Harish-Chandra's Plancherel formula up to the classification  of the discrete spectrum of the inducing datum (see Section \ref{group case}).
Likewise  holds for the Plancherel theorem for symmetric spaces
as obtained by Delorme \cite{Delorme} and van den Ban-Schlichtkrull \cite{vdBS} and we refer
to Section \ref{section DBS} for the complete account.

\par  As in the work of Harish-Chandra on the Plancherel theorem
for a real reductive group, a constant term approximation \cite{HC1} lies at the heart of the proof. Let us explain that.
A Harish-Chandra module $V$ endowed with a linear functional $\eta$, such that $\eta$ extends to a continuous
$H$-invariant functional on the unique smooth moderate growth completion $V^\infty$,  will be called a spherical pair and denoted
$(V,\eta)$.  The continuous dual of $V^\infty$ is denoted $V^{-\infty}$, and from \cite{KKS2} originates
a natural linear map
\begin{equation}\label{eta corresp} (V^{-\infty})^H \to  (V^{-\infty})^{H_I}, \ \eta\mapsto \eta^I\, .\end{equation}
Attached to $\eta$ are the generalized matrix coefficients $m_{v,\eta}(gH)=\eta(g^{-1}v)$ which define smooth functions
on $Z$ for all $v\in V^\infty$. Likewise we obtain smooth functions $m_{v,\eta^I}$ on $Z_I:=G/H_I\subset \uZ_I(\R)$.  An appropriate notion
of temperedness for functions on a real spherical spaces was defined in \cite{KKSS2}, and accordingly
$\eta$ is called tempered if all associated matrix coefficients are tempered functions. The map \eqref{eta corresp}
then gives rise to a linear map of tempered functionals
$$(V^{-\infty})_{\rm temp}^H \to  (V^{-\infty})_{\rm temp}^{H_I}\, .$$
The constant term approximation \cite{DKS} measures the differences

$$ |m_{v,\eta}(g\exp(tX) H)  - m_{v,\eta^I}(g\exp(tX)H_I)|$$
for $g\in \Omega$,  a compact subset of $G$,  and $t\to \infty$ for $X\in \af_I^{--}$.  We refer to
Theorem \ref{loc ct temp} below for the detailed statement.

In case of the group
Harish-Chandra obtained in \cite{HC1} such an approximation for a fixed representation.
Using his strong results on the discrete series \cite{HC}
it was made uniform for all
tempered representations  in \cite{HC2}. For spherical spaces the
uniformity of the constant term approximation is obtained in \cite{DKS}
via the spectral gap theorem of \cite{KKOS} for the
twisted discrete spectrum.
\par

Let us mention that our constant term approximation is also uniform in the category of smooth vectors
so that there is no need for expansion of functions in terms of $K$-types. On a geometric level this allows us to
view $Z_I$ and $Z$ in terms of  the orbit geometry of the minimal parabolic subgroup $P$.  In more detail, we show that there is a natural
injective map of open $P$-orbits  $(P\bs Z_I)_{\rm open} \to (P\bs Z)_{\rm open}$. This in turn allows us
to identify $Z_I$ inside $Z$, up to measure zero via the open $P$-orbits.  We refer to
Section \ref{main remainder} for the analytic implementation of this $P$-equivariant point of view.
Let us point out that the auxiliary "exponential maps" of \cite{SV}, which allowed an identification of $\uZ(k)$ and $\hat \uZ_I(k)$
near the vertex $\hat z_{0,I}$, are no longer needed in our context of $P$-equivariant matching of $Z_I$
with $Z$ up to measure zero.

\par For almost all irreducible Harish-Chandra modules in the spectrum of $L^2(Z_I)$ the
multiplicity space $(V^{-\infty})_{\rm temp}^{H_I}$ is a finite dimensional semisimple module for $\af_I$ and accordingly
every $\eta^I\in (V^{-\infty})_{\rm temp}^{H_I}$ decomposes into eigenvectors
$$\eta^I=\sum_{\lambda \in \rho +i\af_I^*}\eta^{I,\lambda}\, .$$
Our Maass-Selberg relations are then expressed in the form that $\eta\mapsto \eta^{I,\lambda}$ is a partial
isometry, see Theorem \ref{eta-I continuous}. Notice that the $\eta^{I,\lambda}$ reflect the asymptotics
of the matrix coefficients $m_{v,\eta}$  through the constant term approximation.
Finally we define the Bernstein morphisms spectrally via the technique of tempered embedding developed in
\cite[Sect. 9]{KKS2}.

 \par As a corollary of the Bernstein decomposition we obtain a partial geometric characterization of the existence of
 the discrete spectrum:

\begin{equation}\label{DS} \operatorname{int} \hf_{\rm ell}^\perp \neq \emptyset\quad \Rightarrow\quad L^2(Z)_{\rm d}\neq \emptyset\,,
\end{equation}
see Theorem \ref{thm discrete}.  This formulation reflects the known geometric characterization for groups
and symmetric spaces, going back to Harish-Chandra \cite{HC} and Flensted-Jensen \cite{FJ}.
Actually we expect that the converse implication in \eqref{DS} holds as well, and we provide
a geometric analogue of the expected equivalence via moment map geometry in Theorem
\ref{thm moment discrete}.

 \bigskip
 \par{\it Acknowledgement:} We are grateful to Joseph Bernstein who
 provided us with many useful remarks to a preliminary version of this article.

\section{Notions and Generalities}

Throughout this paper we use upper case Latin letters $A,B, C\ldots$ to denote Lie groups and write
$\af, {\mathfrak b}, \cf,\ldots$ for their corresponding Lie algebras. If $G$ is a Lie group, then we denote by
$G_0$ its identity component.

\par If $M$ is a set and $\sim$ is an equivalence relation on $M$, then we denote by $[m]$ the  equivalence class of $m\in M$. Often
the equivalence class is obtained  by orbits of a group $G$ acting on $M$.  More specifically if $X, Y$ are sets and $G$ is
a group which acts on $X$ from the right and acts on $Y$ from the left, then we obtain a left $G$-action on $X\times Y$ by $g\cdot(x,y):=
(x\cdot g^{-1}, g\cdot y)$ whose set of equivalence classes we denote by $X\times_G Y$.  We often abbreviate and simply write
$[x,y]$ instead of $[(x,y)]$ to denote the equivalence class of $(x,y)$.

\par  Given a group $G$ and subgroup $H\subset G$ we use for $g\in G$ the notation $H_g:=gHg^{-1}$, i.e.
$H_g$ is the $G$-stabilizer of the point $gH\in G/H$.

\par For a Lie algebra $\gf$ we write $\U(\gf)$ for the universal enveloping algebra of $\gf_\C$.  Further we denote
by $\Zc(\gf)$ the center of $\U(\gf)$.

\par If $\uZ$ is an algebraic variety defined over $\R$ and $k\supset\R$ is a field, then we
denote by $\uZ(k)$ the set of $k$-points.  Since we only consider fields $k=\R,\C$ in this
paper we abbreviate in the sequel and simply set $\uZ:=\uZ(\C)$.
\smallskip

\par Let now $\uG$ be a connected
reductive algebraic group defined over $\R$
and let $G:=\uG(\R)$.  As a general rule we use the following notation:
if $\uR$ is an algebraic subgroup of $\uG$ and defined over $\R$, then we set $R:= \uR(\R)$
and note that $R$ is closed Lie subgroup of $G$. We regard $G\subset \uG$ and then
$R=G\cap\uR$.
We let $\uH<\uG$ be an algebraic subgroup defined over $\R$, and
define $H<G$ according to this rule.
For intersections with $\uH$ we adopt the notation $\uR_{\uH}:= \uR\cap \uH$
and likewise $R_H:= R \cap H=\uR_{\uH}(\R)$.

Set $\uZ:= \uG/ \uH$ and observe that $\uZ$ is a smooth $\uG$-variety defined over $\R$.  Set $Z:=G/H$ and observe that
$Z$ is a $G$-orbit of $\uZ(\R)$.   In general $\uZ(\R)$ is a finite union of $G$-orbits, but typically
not equal to $Z$.   For example if  $\uG=\Sl(n,\C)$ and $\uH=\SO(n,\C)$ then $\uZ(\R)\simeq  \bigcup_{2k \leq n} \Sl(n,\R)/ \SO(n-2k, 2k)$ identifies with the real symmetric matrices with unit determinant, whereas $Z$ comprises
the set of positive definite symmetric matrices  therein.  In particular,
in this case  $Z=G/H\subsetneq \uZ(\R)$.
This shows,  when taking real points of the principal bundle

\begin{equation} \label{Z-exact} \1 \to \uH  \to \uG  \twoheadrightarrow \uZ\end{equation}
we have to act with care, as  the functor of taking real points in (\ref{Z-exact})
is only left exact
\begin{equation} \label{ZR-exact} \1 \to H \to  G  \to\uZ(\R)\end{equation}
and extends to a long exact sequence of pointed sets \cite[I.5.4, Prop. 36]{Serre} in Galois cohomology

\begin{equation} \label{ZR-exact long}
\1 \to H  \to G  \to\uZ(\R)\to  H^1 (\operatorname{Gal}(\C|\R), \uH) \to H^1(\operatorname{Gal}(\C|\R), \uG)\, .
\end{equation}

In this context we recall from \cite[Prop. 13.1]{KK} that:

\begin{lemma} \label{Example exact}   If  $\uG$ is anisotropic over $\R$, i.e. $\uG(\R)$ is compact, then
\eqref{ZR-exact} is right exact.
\end{lemma}

We denote by $z_0=\uH$ the standard base point of $\uZ$ and observe the $G$-equivariant embedding

$$ Z \to \uZ= \uG/ \uH, \ \ gH \mapsto g\uH= g\cdot z_0\,.$$

\par If $\uR$ is a unipotent group, then note that $R$ is connected for the Euclidean topology.  This
is because unipotent groups $\uR$ are isomorphic (as varieties) to their Lie algebras $\rf_\C$ via the algebraic
exponential map.

\subsection{Real spherical spaces and the local structure theorem}\label{subsection LST}

\par Let $\uP<\uG$ be a parabolic subgroup of $\uG$ which is minimal with respect to being defined over $\R$.
We denote by $\underline{N}$ the unipotent radical of $\uP$.

\par We assume that $Z$ is {\it real spherical}, that is,
the action of $P$ on $Z$ admits an open orbit.  After replacing $P$ by a
conjugate we will assume that $P\cdot z_0$ is open in $Z$.
The local structure theorem (see \cite[Th. 2.3]{KKS} and \cite[Cor. 4.11]{KK}) asserts the existence  of a parabolic subgroup
$\uQ\supset \uP $  with Levi-decomposition $\uQ=\uL \ltimes \uU$ defined over $\R$ such that one has

\begin{eqnarray} \label{lst1}\uP\cdot z_0&=& \uQ \cdot z_0\\
\label{lst2}\uQ_\uH &=& \uL_\uH \\
\label{lst3}\uL_{\rm n} &\subset& \uL_{\uH}\end{eqnarray}
where $\uL_{\rm n}$ is the unique  connected normal $\R$-subgroup of $\uL$  such that
the Lie algebra $\lf_{\rm n}$ is the sum of all non-compact, non-abelian simple ideals
of $\lf$.

\begin{rmk} \label{rmk choice of L} In addition to \eqref{lst1} - \eqref{lst3} we request from our choice of $\uL$ that it is
obtained via the constructive proof of the local structure theorem.
In case that $\uZ=\uG/\uH$ is
quasi-affine, this means that there exists $\xi\in \hf^\perp \subset \gf$ such that
$$\uL= Z_{\uG}(\xi)= \{g \in \uG\mid  \Ad^*(g) \xi= \xi\}\, .$$
In case $\uZ$ is not quasi-affine one uses  a quasi-affine cover (cone construction)
to reduce to the quasi-affine case:  extend $\uG$ to $\uG_1=\uG\times \C^\times$ and
let $\psi: \uH \to \C^\times$ be a character defined over $\R$ which is obtained from a Chevalley embedding
of $\uZ$ into projective space which is defined over $\R$. With $\uH_1=\{ (h, \psi(h))\mid h\in \uH\}$
we obtain a real spherical subgroup $\uH_1\subset \uG_1$ such that $\uZ_1=\uG_1/\uH_1$ is quasi-affine.
The local structure theorem for $\uZ_1$ then descends to a local structure theorem for $\uZ$.
\par With this choice of $\uL$ it is then guaranteed that the slice $\uL/\uL_{\uH}$ can be extended to
suitable compactifications of $\uZ$ which will be used later in this text.\end{rmk}

In particular, we obtain from (\ref{lst1}) - (\ref{lst2}) via the obvious multiplication map

\begin{equation} \label{LST1} \uP\cdot z_0
\simeq \uU \times \uL / \uLH\end{equation}
an isomorphism of  algebraic varieties defined over $\R$.
If we take real points in (\ref{LST1}) we get

\begin{equation} \label{LST1R} [\uP\cdot z_0](\R)
\simeq U \times (\uL / \uLH)(\R).\end{equation}
In the next step we wish to describe $(\uL/\uLH)(\R)$ in more detail.
For that let  $\uA\subset \uL\cap \uP $ be a maximal split torus and set $\uAZ:= \uA/ \uA_{\uH}$.
We also view the torus $\uAZ$ as a subvariety of $\uZ$.
Further we define $A_Z$ to be the identity component of $\uAZ(\R)$.

The number $r:=\rank_\R Z:=\dim A_Z$ is an invariant of $Z$ and referred to as the
{\it real rank of $Z$}.

\par Let
$K$ be a maximal compact subgroup of $G$. Note that $K$ is algebraic, i.e. $K=\uK(\R)$.
Further we denote by $\zf(\gf)$ the center of $\gf$, and we fix with $\kappa: \gf\times \gf\to\R$ a non-degenerate $\Ad(G)$-invariant bilinear form
which yields an orthogonal decomposition of the center $\zf(\gf) = (\zf(\gf)\cap \af)\oplus (\zf(\gf)\cap \kf)$.
In case $\gf$ is semi-simple, the Cartan-Killing form can be used for $\kappa$.
It is a standing further requirement for $K$ that $\kf \perp \af$.
Then $\uM:=Z_{\uK}(\uA)$, the centralizer
of $\uA$ in $\uK$, does not depend on the particular choice of $K$ with $\kf \perp \af$.

\par  Notice that $Z_{\uG}(\uA)$ is a Levi-subgroup of $\uP$ and as such connected.
Moreover we have $Z_{\uG}(\uA)= \uM \uA$. Notice that (\ref{lst3}) implies that $\uM \uA$ acts transitively on $\uL/\uLH$.

\par In the next two paragraphs we recall some elementary facts from \cite[Sect. 1 and App. B]{DKS}.
Define
$$\hat \uMH=\{  m\in \uM\mid   m\cdot z_0 \in \uAZ\}$$
and note that $\hat \uMH$ is the isotropy group for the action of $\uM$ on $ \uL/ \uLH \uA$.
In particular,  $\hat\uMH$ is an algebraic
subgroup of $\uG$ defined over $\R$.  Moreover, $\hat\uMH$  contains $\uMH$ as
a normal subgroup such that $F_M:= \hat \uMH/ \uMH\simeq \hat M_H/ M_H$ is
a finite $2$-group. Here $\hat M_H=\hat\uM_\uH(\R)\subset M$ by our notational conventions.

\par Now $\uL/\uLH$ is homogeneous for $\uM\uA$ and thus

\begin{equation} \label{deco L} \uL/ \uLH  \simeq    \uM \times_{\hat \uMH} \uAZ  =  \uM/\uMH \times_{F_M}  \uAZ\, .\end{equation}
In particular, by \cite[Prop. B.2]{DKS}
\begin{equation} \label{LST2} (\uL/ \uLH)(\R)   \simeq   M \times_{\hat M_H} \uAZ(\R)  =
M/M_H \times_{F_M}  \uAZ(\R)
\end{equation}
where  $\simeq$ refers to an isomorphism of real algebraic varieties.

{}From (\ref{LST1}) and (\ref{LST2}) we obtain the following form of the local structure theorem, which we will use later on:
\begin{equation} \label{LST3} [\uP \cdot z_0](\R)\simeq U \times \left[M/M_H \times_{F_M}  \uAZ(\R)\right]\, .
\end{equation}

\par Recall that $\uAZ(\R) \simeq (\R^\times)^r $, with $r=\rank_\R(Z)$, is a split torus viewed as a
subvariety of $\uZ(\R)$.
Set
$$A_{Z,\R} := \uAZ(\R) \cap Z\, .$$
Then it is clear that $A_Z\subset A_{Z,\R}\subset \uAZ(\R)$.
In general however, $A_{Z,\R}$ is not a group, but carries only  the structure of an $A_Z$-set (see Example
\ref{ex SL3} below for $Z=\Sl(3,\R)/\SO(2,1)$).
\par Let $F_\R=\{-1,1\}^r<\uAZ(\R)= (\R^\times)^r$ be the $2$-torsion subgroup of $\uAZ(\R)$.
Since $A_Z$ is defined to be the identity component of  $\uAZ(\R)$ we obtain the following isomorphism of groups

\begin{equation}\label{AF}
\uAZ(\R)  = A_Z F_\R \simeq A_Z \times F_\R\, .  \end{equation}
Let $F\subset F_\R$ be the subset such that $A_{Z,\R}= A_Z F$, i.e. $F=F_\R\cap  A_{Z,\R}$.
Set $T_Z:=\exp_{\uA}(i\af_H^\perp)<\uA$ and note that
$F_\R\subset T_Z\cdot z_0$
as $T_Z\cdot z_0$ contains all torsion elements of $\uAZ$.

\par Since $F_M$ maps  faithfully into $F_\R$ we view it in the sequel
as a subgroup of $F_\R$. Note that $F_M\subset F$ and that $F_M$ acts on $F$.

With this terminology we obtain from (\ref{LST2}) that

\begin{equation} \label{LST4} Z \cap (\uL/ \uLH)(\R)   \simeq    M/M_H \times_{F_M}  A_{Z,\R}\, ,\end{equation}
and accordingly from \eqref{LST3}

\begin{equation} \label{LST5} Z \cap [\uP\cdot z_0](\R)  \simeq   U \times
\left[ M/M_H \times_{F_M}  A_{Z,\R}\right]\, .\end{equation}

The set of open $P$-orbits in $Z$, resp.~$\uZ(\R)$,  is an important geometric invariant
and plays a dominant role in the harmonic analysis on $Z$, resp.~$\uZ(\R)$.
For a symmetric space it is known from \cite{Mat1} that
the open $P$-orbits are parametrized by a quotient of a Weyl group with a subgroup.
Although no such parametrization is known in general we denote
$$W_\R:= (P\bs \uZ(\R))_{\rm open}  \quad\text{and} \quad W:=(P\bs Z)_{\rm open}\, ,$$
motivated by the special case.

From \eqref{LST3} and (\ref{LST5})  we deduce:

\begin{lemma} \label{lemmaW1} The maps
$$ F_M\bs F_\R    \to W_\R, \ \  \st=F_Mt  \mapsto Pt$$
and
$$ F_M\bs F   \to W, \ \ \st=F_Mt   \mapsto Pt $$
are bijections.
\end{lemma}

It is often convenient to select representatives of $W$ in $G$.  For any $\st \in F_M\bs F$ we pick a representative
$t\in F$ such that $\st=F_M t$. Then $Pt\in W$
and $t\in Z=G\cdot z_0$ implies that there is a lift
$w=w(\st)\in G$ of $t$ to $G$ such that $t =w\cdot z_0$.
If $\Wc=\{w(\st)\mid \st \in F_M\bs F\}$, then
the assignment
$$\Wc   \to W, \ \ w \mapsto Pw\cdot z_0$$
is a bijection.
\par Let $w=w(\st)\in\Wc$ and let $\tilde t \in T_Z$ be a lift of $t$, i.e.
$\tilde t\cdot z_0 = t$. Then

\begin{equation}\label{th-deco} w= \tilde t h \end{equation}
for some $h\in \uH$.

\subsection{Spherical roots and the compression cone}

Let $\Sigma=\Sigma(\gf,\af)$ be the restricted root system for the pair $(\gf,\af)$ and let

$$ \gf = \af \oplus \mf \oplus \bigoplus_{\alpha\in \Sigma} \gf^\alpha$$
be the attached root space decomposition.
Write $(\lf \cap \hf)^{\perp_\lf} \subset \lf$ for the orthogonal complement of $\lf \cap \hf$ in $\lf$
with respect to $\kappa$.
From $\gf=\qf +\hf=\uf \oplus (\lf\cap\hf)^{\perp_\lf}\oplus \hf$ and $\gf=\qf\oplus \oline{\uf}$ we infer the existence of a linear
map $ T:\oline{\uf}\to \uf \oplus (\lf\cap\hf)^{\perp_\lf}$ such that
$\hf=(\lf\cap \hf) \oplus \G(T)$ with $\G(T) \subset \oline{\uf} \oplus \uf \oplus (\lf\cap\hf)^{\perp_\lf}$ the graph of $T$.

\par Set $\Sigma_\uf:=\Sigma(\uf,\af)\subset \Sigma$. For $\alpha \in \Sigma_\uf$ and
$X_{-\alpha} \in \gf^{-\alpha}$
let $T(X_{-\alpha})= \sum_{\beta\in \Sigma_\uf\cup\{0\}} X_{\alpha,\beta}$
with $X_{\alpha,\beta} \in \gf^\beta$ for $\beta\in \Sigma_\uf$ and $X_{\alpha, 0}\in (\lf\cap\hf)^\perp$.
Let $\M\subset \af^*\bs\{0\}$ be the additive semi-group generated by

$$\{ \alpha+\beta\mid \alpha\in \Sigma_\uf, \exists X_{-\alpha} :  \ X_{\alpha,\beta}\neq 0\}\, .$$

Note that all elements of $\M$ vanish on $\af_H$ so that we can view $\M$
as a subset of $\af_Z^*$.  A bit more precisely the elements of $\M$, seen as characters of $\uA$, are
trivial when restricted to $\uA_\uH$ and therefore factor to characters of $\uAZ$.  Thus if we  denote by $\Xi_Z:= \Hom (\uAZ, \C^\times)\simeq \Z^r$ the character group, seen as
a lattice in $\af_Z^*$, we have $\M\subset \Xi_Z$.

\par Define
$$\af_{Z,E}:= \{X\in \af_Z \mid (\forall \alpha \in \M)\  \alpha(X)=0\}$$
and note that $\M$ belongs to $\af_{Z,E}^\perp\subset \af_Z^*$.
Next, according to \cite[Cor. 9.7]{KK}, the convex cone  $\R_{\geq 0} \M$ is
simplicial in $\af_{Z,E}^{\perp}$.
Generators of this cone, suitably normalized, will be called spherical roots and denoted $S$.

\par  The standard normalization of $S$  is that a generator $\sigma$ of
$\R_{\geq 0} \M$ belongs to $S$ provided it is integral and indivisible, i.e.
$\sigma \in \Xi_Z$ and $\frac1n \sigma \not \in
\Xi_Z$ for all $n\geq 2$.
\par Next we define the {\it compression cone}  by

$$\af_Z^-:=\{ X\in\af_Z\mid (\forall \alpha\in S) \ \alpha(X)\leq 0\}.$$
\begin{rmk}  The set of spherical roots $S$ and the associated co-simplicial compression
cone $\af_Z^-$ make up an algebro-geometric invariant of the real spherical space $Z$, see \cite{KK}. This
is important for this article, as the Bernstein morphisms defined later  have an inherent parametrization
by subsets $I\subset S$, i.e. faces of $\af_Z^-$.

Let us also mention that there is an alternative elementary approach
to the compression cone as a fundamental domain of a finite Coxeter group, see \cite{KuitSayag}.
\end{rmk}
Let us define by $\af_{Z,E}=\af_Z^- \cap (-\af_Z^-)$ the edge of $\af_Z^-$ and record
$$\# S = \dim \af_Z/\af_{Z,E}\, .$$

\par Following \cite[Sect. 3]{KKS2} we define for $I\subset S$ the boundary degeneration of $\hf_I$ of $\hf$ by
\begin{equation}\label{eq-hi1}  \hf_I:= \lf\cap\hf \oplus \G(T_I)\end{equation}
where
\begin{equation} \label{eq-hi2} T_I(X_{-\alpha}):= \sum_{\alpha+\beta\in \N_0[I]} X_{\alpha,\beta}\, .
\end{equation}
Observe that $\hf_S=\hf$ and $\hf_\emptyset=\oline \uf \oplus \lf\cap\hf$.

We also set
\begin{align} \af_I&:=\{ X\in \af_Z\mid (\forall \alpha \in I) \ \alpha(X)=0\} ,\\
\af_I^-&:=\{ X\in \af_I \mid (\forall \alpha\in S\bs I) \ \alpha(X)\leq0\} ,\\
\af_I^{--}&:=\{ X\in \af_I \mid (\forall \alpha\in S\bs I) \ \alpha(X)<0\} .\end{align}
We recall from \cite[Sect. 3]{KKS2} that for all $X\in \af_I^{--}$
\begin{equation}\label{I-compression}  \hf_I=\lim_{t\to \infty} e^{t\ad X}\hf\, .\end{equation}

Notice that $\af_Z=\af/\af_H$ is a quotient and not canonically a subalgebra of $\af$.
In general it is convenient and notation saving
to identify $\af_Z$ as a
subalgebra of $\af$ by means of the identification $\af_Z\simeq \af_H^{\perp_\af}$.
Then $\af_I$ normalizes $\hf_I$ and we obtain with
$$\hat \hf_I:=\af_I +\hf_I$$
a Lie subalgebra of $\gf$. It follows from (\ref{I-compression}) that $L_H$ normalizes each $\hf_I$.
Further we define $A_Z=\exp(\af_Z)\subset A$ as a connected subgroup of $A$ and set
$A_Z^-:=\exp(\af_Z^-)$.

%We conclude with a final piece of notation.  If a group $G$ acts on a space $X$ we denote by $L_g: X\to X$
%the left-displacement by $g$, i.e. $L_g(x)=g\cdot x$ for $x\in X$.

\section{Equivariant smooth compactifications of $\uZ(\R)$}\label{section compact}

In this section we explain and recall the principles of $G$-equivariant compactification
theory of $\uZ(\R)$ as developed in  \cite [Sect. 7]{KK}.

The main idea is to use a partial toric completion
of the torus $\uAZ$ via a fan $\Fc$ supported in all of $\af_Z^-$ (in \cite{KK} these fans are called {\it complete}).
Let us call this partial completion $\uAZ(\Fc)$.

\par Given a complete fan supported in $\af_Z^-$, we inflate (\ref{LST1}) and form the $\uP$-variety

$$ \uZ_0(\Fc):= \uU \times (\uL/\uLH \times_{\uAZ} \uAZ(\F))\, . $$
Now it is the content of \cite[Th. 7.1]{KK} that there exists a $\uG$-variety
$\uZ(\Fc)$ of the form $ \uZ(\Fc)= \uG\cdot \uZ_0(\Fc)$ containing
$\uZ_0(\Fc)$ as an open subset.  Note that $\uZ(\Fc)(\R)$ is compact by \cite[Cor. 7.12]{KK}.
The compactifications $\uZ(\Fc)$ of $\uZ$ just constructed are usually called {\it toroidal} as they origin
from partial compactifications of the torus $\uAZ$.

\par For a cone $\Cc\in \Fc$  in the fan $\Fc$ we denote by $\operatorname{int} \Cc$ its relative interior, i.e. the interior
with respect to $\af_\Cc:=\Span_\R \Cc\subset \af_Z$.
Now to every cone $\Cc\in \Fc$ corresponds a radial limit $\hat z_\Cc\in \uAZ(\Fc)\subset \uZ(\Fc)$ defined as
follows. The limit

$$ \hat z_\Cc:= \lim_{s\to \infty} \exp(sX)\cdot z_0$$
exists for every $X\in \operatorname{int} \Cc$ and is independent of $X$.
Moreover, the $\uG$-orbits in $\uZ(\Fc)$ are parametrized by the cones
$\Cc\in\Fc$  by way of $\Cc\mapsto \hat \uZ_\Cc:=\uG\cdot \hat z_\Cc$, see \cite[Cor. 7.5]{KK}.

\par Define $\uA_\Cc\subset \uAZ$ as the torus which fixes $\hat z_\Cc$ and note that its Lie algebra
is given by the complexification of $\af_\Cc$ defined above.
 Hence if $I=I(\Cc)$ is the set of spherical roots  vanishing on $\Cc$, then $\af_\Cc\subset \af_I$.
Further if we denote by $\hat \uH_\Cc $ the $\uG$-stabilizer of $\hat z_\Cc$, then we have the following relation
for Lie algebras:
\begin{equation} \label{formula hc} \hat \hf_\Cc= \hf_I +\af_\Cc\end{equation}
with $\hf_I$ defined as in \eqref{I-compression}. In case $\uZ(\Fc)$ is smooth we provide a simple argument for \eqref{formula hc} below.

\par For our purpose we need that $\uZ(\Fc)$ is a smooth manifold. By the construction of $\uZ(\Fc)$ this is the case
if and only if  $\uAZ(\Fc)$ is smooth. Let us now provide a standard construction of a complete fan which yields
a smooth partial completion $\uAZ(\Fc)$.
For that we denote by $\Xi_Z=\Hom( \uAZ, \C^*)\simeq \Z^r$ the character group of $\uAZ$. Likewise
we let $\Xi_Z^\vee= \Hom(\C^*, \uAZ)$ be the co-character group and note
the natural identification $\af_Z\simeq \Xi_Z^\vee \otimes_\Z\R$.

\par Best results are obtained when $S$ is a
 $\Z$-basis
for the character lattice $\Xi_Z$. In this case the standard fan $\Fc_{\rm st}$ obtained by the faces of $\af_Z^-$
is smooth and $\uZ(\Fc_{\rm st})$ is the  wonderful compactification of $\uZ$ (see \cite[Definition 11.4]{KK}).

\begin{rmk} \label{rmk non smooth}In general $S$ is not a basis of $\Xi_Z$. This  can have several
natural reasons, for example if $\af_{Z,E}\neq 0$ as $\#S:=\dim \af_Z/\af_{Z,E}< r=\rank_\R(Z)=\dim \af_Z$.
One might overcome this by passing from $\uH$ to $\uH= \uH \cdot \uA_{Z,E}$.  But even
if $\#S =r$ it might happen that there is torsion, i.e. $\Xi_Z/ \Z[S]\neq 0$ which destroys
smoothness of $\uAZ(\Fc)$ for $\Fc$ the fan generated by $\af_Z^-$.
\end{rmk}

One can overcome both issues indicated in Remark \ref{rmk non smooth}
simultaneously by subdividing $\af_Z^-$ into finitely many
simple simplicial cones $C_1, \ldots, C_N$ such that
\begin{itemize}
\item $\af_Z^- =\bigcup_{j=1}^N C_j$,
\item $C_i \cap C_j$ is a common face of both $C_i$ and $C_j$ for all $1\leq i, j\leq N$,
\item For each $1\leq j\leq N$ there exists a basis $(\psi_{ji})_{1\leq i\leq r}$ of $\Xi_Z$ such that
 $C_j=\{ X\in \af_Z\mid  (\forall 1\leq i \leq r)  \psi_{ji}(X)\leq 0\}$.
\end{itemize}

The existence of such a decomposition is a standard fact of  toric geometry, see
\cite[Ch.~3]{KKMS}.  Let us denote by $\Fc_i$ the fan generated by $C_i$,
i.e.~the set of all faces of $C_i$. Then define the fan $\Fc:=\bigcup_{i=1}^N \Fc_i$.
Notice that $\uAZ(\Fc)$ is smooth and is obtained from gluing together
the various open pieces $\uAZ(\Fc_i)\simeq \C^r$.
From $\af_Z^- =\bigcup_{j=1}^N C_j$ we obtain $\af_I^{--}= \bigcup_{j=1} C_j\cap \af_I^{--}$.
Now for every $I\subset S$ we let $J_I\subset\{ 1, \ldots, N\}$ be
the set of indices $j$ for which
$C_j\cap \af_I^{--}\neq \emptyset $.  Then $\af_I^{--}= \bigcup_{j\in J_I}  (C_j \cap \af_I^{--})$. Note that in general
$J_I$ is not a singleton as for example $J_\emptyset=\{1, \ldots, N\}$.

\par We fix now a simplicial subdivision as above and the corresponding
complete fan $\Fc$.
To abbreviate notation
we set $\hat \uZ_0:=\uZ_0(\Fc)$ and
$\hat \uZ:=\uZ(\Fc)$. We denote by $\hat Z$ the closure of $Z$ in $\hat \uZ(\R)$ which is then
a manifold with corners \cite[Sect. 14]{KK}.

\par For every $I\subset S$ we fix now an $j_I\in J_I$
and let $\cf_I^-= C_{j_I} \cap \af_I^{-}\subset \af_I^-$.
We denote by $\cf_I^{--}$ the relative interior of $\cf_I^{-}$.
We recall $z_0=\uH\in\uZ$ the standard base point. Then for $X\in \cf_I^{--}$
the limit
\begin{equation} \label{def limit z0I}  \hat z_{0,I}:=\lim_{s\to \infty}  \exp(sX)\cdot z_0 \in \uAZ(\F_{j_I})(\R) \subset \hat \uZ_0(\R)\end{equation}
exists and is independent of the choice of $X\in\cf_I^{--}$ (but depends on $j_I$).

\begin{rmk}\label{rmk face FI} Our choice of $j_I\in J_I$ yielding $\cf_I^-$ can also be seen in the following context.
Set
\begin{equation}\label{def FI}\Fc_I:=\{ \Cc\in \Fc\mid  \af_\Cc=\af_I\}\, .\end{equation}
Then our choice of $j_I\in J_I$ picks an element $\cf_I^-\in \Fc_I$ together with an $1\leq j_I\leq N$ such that
$\cf_I^-\subset C_{j_I}$.
\end{rmk}

Let us denote by $\hat \uH_I$ the stabilizer of $\hat z_{0,I}$ in $\uG$. Note that $\hat \uH_I$ is defined over $\R$.
We claim that  $\hat H_I$
has Lie algebra
\begin{equation} \label{Lie hat HI} \hat \hf_I =\hf_I +\af_I\end{equation}
with $\hf_I$ defined in \eqref{eq-hi1}.
In order to see that we
note that the $G$-stabilizer of  $z_t:=\exp(tX)\cdot z_0$ is $H_t:=\exp(tX) H \exp(-tX)$.
Moreover the fact that $z_t\to \hat z_{0,I}$ in the smooth manifold $\hat \uZ(\R)$ implies that
the stabilizer Lie algebra of the limit  $\hat z_{0,I}$ contains the limit $\hf_I$
of \eqref{I-compression}.  Now the claim follows from
\cite[Th. 7.3]{KK}.
\par We define $\hat \uZ_I= \uG \cdot \hat z_{0,I}\simeq \uG/ \hat \uH_I$ . The next proposition shows that
this definition is  independent of the choice of $\cf_I^-\in \Fc_I$.

\begin{prop} \label{prop independence choice} We have $\hat \uH_\Cc=\hat \uH_I$ for all
$\Cc\in \Fc_I$. Moreover, $\hat\uH_I$ does not depend on the choice of the smooth complete  fan $\Fc$ defining
the smooth toroidal compactification $\hat \uZ=\uZ(\Fc)$ of $\uZ$. In other words, for every $I\subset S$
 the $\uG$-variety
 $\hat \uZ_I=\uG/\uH_I$ is up to $\uG$-isomorphism canonically attached to the $\uG$-variety $\uZ=\uG/\uH$.
 \end{prop}

\begin{proof} We prove the first assertion by induction on $n=\# S$. We start with $n=0$, the case of horospherical varieties, see \cite[Sect. 8]{KK}.  In this situation $\uA$ normalizes $\uH$  and moreover $\uH = (\uH \cap \uL)  \uU^{\rm opp}$ with $\uU^{\rm opp}$ the opposite  of $\uU$. In particular, $\uAZ=\uA/\uA_{\uH}$ acts naturally on the right of
$\uZ=\uG/\uH$. By the construction of the toroidal compactification as the unique minimal $\uG$-extension of
$\uZ_0(\Fc) = [\uQ/ \uQ_{\uH}] \times_{\uAZ} \uAZ(\Fc)$  we obtain that
$$\uZ(\Fc)= \uG/\uH\times_{\uA_Z} \uA_Z(\Fc)$$
and hence  $\hat \uH_\Cc= \uH \uA$ for all $\Cc\in \Fc_S$.
\par Let now $n>0$ and $I\subset S$. We first treat the case for $I=S$. Then $\af_S=\af_{Z,E}$
and we note for all $\Cc\in \Fc_S$ the natural isomorphism
$$ \uAZ\cdot \hat z_\Cc \simeq \uAZ/\uA_{Z,E}\, .$$
Hence we obtain that
$$\uQ\cdot \hat z_\Cc= [\uQ/\uQ_{\uH}]\times_{\uAZ} [\uAZ\cdot \hat z_\Cc]\simeq \uQ/ (\uQ\cap \uH)\uA_{Z,E}\, .$$
This means for the $\uG$-extension $\hat\uZ_\Cc$ of $\uZ_0(\Cc)=\uQ\cdot \hat z_\Cc$
$$\hat \uZ_\Cc\simeq \uG/\uH\uA_{Z,E}\, , $$
i.e. $\hat \uH_\Cc=\uH \uA_{Z,E}$.
\par Suppose now that $I\subsetneq S$ and let
$\Cc, \Cc'\in \Fc_I$. We connect now $\Cc$ and $\Cc'$ in $\af_I^-$ face to face, i.e.
we find $\Cc_1, \dots \Cc_m\in \Fc_I$ such that
 $I(\Cc\cap \Cc_1)=I$, $I(\Cc_i \cap \Cc_{i+1}) = I$ for $1\leq i\leq m-1$ and
 $I(\Cc'\cap \Cc_m)=I$. Hence we may assume that $I(\Cc\cap \Cc')=I$. Set $\Cc_0:=\Cc\cap \Cc'$.
 Set
 $$\Fc(\Cc_0):=\{  \Cc\in \Fc\mid \af_{\Cc_0}\subset \af_\Cc\}$$
 and note that $\Cc, \Cc' \in \Fc(\Cc_0)$.  Set $\uZ_0:= \uG\cdot \hat z_{\Cc_0}\simeq  \uG/ \uH_{\Cc_0}$ and note that
 $\af_{Z_0}=\af_Z/ \af_{\Cc_0}$.  Moreover, $\Fc_0:=\Fc(\Cc_0)/ \af_{\Cc_0}$ is a complete smooth fan for
 $\uZ_0$ featuring $\uZ_0(\Fc_0)\subset \uZ(\Fc)$ as the Zariski closure of $\uZ_0$ in $\uZ(\Fc)$. Now $S_0=S(\uZ_0)=I\subsetneq S$
 and we obtain by induction that
 $\hat \uH_\Cc=\hat\uH_{\Cc'}$.
 \par Finally we note that if $\Fc_1$ and $\Fc_2$ are smooth fans, then there exists a smooth fan $\Fc_3$
 containing both $\Fc_1$ and $\Fc_2$, i.e. $\uZ(\Fc_1), \uZ(\Fc_2)\subset \uZ(\Fc_3)$. This completes the proof
 of the proposition.
   \end{proof}

For the purpose of this paper our interest is not so much with $\hat \uZ_I$
but with the real $G$-orbit  $\hat Z_I=G \cdot \hat z_{0,I}\simeq G/ \hat H_I$. Note that $\hat Z_I\subset \hat Z$.

For $I\subset S$ we denote by
$\uA_I$ the subtorus of $\uAZ$ corresponding to $\af_I\subset \af_Z$. For our fixed $j =j_I\in J_I$ with regard to $\cf_I^-$ we now set
$\psi_i^I:= \psi_{ji}$ for $1\leq i \leq r$.  Let $k:=r-|I|$. We may order the basis $(\psi_i^I)_{1\leq i\leq r}$ then in such
a way that $\Q[I]=\Q[\psi^I_{k+1}, \ldots, \psi^I_r]$ and then

$$\cf_I^{--}=\{ X\in\af_I^{-}\mid  (\forall i\leq k)\ \psi_i^I(X)<0\}\, .$$
With the basis $(\psi^I_i)_i$ we identify
$\uAZ $ with $(\C^\times)^r$ via

\begin{equation}\label{AZ-iso1} \uAZ\to (\C^\times)^r, \ \ a\mapsto (a^{\psi^I_i})_{1\leq i\leq r}\, .\end{equation}
In these coordinates $A_I$ corresponds to the subgroup $(\C^\times)^{r-k}\simeq \1 \times (\C^\times)^{r-k}
\subset (\C^\times)^r$.

Let us denote by $({\bf e}_i^I)_{1\leq i\leq r}\subset \af_Z$ the basis dual to  $(\psi_i^I)_{1\leq i\leq r}$.
We define the $\uAZ(\R)$-modules
$V_I:=\bigoplus_{i \leq k}  \R {\bf e}_i^I\simeq \af_I$ and  $V_I^\perp :=\bigoplus_{i>k}  \R {\bf e}_i^I$, which are both diagonal
with respect to the fixed basis $(\psi_i^I)_{1\leq i\leq r}$ of $\Xi_Z$.
Via the coordinates of \eqref{AZ-iso1} we view $\uAZ$ as open subset of $V_\C=\C^r=\uAZ(\Fc_{j_I})$
and obtain in particular that
\begin{equation} \label{LST-coo} \uZ_0(\Fc_{j_I})= \uU \times  [\uM/ \uMH\times_{F_M}  V_\C]\end{equation}
where we view $F_M=\hat M_H/M_H$ as a  subgroup of $\{-1, 1\}^r$ acting on $V_\C$ by sign changes in the coordinates.
Set $V=\R^r$.
\begin{lemma}  The real points of $ \uZ_0(\Fc_{j_I})= \uU \times  [\uM/ \uMH\times_{F_M}  V_\C]$ are given by
\begin{equation}\label{inf slice} \uZ_0(\Fc_{j_I})(\R)= U \times  [M/M_H \times_{F_M}  V]\, .\end{equation}
\end{lemma}

\begin{proof}  Let $x=(u,[m\uM_\uH, v])\in \uZ_0(\Fc_{j_I})$ where $u\in \uU$, $m\in\uM$ and $v\in V_\C$.  Then $x$ is real if and only if $\bar x = x$, that is
$$ (\bar u, [\bar m \uM_\uH, \bar v])= (u,[m\uM_\uH, v])$$
and in particular $u=\bar u$.  Moreover, as $F_M$ has representatives in $\hat M_H$, we obtain that
$\bar m \uM_\uH \in m\hat M_H\uM_{\uH}$. Now it follows from Lemma \ref{Example exact} that the polar map

$$ M \times_{M_H}  \mf_H^\perp \to \uM/ \uM_\uH, \ \ [g,X]\mapsto  g\exp(iX) \uM_\uH$$
is a diffeomorphism.  Hence if $y=m\uM_\uH=[g,X]$ is such that $\bar y \in m\hat M_H  \uM_H$ we obtain
$\bar y=[g,-X]= [ g\hat m^{-1} ,  \Ad(\hat m)X]$ for some $\hat m \in \hat M_H$.  But this gives $\hat m \in M_H$ and thus
$\bar y = y$, i.e. $X=0$.  Therefore $y=m\uM_\uH=[g,0]$ and we may choose $m=g\in M$.
This yields in turn that
$\bar v = v$ which concludes the proof of the lemma.
\end{proof}

\par Let ${\bf e}_I:=\sum_{j=1}^k {\bf e}_j^I\in V_I$.
Set $F_{M,I}:=F_M\cap \uA_I$ and note that $F_{M,I}$ is the $F_M$-stabilizer of
${\bf e}_I\in V_I$. Further put  $F_M^I:=F_M/ F_{M,I} $.
Denote by $V_I^{\perp, \times}\subset V_I$ the subset with all coordinates non-zero and observe that
$$ V_{I,\C}^{\perp, \times} =\uAZ\cdot {\bf e}_I\simeq \uAZ/ \uA_I
\, .$$
Then we obtain from \eqref{LST-coo} and \eqref{inf slice} the isomorphisms

\begin{eqnarray} \label{LST-Ic}   \uP \cdot \hat z_{0,I}&\simeq &\uU \times   \big[[\uM/ \uM_\uH F_{M,I}] \times_{F_M^I}
V_{I,\C}^{\perp, \times}\big ]\\
\notag &\simeq & \uU \times \big[[\uM/ \uM_\uH F_{M,I}] \times_{F_M^I}
[\uAZ/ \uA_I]\big] \end{eqnarray}
and
\begin{eqnarray} \label{LST-I}   [\uP \cdot \hat z_{0,I}](\R)&\simeq &U \times   \big[[M/ M_H F_{M,I}] \times_{F_M^I}
V_I^{\perp, \times}\big ]\\
\notag &\simeq & U \times \big[[M/ M_H F_{M,I}] \times_{F_M^I}
[\uAZ(\R)/ \uA_I(\R)]\big] \end{eqnarray}
which are given in coordinates as
$$ (u,[mM_HF_{M,I},v])=umv\cdot \hat z_{0,I}\, .$$

\subsection{ Relatively open $P$-orbits in $\hat Z$}\label{subsection rel open P-orbits}

The structure of the finite set of  $G$-orbits in $\hat Z$ is in general complicated
and the $G$-orbits through the boundary points $\hat z_{0,I}\subset \hat Z$ do typically
not give all $G$-orbits in $\hat Z$ (see Example \ref{ex SL3} below).

\par In general, let us call a $P$-orbit $P\cdot \hat z\subset \hat Z$  {\it relatively open} provided
$P\cdot \hat z$ is open in the $G$-orbit $G\cdot \hat z$.  The goal of this subsection is
to describe the set of all relatively open $P$-orbits in $\hat Z$, denoted by
$(P\bs \hat Z)_{\rm rel-op}$ in the sequel.

\par Recall from the end of Subsection \ref{subsection LST} the set $\Wc\subset G$ which parametrizes
$(P\bs Z)_{\rm open}$. In addition we remind that elements
$w\in \Wc$ have a representation as $w=\tilde th$ with $\tilde t \in T_Z=\exp(i\af_H^\perp)\subset \uA$ and  $h\in \uH$ such that $t:=\tilde t \cdot z_0\in F$ where $F=F_\R\cap Z$ and $F_\R$ the finite group of $2$-torsion points of
$\uAZ(\R)\subset \uZ(\R)$ (see Subsection \ref{subsection LST} for the notation).

For $w\in \Wc$ we now define the shifted base points:
$$z_w:=w\cdot z_0= \tilde t \cdot z_0 =t\in F\subset Z\, .$$
Likewise for $I\subset S$ and $X\in \cf_I^{--}$ we define in analogy to \eqref{def limit z0I}

$$\hat z_{w,I}:= \lim_{s\to \infty} \exp(sX)\cdot  z_w =  \tilde t \cdot \hat z_{0,I}$$
and note that the second equality (immediate from the definitions) implies that
$\hat z_{w,I}$ is independent of the choice of $X\in \cf_I^{--}$.  As $\tilde t  \cdot \hat z_{0,I}$
is  independent of the chosen lift $\tilde t$ of $t$  we can define for $t\in F_\R$
$$t \cdot \hat z_{0,I}:= \tilde t \cdot \hat z_{0,I}\, .$$
Since the limit defining $\hat z_{w,I}$ exists and $z_w\in Z$ we infer that
$\hat z_{w,I}\in \hat Z$.  Moreover, as $ \hat z_{w,I} \in F_\R \cdot \hat z_{0,I}$
with the notation defined above, we infer from the local
structure theorem as recorded in \eqref{LST-I} that $P\cdot \hat z_{w,I}$ is open in
$G\cdot \hat z_{w,I}$.  With that we obtain in fact all relatively open $P$-orbits in the wonderful situation:

\begin{lemma} \label{lemma rel open}Suppose that $\hat \uZ$ is wonderful.
Then the set of relatively open $P$-orbits in $\hat Z$ is given by
$$(P\bs \hat Z)_{\rm rel-op}=\{ P \cdot \hat z_{w,I}\mid w\in \Wc, I\subset S\}\, .$$
\end{lemma}

\begin{proof}   The inclusion $\supset$ was already seen above.
In the wonderful situation the $\uG$-orbits in $\hat \uZ$ are
precisely the $\uG\cdot \hat z_{0,I}\simeq \uG/\hat \uH_I$  for $I\subset S$ and accordingly every
relatively open $P$-orbit in $\hat \uZ(\R)$ lies in some
$[\uP \cdot \hat z_{0,I}](\R) $.  Hence any relatively open $P$-orbit in
$\hat Z$ is of the form $P t_1 \cdot \hat z_{0,I}$
for some $t_1\in F_\R$ by  \eqref{LST-I} and \eqref{AF}.
Since $\hat Z$ is $G$-invariant, and in particular $P$-invariant, it follows that
$t_1 \cdot \hat z_{0,I}\in \hat Z$. Further  the local structure theorem \eqref{LST-I} implies that
$t_1 \cdot \hat z_{0,I}\in \partial Z$ is approached by a curve in $Z$ of the form  $\exp(sX) t_2 \in Z$ for some
$t_2\in F$ and $X\in \af_I^{--}=\cf_I^{--}$, for $s\to \infty $.  In other words
$t_1 \cdot \hat z_{0,I} = \lim_{s\to \infty} \exp(sX) t_2 = t_2 \cdot \hat z_{0,I}$.
With  Lemma  \ref{lemmaW1} this concludes the proof.
\end{proof}

\begin{rmk}\label{remark rel open} (a) In the wonderful case we have
a stratification $\hat \uZ(\R)= \coprod_{I\subset S} \hat \uZ_I(\R)$ of $\hat\uZ(\R)$ in real spherical
$G$-manifolds  with $P\cdot \hat z_{w,I}\subset \hat \uZ_I(\R)$ for each $w\in \Wc$.  In particular if $I\neq J\subset S$ we have
$P\cdot \hat z_{w,I}\neq P\cdot \hat z_{w',J}$ for all $w,w'\in \Wc$. However, for fixed
$I$ it can and will happen that $P\cdot \hat z_{w,I}=P\cdot \hat z_{w',I}$ for some $w\neq w'$. The extremal case
is $I=\emptyset$ where  $\hat z_\emptyset=\hat z_{w,\emptyset}$ does not depend on $w\in \Wc$ at all.

\par (b) In case $\hat \uZ$ is not wonderful, the assertion in Lemma \ref{lemma rel open}
needs to be modified as follows. For every cone $\Cc\in \Fc$ and $w\in \Wc$ let us define
\begin{equation}\label{zwC}
\hat z_{w,\Cc}:= \lim_{s \to \infty} \exp(sX)\cdot  z_w
\end{equation}
which does not depend on $X\in \operatorname{int} \Cc$.
Recall that the $\uG$-orbits in the toroidal
compactification $\hat \uZ$ are parametrized by $\Cc\in \Fc$ and explicitly given
by $\uG\cdot \hat z_\Cc$.  Then for each  $\Cc\in \Fc$
the relatively open $P$-orbits in $\partial Z$ contained in $\hat Z_\Cc= \uG \cdot \hat z_\Cc$ are given
by the $P\cdot \hat z_{w,\Cc}$ with $w\in \Wc$.
\par (c) As every open $G$-orbit in $\hat Z_\Cc$ is open and contains an open $P$-orbit, we deduce from (b) that
$$\hat Z_\Cc=\bigcup_{w\in \Wc} G \cdot \hat z_{w,\Cc}\,.$$

\end{rmk}

\section{Normal bundles to boundary orbits in a smooth compactification}

Let $\sX$ be a  manifold and $\sY\subset \sX$ be a submanifold.  We denote by $T\,\sX$ and $T\, \sY$ the associated
tangent bundles of $\sX$ and $\sY$.  The normal bundle of $\sY$ is then defined
to be
$$N_\sY:=  T\,\sX|_\sY/ T\,\sY\, .$$
Note that $N_\sY \to \sY$ is a vector bundle with fibers $(N_\sY)_y=T_y\sX / T_y\sY$.

\par We are mainly interested in the case where $\sX$ is a smooth $G$-manifold
for a Lie group $G$, and $\sY:=G\cdot y$ is a locally
closed orbit. In this case we have a natural action of the stabilizer $G_y$  on $(N_\sY)_y$ and

\begin{equation} \label{normal identifier}N_\sY=  G \times_{G_y}  (N_\sY)_y\end{equation}
reveals the $G$-structure of $N_\sY$.

\subsection{Normal bundles to boundary orbits}  After this interlude on normal bundles
we return
to our basic setting with $G$ a real reductive algebraic group, and
let $\sX:=\hat \uZ(\R)$ be a smooth $G$-equivariant
compactification of $Z$ as constructed in Section \ref{section compact}.

Fix $I\subset S$ and let $\sY:=\hat Z_I\subset \sX$  be a boundary orbit with base point $y:=\hat z_{0,I}$.
Recall the basis $(\psi_i^I)_{1\leq i\leq r}$  of $\Xi_Z$,  its dual basis
$({\bf e}_i^I)_i$ and $V_I:= \bigoplus_{i\leq k}\R {\bf e}_i^I$.
By means of the basis it is often convenient to identify $V_I$ with $\R^k$
where $k=r-|I|$.
Define $V_I^\times:= \bigoplus_{i\leq k} \R^\times {\bf e}_i^I $
and $V_I^0\subset V_I^\times$ by

$$V_I^0:=  \bigoplus_{i\leq k} \R^+ {\bf e}_i^I  \simeq (\R^+)^k\, .$$
Set $V:=V_I \oplus V_I^{\perp}$ and recall ${\bf e}_I=\sum_{j=1}^k {\bf e}_j^I\in V_I^0$.

\par Let  $\U_M\subset M/ M_H$ be an open neighborhood of the base point
$M_H\in M/ M_H$ such that $\U_M \cap \U_M\cdot  x=\emptyset$ for $x\in F_M$, $x\neq 1$.

Recall that $V_I^{\perp,\times} =\uAZ(\R)/\uA_I(\R)$.
According to (\ref{LST-I}) the mapping

$$ \Psi_1: U \times \U_ M \times V_I^{\perp, \times}
\to [\uP\cdot \hat z_{0,I}](\R)
=U\times\big[[ M/ M_HF_{M,I}]\times_{F_M^I}
V_I^{\perp, \times}
\big]$$
given by
$$(u, mM_H, v)\mapsto (u,  [mM_H F_{M,I}, v])$$
is a diffeomorphism onto an open subset of  $[\uP\cdot \hat z_{0,I}](\R)$ and hence also of
$\hat \uZ_I(\R)$. Set
$$\V:= \Psi_1^{-1}(\sY) \, .$$
Thus we obtain two diffeomorphisms onto their images

$$\Psi_0:  \V \to \sY=\hat Z_I, \ \ (u,mM_H, a\uA_I(\R))\mapsto uma\cdot y$$
and
\begin{equation*}  \Psi:  \V\times V_I \to   U \times [ M/M_H \times_{F_M}  V] \subset \sX=\hat \uZ(\R)\, ,\end{equation*}
the latter one being given by
$$(u,mM_H, a\uA_I(\R), v_I)\to  (u, [m, a\cdot {\bf e}_I+ v_I]).$$

\par Set $F_I:=F\cap \uA_I(\R)$  and note that
$F_I$ identifies with a subset of
$ \{ -1, 1\}^k$ upon identification of $\uA_I(\R)\simeq (\R^\times)^k$.
From the definition of $\Psi$ we
then get
\begin{equation}  \label{points to Z} \Psi^{-1}(Z)=  \V \times F_I \cdot V_I^0\, .\end{equation}
It is worth to note that
\begin{equation} \label{normalization normal bundle}   \Psi(y,{\bf e}_I)= z_0\, . \end{equation}

\par With $\Psi$ being diffeomorphic we record the following property of transversality

\begin{equation} \label{tangent decomp} d\Psi(x,0) (0 \times V_I )   \oplus  T_x\sY = T_x\sX\qquad (x\in \V)\, .\end{equation}

In the sequel we use \eqref{tangent decomp} to identify the spaces $V_I\simeq (N_\sY)_y$  for $y=\hat z_{0,I}$.
On $V_I=(N_\sY)_y$ there is a natural
linear action of $G_y=\hat H_I$, the isotropy representation, which we call
$$\rho: \hat H_I \to \GL(V_I)\, .$$

The representation $\rho$ is algebraic, i.e. it originates from the complex isotropy representation

$$\underline{\rho}:   \hat\uH_I\to \GL(V_{I,\C})\, .$$
We write $\uH_I=\ker \underline \rho$ and note  that $H_I=\uH_I(\R)$ is
given by  $H_I=\ker \rho$.  Observe that $\uH_I\triangleleft\hat \uH_I$
and $H_I \triangleleft \hat H_I$
are closed normal subgroups.

\begin{theorem} \label{thm normal}The following assertions hold:
\begin{enumerate}
\item \label{1one1} The Lie algebra of $H_I$ is given by  $\hf_I$, as defined in  \eqref{eq-hi1}.
\item \label{2two2}$\hat \uH_I/  \uH_I  \simeq \uA_I$.
\end{enumerate}
\end{theorem}

The proof of Theorem \ref{thm normal} will be prepared by several intermediate steps.
The key is the following lemma and the techniques contained in its proof.

\begin{lemma} \label{limit normal}The Lie algebra of $H_I$ contains $\hf_I$, as defined by  \eqref{eq-hi1}.
\end{lemma}

\begin{proof}  Let $Y\in \hf_I$, then $h_I:=\exp(Y)\in \hat H_I$
as explained above Proposition \ref{prop independence choice}.
We claim that $\rho(h_I)=\1$.

For all $X\in \cf_I^{--} \cap \Xi_Z^\vee$  we consider the curve
$$ \gamma_X: [0, 1] \to  \sX=\hat \uZ(\R), \ \ s\mapsto \exp(-(\log s) X)\cdot z_0 ,$$
which connects $\hat z_{0,I}$ to $z_0$.
Note, that in coordinates of \eqref{AZ-iso1} we have $\uAZ(\F_{j_I})(\R)\simeq \R^k$ (with $j_I\in J_I$ the selected
element for $\cf_I^-$), and
$$\gamma_X(s) =  (s^{m_1}, \ldots, s^{m_k}) \in V_I$$
for some $m_i\in \N$.  Notice that all  tuples of $m_i\in\N$ occur for some $X$.
Hence $\gamma_X$ is differentiable with $\gamma_X(0)= y=\hat  z_{0,I}$
and $\gamma_X'(0)=(\delta_1, \ldots, \delta_k)$ with $\delta_i=1$ if $m_i=1$ and $\delta_i=0$ otherwise.

\par Since $\rho(h_I) (\gamma_X'(0))  = \frac{d}{ds}\big|_{s=0} h_I \gamma_X(s)$,  the lemma will follow provided we can show that
$\frac{d}{ds}\big|_{s=0} h_I \gamma_X(s)=\gamma_X'(0)$ for all
$X$ as above.   Now for $h_I\in L\cap H$ this is clear and thus we
may assume that $Y$ is of the form  (see \eqref{eq-hi2})
$$Y= \sum_{\alpha\in \Sigma(\af, \uf)} (X_{-\alpha}  + \sum_{\alpha +\beta \in \N_0[I]}
X_{\alpha, \beta})\,.$$
Set now for $s>0$
$$Y_s:=  \sum_{\alpha\in \Sigma(\af, \uf)} (X_{-\alpha}  + \sum_{\beta} e^{- (\log s) (\alpha +\beta)(X)}
X_{\alpha, \beta})\in \Ad(\gamma_X(s)) \hf\, .$$

Note that $Y_s\to Y$ for $s\to 0$. Likewise we set $h_{I,s}:=\exp(Y_s)$ and note $h_{I,s} \to h_I$.
Now we use that  $\M\subset \Xi_Z$ in order to conclude that $h_{I,s}$ is right differentiable
at $s=0$.  The Leibniz-rule yields

$$ \frac{d}{ds}\Big|_{s=0} h_{I,s} \gamma_X(s)= \frac{d}{ds}\Big|_{s=0} h_I \gamma_X(s) +
\underbrace{ \frac{d}{ds}\Big|_{s=0} h_{I,s} y}_{\in T_y \sY}$$
and thus we get
$$\frac{d}{ds}\Big|_{s=0} h_I \gamma_X(s) = \sP \left(\frac{d}{ds}\Big|_{s=0} h_{I,s} \gamma_X(s)\right)$$
with $\sP$ the projection $T_y \sX\to V_I$  along $T_y\sY$.
Now observe that
\begin{eqnarray*} h_{I,s} \gamma_X(s) &=&h_{I,s} \exp(-(\log s) X) \cdot z_0\\
&=&   \exp(-(\log s) X)\underbrace{ \exp ((\log s) X)h_{I,s}
 \exp(-(\log s) X)}_{\in H}  \cdot z_0\\
 & =& \gamma_X(s)\end{eqnarray*}
and the lemma follows.
\end{proof}
\subsubsection{Normal curves and the proof of Theorem \ref{thm normal}}
Recall the base  points $\hat z_{0,I}\in \hat Z_I \subset \hat Z$. Now $\hat z_{0,I}\in \partial Z$ is a boundary point
of $Z$ provided that $I\subsetneq S$ or $\af_{Z,E}\neq 0$ -- in case $\af_{Z,E}=\{0\}$ we have $\hat Z_S= Z$
and $\hat z_S=z_0$ is not a boundary point.

\par The proof of Lemma \ref{limit normal} contains an important concept, namely smooth curves in $Z$
which approach the point $\hat z_{0,I}\in \hat Z_I$ in
normal direction.   Let $X_I \in \af_I^{--}$ correspond to $-{\bf e}_I\in V_I$ after the natural identification of $\af_I$
with $V_I$.  Then we saw that
the curve

$$\gamma_I: [0,1]\to \hat \uZ(\R), \ \ s\mapsto  \exp(-(\log s) X_I)\cdot z_0$$
is smooth with the following properties

\begin{itemize}
\item $\gamma_I\big((0,1]\big) \subset Z$,
\item $\gamma_I(0)=\hat z_{0,I}$,
\item $\gamma_I'(0)= {\bf e}_I \in V_I \subset T_{\hat z_{0,I}} \hat \uZ(\R)$.
\end{itemize}

More generally, let $v\in V_{I,\C}^\times $.  Then there exists a unique $a(v)\in \uA_I$ such that
$v= a(v)\cdot {\bf e}_I$.  If we consider now $\uA_I\subset \uZ$, then
we obtain a smooth curve

$$\gamma_v:  [0,1]\to \hat \uZ, \ \ s\mapsto  \exp(-(\log s) X_I)\cdot  a(v) $$
such that
\begin{itemize}
\item $\gamma_v\big((0,1]\big) \subset \uZ$,
\item $\gamma_v(0)=\hat z_{0,I}$,
\item $\gamma_v'(0)= v  \in V_{I,\C} \subset T_{\hat z_{0,I}} \hat \uZ$.
\end{itemize}

For $g\in \uG$ we now shift
$\gamma_v$ by $g$, i.e. we
set
$$\gamma_{g,v}(s):= g\cdot \gamma_v(s) \in \hat \uZ \qquad s\in [0,\e)\, .$$
Notice that $\gamma_{g,v}(0)= g\cdot \hat z_{0,I}$ and $\gamma'_{g,v}(0) = dL_g(\hat z_{0,I}) v
\in T_{g\cdot \hat z_{0,I}} \hat \uZ$,
where $dL_g$ denotes the differential of the displacement $L_g(z)=g\cdot z$.
If $v={\bf e}_I$ we simply set $\gamma_{g,I}:=\gamma_{g,v}$.
Specifically we are interested when $g\in  \hat \uH_I$ so that $\gamma_{g, v}(0)= \hat z_{0,I}$.

\par Next  we recall the decomposition of the complex tangent spaces
$$ T_{\hat z_{0,I}} \hat \uZ= T_{\hat z_{0,I}} \hat \uZ_I \oplus V_{I,\C}$$
which identifies $V_{I,\C}$ with the complex normal space to
the boundary orbit $\hat \uZ_I =\uG \cdot \hat z_{0,I}$ at  the point $\hat z_{0,I}$.
We denote by
$$\sP: T_{\hat z_{0,I}} \hat \uZ
\to V_{I,\C}, \ \ u\mapsto u_{\rm n}$$
the projection of a tangent vector $u\in  T_{\hat z_{0,I}} \hat \uZ$ to its normal part
$u_{\rm n}$.
With this notation we then obtain from the definition of $\uH_I=\ker \rho$ that

\begin{equation} \label{character HI-0} \uH_I=\{ g \in \hat\uH_I\mid (\forall v \in V_I^\times) \  [\gamma_{g,v}'(0)]_{\rm n}=v\} \,.\end{equation}

\begin{proof}[Proof of Theorem \ref{thm normal}]
First recall that $\hat \hf_I = \hf_I +\af_I$ from \eqref{Lie hat HI}.  As further  $\uA_{\uH}\subset \ker\underline{\rho}$
we see that $\underline{\rho}$ induces a representation of $\uA_I$ on $V_{I,\C}$ which is given by the faithful standard representation
$\underline{\rho}(a)(v)= a\cdot v$.   In fact,  if we denote by
$\tilde a\in \uA$ any lift of $a\in \uAZ$ for the projection $\pi:\uA\to \uA_Z$, then for $a\in\uA_I$  we have
$\rho(a)(v)= \tilde a  \cdot \gamma_v'(0)= a\cdot v$.
Notice that $\underline{\rho}(\uA_I) \simeq  \diag(k, \C^\times)$ within our identification $V_{I,\C}\simeq \C^k$.
It follows in particular that  $\af_I \cap \Lie(H_I)=\{0\}$ and thus $\hf_I=\Lie(H_I)$ by Lemma \ref{limit normal}.
This shows \eqref{1one1}.

\par Moving on to \eqref{2two2} we first observe that $\uP \hat\uH = \uP \uH$
for any spherical subgroup $\uH$.
In fact,  since $\hat\uH$ normalizes $\uH$ it follows that
$\uP\hat\uH$ is a union of open right $\uH$-orbits. Since $\uG$ is connected the identity
$\uP \hat\uH = \uP \uH$ follows. Equivalently,
\begin{equation} \label{hat uH} \hat \uH =(\uP\cap\hat\uH)\uH \,.\end{equation}

We apply this to the spherical subgroup $\uH_I$.  Now if
$p \in \uP\cap\hat\uH_I$ then
\begin{equation} \label{am fix} p\cdot \hat z_{0,I}= \hat z_{0,I}\, .\end{equation}
Let $\widetilde \uA_I:= \pi^{-1} (\uA_I)$.
Then \eqref{am fix} and the local structure theorem in the form of \eqref{LST-Ic}
implies $p \in \uM_{\uH} \widetilde  \uA_I\subset \uH_I \uA_I$, and hence
$\hat\uH_I=\uH_I\uA_I$ by \eqref{hat uH}.\end{proof}

In particular it follows from Theorem \ref{thm normal} that $\underline{\rho} (\hat \uH_I)\simeq \diag(k, \C^\times)$
and thus for $g \in \hat \uH_I$ that $\underline{\rho}(g)=\1$ if and only if
$\underline{\rho}(g)(v) =v$ for some $v\in V_I^\times$. Thus we obtain the following strengthening of
\eqref{character HI-0} to

\begin{eqnarray}\label{character HI} \uH_I&=&\{ g \in \hat\uH_I\mid [\gamma_{g,I}'(0)]_{\rm n}={\bf e}_I\}\, \\
\notag &=& \{ g \in \hat\uH_I\mid [\gamma_{g,v}'(0)]_{\rm n}=v\} \quad(v\in V_I^\times)\, . \end{eqnarray}

\subsection{The part of the normal bundle which points to $Z$}\label{nb points to Z}

We denote by $A_I$ the identity component of $\uA_I(\R)$.

According to Theorem \ref{thm normal} there is the exact sequence
\begin{equation} \label{exact1} \1 \to \uH_I \to \hat \uH_I\to \uA_I\to \1\, .\end{equation}
In (\ref{exact1}) we take real points, which is
only left exact, and obtain

\begin{equation} \label{exact2} \1 \to H_I \to \hat H_I\to
\uA_I(\R)\, .\end{equation}

The image of the last arrow in \eqref{exact2} is an open subgroup
since taking real points is exact on the level of Lie algebras.
We denote this open subgroup by $A(I)$ and record the exact sequence

\begin{equation} \label{exact3} \1 \to H_I \to \hat H_I\to
A(I)\to \1\, .\end{equation}
In particular,
\begin{equation}  A(I) =  A_I F(I),\end{equation}
where $F(I)<\{-1, 1\}^k \subset \uA_I(\R)$ is a subgroup of the $2$-torsion group $\{-1, 1\}^k$ of $\uA_I(\R)\simeq
(\R^\times)^k $.

\begin{rmk} (a) The non-compact torus $A(I)\simeq \hat H_I / H_I$ acts naturally on $Z_I=G/H_I$ from the right
and thus commutes with the left $G$-action on $Z_I$.
\par (b)  Since $\uA_I \subset \uAZ$  we obtain that $A(I)$ is naturally a subgroup of $\uAZ(\R)$.
In particular we stress
that it is not possible in general to realize $A(I)$ as a subgroup of $A=\uA(\R)\subset G$.
\end{rmk}

We return to the normal bundle of the boundary orbit $\sY=\hat Z_I$:

$$N_\sY= G \times_{G_y}V_I =G\times_{\hat H_I} V_I\, .$$
From \eqref{exact3} we obtain that
\begin{equation}\label{im rho} \rho(\hat H_I){\bf e}_I =A(I)\cdot {\bf e}_I=F(I) \cdot V_I^0\, .\end{equation}

Recall the set $F_I= F\cap \uA_I(\R)\subset \{-1, 1\}^k$ with $F(I)\subset F_I$.  We define an $A(I)$-stable open cone in $V_I$ by

\begin{equation} \label{chamber union}   V_{Z,I} =  F_I \cdot V_I^0= F_I\cdot(\R^+)^k\, ,\end{equation}
and we define the  cone-bundle

\begin{equation} \label{pointed normal bundle} N_\sY^Z:=  G \times_{G_y}  V_{Z,I}\end{equation}
as part of the normal bundle $N_\sY$ which {\it points to} $Z$.
To explain the term "points to $Z$" we recall the curves $\gamma_v$ and note that
$\gamma_v\big((0,\e)\big)\subset Z$ if and only if $a(v)\in A(I)$, that is $a(v)\cdot {\bf e}_I =v\in F_I V_I^0$.

\par Observe that the coset space $\sF_I:=F(I) \bs F_I $  identifies with the $N_\sY^Z/G$.
For every $\st=F(I)t\in \sF_I$  with $t\in F_I$  we now denote by

$$N_\sY^{Z,\st}=  G \times_{G_y}    [F(I)t \cdot V_I^0] $$
and note that
$$N_\sY^Z=\coprod_{\st\in \sF_I}  N_\sY^{Z,\st}$$
is the disjoint decomposition into $G$-orbits.

\par Presently we do not have a good understanding of $F(I)$ and the coset space
$\sF_I=F(I)\bs F_I$, except when $G$ is complex, where $F_I=F(I)$ for all $I\subset S$.
Here are two further instructive examples:

\begin{ex}\label{ex=SL2}  (cf. \cite[Ex. 14.6]{KK}) (a) Let $G=\Sl(2,\R)$ and $H=\SO(1,1)$.  We identify $Z=G/H$ with the one sheeted hyperboloid

$$Z=\{ (x_1, x_2, x_3)\in \R^3\mid  x_1^2 - x_2^2 -x_3^2 = -1\}\, .$$
We note that $Z=\uZ(\R)$ and we embed $Z$ into the projective space $\Pb(\R^4)$.  The closure of $Z$ in
projective space is given by

$$\hat Z=\{ [x_1, x_2, x_3, x_4] \in \Pb(\R^4) \mid  x_1^2 + x_4^2 = x_2^2 + x_3^2\}\simeq \Sb^1 \times \Sb^1$$
and coincides with the wonderful compactification $\hat \uZ(\R)$.  In the identification
$\hat Z=\Sb^1\times \Sb^1$ from above, the unique closed $G$-orbit is given  by $\sY= \{ \1\} \times \Sb^1$ and
$$ \hat Z= Z \cup \sY\, .$$
In particular both directions of the normal bundle $N_\sY$  point to $Z$.  In our notation above this means
that $F_I=\{ -1, 1\}$, $F(I)=\{1\}$ and
$$N_\sY=N_\sY^Z= N_\sY^{Z, +1}\amalg N_\sY^{Z, -1}\, .$$
\par (b) The situation becomes different when we consider $G=\Sl(2,\R)$ with $H=\SO(2)$.  We identify $Z=G/H$ with the upper component of the two sheeted hyperboloid $\uZ(\R)$, in formulae:

$$Z=\{ (x_1, x_2, x_3)\in \R^3\mid  x_1^2 - x_2^2 -x_3^2 = 1, x_1>0\}\, ,$$
and
$$\uZ(\R)=\{ (x_1, x_2, x_3)\in \R^3\mid  x_1^2 - x_2^2 -x_3^2 = 1\}\, .$$
We emphasize that $\uZ(\R)$ has two connected components, one of them being $Z$.
As before we view $\uZ(\R)$ in the projective space $\Pb(\R^4)$ and obtain the wonderful compactification
$\hat \uZ(\R)$ as the closure

$$\hat \uZ(\R)=\{ [x_1, x_2, x_3, x_4] \in \Pb(\R^4) \mid  x_1^2  = x_2^2 + x_3^2+x_4^2\}\simeq \Sb^2\, . $$
The unique closed orbit $\sY=\Sb^1$ is identified with the great circle $\Sb^1\subset\Sb^2$ which divides
$\hat \uZ(\R)$ into the two open  $G$-orbits.
In particular, only one direction of the normal bundle $N_\sY$  points to $Z$.   We obtain that
$F_I=F(I)=\{1\}$ with
$$N_\sY \supsetneq N_\sY^Z\, .$$
By this we end Example \ref{ex=SL2}.
\end{ex}

Define

$$Z_I:=G/H_I$$
and write $z_{0,I}=H_I$ for its standard base point.

Let $\st =\sF_I$ and fix with $t\in F_I$ a representative so that $\st = F(I)t$.  We then claim that

$$ G/H_I\to N_\sY^{Z,\st}, \ \ gH_I\mapsto [g, t\cdot {\bf e}_I]$$
defines a $G$-equivariant diffeomorphism
for each $\st\in \sF_I$.  In fact, with $A(I)\simeq F(I)V_I^0$ via $a\mapsto a\cdot {\bf e}_I$, this follows
from:

\begin{equation} \label{N-ident} N_\sY^{Z,\st}\simeq G/H_I  \times_{A(I)}  (A(I) \cdot t \cdot {\bf e}_I)
\simeq G/H_I \times_{A_I}  V_I^0\simeq G/H_I\, .\end{equation}

 \subsection{Speed of convergence} Next we wish to describe
 a quantitative version of  the fact that $\uH_I$ asymptotically preserves normal limits,
 i.e.~of \eqref{character HI}.
 For that recall the curves $\gamma_{g,v}$.

\begin{lemma}\label{speed lemma}  Let $g\in \hat \uH_I$ and $v\in V_I^\times$. Then there exists
a smooth curve $[0,\e)\to \uP, \ s\mapsto p_s$ such that
$$ \gamma_{g,v}(s)= p_s\cdot \gamma_v(s)\qquad (s\in [0,\e)\,)$$
and:
\begin{enumerate}
\item \label{curve1}$p_0 \in \uA_I$.
\item\label{curve2}  If $g\in \uH_I$ then $p_0=\1$.
\item \label{curve3} If $g\in \hat H_I$ we can assume that $p_s\in P$.
\end{enumerate}
\end{lemma}

\begin{proof}  Note that $g\cdot \hat z_{0,I}=\hat z_{0,I}$ by assumption, and hence $\gamma_{g,v}(s)\to
\hat z_{0,I}$ for $s\to 0^+$ in a smooth fashion.
\par The local structure theorem gives us coordinates near $\hat z_{0,I}$, see
 \eqref{LST-Ic}. In particular, it implies that we can find a smooth curve $s\mapsto \tilde p_s\in \uP$ such that
$\tilde p_s\cdot \gamma_v(s) = \gamma_{g,v}(s)$.
Note that $\tilde p_0\cdot \hat z_{0,I} =\hat z_{0,I}$  and hence
$\tilde p_0 \in  \uA_I (\uP\cap \uH)$ by \eqref{LST-Ic}. With that
we obtain an element $p_H\in \uP\cap \uH$ such that $p_s:= \tilde p_s p_H$ satisfies \eqref{curve1}. Here
we used the fact that  $\gamma_v=\gamma_{p,v}$ for all $p\in \uP\cap \uH$.
\par
We move on to \eqref{curve2}.  For that we recall the decomposition of the tangent space
$$T_{\hat z_{0,I}} \hat \uZ= T_{\hat z_{0,I}} \hat \uZ_I \oplus V_{I,\C}$$
and the normal part $u_{\rm n}\in V_{I,\C}$ of a tangent vector $u\in  T_{\hat z_{0,I}} \hat \uZ$.
Now if $g\in \uH_I$, then by the definition of $\uH_I$ as the kernel of the isotropy representation, we obtain that $[\gamma_{g,v}'(0)]_{\rm n} = v$. On the other hand,
using the identity $\gamma_{g,v}(s) = p_s\cdot \gamma_v(s)$ we obtain
$[\gamma_{g,v}'(0)]_{\rm n}= p_0 \cdot v$.  As $p_0\in \uA_I$, this implies $p_0=\1$.
\par The last assertion \eqref{curve3} is proved using the real version of the argument {\text for \eqref{curve1}}.
\end{proof}

Let $d_\uG$ a left invariant Riemannian metric on $\uG$.  Then the quantitative version of
\eqref{character HI} reads as follows:

\begin{cor} \label{lemma H_I limit} Let $X_I\in \cf_I^{--}\subset \af_I^{--}$ correspond to $-{\bf e}_I\in V_I$
in the identification $V_I \simeq \af_I$. Set $a_t:=  \exp(tX_I)$ for $t\geq 0$.
Let $h_I \in \uH_I$.
Then there exist constants $C,\e,  t_0>0$  and for each $t\ge t_0$ an element $x_t \in \uP$ such that
$d_\uG(x_t, \1) \leq C e^{-\e t}$
and
\begin{equation} \label{normal-approx}  h_I a_t\cdot z_0 =  x_t a_t \cdot z_0\, .
\end{equation}
If further, $h_I\in H_I$, then we can choose $x_t\in P$.
\end{cor}
\begin{proof} Apply the lemma to $g=h_I$ and $v=e_I$.
Set $x_t:= p_{e^{-t}}$ and use that $p_0=\1$ and $s\mapsto p_s$ is differentiable
at $s=0^+$.
 \end{proof}

\subsection{The intersection of $\uH_I$ with $\uL$}

\par  For later reference we record the following  fact, which is more or less
immediate from \eqref{LST-Ic}. Since it is crucial for the paper
we include a detailed argument.

\begin{lemma}\label{equal L cap H}  For all $I\subset S$ one has
\begin{equation}  \label{same L cap H} \uL \cap \uH =\uL\cap \uH_I \, .\end{equation}
\end{lemma}
\begin{proof} First note that $\uL= \uM \uA \uL_{\rm n}$ and from $\uL_{\rm n} \subset \uH\cap \uH_I$
we obtain that $\uH\cap \uL= \uL_{\rm n}[(\uM\uA)\cap \uH]$ and likewise $\uH_I \cap \uL = \uL_{\rm n}[(\uM\uA)\cap \uH_I]$.
Hence it suffices to show that $\uH\cap (\uM\uA )= \uH_I \cap (\uM\uA)$.
\par We first show that $\uH\cap (\uM\uA) \subset  \uH_I \cap (\uM\uA)$. For that we recall the
isotropy representation $\rho$ which we view here as a representation of $\hat \uH_I$ so that
$\uH_I=\ker \rho$. Recall  the curves $\gamma_X$ from the proof of Lemma \ref{limit normal}. Now for
$g\in (\uM\uA)\cap H$ we have $g\gamma_X(s) = \gamma_X(s)$ and thus $g\gamma_X'(0)= \gamma_X'(0)$.
Hence $g\in \ker \rho= \uH_I$ and "$\subset$" is established.
\par  For the converse inclusion we first note that both $\uH\cap (\uM\uA )$ and $\uH_I \cap (\uM\uA)$
are elementary algebraic groups (see \cite{KK} or \cite[Appendix B]{DKS} for the notion "elementary").
 Together with $\lf \cap \hf = \lf\cap \hf_I$  (which we obtain from \eqref{eq-hi1} )
 we infer that $\uH\cap (\uM\uA )$ and $\uH_I \cap (\uM\uA)$
have the same Lie algebra, namely $[\mf_H + \af_H]_\C$.
Further as $\uM \uA$ is
an elementary group we obtain $\uH_I \cap (\uM\uA) = (\uM \cap \uH_I)  (\uA_{\uH_I})_0$, see \cite[Appendix B]{DKS}.

From $\af\cap \hf =\af \cap \hf_I$ again obtained from \eqref{eq-hi1} we derive
$(\uA_{\uH_I})_0=(\uA_{\uH})_0$. Hence   we only need to show that
$\uM\cap \uH_I \subset \uM\cap \uH$.  Let now $m \in \uM\cap \uH_I$. In particular $m\in \hat \uH_I$ fixes $\hat z_{0,I}$ and thus
we obtain from \eqref{LST-I} that $m \in \uM_\uH F_{M,I}$.  Hence we may assume that $m \in F_{M,I}$.
From $\rho(m)=\1$ we then obtain that  $m\in F_{M,I}\subset  \{-1, 1\}^k$ needs to have all coordinates to be $1$, i.e.
$m=\1$ and the proof is complete.
\end{proof}

\subsection{The structure of $\uZ_I(\R)$}\label{structure of Z_I}
From Theorem \ref{thm normal} and Proposition \ref{prop independence choice}  we obtain:

\begin{lemma}  For any $I\subset S$,
 the $\uG$-isomorphism class of the variety $\uZ_I$ is canonically attached to $\uZ$,  i.e.
 independent of the particular smooth toroidal  compactification $\hat \uZ=\uZ(\Fc)$ of $\uZ$.
 \end{lemma}

In particular, it follows that up to $G$-isomorphism
$\uZ_I(\R)$ is canonically attached to $\uZ(\R)$.
However, for $Z_I$ the situation is different.
We recall the shifted base points $z_w=w\cdot z_0$ and $\hat z_{w,I}$ from Subsection \ref{subsection rel open P-orbits}.
For $I\subset S$ we then define
the set of $G$-orbits

$$\sC_I:=\{  G\cdot \hat z_{w,I}\mid w\in \Wc\},$$
and note that different orbits in $\sC_I$ may not be isomorphic, see Example \ref{ex SL3} below.
In particular, the isomorphism class of $Z_I=G/H_I$ is not canonically attached to $Z$.
In this sense only the collection of $G$-spaces $\{ G/(H_w)_I \mid w\in \Wc\}$
 (where $(\hat{H_w})_I$ is the stabilizer of $\hat z_{w,I}$)
is canonically attached to the $G$-space
$Z=G/H$.

\begin{rmk}\label{rmk DI}  In case $\hat \uZ$ is wonderful the set
$\sC_I$ equals the set of $G$-orbits $\partial Z \cap \hat \uZ_I(\R)$.
This follows from Lemma \ref{lemma rel open}. The general case is a bit more complicated, see
Remark \ref{remark rel open} (b). Recall the boundary points $\hat z_{w,\Cc}$ from \eqref{zwC}.
Then
\begin{eqnarray*}\sD_I&:=&\{ G \cdot \hat z_{w,\Cc} \mid G\cdot \hat z_{w,\Cc}\subset \hat \uZ_I(\R), w\in \Wc,  \Cc\in \Fc\}\\
&=&\{G \cdot \hat z_{w,\Cc}\mid  w\in \Wc, \ \Cc\in \Fc_I\}\end{eqnarray*}
 yields all $G$-orbits in $\partial Z\cap \hat Z_I(\R)$.
\end{rmk}

For $\sc\in\sC_I$ we set
 $$\Wc_\sc:=\{w\in \Wc\mid G\cdot \hat z_{w,I}=\sc\}$$
and obtain the partition
\begin{equation} \label{party W} \Wc=\coprod_{\sc \in \sC_I}  \Wc_\sc\, .\end{equation}
Given $\sc\in\sC_I$ we choose a representative $w(\sc)\in \Wc_\sc$.
In case $\sc= G\cdot \hat z_{0,I}=\hat Z_I$ we make the request that
$w(\sc)=\1$.   We then define
$$H_{I,\sc}:=  (H_{w(\sc)})_I.$$
We will see in Lemma \ref{lemma HI comp} that the $G$-conjugacy class of $H_{I,\sc}$ is independent of the representative
$w(\sc)$ used for its definition.

\begin{ex} \label{ex SL3} Consider $\uZ= \Sl(3,\C)/\SO(3,\C)$ which is defined over $\R$.
We will use the identification
$$\uZ= \Sym(3\times 3, \C)_{\det =1}$$
with $\Sym$ denoting the symmetric matrices.  Hence
$$\uZ(\R) =  G/ K\amalg G/H$$
consists of two $G$-orbits with $K=\SO(3,\R)$ and $H=\SO(1,2)$, both real forms of $\uH = \SO(3,\C)$.
Our interest is here with $Z=G/H$. If we
identify $\uAZ$ with the diagonal matrices in $\uZ$, then $F_\R$,
the $2$-torsion group
of $\uAZ(\R)$, is given by

$$F_\R=\{  (1,1,1),  (1, -1, -1), (-1, 1, -1), (-1, -1, 1)\} = \{t_0, t_1, t_2, t_3\}$$
which in this case parametrize the open $P$-orbits in $\uZ(\R)$ -- we have $F_M=\{1\}$
in this example.  Notice that $t_0 \in G/K$ whereas $t_1, t_2, t_3\in Z=G/H$.
In particular $F=\{t_1, t_2, t_3\}$ is not a group. Let us denote by
$w_1, w_2, w_3\in G$ lifts of $t_i$ to $G$ so that $\Wc=\{ w_1, w_2, w_3\}$.

In this case the spherical roots comprise a system
of type $\sA_2$. With $I=\{\alpha_2\}$
we can take
$a_t = \diag(t^{-2}, t, t)$ for our ray.

\par  Our example $\uZ$ has a wonderful compactification which is given by the closure of the
image of its standard embedding into projective  space
$$ \uZ=\Sym(3\times 3, \C)_{\det =1}\to {\mathbb P}( \Sym(3\times 3, \C)\times  \Sym(3\times 3, \C))$$
$$X\mapsto  \C\cdot (X,X^{-1}). $$

\par  Note $H_{w_1}=H$ and an elementary
calculation in the above model for $\hat \uZ$ yields

$$ H_I = (H_{w_1})_I  =  S (\OO(1) \OO(2)) U_I\quad\text{and} \quad
(H_{w_2})_I  =  (H_{w_3})_I=S (\OO(1) \OO(1,1)) U_I$$
where
$$U_I = \begin{pmatrix}  1 & & \\ * & 1 & \\ * & & 1\end{pmatrix}\subset G\, .$$
In particular, we see $H_{I,{\mathsf 1}}:=H_I$ is not conjugate to
$H_{I,{\mathsf 2}}:= (H_{w_2})_I$.

\par We further note that $\hat \uZ_I(\R)= \partial Z\cap \hat \uZ_I(\R)$ consists of two $G$-orbits
$\hat Z_{I,{\mathsf 1}}= G/\hat H_{I,{\mathsf 1}}$ and $\hat Z_{I,{\mathsf 2}}= G/ \hat H_{I,{\mathsf 2}}$, and accordingly
$\sC_I \simeq \{{\mathsf 1} , {\mathsf 2}\}$ has two elements.
Note that $\hat Z_I(\R)$ is $G$-isomorphic to the projective space of the rank two real symmetric matrices.
Within this identification $\hat Z_{I,{\mathsf 1}}\subset \hat Z(\R)$ consists of the  rank two symmetric matrices (viewed projectively) with equal
signature  (i.e.  $0++$ or $0--$), and $\hat Z_{I,{\mathsf 2}}\subset \hat Z(\R)$ of the rank two symmetric matrices with signature $0+- $.  Finally note that
$\Wc_{\mathsf 1}=\{w_1\}$ and $\Wc_{\mathsf 2}=\{ w_2, w_3\}$.
\end{ex}

\section{Open $P$-orbits on $Z_I$ and $Z$}\label{subsection WI}

Recall the set $\Wc\subset G$  of representatives
for $W =(P\bs Z)_{\rm open}$.  Let  $W_I =(P\bs Z_I)_{\rm open}$, the set of open $P$-orbits in $Z_I$.
The objective of this section is to obtain a good set $\Wc_I$ of representatives for $W_I$
which results in a natural injective map ${\bf m}:  \Wc_I \to \Wc$  (or $W_I \to W$ if one wishes),
and thus matches each open $P$-orbit in $Z_I$  with a particular open $P$-orbit in $Z$. This map
is important for various constructions of the paper.
\par  In general the map ${\bf m}$ is not surjective and this originates from the fact that $Z_I$ is only one $G$-orbit
in $\uZ_I(\R)$ which points to $Z$.
We will show in this section that the $G$-orbits in $\uZ_I(\R)$  which point to $Z$
are given by

\begin{equation} \label{normal union} \widetilde Z_I :=\coprod_{\sc \in \sC_I} \coprod_{\st \in \sF_{I,\sc}} Z_{I,\sc, \st}\, .
\end{equation}
Here $Z_{I,\sc,\st}\simeq Z_{I,\sc}= G/ H_{I,\sc}$ for all $\st \in \sF_{I,\sc}$ with
$\sF_{I,\sc}$ the set corresponding to $\sF_I=F_I/ F(I)$ when $Z_I$ is replaced by $Z_{I,\sc}$.
For every pair $\sc,\st$ this then leads to an injective matching map ${\bf m}_{\sc,\st}:
\Wc_{I,\sc} \to \Wc$ with $\Wc_{I,\sc}$ a parameter set for $W_{I,\sc}=(P\bs Z_{I,\sc})_{\rm open}$. The case of
$\sc=\st=\1$  corresponds to the original map ${\bf m}={\bf m}_{\1,\1}$. The decomposition
\eqref{normal union} then leads to
a partition
\begin{equation} \label{W partition}
\Wc=\coprod_{\sc\in \sC_I} \coprod_{\st \in \sF_{I,\sc}} {\bf m}_{\sc, \st}(\Wc_{I,\sc})
\end{equation}
refining \eqref{party W}.

\par This section has several parts. It starts with the construction of the injective map
${\bf m}: \Wc_I \to \Wc$.  For a better understanding of the matching map ${\bf m}$ we
then illustrate  the case where $Z$ is a symmetric space
and relate ${\bf m}$ to Matsuki's description \cite{Mat1}
of the  open $P\times H$-double cosets in $G$ in terms
of Weyl groups.  After that we derive the general partition of $\Wc$ in terms of the ${\bf m}_{\sc,\st}$. This last part is
a bit more  technical and can be skipped at a first reading.
\par Throughout this section $I\subset S$ is fixed.

\subsection{Relating $W_I$ to $W$}\label{subsection 5.1}
We recall from Lemma  \ref{lemmaW1} the natural bijection of $W_\R=(P\bs \uZ(\R))_{\rm open}$ with
$F_M\bs F_\R$  where $F_\R= \uAZ(\R)_2$ denotes the $2$-torsion subgroup of $\uAZ(\R)$.
On the other hand we recall from Lemma \ref{equal L cap H}  that  $\uL\cap \uH= \uL \cap \uH_I$.
Intersecting this identity with $\uA$ we obtain that  $\uA\cap \uH= \uA\cap \uH_I$
and hence an identity of homogeneous spaces
$$ \uAZ= \uA/ \uA\cap \uH  = \uA/ \uA\cap \uH_I =\uA_{\uZ_I}\, .$$
In particular, $\uAZ(\R)$ and  $\uA_{\uZ_I}(\R)$ have the same $2$-torsion groups, namely $F_\R$. In addition
 $\uL\cap \uH= \uL \cap \uH_I$ implies
 that the two open $\uP$-orbits $\uP\cdot z_0 \subset \uZ$ and
$\uP\cdot z_{0,I}$ carry canonically isomorphic local structure theorems, see \eqref{LST1}, \eqref{deco L} and
\eqref{LST2}.  Hence the group $F_M$ is identical in both cases
and we obtain a natural bijection (the identity map)
$${\bf m}_\R:  W_{I,\R}=(P\bs \uZ_I(\R))_{\rm open}  \to (P\bs \uZ(\R))_{\rm open}=W_\R\, .$$

\begin{rmk} On the one hand side we have an identity of homogeneous spaces $\uAZ=\uA_{\uZ_I}$, but on the other
hand we also view $\uAZ$ as a subvariety of $\uZ$ and $\uA_{\uZ_I}$ as a subvariety of $\uZ_I$. In the latter picture
the identity of homogeneous spaces yields a natural identification of subvarieties of $\uZ$ and $\uZ_I$. \end{rmk}

\begin{prop} \label{prop m} One has ${\bf m}_\R(W_I)\subset W$. \end{prop}

In order to prove this proposition we first recall another natural map
${\bf m}: \Wc_I \to \Wc$ which first arose in \cite[Sect.~3]{KKS2}.
We fix with $\Wc_I\subset G$ a set of representatives of $W_I$ with elements
$w_I\in \Wc_I$ of the form  $w_I =\tilde t_I h_I $ where $h_I \in \uH_I$ and $\tilde t_I \in T_Z$.  Upon the
identification of varieties $\uAZ(\R)\simeq \uA_{\uZ_I}(\R)$ we view $t_I:=\tilde t_I \cdot z_{0,I}=w_I\cdot z_{0,I}$
as an element of $F_\R$.

\smallskip \par  Let now  $Pw_I\cdot z_{0,I}\in W_I$ be an open $P$-orbit in $Z_I$ with $w_I\in \Wc_I$.
Next let $X\in \af_I^{--}$ and set
$$a_s:=\exp(sX)\in A_I^{--}\subset A \qquad (s>0)\, .$$
It follows from \cite[Lemma 3.9]{KKS2} that there exist
$s_0=s_0(X)>0 $ and a unique $w=\tilde th\in \Wc$ such that
\begin{equation} \label{eq-corr}P w_I a_s \cdot z_0 = Pw \cdot z_0\qquad (s\geq s_0)\, .
\end{equation}

\begin{lemma} Given $w_I= \tilde {t_I} h_I\in\Wc_I$ as above, the element $w\in \Wc$ such that \eqref{eq-corr} holds
 does not depend on the choice of $X\in \af_I^{--}$.
\end{lemma}
\begin{proof} In order to record the possible dependence on $X$ we write
$a_s(X)=\exp(sX)$ and $w(X)$ for the corresponding $w$. Now we recall the argument of \cite[Lemma 3.9]{KKS2}:  For fixed $X$
we have $\lim_{s\to \infty} e^{s\ad X} \hf= \hf_I$ by \eqref{I-compression}. Thus there exists
an $s_0(X)$ such that $\pf + \Ad(w_I) e^{s\ad X}\hf =\gf$ for all $s\geq s_0(X)$.
In particular, we obtain that $Pw_I a_s(X) \cdot z_0$ is open for all
$s\geq s_0(X)$. Since the limit  \eqref{I-compression} is locally uniform in $X\in \af_I^{--}$, it follows that
$w(X)$ is locally constant. The lemma follows.
\end{proof}

With this lemma  we obtain in particular a natural map
\begin{equation}\label{bfm}
{\bf m}:  \Wc_I \to \Wc,  \ w_I\mapsto w={\bf m} (w_I)\, .
\end{equation}
With
the identifications $W_I \simeq \Wc_I$
and $W\simeq \Wc$
we view ${\bf m}$ also as a map $W_I \to W$, which by slight abuse of notation is denoted
as well by ${\bf m}$.
\par Since the choice of $X\in \af_I^{--}$ was irrelevant for the definition of ${\bf m}$ we may henceforth assume that
$X=X_I\in \cf_I^{--}\subset \af_I^{--}$ was such that it corresponds to $-{\bf e}_I\in V_I$ under the identification
$V_I \simeq \af_I$.
\par Proposition \ref{prop m} will now follow from:
\begin{lemma} \label{lemma m} ${\bf m}_\R |_{W_I} = {\bf m}$.\end{lemma}

\begin{proof} Let $w_I=\tilde t_I h_I\in\Wc_I$ and $w={\bf m}(w_I)= \tilde t h$.
From Lemma \ref{speed lemma} for $g=h_I \in \uH_I$ we obtain
a $C^1$-curve $[0,1]\to \uP, \ u\mapsto p_u$  with $p_u\to \1$ for $u\to 0^+$ and
$$ h_I a_s \cdot z_0 = p_{e^{-s}} a_s \cdot z_0 \qquad (s\geq s_0)\, .$$
Hence with $ p'_s:= \tilde t_I p_{e^{-s}} \tilde t_I^{-1}$ we obtain that
$$ w_I a_s\cdot z_0 = p_s' \cdot  t_I\in Z\, .$$
Since $t_I \in Z$ and $p_s'\to \1$ for $s\to \infty$ we may assume that $p_s'\in P$ is real
as well (use the local structure
theorem of the form \eqref{LST3}). On the other hand the matching property
\eqref{eq-corr} yields
$$w_I a_s\cdot z_0 =  p_s'' \cdot t$$
for some $p_s'' \in P$.  Thus we get
$$ P\cdot t_I=P\cdot t.$$
This implies the lemma, and with that Proposition \ref{prop m}.
\end{proof}

In the sequel we adjust $\Wc\subset G$ (by possibly multiplying the previous $w = \tilde t h$
by an  element of $F_M$) in such a way that for each $w_I =\tilde t_I h_I\in\Wc_I$ one has
\begin{equation}\label{t=t_I} t=t_I \qquad  \text{when} \ w=\tilde t h = {\bf m}(w_I)\, .\end{equation}
We note that this adjustment of $\Wc$ depends on our fixed choice of $\Wc_I$ and hence on $I$.

\par \begin{rmk} \label{rmk ZZI}Notice that we typically have ${\bf m}(\Wc_I)\subsetneq \Wc$ as the example of
$Z=\Sl(2,\R)/\SO(1,1)$ with $I=\emptyset$ already shows (cf. Example \ref{ex=SL2} (1) ). Here we have $H_\emptyset=M \oline N$ and $\hat H_\emptyset= MA \oline N$
and thus $\Wc_\emptyset =\{\1\}$ while $\Wc=\{1,w\}$ has two elements.
\end{rmk}

\begin{prop}\label{prop cr1} {\rm (Consistency relations for stabilizers)} Let $w_I \in \Wc_I$ and
$w={\bf m}(w_I)\in \Wc$.  Then
\begin{equation} \label{ConsisT1}  (H_w)_I = (H_I)_{w_I}\,.\end{equation}
\end{prop}

\begin{proof}  Recall from \eqref{character HI} that $H_I$ is the subgroup of $G$ which asymptotically
preserves the curves $\gamma_v$ in normal direction,  i.e. is the group of elements $g\in G$
with $g\cdot [\gamma_{v}'(0)]_{\rm n}= [\gamma_{v}'(0)]_{\rm n}= v$. Hence   $(H_I)_{w_I}\subset G$ is the group of elements $g\in G$
with $g\cdot [\gamma_{w_I, v}'(0)]_{\rm n}= [\gamma_{w_I, v}'(0)]_{\rm n}=t_I \cdot v$.  On the other hand
we can characterize $(H_w)_I$ as follows: define the curve
$$\sigma_{w,v}(s):= \tilde a(v) \exp(-(\log s)  X_I) \cdot z_w = \tilde a(v) \exp(-(\log s)  X_I) \cdot t$$
where  $\tilde a(v) \in \uA$ is any lift of $a(v)\in \uA_I$ with respect to the projection
$\pi: \uA\to \uA_Z$.  Then $(H_w)_I$
is the group of elements $g\in G$ with $g\cdot [\sigma_{w,v}'(0)]_{\rm n}= [\sigma_{w,v}'(0)]_{\rm n}=t\cdot v$.
As $t=t_I$, the desired equality of groups follows.
\end{proof}

\subsection{Symmetric spaces}\label{Subsection m symmetric}
The nature of the map ${\bf m}$ becomes quite clear in the special case where $Z$ is a symmetric space. In this special situation
we can make the matching map explicit in terms of certain Weyl groups.

\par For this subsection $Z=G/H$ is symmetric, that is, there exists an involution $\tau: \uG\to \uG$, defined over $\R$,
such that $\uH$ is an open subgroup of the  $\tau$-fixed point group $\uG^\tau$.
We choose our maximal anisotropic group $\uK \subset \uG$ in such a way that the Cartan involution
$\theta$,  which defines $\uK$,  commutes with $\tau$. By slight abuse of notation  we use $\tau$ and $\theta$ for the
induced derived involutions on $\gf$ as well.

\subsubsection{The adapted parabolic}\label{Subsubsection adapted}  With $\gf=\kf\oplus \kf^\perp$, resp $\gf= \hf\oplus \hf^\perp$,
we obtain  the decomposition of $\gf$ in eigenspaces of $\tau$, resp. $\theta$, with eigenvalues $+1$ and $-1$.
We let  $\af_Z\subset \hf^\perp \cap \kf^\perp$ be a maximal abelian subspace and extend $\af_Z$ to
a maximal abelian subspace $\af\subset \kf^\perp$.  Now, according to Rossmann,  the root system $\Sigma=\Sigma(\gf,\af)$ restricts
to a root system
$$\Sigma_Z=\Sigma|_{\af_Z}\bs \{0\}$$
on $\af_Z$. The Weyl group of $\Sigma_Z$ is denoted by $\W=\W_Z$.
\par Let $\Sigma_Z^+\subset \Sigma_Z$ be a positive system, and let $\Sigma^+\subset \Sigma$ be
a positive system such that $\Sigma^+|_{\af_Z}\bs \{0\}= \Sigma_Z^+$.
Then $PH\subset G$ is open for the minimal parabolic subgroup $P=MAN$, for which
$\nf$ is the sum of the positive root spaces. The adapted parabolic $Q=LU\supset P$ is then characterized by $L=Z_G(\af_Z)$.
It is the unique minimal $\theta\tau$-stable parabolic subgroup of $G$ containing $P$.

\subsubsection{The deformations $H_I$}
The spherical roots $S\subset\af_Z^*$ are given by the simple roots in $\Sigma_Z$ with respect to $\Sigma_Z^+$.
Hence for any $I\subset S$ we obtain parabolic subgroups $P_I \supset Q$ with $L_I=Z_G(\af_I)$.
As before we realize $\af_I\subset \af$ so that $A_I=\exp(\af_I)$ becomes a subgroup of $A$.
Then $L_I=M_I A_I\simeq M_I \times A_I$ for a unique $\tau$-stable subgroup $M_I \subset L_I$.
 Now the deformations $H_I$ are given by

$$ H_I = (M_I \cap H) \oline{U_I}$$
with $M_I \cap H\subset M_I$ a symmetric subgroup, i.e. $M_I/ M_I\cap H\subset G/H$ is a symmetric subspace.
Note that the $H_w$, $w\in \Wc$, can be treated on the same footing, i.e.
$(H_w)_I =  (M_I \cap H_w) \oline{U_I}$ and  $M_I \cap H_w\subset M_I$ a symmetric subgroup.
As seen in Example
\ref{ex SL3} the subgroups
$M_I\cap H$ and $M_I\cap H_w$
are not necessarily conjugate in $M_I$.

\subsubsection{Open double cosets}\label{open double}
For later reference in Section \ref{section DBS} (where we derive the Plancherel formula
for symmetric spaces) we consider here both $(P\bs Z_I)_{\rm open}$ and $(P_I\bs Z)_{\rm open}$
together.

Recall that for symmetric spaces the set $W=(P\bs Z)_{\rm open}$
allows a description in terms of Weyl groups. For that we identify $\W=\W_Z\simeq
 [N_K(\af)\cap N_K(\af_Z)]/ M$  and  define a subgroup of $\W$ by
 $\W_H= [N_{K\cap H}(\af)\cap N_{K\cap H}(\af_Z)]/ M$.
Then Matsuki \cite{Mat1} has shown that

\begin{equation}\label{Matsuki PZ} \W/\W_H \to (P\bs Z)_{\rm open}, \ \ w\W_H  \mapsto  PwH
\end{equation}
is a bijection. In particular, $W\simeq \W/\W_H$.

When applied to the symmetric space $M_I/(M_I\cap H)$, Matsuki's result becomes
\begin{equation}\label{Matsuki M_I}
\W(I)/ (\W(I)\cap \W_H)\simeq(( P\cap M_I)\bs M_I / (M_I \cap H))_{\rm open}.
\end{equation}
Now
\begin{equation}\label{Pz to PIz}
(P\bs Z)_{\rm open} \to (P_I\bs Z)_{\rm open}, \ \ Pz\mapsto P_Iz
\end{equation}
is surjective. It follows from \eqref{Matsuki M_I} that the composition of \eqref{Matsuki PZ}  with
\eqref{Pz to PIz} factorizes to a bijection  (see also \cite{Mat2})
$$ \W(I)\bs \W/ \W_H\to   (P_I \bs Z)_{\rm open}$$
where $\W(I)<\W=\W_Z$ is the subgroup generated by the reflections $s_\alpha$ for $\alpha \in I$.

In particular we obtain an action of $\W(I)$ on $\Wc\simeq \W/\W_H$ and record:

\begin{lemma} \label{Lemma Mats1}For $I\subset S$ the following assertions hold:
\begin{enumerate}
\item \label{onemats}$(P_I\bs Z)_{\rm open}\simeq\W(I)\bs \Wc$.
\item \label{twomats} $(P\bs Z_I)_{\rm open}\simeq \Wc_I\simeq\W(I)/(\W(I)\cap \W_H)$.
\end{enumerate}
\end{lemma}
\begin{proof} The first assertion we have just shown. For the second, recall first that
$H_I= (M_I \cap H )\oline {U_I}$. Hence  the Bruhat decomposition   yields that
$$ (P\bs Z_I)_{\rm open } \simeq  (( P\cap M_I)\bs M_I / (M_I \cap H))_{\rm open}\, ,$$
so that \eqref{twomats} follows from \eqref{Matsuki M_I}.
\end{proof}

\begin{lemma}  \label{lemma 58}Upon identifying $\W(I)/ \W(I)\cap \W_H$ with $\Wc_I$ and
$\W/ \W_H$ with $\Wc$, the map ${\bf m}: \Wc_I \to \Wc$ corresponds to the natural inclusion
map $\W(I)/(\W(I)\cap \W_H) \hookrightarrow \W/\W_H$.\end{lemma}

\begin{proof} We recall the construction of the map ${\bf m}$ via considering
the limits of the double cosets $Pw_Ia_s H$.  So let $w_I\in \W(I)$ and observe that $\W(I)$ keeps $\af_I$
pointwise fixed. Thus we have $Pw_I a_s H= Pw_I H$ and the lemma follows.
\end{proof}

Also of later relevance are the open $H\times \oline P_I$-double cosets in $G$ which we
treat here as well. Since the anti-involution
$$G\to G, \ \ g\mapsto g^{-\theta}:=\theta(g^{-1})$$
leaves $H$ invariant and maps $P_I$ to its opposite $\oline{P_I}$, we obtain
a bijection of double cosets

$$ P_I\bs G/H \to H\bs G/ \oline {P_I}  , \ \  P_IgH\mapsto   H g^{-\theta} \oline{P_I}\, .$$
%Since $G=P_IK$ we can write any $g\in G$ as $g=p_I k$ in self-explaining notation
%and deduce that $g^{-\theta}= k^{-1} \theta{p_I}^{-1}$ and
%hence

With Lemma \ref{Lemma Mats1} we thus obtain a bijection
\begin{equation}\label{Matsuki Pbar}
\W(I)\bs \Wc \to (H\bs G/\oline{P_I})_{\rm open}, \ \ \W(I)w \mapsto H w^{-\theta}\oline{P_I }\, .
\end{equation}

\subsection{Relating $W_I$ to $\hat W_I$} We now return to the setup of a general real spherical space.
In this subsection  we provide some complementary material on the relation of $W_I$
to $\hat W_I:=(P\bs \hat Z_I)_{\rm open}$.  This will lead to a better  geometric understanding
of what to come next.

Recall that $A_I$ is the connected component of $A(I)= A_I F(I)\simeq
A_I \times F(I)$.  Notice that $A(I)$ acts naturally on the right of $Z_I=G/H_I$ and thus induces an action of
$A(I)$  on $W_I=(P\bs Z_I)_{\rm open}$.
The following lemma is then a consequence of the fact that the connected group
$A_I$ acts trivially on the finite set $W_I$.

\begin{lemma} \label{lemma WWI}The natural map
$$ W_I= (P\bs Z_I)_{\rm open}\to \hat W_I =(P\bs \hat Z_I)_{\rm open}, \ \  PwH_I\mapsto Pw\hat H_I$$
is surjective and induces an isomorphism $W_I/ F(I)\simeq \hat W_I$.
\end{lemma}

\begin{proof}  Let $P w\hat H_I\subset G$ be open for some $w\in G$. We first show that $PwH_I A_I= PwH_I$.
According to  \eqref{th-deco} applied to the real spherical space $\hat Z_I$
we may write $w=\tilde t\hat h$ with $\tilde t\in T_Z$ and $\hat h\in \hat \uH_I$.  Since $\hat \uH_I  =\exp_{\uG}(\af_{I,\C}) \uH_I$
by (\ref{exact1}), we have $\hat h= \tilde t_I h $ with $h\in \uH_I $ and $\tilde t_I\in
\exp_{\uG} (\af_{I, \C}) $.  Let $a\in A_I\subset A$. Then
$$(aw)^{-1} w a
= (a\tilde t\tilde t_I  h)^{-1} \tilde t \tilde t_I  h a
=  h^{-1}a^{-1} h a \, .$$
Now we observe that $a w \in G$ and
$a^{-1}ha \in \uH_I$.
Hence $(aw)^{-1} w a\in H_I$, and thus $Pw H_I a=PwaH_I= PwH_I$ as claimed.

Since  $\hat H_I = H_I  A_IF(I)$ we obtain that $Pw \hat H_I = \bigcup_{t \in F(I)} P w H_I t$.  In particular
$PwH_I$ is open. Hence the map $W_I \to \hat W_I$ is onto.  The last assertion also follows.
\end{proof}

\par Similar to $W\simeq F_M\bs F $ (see Lemma \ref{lemmaW1}) we obtain with
$$F_I^\perp :=  F \uA_I(\R)/ \uA_I(\R)\subset \uA_\uZ(\R)/\uA_I(\R)$$ that
$$\hat W_I \simeq F_M\bs F_I^\perp$$
 as a consequence of  \eqref{LST-I}. We further recall that
we view $F_I\subset \{-1,1\}^r \cap V_I \subset  V$ and accordingly the group
$F_M\cap F_I= F_M \cap F(I)$  acts on $F_I$.
Thus we obtain an exact sequence of pointed sets
$$   (F_M\cap F(I)) \bs  F_I \hookrightarrow  F_M\bs F \twoheadrightarrow F_M\bs  F_I^\perp $$
or, equivalently,
$$     (F_M \cap F(I))\bs F_I \hookrightarrow W \twoheadrightarrow \hat W_I\, .$$
From the injectivity of ${\bf m}$ and Lemma \ref{lemma WWI}
we thus obtain the commutative diagram:

 \begin{equation}   \label{small 2-group diagram}
\xymatrix{  (F_M\cap F(I)) \bs F_I \ar@{^{(}->}[r]&   W\ar@{>>}[r] &\hat W_I\\
\ar@{^{(}->}[u] \ar@{^{(}->}[r] (F_M\cap F(I))\bs F(I) & \ar@{>>}[r]\ar@{^{(}->}[u]^{{\bf m}} W_I&
\ar@2{-}[u]\hat W_I}
\end{equation}

\begin{rmk}\label{rmk W does not split} Let us emphasize that the upper horizontal sequence in \eqref{small 2-group diagram} is exact
in the category of pointed  sets, but not in the category of sets, i.e. we do not have
$W\simeq \hat W_I \times (F_M\cap F(I))\bs F_I$ as sets, see Example
\eqref{ex SL3 continued} below.
\end{rmk}

This phenomenon disappears if we consider $W_\R$ and $\hat W_{I,\R}= (P\bs \hat \uZ_I(\R))_{\rm open}$ instead
of $W$ and $\hat W_I$. In more detail, recall the basis
$(\psi_i^I)_{1\leq i\leq r}$ by means of which we get a decomposition (see \eqref{AZ-iso1})
\begin{equation}\label{AI torus deco}  \uAZ(\R)= \underbrace{\uA_I(\R)}_{\simeq (\R^\times)^k}  \times \underbrace{\uA_I^\perp(\R)}_{\simeq (\R^\times)^{r-k}}
\simeq (\R^\times)^r \end{equation}
analogous to the decomposition $V= V_I \oplus V_I^{\perp}$.  In particular,  $F_\R$, the $2$-torsion subgroup of
$\uAZ(\R)$,  decomposes as $F_\R=F_{I,\R}\times F_{I,\R}^\perp$ in self explaining notation.
Hence any $t\in F_\R$ decomposes as $t= t^{\|} t^\perp$ with $t^{\|} \in F_{I,\R}$ and $t^\perp \in F_{I,\R}^\perp$.

In this situation we obtain that the map
$$F_\R \to F_{I,\R}^\perp, \ \  t \mapsto t^\perp$$
induces an epimorphism
$$W_\R\simeq F_M\bs F_\R\to \hat W_{I,\R} \simeq F_M\bs F_{I,\R}^\perp, \ \ F_M t \mapsto F_M\cdot t^\perp\, ,$$
leading to a decomposition
$$ W_\R \simeq \hat W_{I,\R} \times   (F_{I,\R} \cap F_M)\bs F_{I,\R}\, .$$

\subsection{The fine partition of $W$ with respect to $I$}\label{fine partition}

Our next goal is to explore the issue of non-surjectivity of ${\bf m}$.

We recall from Subsection \ref{structure of Z_I} the set
$\sC_I$, the partition
$$ \Wc=\coprod_{\sc \in \sC_I}  \Wc_\sc\, $$
and the groups $H_{I,\sc}$ for $\sc\in \sC_I$.
Thus the understanding of $\Wc$ with respect to $I$ comes down to understanding the various
$\Wc_\sc$.   Once we have fixed $\sc$ we will see below that we obtain a natural geometric
splitting of $\Wc_\sc\simeq \hat \Wc_\sc\times (F_M\cap F_I) \bs F_I$ contrary to what happens for
$\Wc$ (see Remarks \ref{rmk W does not split} and \eqref{splitting W}).

For expository reasons we start with $\sc=\1\in\sC_I$, by which we mean $\sc=\hat Z_I=G\cdot \hat z_{0,I}$.
Thereupon we consider the other cases by replacing
$H_I$ with $H_{I,\sc}$ and adding a further index $\sc$ to the notation.

\begin{rmk}  Even in case $\sC_I=\{\1\}$  and $\Wc=\Wc_\1$ it can happen
that ${\bf m}(\Wc_I)\subsetneq \Wc$.  As we will see below this is related to the set $\sF_I=F(I)\bs F_I$ originating from the normal bundle geometry in  Subsection \ref{nb points to Z}.
\end{rmk}

\subsubsection{The case $\sc=\1$}\label{subsection c=1}
We assume that $w\in \Wc_\1$, i.e.~$\hat z_{w,I}\in \hat Z_I$. Let $W_\1=\{P \cdot w\mid
w \in \Wc_\1\}$.
Let $F_\1:=\{t\in F\mid Pt\in W_\1\}\subset F_\R$. Then we can describe $F_\1$ and thus
$\Wc_\1\simeq F_M\bs F_\1$ geometrically as follows.

Recall from \eqref{AI torus deco}
that any $t\in F_\R$ decomposes as $t= t^{\|} t^\perp$ with $t^{\|} \in F_{I,\R}$ and $t^\perp \in F_{I,\R}^\perp$.
 Let $w\in \Wc_\1$, write it
as $w=\tilde th$, and decompose
$\tilde t=\tilde t^{\|} \tilde t^\perp$  such that $\tilde t^{\|} \cdot z_0 = t^{\|}$ and $\tilde t^\perp \cdot z_0=t^\perp$.
Consider the curve $s\mapsto  a_s  \cdot z_w= a_s\cdot t $ where $a_s=\exp(sX)$ with $X\in\cf_I^{--}$.
Then,  as $t^{\|}\in \uA_I(\R)$
fixes $\hat z_{0,I}$, we obtain  in the limit for $s\to \infty$ that
$\tilde t\cdot \hat z_{0,I}= t^\perp \cdot \hat z_{0,I}$, and as $w\in \Wc_\1$ this limit belongs to an open $P$-orbit of $\hat Z_I$.

Furthermore the coordinate $t^{\|}\in F_{I,\R}$ tells us
in which direction we approach the limit $t^\perp\cdot \hat z_{0,I}$, i.e. in which component of the cone
$V_{Z,I}=F_I V_I^0$ we approach the limit.
With $F_{I,\1}^\perp:=
F_\1 \uA_I(\R)/ \uA_I(\R) \simeq F_\1 F_{I,\R}/ F_{I,\R}$
we obtain the following.
\begin{lemma}\label{lemma F-product} By restriction the map $t \mapsto (t^{\|},t^\perp)$
yields a bijection $F_\1 \simeq F_I\times F_{I,\1}^\perp$.
\end{lemma}

\begin{proof} First we claim that $F_I=F_\1\cap \uA_I(\R)$.
The inclusion $\supset$ is clear since by definition $F_I=F\cap \uA_I(\R)$.
Conversely, each $t\in F_I$ corresponds to a $w=\tilde th\in \Wc$
with $t\in \uA_I(\R)$. Then $\hat z_{w,I}=\hat z_{0,I}$, and hence $t\in F_\1$ as claimed.

In particular it follows that $(t^{\|},t^\perp)\in F_I\times F_{I,\1}^\perp$ for all $t\in F_\1$.
Since $t\mapsto (t^{\|},t^\perp)$ is injective
by its definition in \eqref{AI torus deco}, it remains
to see that $t=t^{\|}t^\perp\in F_\1$
for all pairs $(t^{\|},t^\perp)\in F_I\times F_{I,\1}^\perp$.
Since $ t^{\|}\in \uA_I(\R)$ we know that $t\cdot \hat z_{0,I}
=t^\perp\cdot \hat z_{0,I}\in \hat Z_I$ which is the limit of the curve $\gamma(s)= a_s \cdot t$
for $s\to \infty$.  The coordinate $t^{\|}\in F_I$ shows that $\gamma$ approaches
the limit $t\cdot \hat z_{0,I}$ in a direction pointing to $Z$ (see also the end of the proof of Lemma \ref{lemma rel open} for a
more formal argument).
Hence $t\in F$ and $Pt\in W_\1$.
\end{proof}

Lemma \ref{lemma F-product} implies the splitting
\begin{equation} \label{splitting W}\Wc_\1\simeq \hat \Wc_I\times (F_M\cap F(I)) \bs F_I\end{equation}
and we can rephrase Lemma \ref{lemma WWI} as:

\begin{lemma} We have ${\bf m}(\Wc_I)\subset \Wc_\1$ and under the identification \eqref{splitting W}
we have
\begin{equation}\label{splitting WI}  {\bf m}(\Wc_I)\simeq \hat \Wc_I \times (F_M\cap F(I))\bs F(I) \end{equation}
\end{lemma}

From \eqref{splitting W} and \eqref{splitting WI} we obtain that
\begin{equation} \label{splitting WI1} \Wc_\1 \simeq {\bf m}(\Wc_I) \times  \sF_I \end{equation}
with $\sF_I=F(I)\bs F_I$.

\begin{rmk}\label{remark normal bundle}  It is instructive for the following to recall from
Subsection \ref{nb points to Z} the part
$N_\sY^Z= \coprod_{\st\in \sF_I} N_\sY^{Z, \st}$ of the normal bundle $N_\sY$ which points to $Z$.
Here
$$Z_{I,\st}:=N_\sY^{Z,\st}\simeq G/H_I$$
by the isomorphism \eqref{N-ident}, and $\sF_I=F(I)\bs F_I$ parametrizes the components of $N_\sY^Z$. \end{rmk}
\par  Note that $F_{I,\R}$ is a $\Z_2$-vector space and thus we can find
a splitting $F_{I,\R} = F(I)\oplus F_{I,\R}^0$ of vector spaces. In particular,
we obtain $F_I = F(I) \oplus F_I^0$ for a subset $F_I^0\subset F_{I,\R}^0$.
In particular the map

$$ F_I^0\to \sF_I,\ \ t\mapsto \st:= tF(I)$$
is a bijection.
\par Now, using the isomorphism \eqref{splitting WI1} and the identification $\sF_I\simeq F_I^0$
we obtain injective maps
$${\bf m}_\st: \Wc_I \to \Wc_\1\simeq {\bf m}(\Wc_I) \times  F_I^0, \ \ w_I \mapsto ({\bf m}(w_I), t)$$
which yields the partition
\begin{equation} \label{partition}  \Wc_\1= \coprod_{\st\in \sF_I}  {\bf m}_{\st} (\Wc_I)\, .\end{equation}

Let us explain the map ${\bf m}_\st$ more geometrically using the normal bundle, see  Remark \ref{remark normal bundle}.  The subset ${\bf m}_{\st} (\Wc_I)\subset \Wc_\1$
corresponds to those $w=\tilde t_wh\in \Wc_\1$  for which the curve $s\mapsto a_s\cdot z_w= a_s t_w \cdot z_0$
approaches the boundary point $\hat z_{w,I}= t_w\cdot \hat z_{0,I}=t_w^\perp\cdot \hat z_{0,I}$
in direction of $t F(I) V_I^0\subset V_{Z,I}$.   Let us emphasize that
our initial map ${\bf m}$ corresponds then to the case where
$\st=F(I)$ is the identity coset.

\par  Recall that $\st \in \sF_I$ corresponds to a unique $t \in F_I^0$. Further we
let $\tilde t\in T_Z$  be a lift of $t$, i.e. $\tilde t \cdot z_0 =t$.
We assume that $t={\bf 1}$ in case $\st= F(I)$.

\begin{rmk}  Let $w_I=\tilde t_I h_I$ and  $w_1={\bf m}(w_I)= \tilde t_I h\in \Wc$. Then note that
$$ {\bf m}_{\st }(w_I) =  \tilde t_M  \tilde t \tilde t_I  \tilde h $$
for some $\tilde t_M \in F_M $, depending on the choice of representatives
for $w:= {\bf m}_{\st }(w_I)\in \Wc$,  and $\tilde h  \in \uH$.
Thus by changing $w={\bf m}_{\st} (w_I)\in \Wc$ to $\tilde t_M w \tilde h\in G$ for some $ h'\in\uH$
we may assume that the compatibility
conditions
$${\bf m}_\st(w_I)=\tilde t \tilde t_I h''$$
hold  for some $h''\in \uH$.  In particular, we have
$${\bf m}_\st(w_I)\cdot z_0= \tilde t {\bf m} (w_I)\cdot z_0=tt_1$$
 for all $\st\in \sF_I,  w_I\in \Wc_I$.
 \par Notice that this correction of choice of $\Wc$ (by harmless left displacements of elements of $F_M$)  with respect to
$\Wc_I$ depends on $I$. In general it seems to be  not possible to make a consistent choice
of $\Wc$ which would be valid for all $I$ simultaneously. \end{rmk}

\par

Recall the notation $H_g=gHg^{-1}$ for a subgroup $H$ in a group $G$ and $g\in G$.   Then note that
$\uH_{\tilde t}$ is defined over $\R$ and $H_{\tilde t}:= (\uH_{\tilde t})(\R)$ is conjugate to $H$ as $t\in Z$.
Likewise we define $z_{t,I}:=\tilde t \cdot z_{0,I}\in \uZ_I(\R)$ and note that
$G$-stabilizer of $z_{t,I}$ is $H_I$ as we have $(H_{\tilde t})_I = H_I$ as a consequence of the fact that
$\tilde t$ fixes the vertex $\hat z_{0,I}$.

With then obtain the following extension of the consistency relations from Proposition \ref{prop cr1}:

\begin{lemma} \label{lemma HI comp} Let $w_I\in \Wc_I$, $\st \in \sF_I$ and $w= {\bf m}_\st (w_I)\in \Wc_\1$.
Then
\begin{equation} \label{WWI2} (H_w)_I =  (H_I)_{w_I} \, . \end{equation}
In particular, $(H_w)_I$ only depends on $w_I$ and is  independent of $\st$.
\end{lemma}

\begin{proof} For $w_I =\tilde t_I h_I$ we have $w_1:={\bf m}(w_I)=\tilde t_I h'$ for some $h'\in \uH$.
Hence ${\bf m}_\st (w_I) = \tilde t \tilde t_I h$ for some $h\in \uH$.  We further have
$$(H_w)_I= \big((H_{w_1})_{\tilde t}\big)_I=(H_{w_1})_I$$
and now Proposition \ref{prop cr1} applies.
\end{proof}

\subsubsection{The general decomposition of $\Wc$ }
In general we obtain a partition

\begin{equation} \label{full deco W} \Wc= \coprod_{\sc\in \sC_I}  \coprod_{\st \in \sF_{I,\sc}}  {\bf m}_{\sc,\st} (\Wc_{I,\sc})\end{equation}
where $\Wc_{I,\sc}$ are the open $P$-orbits for $Z_{I,\sc}:= G/H_{I,\sc}$ parametrized
as in the previous section with $H_I$ replaced by $H_{I,\sc}$.
The set $\sF_{I,\sc}$ is then $\sF_I$, but
for $H_I$ replaced by $H_{I,\sc}$. We define ${\bf m}_{\sc, \st} $ similarly. Regarding
our choices $w(\sc)\in \W$ which defined $H_{I,\sc}$ we normalize ${\bf m}_{\sc,\1}$ such that
${\bf m}_{\sc,\1}(\1)=w(\sc)$.

\begin{rmk}  If we let $F_c\subset F$ correspond to $\Wc_\sc\subset \Wc$ we define as  before
$F_{I,\sc}:= F_\sc \cap \uA_I(\R)$ and $F_{I,\sc}^\perp:=  F_\sc F_{I,\R}/ F_{I,\R}$.
As in Lemma \ref{lemma F-product} we then obtain
\begin{itemize}
\item  $F_{I,\sc}=F_I$.
\item $F_\sc\simeq F_{I,\sc} \times F_{I,\sc}^\perp$ under $t\mapsto (t^{\|},t^\perp)$.
\end{itemize}
The first item tells us that $F_{I,\sc}$ is independent of $\sc$. However $F(I)_\sc$ does depend
on $\sc$ as Example \ref{ex SL3 continued} below shows. In particular the dependence of
 $\sF_{I,\sc}= F(I)_\sc\bs F_I$ n $\sc$ is caused by the $\sc$-dependence of $F(I)_\sc$ only.
\end{rmk}

Further we denote by
$z_{0,I,\sc} = H_{I,\sc}$ the standard base point of $Z_{I,\sc}$,
and state the general version of \eqref{WWI2}:  let $\sc \in \sC_I$ and $\st \in \sF_{I,\sc}$
such that $w={\bf m}_{\sc,\st}(w_{I,\sc})\in \Wc_\sc$ for some $w_{I,\sc}\in \Wc_{I,\sc}$. Then $(H_w)_I$
does not depend on $\st$ and
\begin{equation} \label{WWI2 general} (H_w)_I = (H_{I,\sc})_{w_{I,\sc}} \qquad (w= {\bf m}_{\sc,\st}(w_{I,\sc}))\, .\end{equation}

If we define $w(\sc,\st):= {\bf m}_{\sc,\st}(\1)\in \Wc$ and set $Z_{I,\sc,\st} =(H_{w(\sc,\st)})_I$, then
$Z_{I,\sc,\st}= Z_{I,\sc}$ and the decomposition \eqref{normal union} follows.

\begin{ex}  \label{ex SL3 continued}We continue Example \ref{ex SL3} of $Z=\Sl(3,\R)/ \SO(1,2)$
with $\Wc=\{w_1, w_2, w_3\}$  and $H_{w_1}=H$.  We chose $I=\{\alpha_2\}$ and obtained
$\sC_I=\{{\mathsf 1},{\mathsf 2}\}$ with $\Wc_{\mathsf 1}=\{w_1\}$ and $\Wc_{\mathsf 2}=\{w_2, w_3\}$. Further we had
$H_{I,{\mathsf 1}}= H_I=S(\OO(1)\OO(2)) U_I$ and $H_{I,{\mathsf 2}}= (H_{w_2})_I = (H_{w_3})_I=S(\OO(1)\OO(1,1)) U_I$.
\par Next we claim that both $\Wc_{I,{\mathsf 1}}=\{\1\}$ and $\Wc_{I,{\mathsf 2}}=\{\1\}$ are are one-elemented. In fact this
follows from the fact that  the open $P$-orbits in $G/H_{I,j}$  are induced:
if we denote by $G_I\simeq \GL(2,\R)$ the Levi for the parabolic defined by $I$, then the open $P$-orbits
on $G/H_{I,j}$ correspond to the open $P\cap G_2$ orbits in
$\GL(2,\R)/\OO(2)$ respectively $\GL(2,\R)/ \OO(1,1)$. Both cases feature only one open orbit for $P\cap G_I$
and establish our claim.

\par Finally we determine $\sF_{I,{\mathsf 1}}$ and $\sF_{I,{\mathsf 2}}$.  Since $F=\{t_1, t_2, t_3\}$ with $t_i t_j=t_k$ for all
$i, j,k$ pairwise different, we readily deduce that $F_{\R,I}=F_{I,{\mathsf 1}}=F_{I,{\mathsf 2}}\simeq \Z_2$ is a group.
Recall that we described $\hat H_{I,{\mathsf 1}}$ and $\hat H_{I,{\mathsf 2}}$ already in Example \ref{ex SL3}.
From that we deduce that $\hat H_{I,{\mathsf 1}}/ H_{I,{\mathsf 1}}\simeq A_I$ is connected and thus $F(I)_{\mathsf 1}=\{{\mathsf 1}\}$. In particular,
$\sF_{I,{\mathsf 1}}\simeq \Z_2$.

On the other hand we have
$$u=\begin{pmatrix} 1 & 0&0 \\ 0& 0& 1\\ 0& -1 &0\end{pmatrix}\in \hat H_{I,{\mathsf 2}}$$
as it preserves the diagonal quadratic form $(0,1,-1)$ projectively (i.e. up to sign).
Since $u\not \in H_{I,{\mathsf 2}}$ and commutes with $A_I=\{ \diag (t^{-2}, t, t): t>0\}$ we thus have $F(I)_{\mathsf 2}\simeq \Z_2$. In particular,
$\sF_{I,{\mathsf 2}}=\{\1\}$.
\end{ex}

\begin{rmk} The above example shows that the group $A(I)=A_I \times F(I)$ is sensitive to
the orbit type in $\sC_I$.  More explicitly, we do not have $A(I)\simeq \hat H_{I,\sc}/ H_{I,\sc}$
for all $\sc\in \sC_I$. \end{rmk}

\section{Abstract Plancherel theorem and tempered representations}\label{Section AbsPlanch}

This section has several parts.  We begin with a brief recall on Banach representations and their
 smooth vectors, followed by a recap of smooth completions of Harish-Chandra modules.
 Then we turn our attention to  the abstract Plancherel theorem
for real spherical spaces. In fact there is no much difference to the case of a general unimodular
homogeneous space and "real spherical"  only enters via finite multiplicities.
Finally we recall the basic tempered theory for homogeneous spaces, initiated by Bernstein \cite{B}
in a general setup, and then made concrete for real spherical spaces
 in \cite{KKSS2}.

\subsection{Generalities on Banach representations and their smooth vectors}
We begin with a few facts on Banach representations of a Lie group $G$.
%Let $E$ be a Banach or a Fr\'echet  space.
By a Banach (or a Fr\'echet) representation of a Lie group $G$ we understand a continuous linear action

$$G \times E \to E,  \ \   (g,v) \mapsto \pi(g) v\, $$
on a Banach (or Fr\'echet) space $E$.
As customary we use the symbolic pair $(\pi, E)$ to denote the representation. Sometimes we abbreviate and use
$g\cdot v$ instead of $\pi(g)v$.

\par Let now $(\pi, E)$ be a Banach representation.  Further we fix with $p$ a norm which induces the topology on $E$.
In case $E$ is a Hilbert space and $p$ originates from the defining scalar product,
then we say $p$ is the Hermitian norm on $E$.
As the space $E$ does not necessarily allow an action of the Lie algebra we pass to the subspace $E^\infty\subset E$
of smooth vectors.  Here $v\in E$ is called smooth provided the $E$-valued orbit map $f_v: G \to E, \ \ g\mapsto \pi(g)v$
is smooth. In this sense we obtain a $G$-invariant subspace $E^\infty\subset E$ which is dense in $E$.
The space $E^\infty$ carries a Fr\'echet topology
for which the $G$-action is smooth.
For further reference we briefly recall a few standard possibilities on how to define the Fr\'echet topology. To begin with let
$\B:=\{ X_1, \ldots, X_n\}$ be an ordered basis of $\gf$. For a multi-index $\alpha\in \N_0^n$ we set
${\bf X}^\alpha:=X_1^{\alpha_1}\cdot \ldots \cdot X_n^{\alpha_n}\in \U(\gf)$.
For each $k\in \N_0$ we now define a norm on $E^\infty$ by

$$p_{\B, k}(v):=  \Big(\sum_{\alpha\in \N_0^n \atop |\alpha|\leq k}  p( {\bf X}^\alpha\cdot  v)^2 \Big)^{1\over 2} \qquad  (v \in E^\infty)\, .$$
Notice that $p_{\B, k}$ is Hermitian in case $p$ is Hermitian. If $\Cc$ is any other choice of ordered
basis we note that there exist constants $C_k =C_k(\B, \Cc)>0$, depending
on $\B$ and $\Cc$ but not on the space $E$ and its norm, such that
${1\over C_k}  p_{\B, k} \leq p_{\Cc, k} \leq C_k p_{\B, k}$ for all $k\in \N_0$.
In particular the locally convex topology on $E^\infty$ induced from the family $(p_{\B, k})_{k\in \N_0}$
does not depend on the particular choice of $\B$.
In the sequel we fix a basis $\B$, set $p_k:=p_{\B, k}$, and refer to $p_k$ as a $k$-th Sobolev norm
of $p$.  We denote by $E_k$ the completion of $E^\infty$ with respect to the norm $p_k$.  Note that
$G$ leaves $E_k$ invariant and defines  a Banach representation $(\pi_k, E_k)$ of $G$.
It follows that the Fr\'echet representation $(\pi^\infty, E^\infty)$ is of moderate growth (see \cite[Lemma 2.10]{BK}).
\par A second possibility to define the Fr\'echet structure is by Laplace Sobolev norms.  Let
\begin{equation} \label{def Delta}\Delta:= - (X_1^2 + \ldots + X_n^2)\in \U(\gf)\end{equation}
be a Laplace element attached to the basis $\B$, and set
\begin{equation}\label{def DeltaR} \Delta_R=\Delta+R^2\cdot\1\end{equation}
for $R\in\R$.
We recall the following from \cite[Cor. 3.3, Rem. 3.4]{GK}.

\begin{lemma} Let $(\pi, E)$ be a Banach representation of a unimodular Lie group $G$. Then there exists
a constant $R_E\geq 0$ such that for all $R>R_E$ the operator
$$ d\pi(\Delta_R): E^\infty\to E^\infty$$
is an isomorphism of Fr\'echet spaces. Moreover,  one can take $R_E=0$ in case $(\pi,E)$ is unitary.
\end{lemma}

From now on we assume that $G$ is a unimodular Lie group. For a Banach representation $(\pi, E)$ and
fixed $R>R_E$, we define
Laplace Sobolev norms of even order for any $k\in \Z$ by

\begin{equation} \label{def  Laplace} ^\Delta p_{2k}(v):=  p( \Delta_R^k v) \qquad (v\in E^\infty)\, .\end{equation}
Strictly speaking $^\Delta p_{2k}$ depends on $R>R_E$ but we suppress this in the notation. In case
$(\pi, E)$ is unitary we use $R=1$ and thus $^\Delta p_{2k}(v) = p(\Delta_1^k v)$.

For $k\geq 0$, it is clear that $^\Delta p_{2k} \leq c_k\cdot p_{2k}$ for a constant $c_k>0$ which is independent of $p$ and $E$.

Further, for $k\geq 0$
\cite[Prop. 4.12]{GK} yields constants $C_k >0$ only depending on
$\B$ and not on $E$ or $p$ such that
\begin{equation} \label{Laplace bound}  p_{2k} (v)\leq  C_k\cdot   {^\Delta p}_{2k +n^* }(v) \qquad (v\in E^\infty)\, \end{equation}
where

\begin{equation}\label{defi nstar}
n^*=\min\{k\in 2\N \mid  1+\dim G\le k\}
\end{equation}

\par For the rest of this section we request that $G$ is real reductive and  $G<\GL(m,\R)$
for some $m$. In this situation we take the basis $\B=\{X_1, \ldots, X_n\}$ such that the Laplace element $\Delta$ as defined in \eqref{def Delta}
satisfies
$$\Delta=  -\Cc_G +2\Cc_K$$
with $\Cc_G$ and $\Cc_K$ appropriate Casimir elements (unique if $\gf$ and $\kf$ are semisimple).

\begin{lemma}\label{unitary Laplace inequality2}
Assume $(\pi,E)$ is irreducible and unitary and let $p$ be any continuous $K$-invariant Hermitian
norm on $E^\infty$.  Let $R>0$. Then for each $k\in\N$
there exists a constant $C=C(k,R)>0$, independent of $p$ and $\pi$,  such that
$$ p( \Delta_R^kv)\leq C p(\Delta_1^k v) \qquad (v\in E^\infty)\, .$$
\end{lemma}

\begin{proof} It suffices to prove this for $k=1$. Notice that any
$v\in E^\infty$ admits a convergent expansion $v=\sum_{\tau\in \hat K} v_\tau$ in $K$-types which is orthogonal
with respect to any $K$-invariant Hermitian norm on $E^\infty$.
Since $\Delta_R$ is $K$-invariant,  the norm
$p(\Delta_R\cdot)$ is $K$-invariant and Hermitian.
Hence it suffices to show that $p(\Delta_R v)\leq C p(\Delta_1 v)$ for $v$ belonging to a $K$-type $E[\tau]$.
Then both $\Cc_G$ and $\Cc_K$ act
by scalars on $E[\tau]$. Hence $\Delta_R v=(c_\tau+R^2) v$ for some scalar $c_\tau$, which has to be
 $\ge 0$ as the representation $\pi$ was unitary:  use $\la \Delta v, v\ra \geq 0$ for all
 $v\in E^\infty$ and  $\la\cdot, \cdot\ra$
 a unitary inner product on $E$. Then
 $$p(\Delta_R v)=(c_\tau+R^2)p(v)\le C(c_\tau+1)p(v)=Cp(\Delta_1 v)$$
for $C=\max\{1,R^2\}$ and the lemma follows.
\end{proof}

\subsection{Smooth completions of Harish-Chandra modules and spherical pairs}

We move on to Harish-Chandra modules  and their canonical smooth completions.
A useful reference for the following summary might be \cite{BK}.

\par If $V$ is a complex vector space and $p$ is a norm on $V$, then we denote by $V_p$ the Banach
completion of the normed space $(V,p)$.

\par  Let $V$ be a Harish-Chandra module (with regard to a fixed choice of a maximal compact group $K$ of $G$).
A norm $p$ on $V$ is called $G$-continuous provided the infinitesimal action of $\gf$ on $V$ exponentiates
to a Banach representation of $G$ on $V_p$.   Note that every Harish-Chandra module admits a
$G$-continuous norm, as a consequence of the  Casselman embedding theorem.

\par The Casselman-Wallach globalization theorem asserts that the space  of smooth vectors $V_p^\infty$ is independent of the particular $G$-continuous norm $p$,
i.e.~if $q$ is another $G$-continuous norm, then the identity map  $V\to V$ extends to a
$G$-equivariant isomorphism of Fr\'echet spaces $V_p^\infty \to V_q^\infty$.
Stated differently, up to $G$-isomorphism of Fr\'echet spaces, there is a unique Fr\'echet completion
$V^\infty$ of $V$ such that the $G$-action on $V^\infty$ is smooth and of moderate growth.

\par We extend $\af$ to an abelian subalgebra $\jf = \af +i\tf\subset \gf_\C$  with $\tf\subset \mf$ a maximal torus.
Note that $\jf_\C$ is a Cartan subalgebra of $\gf_\C$ for which the roots are real valued on $\jf$, i.e.
$\Sigma(\gf_\C, \jf_\C) \subset \jf^*$.  We denote by $\W_\jf = \W(\gf_\C,\jf_\C)$ the corresponding Weyl group and
let $\rho_\jf\in \jf^*$ be a half-sum with $\rho_\jf|_\af=\rho$,
where $\rho$ is the half sum defined by $\nf$.

\par Assume now that $V$ is an irreducible Harish-Chandra module and denote by $\Zc(\gf)$ the center of
$\U(\gf)$.   By the Schur-Dixmier lemma the elements of  $\Zc(\gf)$ act by scalars on $V$ and we thus
obtain an algebra morphism $\chi_V:  \Zc(\gf) \to \C$, the infinitesimal character of $V$.
Via the Harish-Chandra isomorphism
we identify $\Zc(\gf)\simeq S(\jf_\C)^{\W_\jf}$, and consequently we may identify
$\chi_V$ with an element of $\jf_\C^*/ \W_\jf$.

Let $V$ be an irreducible Harish-Chandra module and $V^\infty$ its  canonical smooth completion.
Further let $V^{-\infty}:= ({V^\infty})'$ be the  continuous dual of $V^{\infty}$ and let
$\eta\in (V^{-\infty})^H$ be an $H$-fixed element.  We refer to $(V,\eta)$  as a {\it spherical pair} provided $\eta\neq 0$.

Let now $(V,\eta)$ be a spherical pair and $v\in V^\infty$. We form the generalized
matrix coefficient

$$m_{v,\eta}(g\cdot z_0):= \eta( g^{-1}\cdot v)   \qquad (g\in G)$$
which is a smooth function on $Z$.

\subsection{Abstract Plancherel theory}\label{subs APt}
We denote by $\hat G$ the unitary dual of $G$ and pick for every equivalence class $[\pi]$ a
representative $(\pi, \Hc_\pi)$, i.e. $\Hc_\pi$ is a Hilbert space  and $\pi: G \to U(\Hc_\pi)$ is an
irreducible unitary representation in the equivalence class of $[\pi]$.  We denote by $(\oline \pi,
\Hc_{\oline \pi})$ the dual representation.
We recall the $G$-equivariant  antilinear equivalence
$$\Hc_\pi \to \Hc_{\oline \pi}, \ \ v \mapsto \oline v:=\la \cdot , v\ra_{\Hc_\pi}$$
which induces the $G$-equivariant antilinear isomorphism:

$$\Hc_\pi^{-\infty}  \to  \Hc_{\oline \pi}^{-\infty},  \ \ \eta\mapsto \oline \eta; \ \oline \eta(\oline v)
:=\oline{ \eta(v)}\, $$
and a linear embedding $\Hc_\pi^{\infty}\hookrightarrow \Hc_{\oline\pi}^{-\infty}$.

In this context we recall  the mollifying map

$$ C_c^\infty(G) \otimes \Hc_{\oline \pi}^{-\infty}  \to \Hc_\pi^\infty
\subset\Hc_{\oline\pi}^{-\infty},  \ \ f\otimes \oline \eta \mapsto
\oline\pi(f)\oline \eta:=\int_G  f(g) \oline \eta (\oline \pi(g)^{-1}\cdot) \ dg\, . $$
The mollifying map restricted to $H$-invariants induces a map
$$ C_c^\infty(G/H) \otimes (\Hc_{\oline \pi}^{-\infty})^H  \to \Hc_\pi^\infty\, ,$$
$$ F\otimes \oline \eta \mapsto
\oline\pi(F)\oline \eta:=\int_{G/H}  f(gH) \oline \eta (\oline \pi(g)^{-1}\cdot) \ d(gH)\, . $$

The abstract Plancherel Theorem for the unimodular real spherical space $Z=G/H$
asserts the following (see \cite{Penney}, \cite{vanDijk}, or \cite[Section 8]{KS2}) :
There exists a Radon measure $\mu$
on $\hat G$ and for
every $[\pi]\in\hat G$ a Hilbert space $\M_{\pi} \subset (\Hc_{\pi}^{-\infty})^H $,
depending measurably on $[\pi]$,
(note that $(\Hc_{\pi}^{-\infty})^H$ is finite dimensional
\cite{KO}, \cite{KS1}), such that with the induced Hilbert space structure on
$\Hom(\M_{\oline \pi}, \Hc_\pi) \simeq \M_{\pi}\otimes \Hc_\pi$
the Fourier transform

$$ \F: C_c^\infty(Z) \to  \int_{\hat G}^\oplus  \Hom(\M_{\oline \pi}, \Hc_\pi) \ d\mu(\pi) $$
$$ F\mapsto \F(F)= (\F(F)_\pi)_{\pi \in \hat G};  \ \F(F)_\pi(\oline\eta):= \oline \pi(F) \oline \eta\in \Hc_\pi^\infty$$
extends to a unitary $G$-isomorphism from $L^2(Z)$ onto
$\int_{\hat G}^\oplus  \Hom(\M_{\oline \pi}, \Hc_\pi) \ d\mu(\pi)$.

Moreover the measure class of $\mu$ is uniquely determined by $Z$ and we call $\mu$ a {\it Plancherel measure}
for $Z$.
Unique are also the {\it multiplicity  subspaces} $\M_{\pi} \subset (\Hc_{\pi}^{-\infty})^H$ for almost all $\pi$
together with their inner products up to positive scalar.

Note that by definition
\begin{equation}\label{inner product with matrix coefficient}
\langle F, m_{v,\eta} \rangle_{L^2(Z)}= \langle \F(F)_\pi(\bar\eta), v\rangle
 \qquad (F\in C_c^\infty(Z))\, ,
\end{equation} %\label{Fourier1}
for all $\eta\in \M_{\pi}, v\in \Hc^\infty_\pi$, and furthermore the {\it Parseval formula}
\begin{equation}\label{abstract Plancherel} \| F\|^2_{L^2(Z)}  =  \int_{\hat G}   \sH_\pi (F)  \ d \mu(\pi) \qquad (F\in C_c^\infty(Z))\, ,\end{equation}
where $\sH_\pi$ denotes the Hermitian
form on $C_c^\infty(Z)$ defined by
\begin{equation} \label{Hermitian sum} \sH_\pi(F)= \sum_{j=1}^{m_\pi} \|\oline \pi(F) \oline \eta_j \|_{\Hc_\pi}^2\end{equation}
for $\oline \eta_1, \ldots,
\oline \eta_{m_\pi}$ an orthonormal basis of $\M_{\oline \pi}$. Observe that $\sH_\pi(F)$ is the Hilbert-Schmidt
norm squared of the operator $\F(F)_\pi: \M_{\oline \pi} \to \Hc_\pi$  and hence does not depend on the choice of the particular orthonormal basis.

\begin{rmk} (Normalization of Plancherel measure) As mentioned, only
the measure class of $[\mu]$ of $\mu$ is unique. With a choice of Plancherel measure $\mu\in [\mu]$
we pin down uniquely the $G$-invariant Hermitian forms $\sH_\pi$ on $\Hc_\pi\otimes \M_\pi$ for almost all
$\pi$. In particular, together with
a choice of an inner product on $\Hc_\pi$ (unique up to scalar by Schur's Lemma) we pin down the scalar product
on $\M_\pi$ uniquely.
\par Typically the $\Hc_\pi$ are induced representations with a preferred inner product, but in practice
there are several meaningful choices for the inner product on the multiplicity space (see Section \ref{group case} and Section \ref{section DBS}.)  A different choice of inner product on $\M_\pi$ then leads to a rescaling of $\mu$
in its measure class.
\end{rmk}

\begin{rmk} \label{F-inverse}{\rm (Fourier inversion)}  Let $f\in C_c^\infty(Z)$ be of  the form
$f= (F^**F)^H$ where $F\in C_c^\infty(G)$, $F^*(g)=\oline {F(g^{-1})}$ and the upper index $H$ denoting
the right $H$-average of $F^* * F$.  Then $f(z_0)=\|F^H\|^2_{L^2(Z)}$. Hence we deduce
from the Parseval formula \eqref{abstract Plancherel} for all
$f\in C_c^\infty(Z)$ the inversion formula

\begin{equation}  f(z_0) = \int_{\hat G}  \sum_{i=1}^{m_\pi} \Theta_\pi^i(f) \ d\mu(\pi) \end{equation}
where $\Theta_\pi^i$ is the {\it  spherical character}, i.e. the left $H$-invariant distribution

$$ \Theta_\pi^i(f) = \eta_i(\oline \pi(f) \oline \eta_i) \qquad (f\in C_c^\infty(Z))\, .$$
\end{rmk}

\subsection{Tempered norms}
We recall the standard tempered norms on $Z$.  Using the weight functions
$\w$ and $\v$ from \cite{KKSS2} Sections 3 and 4, the following norms on $C_c^\infty(Z)$ are attached to
a parameter $N\in \R$:

\begin{align*}
q_N(f) &:=  \sup_{z\in Z}   |f(z)| \, \v(z)^\frac12   ( 1 +\w(z))^N\, ,\\
p_N(f) &:= \left(\int_Z  |f(z)|^2  (1 + \w(z))^N \ dz\right)^\frac12 .
\end{align*}

Note that the norm $p_N$ is $G$-continuous, $K$-invariant, and Hermitian.
We recall that the two families of Sobolev norms
$q_{N;k}$ and $p_{N;k}$ for $(N,k)\in\R\times\N_0$ define the same topology on
$C_c^\infty(Z)$, and specifically for $k> {\dim G\over 2}$ we recall the inequality

\begin{equation} \label{Sob comparison} q_N(f) \leq  C  p_{N;k}(f)  \qquad (f\in C_c^\infty(Z)) \end{equation}
for a constant $C$ only depending on $k$ and $N$ (see \cite[Lemma 9.5]{KS2} and its proof).

\par We denote by $L_{N;k}^2(Z)$ the completion of $C_c^\infty(Z)$  with respect to
$p_{N;k}$.    We wish to define $L_{N;k}^2(Z)$ and $p_{N;k}$ as well for $k\in -\N$, and we do that
by duality. Given the invariant measure on $Z$, the dual
$L_N^2(Z)'$ is canonically isometric isomorphic
to $L_{-N}^2(Z)$ via the equivariant bilinear pairing
$$L_N^2(Z)\times L_{-N}^2(Z)\to \C, \ \ (f, g) \mapsto \int_Z f(z) g(z)\ dz\, .$$
This leads
to the definition

\begin{equation} \label{negative space} L_{N;-k}^2(Z):=  L_{-N; k}^2(Z)' \qquad (k\in \N)\end{equation}
with
\begin{equation}\label{negative norms}  p_{N;-k} (f) :=\sup_{\phi \in L_{-N;k}^2(Z)\atop
p_{-N;k}(\phi)\leq 1} \left|\int_Z f(z) \phi(z) \ dz\right| \, .\end{equation}

\subsection{ Negative Sobolev norms}The definition of the negative Sobolev norms
$p_{N;-k}$ for the norm $p_N$ fits into a general pattern which we recall in this Subsection.
Given a Banach representation $(\pi, E)$ and a $G$-continuous norm $p$ on $E$ we define the dual
norm $p'$ of $p$ on the continuous dual $E'$ as usual:
$$p'(\lambda)=\sup_{p(v)\leq 1} |\lambda(v)|\qquad (\lambda\in E').$$
In the sequel we assume that $p$ is a Hermitian norm. This guarantees in particular
that the dual action of $G$ on $E'$ is continuous, i.e. $(\pi', E')$ is a representation.  Further we retrieve $p$ from $p'$ via
$p= (p')'$. For any $k\in \N_0$ we write $p'_k:=(p')_k$ for the $k$-th Sobolev norm of the dual norm $p'$ and define the
negative Sobolev norm $p_{-k}$ of $p$ by
\begin{equation}\label{def Sob negative} p_{-k}(v) := (p'_k)'(v) \qquad (v\in E)\, .\end{equation}
Recall that we define Laplace Sobolev norms $^\Delta p_{2k}$ for all integers $k\in \Z$.

\begin{lemma} \label{lemma Sob negative} Let $(\pi, E)$ be a Hilbert representation of $G$ and $p$ a corresponding Hermitian norm.
Then for all $k\in \N_0$ there exists a constant $C_k>0$ such that
$$ ^\Delta p_{-2k-n^*} (v) \leq C_k p_{-2k}(v) \qquad (v\in E^\infty)\, .$$
\end{lemma}

\begin{proof} In view of the definition of the negative Sobolev norm $p_{-2k}$ in \eqref{def Sob negative}
this follows from \eqref{Laplace bound} applied to the dual norm $p'$
and the observation that $$(^\Delta p'_{2k})' = {}^\Delta p_{-2k}$$ for all $k\in\N_0$.
\end{proof}

\begin{lemma}\label{lemma Sobolev norms inequality}
Let $(V,\eta)$ be a spherical pair where
$V=V_\pi$ is the Harish-Chandra module of a unitary irreducible
representation $\pi$,  and let
$N\in\R$ be such that
$p_{N}(m_{v,\eta})<\infty$ for all $v\in V^ \infty$.
Then for each $2k> n^*$
there exists a constant $C>0$, depending on $k$ but not on $(V,\eta)$ and $N$, such that
\begin{equation}\label{Sobolev norms inequality}
 p_{N}(m_{v,\eta}) \le C p_{N; -2k+n^*}(m_{\Delta_1^k v,\eta}) \qquad(v\in V^\infty).
\end{equation}
\end{lemma}

\begin{proof}  In general we have for all $f\in E^\infty =L^2_N(Z)^\infty$ and fixed $R>R_E$
$$ p_N(f)= p_N(\Delta_R^{-k} \Delta_R^k f) =  {^\Delta p}_{N; -2k}(\Delta_R^k f)\, .$$
Upon applying Lemma \ref{lemma Sob negative} we obtain that
$$ p_N(f) \leq C p_{N; -2k + n^*}(\Delta_R^k f)\, .$$
Specifically for  $f=m_{v,\eta}$ we arrive at
$$ p_{N}(m_{v,\eta}) \le C p_{N; -2k+n^*}(m_{\Delta_R^k v,\eta}) \qquad(v\in V^\infty)\, .$$
Now $q(v):= p_{N; -2k +n^*} (m_{v,\eta})$ defines a $K$-invariant continuous Hermitian norm on
$V^\infty$ and thus we may replace $R$ by $1$ according to Lemma \ref{unitary Laplace inequality2}. 
\end{proof}

\subsection{Tempered pairs}
We now define

\begin{equation} \label{def NZ}  N_Z:=2 \rank_\R Z +1\qquad k_Z:={\frac12 \dim\gf}.\end{equation}
Then for all $N\geq N_Z$ and $k  > k_Z$ it follows from \cite[Prop. 9.6]{KS2} combined with \cite[Th. 1.5]{B}
that for $\mu$-almost all
$[\pi]\in \hat G$, the $\pi$-Fourier transform

$$\F_\pi: C_c^\infty(Z) \to \Hom (\M_{\oline \pi}, \Hc_\pi)$$
extends continuously to $L_{N;k}^2(Z)$ and that the corresponding inclusion
\begin{equation} \label{HS11}  L_{N;k}^2(Z) \to \int_{\hat G}^\oplus \Hom(\M_{\oline \pi}, \Hc_\pi)\   d\mu(\pi)\end{equation}
is Hilbert-Schmidt (in the sequel HS for short).

We wish to make this fact a bit more concrete in the context of the Hermitian forms
$\sH_\pi$. For that purpose we fix $N$ and $k$ as above and
denote by $\| \sH_\pi\|_{{\rm HS},  N; k}$ the HS-norm of the operator
$F\otimes \bar\eta\mapsto\bar\pi(F)\bar\eta$ from $L^2_{N;k}(Z) \otimes\M_{\bar\pi}$ to $\Hc_\pi$,
that is
$$ \| \sH_\pi\|^2_{{\rm HS}, N; k} := \sum_{n\in \N}   \sH_\pi(F_n)$$
for any orthonormal basis $(F_n)_{n\in \N}$ of $L_{N;k}^2(Z)$.
The fact that \eqref{HS11} is HS then
translates into the {\it a priori bound}
\begin{equation} \label{global a-priori}
\int_{\hat G}  \| \sH_\pi\|^2_{{\rm HS},N; k}  \ d\mu(\pi)<\infty.\end{equation}

By \eqref{inner product with matrix coefficient} we further infer
\begin{equation}\label{HiSch estimate}
\sum_{j=1}^{m_\pi} p_{-N;-k}(m_{v,\eta_j})^2=
\sum_{j=1}^{m_\pi} \sup_{F\in C_c^\infty(Z)\atop p_{N;k}(F)\le 1} |\langle \F(F)_\pi(\bar\eta_j), v\rangle|^2
\le \|\sH_\pi\|^2_{{\rm HS}, N; k}\, \|v\|_{\Hc_\pi}^2
\end{equation}
for $\mu$-almost all $[\pi]\in\hat G$, all $v\in\Hc_\pi^\infty$, and $\eta_1,\dots,\eta_{m_\pi}$
an orthonormal basis of $\M_\pi$.

Hence it follows
from  (\ref{global a-priori}) that

\begin{equation} \label{global a-priori 2}
\int_{\hat G} \,\sup_{\eta\in\M_\pi\atop \|\eta\|\le 1}\, \sup_{v\in\Hc_\pi^\infty\atop \|v\|\le 1}\,
p_{-N; -k} (m_{v,\eta})^2 \  d\mu(\pi)< \infty\, .\end{equation}
Consequently $p_{-N;-k}(m_{v,\eta})<\infty$ for $N\geq N_Z$ and $k> k_Z$, for
all $v\in\Hc_\pi^\infty$, $\eta\in\M_\pi$, and $\mu$-almost all $[\pi]$ .

In particular with any $k$ with $2k - n^* > k_Z$
we obtain for $N\geq N_Z$ we obtain from Lemma \ref{lemma Sob negative}
that
\begin{align*} p_{-N}(m_{v,\eta}) &= p_{-N}( \Delta_R^{-k} \Delta_R^k m_{v,\eta})
= {^\Delta p}_{-N; 2k}( \Delta_R^k m_{v,\eta})\\
&\leq C p_{-N; -2k + n^*} ( m_{\Delta_R^kv,\eta})<\infty\end{align*}
for all $v\in \Hc_\pi^\infty$ and $\mu$-almost all $[\pi]$.

\begin{definition}\label{defi temp pair} (cf.  \cite[Def. 5.3]{KKSS2} and \cite[Sect. 3.3]{DKS})
Let $(V,\eta)$ be a spherical pair. We say that  $\eta$ is {\it tempered} or
$(V,\eta)$ is a {\it tempered pair} provided that

$$ p_{-N}(m_{v,\eta})<\infty \qquad (v\in V^\infty)$$
for some $N\in\R$.
\end{definition}

The tempered functionals make up a subspace of $(V^{-\infty})^H$ which we denote by
$(V^{-\infty})^H_{\rm temp}$.  We conclude that $\M_\pi\subset (V^{-\infty})^H_{\rm temp}$
for almost all $\pi$.

\begin{rmk}  (a) (About the inclusion $(V^{-\infty})^H_{\rm temp}\subset (V^{-\infty})^H$).  For a tempered
pair $(V,\eta)$ the inclusion $\{0\}\neq (V^{-\infty})^H_{\rm temp}\subset (V^{-\infty})^H$ can be strict.
This already appears for the rank one symmetric spaces $Z= \SO_0(1,n)/\SO_0(1,n-1)$ when $n\geq 4$,
in which case there exists an irreducible Harish-Chandra module which has multiplicity one in $L^p(Z)$
for $p\le n-1$ and multiplicity two for $p>n-1$. For details of this example we refer to \cite{KrKS}.
\par (b) (Tempered Frobenius reciprocity). If we denote by $C^\infty_{\rm temp}(Z) =\bigcup_{N\in \R}
L^2_{N}(Z)^\infty$ the $G$-module of smooth functions of moderate growth on $Z$, then we recall from \cite[3.10]{DKS} the following variant
of Frobenius reciprocity for Harish-Chandra modules $V$:
$$ \Hom (V^\infty, C^\infty_{\rm temp}(Z)) \simeq  (V^{-\infty})_{\rm temp}^H$$
with $\Hom$ referring to continuous morphisms of $G$-modules.
\par (c) (About the inclusion $\M_\pi \subset (V^{-\infty})^H$). For symmetric spaces one has equality
\begin{equation} \label{mult equal} \M_\pi= (V^{-\infty})^H_{\rm temp} \quad\text{ for almost all $\pi$}\, .\end{equation}
This was established by forming wave packets, which was
a central technical step in the proof of the Plancherel formula for symmetric spaces. Since we follow another
approach towards the Plancherel formula in this article,  the equality \eqref{mult equal} together with an explicit description of $(V^{-\infty})^H$ is not an issue in the underlying treatment. However, we do expect that in general $\M_\pi =(V^{-\infty})^H_{\rm temp}$ for almost all $\pi$. 
\end{rmk}

\section{Constant term approximations}\label{section: ct}

 In this section we review the constant term approximation of \cite{DKS} which is a central
technical tool for this paper. In fact, by using our geometric results from
Section \ref{structure of Z_I} on the stabilizer $H_I$, and our combinatorial results
on the open $P$-orbits of Section \ref{subsection WI}, we are able to refine slightly the results
from \cite{DKS}.
\par Recall from \eqref{full deco W} that the set of open $P$-orbits
$\Wc$ of $Z$ admits a combinatorial decomposition $\Wc=\coprod_{\sc\in \sC_I} \coprod_{\st\in \sF_{I,\sc} }
{\bf m}_{\sc, \st}(\Wc_{I,\sc})$.  For the sake of readability we first consider the part
${\bf m}(\Wc_I)\subset \Wc$ corresponding to $\sc=\st=\1$
and treat the notationally heavier case later.

\subsection{Notation}
Let $V$ be an irreducible Harish-Chandra module with smooth completion $V^\infty$ and dual
$V^{-\infty}$.

We recall that $(V^{-\infty})^H$ is a finite dimensional space  for any real spherical
subgroup $H\subset G$. Also we recall
that $A_I$ normalizes $H_I$.  Hence for any $I\subset S$ we obtain an action
of $A_I$ on $(V^{-\infty})^{H_I}$ by $a_I\cdot\xi= \xi(a_I^{-1}\cdot)$
for $\xi\in(V^{-\infty})^{H_I}$. Accordingly we can decompose
$\xi$ into generalized eigenvectors:

$$ \xi= \sum_{\lambda\in \af_{I,\C}^*}  \xi^{\lambda},$$
where $\xi^\lambda$ has generalized eigenvalue $\lambda$. We set
\begin{equation}\label{defi generalized eigenvalues}
\E_{\xi}:=\{ \lambda\in \af_{I,\C}^*\mid \xi^{\lambda}\neq 0\}\, .
\end{equation}

For  $\eta\in(V^{-\infty})^{H}$ and  $w\in \Wc$ we set $\eta_w:=w\cdot \xi$ and note that $\eta_w$ is $H_w$-fixed.

\subsection{Base points from ${\bf m}(\Wc_I)$}
We recall from \eqref{bfm} the injective map
${\bf m}:  \Wc_I \to \Wc$. Let now $w_I\in \Wc_I$ and $w={\bf m}(w_I)$.
Then, given $\xi \in (V^{-\infty})^{H_I}$ we note that
$\xi_{w_I}= w_I \cdot \xi$ is fixed under $(H_I)_{w_I}=(H_w)_I$, see \eqref{ConsisT1}.
Moreover $A_I$ normalizes  $(H_I)_{w_I}$ and we obtain from (a slight adaption of) \cite[Lemma 6.2]{KKS2} that
$(\xi^{\lambda})_{w_I}$ is a generalized eigenvector for the $\af_I$-action to the same spectral value
$\lambda$.

We recall that $\rho|_{\af_H}=0$ by the request that $Z$ is unimodular, see \cite[Lemma 4.2]{KKSS2}.
This allows us  to consider $\rho$ as
a functional on $\af_Z=\af/\af_H$ as well.
In the sequel if not stated otherwise we take $N=N_Z$ (see \eqref{def NZ}).

\begin{theorem} \label{loc ct temp} {\rm(Constant term approximation)}
Let $Z=G/H$ be a unimodular real spherical space and $I\subset S$.  Then for all irreducible
Harish-Chandra modules $V$ there exists  a unique linear map
$$ (V^{-\infty})^H_{\rm temp}  \to (V^{-\infty})^{H_I}_{\rm temp},  \ \ \eta\mapsto \eta^I$$
with the following property.  For all compact sets $\Omega\subset G$ and $\Cc_I\subset \af_I^{--}$
there exist  $k\in \N$, $\e>0$, and  $C>0$,
such that
\begin{equation} \label{cta}|m_{v,\eta}(g a_I w\cdot z_0) -  m_{v,\eta^I} (g a_I w_I\cdot z_{0,I})|
\leq C a_I^{(1+\e)\rho} p_{-N;k} (m_{v,\eta})\end{equation}
for all $\eta\in (V^{-\infty})^H_{\rm temp}$, $v\in V^{\infty}$, $g\in \Omega$,
$a_I\in A_I^{--}$ with $\log a_I \in \R_{\geq 0}\Cc_I$, and
$w={\bf m}(w_I)\in {\bf m}(\Wc_I)\subset \Wc$.
The constants $k$, $\e$, and $C$ can be chosen independently of $V$.

Moreover, with $\chi_V\in \jf_\C^*/\W_\jf$ the infinitesimal character of $V$ one has
\begin{equation}\label{exponents} \E_{\eta^I}\subset  (\rho|_{\af_I}  + i\af_I^*) \cap (\rho - \W_\jf \cdot \chi_V)|_{\af_I}\,,
\end{equation}
where $\E_{\eta^I}$ is defined by \eqref{defi generalized eigenvalues}.
Finally there is the consistency relation
\begin{equation} \label{consist}(\eta_w)^I= (\eta^I)_{w_I} \qquad (w={\bf m}(w_I)\in W)\, .\end{equation}
\end{theorem}

The constant term assignment
$$ (V^{-\infty})^H_{\rm temp}  \to (V^{-\infty})^{H_I}_{\rm temp},  \ \ \eta\mapsto \eta^I$$
is typically  neither injective nor surjective. Let us illustrate that in two examples
before giving the proof of the theorem.

\begin{ex} (a)  Let $H=K$ be a maximal compact subgroup of $G$ and $I=\emptyset$. Then $H_\emptyset = M\oline N$.
Now let  $V$ be a $K$-spherical tempered Harish-Chandra module. Then $\dim V^K=1$.
However, for generic $V$ we have $\dim (V^{-\infty})^{M\oline N}= |\W_\af|$ with
$\sW_\af$ the Weyl group of the restricted root system $\Sigma(\gf, \af)$. This shows that the constant term assignment
is typically not surjective.
\par\noindent (b) Tempered pairs $(V,\eta)$ of the twisted discrete series can be characterized by the vanishing of the
constant term assignments for $I\neq S$, see \cite[Th. 5.12]{DKS}. In particular, if $(V,\eta)$ belongs to the
discrete series of $Z$, then we have $\eta^I=0$ for all $I\neq S$. Hence the constant term assignment is
typically not injective.
\end{ex}

\begin{proof} The existence of an $\eta^I\in(V^{-\infty})^{H_I}_{\rm temp}$ satisfying
\eqref{cta}, \eqref{exponents} and \eqref{consist} is  proved in \cite{DKS}, with the exception that invariance of
$\eta^I$ is only shown for the identity component of $H_I$.  In more precision, \eqref{cta} for $H_I$
replaced  by $(H_I)_0$ is
\cite[Th. 7.10]{DKS} with the caveat
that  in \cite{DKS} the norms to bound the right hand side of \eqref{cta} are Sobolev norms of $q_{-N}$ and not of $p_{-N}$.
However, the passage between
$q_{-N}$ and $p_{-N}$ is justified by the comparison of Sobolev norms in \eqref{Sob comparison} which is valid for any
$N\in \R$.
The inclusion of exponents \eqref{exponents} is part of the general theory in \cite{DKS} and the consistency
relation in \eqref{consist} is \cite[Prop. 5.7]{DKS}.

\par  We turn to the uniqueness of the map $\eta\to \eta^I$. We recall that $(V^{-\infty})^{H_I}$ is a finite dimensional $A_I$-module and thus
 \eqref{exponents} implies
that for any fixed $g\in G$ and $v\in V^\infty $ the map
$$A_I \ni  a_I\mapsto m_{v,\eta^I}(ga_I \cdot z_{0,I})= m_{v,a_I\cdot \eta^I}(g \cdot z_{0,I})$$
is an exponential polynomial with normalized unitary exponents and hence unique as
constant term approximation of $m_{v,\eta}(ga\cdot z_0)$, see Remark \ref{rmk unique approx} below.
In particular, $\eta^I$ is then uniquely determined by the approximation property \eqref{cta}.
\par Finally we will  show that $\eta^I$ is in fact $H_I$-invariant
for all $\eta\in(V^{-\infty})^H_{\rm temp}$.
We do this for the case of $w_I=w=\1$, the more general
case being an easy adaption. We recall Lemma \ref{lemma H_I limit} and the notation used therein.

\par Let $X_I\in \cf_I^{--}$ corresponding to $-{\bf e}_I$ under the identification $\af_I\simeq V_I$.
Set $a_t:=\exp(tX_I)$ for $t\geq 0$.  First notice that both
$m_{v,\eta^I} (gh_I a_t\cdot z_{0,I})$ and $m_{v,\eta^I} (gx_t a_t\cdot z_{0,I})$
approximate
$$m_{v,\eta}(gh_Ia_t \cdot z_0) = m_{v,\eta}(gx_t a_t\cdot z_0)$$
via \eqref{cta}, and thus we get

\begin{equation} \label{inv1} a_t^{-\rho} |m_{v,\eta^I}(gh_Ia_t\cdot z_{0,I})- m_{v,\eta^I}
(gx_ta_t \cdot z_{0,I})| \leq C  e^{-\e t}\end{equation}
for some $C,\e >0$. On the other hand, the coefficients of the
exponential polynomial
$$a_I\mapsto a_I^{-\rho} m_{v,\eta^I}(gx_ta_I\cdot z_{0,I})=a_I^{-\rho}m_{(gx_t)^{-1}v, a_I\cdot \eta^I}(z_{0,I})$$
with unitary exponents depend smoothly on $gx_t$. Hence it follows, after possibly shrinking $\e$,
 from \eqref{normal-approx}
that
\begin{equation} \label{inv2}|a^{-\rho} m_{v,\eta^I} (gx_t a\cdot z_{0,I})- a^{-\rho}  m_{v,\eta^I} (ga\cdot z_{0,I})|\leq C e^{-\e t}\end{equation}
for all $a\in A_I$.  Now the $H_I$-invariance of $\eta^I$ follows from combining (\ref{inv1}) and (\ref{inv2}) together with the
before mentioned uniqueness.
\end{proof}

\begin{rmk}\label{rmk unique approx}  (Uniqueness of the constant term)  Let $f(a)$ be a function on $A_I$ and
$$F(a) = a^\rho \sum_{\lambda\in\E}
q_\lambda(\log a) a^\lambda\qquad (a \in A_I) $$
an exponential polynomial with unitary exponents, i.e.
$\E\subset i\af_I^*$ is finite and $q_\lambda$ are polynomial functions on $\af_I$.  In case there exists an $\e>0$ such that
\begin{equation} \label{unique approx} |f(a) - F(a)| \leq C  a^{(1+\e)\rho}\qquad (a \in A_I^-)\, ,\end{equation}
then $F$ is the unique exponential polynomial with normalized unitary exponents having
the approximation property \eqref{unique approx}. This is a consequence of the following basic lemma,
which we record without proof.\end{rmk}

\begin{lemma} \label{lemma basic ineq} Let $\Lambda\subset\R$ be a finite set
and for each $\lambda\in\Lambda$ let $q_\lambda\in\C[t]$ be a polynomial.
If there exist constants $\e,C>0$ such that
$$\Big|\,\sum_{\lambda\in\Lambda} q_\lambda(t)e^{i\lambda t}\,\Big| < C e^{-\e t}\qquad (t\geq 0)$$
then $q_\lambda=0$ for all $\lambda\in\Lambda$.
\end{lemma}

\subsection{General base points}\label{subsection all base points}

So far we have treated the constant term approximation through the base points
$z_w=w\cdot z_0$ for $w\in {\bf m}(\Wc_I)$. The general case is obtained by adapting the notation to the
partition $\Wc=\coprod_{\sc\in \sC_I} \coprod_{\st\in \sF_{I,\sc}}  {\bf m}_{\sc,\st}
(\Wc_{I,\sc})$ from \eqref{full deco W}.

\par For $\sc\in\sC_I$  and $\st \in \sF_{I,\sc}$ we define $w(\sc,\st):={\bf m}_{\sc, \st}(\1)\in \Wc$
and set  $z_{\sc, \st}=  w(\sc, \st)\cdot z_0$. Further we set $w(\sc)={\bf m}_{\sc, \1}(\1)\in \Wc$
and $z_\sc= w(\sc)\cdot z_0$. Let $H_{\sc, \st}$ and $H_\sc$ denote the $G$-stabilizers
of $z_{\sc,\st}$ and $z_\sc$ respectively.

\par Define for
$\eta\in(V^{-\infty})^{H}$ accordingly $\eta_{\sc, \st}:= w(\sc,\st)\cdot \eta$.  Notice that
$\eta_{\sc,\st}^I$ is invariant under $(H_{\sc,\st})_I$.  From \eqref{WWI2 general}
we infer further that  $(H_{\sc,\st})_I=H_{I,\sc}$ does not depend on $\st$.

\par As before we obtain
that $A_I$ normalizes  $(H_{I,\sc, \st})_{w_I}=(H_{I,\sc})_{w_I}$, so that $A_I$ acts
naturally on $(H_{I,\sc})_{w_I}$-invariant distribution vectors $\xi$ and yields
generalized eigenspace decompositions $\xi= \sum_{\lambda\in {\af_I}_\C^*} \xi^\lambda$.
Within the introduced terminology the general case of the constant term approximation then
reads as follows:

\begin{theorem} \label{loc ct temp2} {\rm(Constant term approximation - general version)}
Let $Z=G/H$ be a unimodular real spherical space and $I\subset S$.  Fix $\sc\in \sC_I$
and $\st\in \sF_{I,\sc}$.  Then for all irreducible
Harish-Chandra modules $V$ there exists  a unique linear map

$$ (V^{-\infty})^H_{\rm temp}  \to (V^{-\infty})^{H_{I,\sc}}_{\rm temp},  \ \ \eta\mapsto \eta_{\sc, \st}^I$$
with the following property: There  exist constants $\e>0$, $k\in \N$,
such that for all compact subsets $\Cc_I\subset \af_I^{--}$
and $\Omega\subset G$ there exists
a constant $C>0$,  such that

 \begin{equation} \label{cta2}|m_{v,\eta}(g a_I w\cdot z_0) -  m_{v,\eta_{\sc,\st}^I} (g a_I w_{I,\sc}\cdot z_{0, I, \sc})|
\leq C a_I^{(1+\e)\rho} p_{-N;k} (m_{v,\eta})\end{equation}
for all $\eta\in (V^{-\infty})^H_{\rm temp}$, $v\in V^{\infty}$, $g\in \Omega$,
$a_I\in A_I^{--}$ with $\log a_I \in \R_{\geq 0}\Cc_I$, and
$w={\bf m}_{\sc,\st}(w_{I,\sc})\in {\bf m}_{\sc,\st}(\Wc_{I,\sc})\subset \Wc$.
The constants $\e$, $k$, and $C$ can all be chosen independently of $V$.

Moreover, with $\chi_V\in \jf_\C^*/\W_\jf$ the infinitesimal character of $V$ one has
\begin{equation}\label{exponents2} \E_{\eta_{\sc,\st}^I}\subset  (\rho|_{\af_I}  + i\af_I^*) \cap (\rho - \W_\jf \cdot \chi_V)|_{\af_I}.
\end{equation}
Finally there is the consistency relation
\begin{equation} \label{consist2}(\eta_w)^I= (\eta_{\sc,\st}^I)_{w_{I,\sc}} \qquad (w={\bf m}_{\sc,\st}(w_{I,\sc})\in W)\, .\end{equation}
\end{theorem}

\begin{proof} By replacing $H_{I,\sc}$ with $H_I$ we may assume that $\sc=\1$. Let then
 $w\in {\bf m}_\st(W_I)$. The passage to $\st=\1$ is obtained via the material in Subsection \ref{subsection c=1} and
via the further base point shift $z_0\to z_\st$. By this we obtain a reduction
to Theorem \ref{loc ct temp}.
\end{proof}

\section{The main remainder estimate}\label{main remainder}

In this section we derive an important uniform estimate which is the key technical tool
for the results in the next section.  The estimate is based on the constant term approximation
of Section \ref{section: ct}.

\subsection{Adjustment of Haar measures}\label{measures}

We assume that $Z=G/H$ carries a $G$-invariant measure.
Then, according to \cite[Lemma 3.12]{KKS2}, the same holds for $Z_I:=G/H_I$.
%Let $z_{I,0}$ be the base point of $G/H_I$.
Since $L\cap H=L\cap H_I$ by Lemma \ref{equal L cap H},
we see that the $P$-orbits through $z_0$ and $z_{0,I}$ are isomorphic
as homogeneous spaces for $Q$, i.e.
\begin{equation}\label{Q-orbit iso}
P\cdot z_0= Q\cdot z_0\simeq Q/ L\cap H\simeq Q\cdot z_{0,I} =  P \cdot z_{0,I}\, .
\end{equation}

\noindent We fix the normalizations of the $G$-invariant measures on $Z$
and $Z_I$ such that on these open pieces they
coincide with a common Haar measure on $Q/L\cap H$, and we denote these measures
on $Z$ and $Z_I$ by $dz$ and $dz_I$, respectively.

\subsection{Right action by $A(I)$}\label{right action}

\par As $A(I)$ normalizes $H_I$ we obtain a right action of $A(I)$ on functions $f$ on $Z_I$  given by

$$ (R(a_I) f) (g\cdot z_{0,I}):= f( g a_I \cdot z_{0,I}) \qquad (g\in G, a_I \in A(I))\, .$$

\begin{lemma}\label{lemma ZI-int}  Let $f\in L^1(Z_I)$ and $a_I\in A(I)$.  Then
\begin{equation} \label{int ZI} \int_{Z_I}  (R(a_I)f)(z_I) \ dz_I =  |a_I^{2\rho}|  \int_{Z_I} f(z_I) \ dz_I\, \end{equation}
In particular, the normalized action $ f\mapsto |a_I^{-\rho}| R(a_I) f$ of $A(I)$ is unitary on $L^2(Z_I)$.
\end{lemma}
\begin{proof}  First note that $|a^\rho|=1$ for all $a\in T_Z=\exp(i\af_H^\perp)\subset \uA$.  Since elements
of $F(I)$ have finite order it is sufficient to consider
$a_I \in A_I\subset A(I)$. The first assertion then follows from \cite[Lemma 8.4]{KKS2}, and the second assertion is a consequence of the first.
\end{proof}

Fix an element $X\in \af_I^{--}$ and set $a_t:=\exp(tX)$ for $t\in\R$.   Let $f\in L^2(Z_I)$ and
define
\begin{equation} \label{defi f_t}
f_t(z):= a_t^{\rho} (R(a_t^{-1})f)(z),\quad (z\in Z_I).
\end{equation}
Notice that the assignment $f\mapsto f_t$ is $G$-equivariant and unitary by Lemma \ref{lemma ZI-int}.
In particular
\begin{equation} \label{match1}  \| f_t\|_{L^2(Z_I)}= \|f\|_{L^2(Z_I)} \qquad(t\in\R)\end{equation}
and, in case $f$ is smooth,
\begin{equation} \label{match equivariant} L_u f_t = (L_u f)_t  \qquad (u \in \U(\gf))\, .\end{equation}

\subsection{Matching of functions}\label{matching functions}
We recall from Section \ref{subsection WI} the injective map ${\bf m}: \Wc_I \to \Wc$ which matches the open $Q$-orbit
$Qw_I\cdot z_{0,I}=Pw_I\cdot z_{0,I}$ in $Z_I$ with the open $Q$-orbit $Qw\cdot z_0=Pw\cdot z_0$ in $Z$ where $w={\bf m}(w_I)$.
As in \eqref{Q-orbit iso} we have

\begin{equation} \label{QwwI} Qw\cdot z_0\simeq Q/ L\cap H\simeq Qw_I \cdot z_{0,I}\, .\end{equation}

Given a smooth function $f$ on $Z_I$ with compact support in $Q\Wc_I \cdot z_{0,I}\subset Z_I$
we define via \eqref{QwwI} a `matching' smooth function $F=\Phi(f)$ on $Z$ with compact support
in $Q{\bf m}(\Wc_I) \cdot z_0\subset Z$
by
\begin{equation}\label{match fF} F(q{\bf m}(w_I)\cdot z_0):= f(qw_I \cdot z_{0,I})\qquad (q\in Q)\, .
\end{equation}
Observe that the space spanned by the smooth functions on $Z_I$ with compact support contained
in the union of the open $Q$-orbits $Q\Wc_I\cdot z_{0,I}$
is dense in $L^2(Z_I)$.

\par

Since the invariant measures on $Z$ and $Z_I$ coincide on the open $Q$-orbits we get
\begin{equation*}  \|\Phi(f)\|_{L^2(Z)}= \|f\|_{L^2(Z_I)}.\end{equation*}
Together with \eqref{match1} this implies
for the function $f_t$ defined in \eqref{defi f_t}
\begin{equation} \label{match2} \|\Phi(f_t)\|_{L^2(Z)}=\|f\|_{L^2(Z_I)}\end{equation}
for all $t\in \R$.

The main result of this section is now reads as follows. Let $N=N_Z$ from \eqref{def NZ}.

\begin{theorem}\label{matching comparison} {\rm (Main remainder estimate)}
There exists $\epsilon>0$ with the following property.
Let $\Omega\subset Q$ be a compact set.
Then for every $s\in\R$ there exist $C>0$ and $m\in \N$ such that
for all
$f\in C_c^\infty(Z_I)$ with
$\supp f \subset \Omega \Wc_I \cdot z_{0,I}$,
all tempered pairs $(V,\eta)$, and all $v\in V^\infty$
the following equality holds

$$ \la \Phi(f_t), m_{v,\eta}\ra_{L^2(Z)}  =  \la f_t, m_{v,\eta^I}\ra_{L^2(Z_I)} + R(t)\qquad (t\ge 0)\, ,$$
with the remainder bounded by

$$| R(t)| \leq  C  e^{ -t \e} \, p_{-N; -s}(m_{v,\eta}) \, p_{N;m}(\Phi(f))\,.$$

\end{theorem}

Before giving the proof we observe the following corollary.
Recall from (\ref{Hermitian sum}) the Hermitian forms $\sH_\pi$ on $C_c^\infty(Z)$.
We fix an orthonormal basis $\eta_1, \ldots, \eta_{m_\pi}$ of $\M_\pi$ and define
a preliminary Hermitian form $\sH_\pi^{I, \rm pre}$ on $C_c^\infty(Z_I)$
by
\begin{equation} \label{Hermitian sum I-side} \sH_\pi^{I, \rm pre}(f)= \sum_{j=1}^{m_\pi}  \|\oline \pi(f) \oline \eta_j^I \|_{\Hc_\pi}^2
\qquad (f\in C_c^\infty(Z_I))\,. \end{equation}
Notice that $\sH_\pi^{I, \rm pre}$ is independent from the particular choice of the
orthonormal basis $\eta_1, \ldots, \eta_{m_\pi}$, being the Hilbert-Schmidt norm squared
of the linear map
$$\M_{\oline\pi} \to \Hc_\pi,\ \ \oline \eta\mapsto\oline\pi(f)\eta^I\, .$$

 We derive from Theorem \ref{matching comparison} and the global a priori bound
 \eqref{global a-priori}  that:

\begin{cor}  \label{main cor}
Let $\epsilon>0$ be as in Theorem \ref{matching comparison} and let
$f\in C_c^\infty(Z_I)$ with support in
$ Q\Wc_I \cdot z_{0,I}$. Then there exists a constant $C>0$ such that
$$\| f \|^2_{L^2(Z_I)} = \int_{\hat G}   \sH_\pi^{I, \rm pre} (f_t) \ d\mu(\pi)  + R(t)$$
with $|R(t)| \leq  C e^{-\e t}$ for all $t\geq 0$.
\end{cor}

\begin{proof} We first observe that by  \eqref{match2}
and \eqref{abstract Plancherel}-\eqref{Hermitian sum}
\begin{equation}\label{first observe}
\| f \|^2_{L^2(Z_I)} =
\int_{\hat G}   \sH_\pi(\Phi(f_t)) \, d\mu(\pi) =
\int_{\hat G}   \sum_{j=1}^{m_\pi} \|\oline \pi(\Phi(f_t))\oline \eta_j\|_{\Hc_\pi}^2\, d\mu(\pi).
\end{equation}
Hence we need to estimate the integral over $\pi\in\hat G$ of
$$
\sum_{j=1}^{m_\pi}
\Big(
\|\oline \pi(\Phi(f_t))\oline \eta_j\|_{\Hc_\pi}^2 - \|\oline \pi(f) \oline \eta_j^I \|_{\Hc_\pi}^2
\Big)\, .
$$
Using the identity $a^2-b^2=2a(a-b)-(a-b)^2$
together with Cauchy-Schwarz
and \eqref{first observe},
we see that it suffices
to show
\begin{equation}\label{suffices to show}
\bigg[\int_{\hat G}
\sum_{j=1}^{m_\pi}
\Big(
\|\oline \pi(\Phi(f_t))\oline \eta_j\|_{\Hc_\pi} - \|\oline \pi(f) \oline \eta_j^I \|_{\Hc_\pi}
\Big)^2\, d\mu(\pi) \bigg]^{1/2}
\le
C e^{-\e t}\,.
\end{equation}

From the dense inclusion $\Hc_\pi^\infty \subset \Hc_\pi$ and \eqref{inner product with matrix coefficient}  we obtain that
$$
\|\oline \pi(\Phi(f_t))\oline \eta_j\|_{\Hc_\pi}
= \sup_{v\in \Hc_\pi^\infty\atop \|v\|=1}
\la \oline \pi(\Phi(f_t))\oline \eta_j, v\ra_{\Hc_\pi}
= \sup_{v\in \Hc_\pi^\infty\atop \|v\|=1}
\la \Phi(f_t), m_{v,\eta_j}\ra_{L^2(Z)}
$$
and similarly
$$
\|\oline \pi(f)\oline \eta_j^I\|_{\Hc_\pi}
= \sup_{v\in \Hc_\pi^\infty\atop \|v\|=1}
\la f, m_{v,\eta_j^I}\ra_{L^2(Z_I)}\, .
$$
Let $s>k_Z$ (see \eqref{def NZ}). Now application of Theorem  \ref{matching comparison} implies
for all $t>0$
$$ \Big|\|\oline \pi(\Phi(f_t))\oline \eta_j\|_{\Hc_\pi}-
\|\oline \pi(f)\oline \eta_j^I\|_{\Hc_\pi} \Big|
\leq  C e^{-t\e} \sup_{v\in \Hc_\pi^\infty
\atop \|v\|=1} p_{-N; -s}(m_{v,\eta_j})\, ,$$
where $C>0$ depends on $f$, but not on $t$ or $\pi$.
Hence \eqref{suffices to show} follows from \eqref{HiSch estimate}
and \eqref{global a-priori}.
\end{proof}

\subsection{Comparing Haar measures}\label{matching measures}

In the proof of Theorem \ref{matching comparison} we will assume for simplicity that
$\supp f\subset \Omega\cdot z_{0,I}$. The general case is obtained using the following
observation.  Recall that the Haar measures of $Z$ and $Z_I$ are both adjusted to agree
with a fixed Haar measure of $Q/Q_H$ on the $Q$-orbits through  $z_0$ and $z_{0,I}$.

Recall from the local structure theorem that
\begin{equation}\label{PQ-corr}Qw\cdot z_0\simeq Q/ Q_H\simeq  U \times L/L_H\end{equation}
and by \eqref{QwwI} likewise $Qw_I\cdot z_{0,I} \simeq Q/ Q_H$.
We claim that the Haar measures of $Z$ and $Z_I$ coincide on every open $Q$-orbit
with the fixed normalized measure on $Q/ Q_H$. Let us verify this for $Z$, the proof for $Z_I$ being analogous.
We first implement  the Haar measure on $Q/Q_H$
via a density $|\omega_Z|$ obtained from a top degree differential form $\omega_Z\in \bigwedge^{\rm top} (\qf/ \qf\cap \hf)^*$.  As usual we decompose
$w=\tilde th$ with $\tilde t\in T_Z$ and $h\in \uH$, see \eqref{th-deco}. Then $\Ad(\tilde t)$ preserves $(\qf/ \qf\cap \hf)_\C$
and thus acts on $\bigwedge^{\rm top} (\qf/ \qf\cap \hf)_\C^*$ by a unit scalar.  Since the scalar has to be real, the claim follows.

\subsection{Matching derivatives}\label{comparison of derivatives}

Before we can give the proof of Theorem \ref{matching comparison} we need the following lemma.

\begin{lemma} \label{lemma match2} Let $\Omega\subset Q$ be a compact subset.
Then the following assertions hold:
\begin{enumerate}
\item \label{lemma match2a} Let $u\in \U(\gf)$. There exist
$u_1, \ldots, u_k\in \U(\qf)$
with $\deg u_j \leq \deg u$
and a constant $C=C(\Omega, u)$ such that
\begin{equation}\label{comparison with L_u}
\big|[\Phi(L_u(f_t))- L_u (\Phi(f_t))](z)\big| \\ \leq C
\underset{\sigma\in S\bs I}\max a_t^\sigma \,\sum_{j=1}^k \left| L_{u_j} (\Phi(f_t))(z)\right|
\end{equation}
for all
$f\in C_c^\infty(Z_I)$ with support in $\Omega \Wc_I \cdot z_{0,I}$, and all
$z\in Z$, $t\ge 0$.
\item \label{lemma match2b}Let $p_0$ denote the $L^2$-norm on $L^2(Z)$.  Then for every $k\in\N_0$ there exists a constant
$C=C(\Omega, k)>0$ such that
\begin{equation}\label{approximately unitary} p_{0;k}(\Phi(f_t)) \leq C p_{0;k}(\Phi(f))\end{equation}
for all
$f\in C_c^\infty(Z_I)$ with support in $\Omega \Wc_I \cdot z_{0,I}$ and $t\ge 0$.
\end{enumerate}

\end{lemma}

\begin{proof}
Since the map
$$ \Phi: C_c^\infty (Q\Wc_I\cdot z_{0,I})\to C_c^\infty(QW\cdot z_0), \ \ f\mapsto \Phi(f)$$
is $Q$-equivariant we have
\begin{equation}
%\label{q-inv} \Phi(L(q)f)&= L(q)\Phi(f) \\
\label{Y-inv} \Phi(L_Yf)= L_Y\Phi(f) .
\end{equation}
for all $Y\in\qf$.

For simplicity we consider the case $\supp f\subset \Omega\cdot z_{0,I}$.
We first calculate $L_X(\Phi(f_t))(qa_t \cdot z_0)$
and $\Phi(L_X (f_t))(qa_t \cdot z_0)$ for $X\in\gf$.

\par For that we recall that $\gf=\oline \uf+\qf$ is a direct sum. More generally for all
$q\in Q$ the sum $\gf=\Ad(q) \oline \uf+\qf$ is direct.
Accordingly we can decompose any $X\in \gf$ as
$$ X=\sum_{\alpha, k}  c_{\alpha, k} (q) \Ad(q)X_{-\alpha}^k  + \sum_j  d_j(q) X_j$$
where $(X_{-\alpha}^k)_k$ is a basis of $\gf^{-\alpha}$, $\alpha\in \Sigma_\uf$,
and $(X_j)_j$ is a basis of $\qf$.   The coefficients $c_{\alpha, k}(q), d_j(q)\in \R$ depend smoothly on $q$.

\par Recall that
$X_{-\alpha}^k+ \sum_\beta X_{\alpha, \beta}^k\in \hf $ by \eqref{eq-hi2} with $I=S$.
Thus we get for every smooth function $F$ on $Z$ and every $q\in Q$, $a\in A_Z$ that

\begin{equation*}
L_X F(qa\cdot z_0)=\sum_j d_j(q)  L_{X_j}F(qa\cdot z_0)
-\sum_{\alpha,\beta, k} c_{\alpha, k}(q) a^{\alpha+\beta} L_{\Ad(q) X_{\alpha,\beta}^k} F(qa\cdot z_0)\,.
\end{equation*}
By expanding each $\Ad(q) X_{\alpha,\beta}^k$ in terms of the $X_j$
we can rephrase this identity as
\begin{equation}\label{LYa}
L_X F(qa\cdot z_0)=\sum_j \Big[ d_j(q)
-\sum_{\alpha,\beta} c_{j,\alpha,\beta}(q) a^{\alpha+\beta}\Big]L_{X_j} F(qa\cdot z_0)
\end{equation}
with coefficients $c_{j,\alpha,\beta}$ depending smoothly on $q$.

On the other hand by \eqref{eq-hi2} we also have
$X_{-\alpha}^k+ \sum_{\alpha +\beta \in \la I\ra} X_{\alpha, \beta}^k\in \hf_I$ which
then similarly yields
for every smooth function $f$ on $Z_I$

\begin{equation*}
L_X f(qa\cdot z_{0,I})=\sum_j \Big[ d_j(q)
-\sum_{\alpha,\beta\atop \alpha +\beta \in \la I\ra}
c_{j,\alpha,\beta}(q) a^{\alpha+\beta} \Big]  L_{X_j} f(qa\cdot z_{0,I})
\end{equation*}
with exactly the same coefficients as before, but for fewer $\alpha$ and $\beta$.
We apply $\Phi$ to this equation with $f$ replaced by $f_t$.
With \eqref{Y-inv} this gives

\begin{equation}\label{LYb}
\Phi(L_X f_t)(qa\cdot z_{0})=\sum_j \Big[d_j(q)
 -\sum_{\alpha,\beta\atop \alpha +\beta \in \la I\ra}
c_{j,\alpha,\beta}(q) a^{\alpha+\beta}  \Big] L_{X_j} (\Phi(f_t))(qa\cdot z_{0})\, .
\end{equation}

From this equation we subtract \eqref{LYa} with $F=\Phi(f_t)$. With $a=a_t$
we obtain
\begin{equation}\label{LYc}
[\Phi(L_X (f_t))-L_X (\Phi(f_t))] (qa_t \cdot z_0)=
 \sum_j c_j(q, t)[L_{X_j} (\Phi(f_t))(qa_t\cdot z_0)]
\end{equation}
with coefficients $c_j(q,t)$, each being
a linear combination $\sum_\mu c_\mu(q) a_t^{\mu}$
of functions $a_t^\mu$ with $\mu\in \la S\ra\setminus \la I\ra$,
and with coefficients $c_\mu\in C^\infty(Q)$ supported
in $\Omega$.
In particular \eqref{comparison with L_u} follows for $\deg u=1$.

\par We now prove by induction on $\deg u$ that
\begin{equation}\label{Lu2}
[L_u (\Phi(f_t))- \Phi(L_u (f_t))] (qa_t\cdot z_0)=
 \sum_j c_j(q, t)[L_{u_j} (\Phi(f_t))(qa_t\cdot z_0)]
\end{equation}
for some $u_j\in \U(\gf)$ with $\deg u_j\leq \deg u$ and coefficients
$c_j(q,t)$ of the same type as required in \eqref{LYc}.
Note that the set of coefficients of this type is stable under
differentiation by elements from $\qf$.

Let $u=Xv$ with $X\in \gf$ and $\deg v<\deg u$.
We write
\begin{align*}
L_u (\Phi(f_t))&-\Phi(L_u (f_t))=\\
&L_X \big[ L_v (\Phi(f_t))-\Phi(L_v f_t)) \big]+\big[L_X \Phi(L_v f_t)-\Phi(L_X(L_v f_t)) \big].
\end{align*}
For the first term we apply \eqref{LYa} to $L_X$ in order to replace the
differentiation with $X\in\gf$ by differentiation with the $X_j\in\qf$.
We then apply
the induction hypothesis \eqref{Lu2} to
$\big[L_v (\Phi(f_t))-\Phi(L_v f_t))\big]$. After the differentiations by
$X_j$ we then obtain  for the first term
an expression of the required form.
For the second term we apply \eqref{LYc} with $f_t$
replaced by $(L_vf)_t=L_vf_t$. This gives
$$\sum_j c_j(q, t)[L_{X_j} (\Phi(L_v f_t))(qa_t\cdot z_0)].$$
Once more we apply the induction hypothesis to $v$, which allows us to replace this
expression by
$$\sum_j c_j(q, t)[L_{X_j} L_v (\Phi(f_t))(qa_t\cdot z_0)]$$
at the cost of additional terms. Since all these terms have the required form
this completes the proof of \eqref{Lu2}.

In order to complete the proof of \eqref{comparison with L_u} we need to
replace the $u_j\in\U(\gf)$ in \eqref{Lu2}
by elements from $\U(\qf)$. By induction on the degree, similar
to the one before, we obtain
from \eqref{LYa} for every $u\in\U(\gf)$ a set of elements
$u_1, \ldots, u_n \in \U(\qf)$ with $\deg u_j\leq \deg u$
such that
\begin{equation}\label{Lu}
L_u \Phi(f_t)(qa_t\cdot z_0)=\sum_j e_j(q,t)  L_{u_j}\Phi(f_t)(qa_t \cdot z_0),
\end{equation}
with coefficients $e_j(q,t)$, each being
a linear combination $\sum_\mu c_\mu(q) a_t^{\mu}$
of functions $a_t^\mu$ with $\mu\in \la S\ra$,
and with coefficients $c_\mu\in C^\infty(Q)$ supported
in $\Omega$.
This finally implies \eqref{comparison with L_u} and
with that the proof of \eqref{lemma match2a} has been completed.

\par For \eqref{lemma match2b} we note that \eqref{Lu} and \eqref{Y-inv}
imply:

$$p_0(L_u \Phi(f_t))\leq C_u  \sum_j p_0 (\Phi(L_{u_j} (f_t))).$$
If we denote by $q_0$ the $L^2$-norm on $L^2(Z_I)$
we obtain from \eqref{match equivariant} and \eqref{match2}

$$p_0(\Phi(L_{u_j} f_t))=q_0(L_{u_j} f)
=p_0(\Phi(L_{u_j}f))=p_0(L_{u_j}( \Phi(f))).$$
Combining this with the preceding inequality, \eqref{lemma match2b} follows.
\end{proof}

\subsection{Proof of Theorem \ref{matching comparison}}

\begin{proof}  In view of the consistency relations $w_I \cdot \eta^I= (\eta_{{\bf m}(w_I)})^I$ for all $w\in \Wc_I$
(see \eqref{consist}), the assertion readily reduces to the case where $\supp f \subset Q\cdot z_{0,I}$.
Let us assume that in the sequel.

 Recall that $a_t=\exp(tX)$ with $X\in\af_I^{--}$ fixed.
For simplicity
we assume again that $\supp f \subset \Omega\cdot z_{0,I}$,
and then $\supp \Phi(f_t)\subset \Omega a_t \cdot z_0$.

\par Recall the Laplace element $\Delta_1\in \U(\gf)$ from \eqref{def DeltaR}. 
In what follows we will apply the Sobolev inequality of Lemma \ref{lemma Sobolev norms inequality} to $V$,
and for this we observe (see Theorem \ref{lead lemma 2} below) that $V$ is unitarizable
since $(V,\eta)$ is tempered.

In the sequel we write $\la\cdot, \cdot\ra$ for $\la\cdot, \cdot\ra_{L^2(Z)}$
and $\la\cdot, \cdot\ra_I$ for $\la\cdot, \cdot\ra_{L^2(Z_I)}$  to save notation.
\par Let $n\in \N$, to be specified at the end of the proof. It will depend
on $s$, but apart from that only on the space $Z$.
We start with the identity $v = \Delta_1^n \Delta_1^{-n} v$ which yields

\begin{equation} \label{start id0}  \la \Phi(f_t), m_{v,\eta}\ra =  \la L_{\Delta_1^n} \Phi(f_t), m_{\Delta_1^{-n}v,\eta}\ra\,.
\end{equation}

Next we have to address the subtle point that $\Phi(L_{\Delta_1^n} f_t)$ does not necessarily equal
$L_{\Delta_1^n} \Phi(f_t)$. However from Lemma \ref{lemma match2}\eqref{lemma match2a} we obtain
constants $\epsilon>0$, $C>0$, and elements $u_j\in \U(\gf)$ of degree $\leq 2n$
such that for all $f$ supported by $\Omega\cdot z_{0,I}$
\begin{equation}\label{F-t-esti} | L_{\Delta_1^n} \Phi(f_t) (z)  - \Phi(L_{\Delta_1^n} f_t)(z) |\leq C e^{-t\e}
\sum_{j} |L_{u_j} (\Phi(f_t))(z)|   \qquad (z\in  Z, t\ge 0).\end{equation}
We rewrite (\ref{start id0}) as
\begin{equation} \label{start id1} \la \Phi(f_t), m_{v,\eta}\ra= \la \Phi(L_{\Delta_1^n} f_t)  , m_{\Delta_1^{-n}v,\eta}\ra   +R_1(t)\end{equation}
with $R_1(t) =  \la L_{\Delta_1^n} \Phi(f_t)   -\Phi(L_{\Delta_1^n} f_t), m_{\Delta_1^{-n}v,\eta}\ra$.
We claim, after shrinking $\e$ to ${\e\over 2}$, that  for $2n>n^*$ where $n^*$ is the even integer given
by \eqref{defi nstar}
\begin{equation}\tag{R1}\label{rem1}
|R_1(t)| \leq C e^{-t\e} p_{N;2n}(\Phi(f)) p_{-N; -2n + n^*} (m_{v,\eta})\,
\end{equation}
with a constant $C>0$ that depends on $\Omega$ and $n$, but not on $f$.
From (\ref{F-t-esti}) and Cauchy-Schwarz we obtain

\begin{equation} \label{R11}  |R_1(t)|\leq C e^{-t\e} p_{N; 2n} (\Phi(f_t)) p_{-N} (m_{\Delta_1^{-n}v,\eta})\, .\end{equation}

We obtain from \cite[Prop. 3.4 (2)]{KKSS2} that $|\w(z)|\leq C ( 1+t)$ for all $ z\in \supp \Phi(f_t)$ for a constant $C$ only depending on
 $\Omega$.  Hence it follows with Lemma \ref{lemma match2} \eqref{lemma match2b} that
 \begin{eqnarray} \notag p_{N;2n}(\Phi(f_t)) &\leq& C(1+t)^{N\over 2} p_{0;2n}(\Phi(f_t)) \\
 \label{R12}&\leq& C (1+t)^{N\over 2} p_{0;2n}(\Phi(f)) \leq C (1+t)^{N\over 2} p_{N;2n}(\Phi(f))\end{eqnarray}
with positive constants $C$ (possibly not equal to each other). Note that these constants
$C$ depend on $n$.

Furthermore it follows from \eqref{Sobolev norms inequality} that for $2n>n^*$
 \begin{equation} \label{R13} p_{-N}(m_{\Delta_1^{-n}v,\eta}) \leq C p_{-N; -2n+ n^*} (m_{v,\eta})\, . \end{equation}
If we insert (\ref{R12}) and (\ref{R13}) into (\ref{R11}) we obtain the claim (\ref{rem1}) by noting that $(1+t)^{N\over 2} e^{-{\e\over 2}t}$ is bounded
for all $t\ge 0$.

We move on with the identity (\ref{start id1})  and wish to analyze
$\la \Phi(L_{\Delta_1^n}f_t)  , m_{\Delta_1^{-n}v,\eta}\ra$ further.
By the definitions of $\Phi$ and $f_t$
\begin{eqnarray}\notag \la \Phi(L_{\Delta_1^n} f_t) , m_{\Delta_1^{-n}v,\eta}\ra &=&
%\int_{Q/Q_H} \Phi(L_{\Delta_1^n} f_t) (q\cdot z_0) \oline{m_{\Delta_1^{-n} v,\eta}(q\cdot z_0)} \ d(qQ_H) \\
%\label{longeq}&=&
\int_{Q/Q_H}
(L_{\Delta_1^n} f)(qa_t\cdot z_{0,I}) a_t^{\rho}\oline{m_{\Delta_1^{-n} v,\eta}(q\cdot z_0)} \ d(qQ_H) \\
\label{longeq} &=& \int_{Q/Q_H}
(L_{\Delta_1^n} f)(q\cdot z_{0,I}) a_t^{-\rho} \oline{m_{\Delta_1^{-n} v,\eta}(qa_t\cdot z_0)} \ d(qQ_H) \,.
\end{eqnarray}
Likewise
\begin{equation}\label{longeq2}
\la L_{\Delta_1^n} f_t , m_{\Delta_1^{-n}v,\eta^I}\ra_I
 = \int_{Q/Q_H}
(L_{\Delta_1^n} f)(q\cdot z_{0,I}) a_t^{-\rho} \oline{m_{\Delta_1^{-n} v,\eta^I}(qa_t\cdot z_{0,I})} \ d(qQ_H)\,.
\end{equation}

Next we wish to replace $m_{\Delta_1^{-n} v,\eta}$ by the constant term approximation
$m_{\Delta_1^{-n}v,\eta^I}$ via Theorem \ref{loc ct temp}.
We then obtain constants $\e>0, k\in \N$,  depending only on $Z$, and a constant
$C>0$ depending  also on $\Omega$ and $n$,
such that with $l:=k + n^*$ one has for all
$q\in \Omega$ and all $v\in V^\infty$
\begin{eqnarray} \notag | m_{\Delta_1^{-n} v,\eta}(qa_t\cdot z_0)  -  m_{\Delta_1^{-n}v,\eta^I} (qa_t \cdot z_{0,I})|
&\leq &  C a_t^{(1+\e)\rho} p_{-N;k} (m_{\Delta_1^{-n}v, \eta})\\
\label{longeq3} &\leq & C a_t^{(1+\e)\rho} p_{-N;l-2n} (m_{v, \eta})\, .\end{eqnarray}
In the passage to the second line of \eqref{longeq3} we used \eqref{R13}.

Now note that \eqref{match equivariant} implies
$$ \la (L_{\Delta_1^n} f)_t, m_{\Delta_1^{-n}v,\eta^I}\ra_I= \la f_t, m_{v,\eta^I}\ra_I,$$
and thus if we insert the bound
(\ref{longeq3}) into the difference between (\ref{longeq}) and \eqref{longeq2},
we obtain the identity

\begin{equation} \la \Phi(L_{\Delta_1^n} f_t) , m_{\Delta_1^{-n}v,\eta}\ra = \la f_t, m_{v,\eta^I}\ra_I +R_2(t)\end{equation}
with
\begin{equation} \label{R22}|R_2(t)|\leq C e^{-t\e}
p_{-N;l-2n} (m_{v,\eta})  \| L_{\Delta_1^n} f\|_{L^2(Z_I)}    \sqrt {\vol_{Z_I} (\Omega \cdot z_{0,I})}    \, .\end{equation}
Now, as in \eqref{Lu} we convert
derivatives,
$$L_{\Delta_1^n}f (q\cdot z_{0,I})= \sum  c_j(q) L_{u_j} f (q\cdot z_{0,I})$$
 with $u_j\in \U(\qf)$ of $\deg u_j \leq 2n$ and smooth coefficients $c_j$. Hence

$$\| L_{\Delta_1^n} f\|_{L^2(Z_I)} \leq C p_{0; 2n}(\Phi(f))\leq C p_{N;2n}(\Phi(f))$$
with constants $C$ depending only on $\Omega$.
Hence we obtain
\begin{equation}\tag{R2}  \label{rem2}|R_2(t)|\leq C e^{-t\e}
p_{-N;l-2n} (m_{v,\eta})  p_{N;2n} (\Phi(f)) \, . \end{equation}
Now the theorem follows from the two remainder estimates
\eqref{rem1} and \eqref{rem2}, by choosing the number $n$ such that
$m=2n \ge s+k+ n^*$.
\end{proof}

\subsection{Matching with respect to $\widetilde Z_I$} \label{subsection full match}
We conclude this section with a slight extension of the preceding results, when we consider instead of
$Z_I$ the union of all $G$-orbits in $\uZ_I(\R)$ which point to $Z$, i.e. the space
$\widetilde Z_I=\coprod_{\sc \in \sC_I} \coprod_{\st \in \sF_{I,\sc}} Z_{I,\sc, \st}$ from
\eqref{normal union} which gives rise to the full partition
$\Wc=\coprod_{\sc\in \sc_I}\coprod_{\st\in \sF_{I,\sc}} {\bf m}_{\sc,\st}(\Wc_{I,\sc})$
from \eqref{W partition}.

Observe that $f\in C_c^\infty (\widetilde Z_I)$ corresponds to a family $f=(f_{\sc,\st})_{\sc,\st}$ with
$f_{\sc,\st}\in C_c^\infty (Z_{I,\sc,\st})$ and $Z_{I,\sc,\st} =Z_{I,\sc}$ as homogeneous spaces.  Suppose now that
$\supp f_{\sc,\st} \subset Q w_{I,\sc}\cdot z_{0,I,\sc}\subset Z_{I,\sc}=Z_{I,\sc,\st}$ for all $\sc,\st$. With
\eqref{W partition} the function
$f$ can then be matched with a function $F=\Phi(f)\in C_c^\infty(Z)$ by requesting
$$F(q{\bf m}_{\sc,\st}(w_{I,\sc})\cdot z_0))=f_{\sc,\st}(q w_{I,\sc}\cdot z_{0,I,\sc})\qquad (q\in Q)\, .$$
Then Corollary \ref{main cor}
extends to all $f\in C_c^\infty(\widetilde Z_I)$ with $\supp f_{\sc,\st}\subset  Q\Wc_{I,\sc} \cdot z_{0,I,\sc}$,
and yields constants $C,\e>0$ such that
\begin{equation} \label{full L2} \| f \|^2_{L^2(\widetilde Z_I)} = \int_{\hat G}  \sum_{\sc,\st} \sH_{\pi,\sc,\st}^{I, \rm pre} ((f_{\sc,\st})_t) \ d\mu(\pi)  + R(t)\end{equation}
with $|R(t)| \leq  C e^{-\e t}$ for all $t\geq 0$. Here $\sH_{\pi,\sc,\st}^{I, \rm pre}$ refers to $\sH_\pi^{I, \rm pre} $
for $Z_I$ replaced by $Z_{I,\sc,\st}$; explicitly

\begin{equation} \label{H-c-t} \sH_{\pi,\sc,\st}^{I, \rm pre} (f_{\sc,\st})= \sum_{j=1}^{m_\pi}  \|\oline \pi(f_{\sc,\st})((\oline \eta_j)_{\sc,\st}^I) \|_{\Hc_\pi}^2
\qquad (f_{\sc,\st} \in C_c^\infty(Z_{I,\sc,\st}))\, .\end{equation}

\section{Induced Plancherel measures}
In this section we show that  the Plancherel measure of $L^2(Z_I)$ is induced from
the Plancherel measure of $L^2(Z)$ in a natural manner, see Theorem \ref{Plancherel induced}
below.  A consequence thereof is a certain variant of the  Maass-Selberg relations
as recorded in Theorem \ref{eta-I continuous}.  Statements and approach
are largely motivated by the reasoning in
Sakellaridis-Venkatesh \cite[Sect.~11.1-11.4] {SV}, which originates from ideas of Joseph
Bernstein.
The  main technical  ingredient is our remainder estimate
of Corollary~\ref{main cor}.

\bigskip
Given a point $[\pi]\in \hat G$ we denote by $\U_{[\pi]}$ the neighborhood filter of $[\pi]$ in $\hat G$.
Let $I\subset S$ and recall  from \eqref{Hermitian sum I-side}
the definition of the Hermitian form $\sH_\pi^{I, \rm pre}$.
Attached to the Plancherel measure $\mu$ we  define its $I$-support by

\begin{equation} \label{def supp mu I}  \supp^I (\mu):=\{ [\pi]\in \hat G\mid  (\forall U \in \U_{[\pi]})  \ \mu(\{ [\sigma] \in U: \sH_\sigma^{I,\rm pre} \neq 0\})>0\}\, .\end{equation}
We denote by $\mu^I$ the restriction of $\mu$ to $\supp^I(\mu)$.
In the sequel we let $(\pi,\Hc_\pi)$ be such that $[\pi]\in \supp^I(\mu)$. Define

\begin{equation} \label{def M-pi-I} \M_\pi^I :=\Span\{  a\cdot \eta^I:  \eta\in \M_\pi, a\in A_I\} \subset (\Hc_\pi^{-\infty})_{\rm temp}^{H_I}\end{equation}
where the latter inclusion is part of Theorem \ref{loc ct temp}.

The elements $\xi\in \M_\pi^I$ decompose into generalized eigenvectors for the
$A_I$-action,
\begin{equation} \label{decomp etaI}
\xi=\sum_{\lambda\in \E_\xi}\xi^\lambda,
\end{equation}
and we recall from  \eqref{exponents} that the generalized eigenvalues $\lambda$ satisfy
\begin{equation} \label{exponents normalized unitary}
\E_\xi\subset
(\rho- \W_\jf\cdot \chi_\pi)|_{\af_I} \cap (\rho|_{\af_I} +i\af_I^*)\, .\end{equation}
It will be seen later that
the  $A_I$-action is semisimple for almost all $\pi \in \supp^I(\mu)$.

\par Recall that the conjugation $\Hc_{\pi}^{-\infty}\to \Hc_{\oline \pi}^{-\infty}, \eta\mapsto \oline \eta$
is a $G$-equivariant isomorphism of topological vector spaces. The conjugation map induces
an antilinear $A_I$-equivariant isomorphism $\M_\pi^I\simeq \M_{\oline \pi}^I$.
In particular, $\M_\pi^I$ is semisimple if and only if
$\M_{\oline\pi}^I$ is semisimple.

\subsection{Averaging}
What follows is motivated by the techniques of \cite[Sect. 10]{SV}. Let $X\in \af_I^{--}$
and set $a_t=\exp(tX)$ as usual.  Throughout this section we let $(\pi, \Hc_\pi)$ be a
representation occurring in $\supp^I(\mu)$.  We recall the notion $f_t$ from \eqref{defi f_t}.

\begin{lemma}\label{averaging lemma} {\rm (Averaging Lemma)} Let $X\in \af_I^{--}$.  Then the following assertions hold:
\begin{enumerate}
\item \label{One I}Suppose that $\M_\pi^I$ is $X$-semisimple.
Then we have for all $f\in C_c^\infty(Z_I)$ and $\xi\in \M_\pi^I$ that
\begin{eqnarray}\notag
\lim_{n\to \infty} {1\over n}  \sum_{t=n+1}^{2n} \|\pi (f_t)\xi\|^2& =&
\sum_{\lambda \in \E_{\xi}}  \| \pi(f)\xi^\lambda\|^2
+ \\
\label{limit 1}
&+& 2 \re \sum_{\lambda\neq \lambda' \in \E_{\xi}\atop
(\lambda - \lambda')(X)\in 2\pi i\Z} \la \pi(f)\xi^\lambda, \pi(f)\xi^{\lambda'}\ra\, .\end{eqnarray}
In particular,  if
$(\lambda - \lambda')(X) \not \in 2\pi i  \Z$ for all
$\lambda, \lambda'\in  \E_{\xi}$ with $\lambda\neq \lambda'$, we have

\begin{equation}\label{limit 2} \lim_{n\to \infty} {1\over n}  \sum_{t=n+1}^{2n} \|\pi (f_t)\xi\|^2 =
\sum_{\lambda \in \E_{\xi}}  \| \pi(f)\xi^\lambda\|^2\, .\end{equation}
\item \label{Two II} Suppose that $\M_\pi^I$ is not $X$-semisimple.
Then for every $\xi\in\M_\pi^I$ which does not belong to the sum of eigenspaces for $X$ we have
\begin{equation}\label{limit 3}  \lim_{n\to \infty}
{1\over n} \sum_{t=n+1}^{2n}   \|\pi (f_t)\xi\|^2=\infty \end{equation}
for every $f \in C_c^\infty(Z_I)$ for which $\pi(f)|_{\M_{\pi}^I}$ is injective.
\end{enumerate}
\end{lemma}

\begin{proof}  \eqref{One I} Assume $\M_{\pi}^I$ is diagonalizable for $X$
and let $\xi\in \M_ \pi^I$. It then follows from
\eqref{decomp etaI} and
\eqref{exponents normalized unitary}
that
$$\pi(a_t)  \xi= \sum_{\lambda \in \rho +i\af_I^*}
a_t^{\lambda} \xi^\lambda\, .$$
In particular we obtain from Lemma \ref{lemma ZI-int} for all $f\in C_c^\infty(Z_I)$ and $t\ge 0$ that
$$ \pi (f_t)  \xi= \sum_{\lambda \in \rho + i\af^*}  a_t^{\lambda -\rho}  \pi(f)
\xi^\lambda$$
and thus
\begin{eqnarray*} \|\pi(f_t)\xi\|^2 &=&
%\| \sum_{\lambda \in \rho+i\af_I^*} \pi(f_t)\xi^\lambda\|^2
\| \sum_{\lambda\in \rho +i\af_I^*}  a_t^{\lambda-\rho} \pi(f)\xi^\lambda\|^2\\
&=& \sum_{\lambda\in \rho+ i\af_I^*}  \|\pi(f)\xi^\lambda\|^2  + 2 \re \sum_{
\lambda, \lambda'\in  \E_{\xi}\atop
\lambda\neq \lambda'}
a_t^{\lambda - \lambda' } \la \pi(f)\xi^\lambda, \pi(f)\xi^{\lambda'}\ra\, .\end{eqnarray*}
Now for any $\gamma\in \R\bs 2\pi\Z$ we have $\lim_{n\to \infty} {1\over n}\sum_{t=n+1}^{2n}
e^{i  t\gamma}=0$
and \eqref{One I} follows.
\par For \eqref{Two II} we remark that
 with the mentioned assumption on $\xi$ we have for some $m\in\N$ and each $\lambda\in \rho + i\af^*$ that
\begin{equation}\label{apply f}  \pi(f_t)\xi^\lambda= a_t^{\lambda-\rho} \pi(f) \sum_{j=0}^m \frac{t^j}{j!}\xi^{\lambda,j}\end{equation}
where $\xi^{\lambda,0}=\xi^\lambda,\xi^{\lambda,1},\dots,\xi^{\lambda,m}\in \M_ \pi^I$.
Moreover we can assume $\xi^{\lambda,  m}\neq 0$
for some $\lambda$. Now \eqref{limit 3} becomes a simple matter
on polynomial  asymptotics:  set
$$ \xi_t^{\rm top}:= \sum_\lambda a_t^{\lambda -\rho} \xi^{\lambda, m} \qquad (t\geq 0)$$
and note that $|a_t^{\lambda- \rho}|=1$ implies that the vectors $\xi_t^{\rm top}$, $t\geq 0$, stay away from $0$ in the finite dimensional
space $\M_\pi^I$. Thus we obtain from \eqref{apply f} and the injectivity of $\pi(f)|_{\M_\pi^I}$ that
$$ \|\pi(f_t) \xi\|\sim t^{m}\|\pi(f)\xi_t^{\rm top}\| \, $$
from which \eqref{limit 3} follows.
\end{proof}

Suppose that $X\in \af_I^{--}$ is such that
$(\lambda - \lambda')(X) \not \in 2\pi i  \Z$ for all
$\lambda, \lambda'\in  \E_\xi$, $\xi \in \M_{\oline \pi}^I$,  with $\lambda\neq \lambda'$. Then we obtain from \eqref{Hermitian sum I-side} and
Lemma \ref{averaging lemma} that

\begin{equation} \label{LIMIT HI-PRE}
\lim_{n\to \infty}
{1\over n} \sum_{t=n+1}^{2n}  \sH_\pi^{I, \rm pre}(f_t) = \begin{cases}
\sum\limits_{j=1}^{m_\pi}\sum\limits_{\lambda\in \E_{\oline \eta_j^I}}  \|\oline \pi (f) \oline \eta_j^{I,\lambda}\|^2 &
\text{if  $\M_\pi^I$ is $X$-semisimple}\\
 \infty & \text{if otherwise and}\\
 & \text{$\oline \pi(f)|_{\M_{\oline \pi}^I}$ is injective}
\end{cases}
\end{equation}
where $\eta_1, \ldots, \eta_{m_\pi}$ is an orthonormal basis for $\M_\pi$.

This motivates the following definition of $\sH_\pi^I$. In case
$\M_\pi^I$ is a semisimple $A_I$-module we set

\begin{equation} \label{I-Hermitian form} \sH_\pi^I (f) :=\sum_{j=1}^{m_\pi}
\sum_{\lambda\in \E_{\oline \eta_j^I}}  \|\oline \pi (f) \oline \eta_j^{I,\lambda}\|^2 \qquad (f \in C_c^\infty(Z_I))\, ,\end{equation}
and otherwise $\sH_\pi^I:=0$.   Observe that   the Hermitian form
$\sH_\pi^I$ is left $G$-invariant, and
normalized-right $A_I$-invariant.
Set

$$ \supp_{ \rm fin}^I(\mu):=\{ [\pi] \in \supp^I(\mu)\mid \M_\pi^I \ \text{is $\af_I$-semisimple}\}\, .$$

\subsubsection{Mollifying on multiplicity spaces}
Throughout this subsection we let $V$ be an irreducible Harish-Chandra module and $V^\infty$ its unique
$SF$-completion.
Let $\Sc(G)$ be the Schwartz algebra of rapidly decreasing
functions on $G$ (see \cite{BK}) and recall the following variant
of the Casselman-Wallach theorem: if
$0\neq v\in V$, then
\begin{equation}\label{cas-wa}\Sc(G)*v =V^\infty\end{equation}
by \cite[Th. 8.1]{BK}, where for $f\in \Sc(G)$ and $v\in V^\infty$ we use the standard notation

$$ f * v = \int_G f(g) g\cdot v \ dg$$
with the right hand side being a convergent integral in the Fr\'echet space $V^\infty$.
Assertion \eqref{cas-wa} can be strengthened further as follows. Let $\widetilde V$ be the Harish-Chandra module dual to $V$.
Then we first record the mollifying property $\Sc(G)* \widetilde V^{-\infty} \subset V^\infty$  which in view
of \eqref{cas-wa} strengthens to
\begin{equation}\label{cas-wa2}\Sc(G)*\eta =V^\infty   \qquad (0\neq \eta \in\widetilde V^{-\infty})\end{equation}
In fact, choose first a left $K$-finite function $f\in C_c^\infty(G)$ such that $0\neq f*\eta \in V$ and then apply
\eqref{cas-wa} with  $\Sc(G)*C_c^\infty(G)\subset \Sc(G)$.
Let now $H\subset G$ be any closed unimodular subgroup of $G$. Then we define $\Sc(G/H)$ as the
space of right $H$-averages of functions $F\in \Sc(G)$, i.e. $f \in \Sc(G/H)$ if and only if there exists
an $F\in \Sc(G)$ such that
$$f(gH)=F^H(g):= \int_H  F(gh) \ dh\qquad (g\in G)\, .$$
With that we can define for $\eta \in (\widetilde V^{-\infty})^H$ and $f=F^H \in \Sc(G/H)$:
$$f*\eta:= F*\eta$$
as the right hand side of this equation is independent of the particular lift $F$ of $f$.
Then we have the following generalization of \eqref{cas-wa2}.

\begin{lemma} \label{lem molly} Let $H\subset G$ be a closed unimodular subgroup and let
$E\subset (\widetilde V^{-\infty})^H$ be a finite dimensional subspace. Then the
map
$$\Phi_E:  \Sc(G/H) \to \Hom(E, V^\infty), \ \ f\mapsto \left(\eta\mapsto f*\eta\right)$$
is continuous and surjective. Moreover $E$ is uniquely determined by $\ker \Phi_E$.
\end{lemma}

\begin{proof} First  of all it is clear that $\Phi_E$ is continuous.  Next we observe that the statement reduces
to $H=\{\1\}$ which we will assume from now on.

Notice that $\Phi_E$ is an $\Sc(G)$-module morphism
with $\Sc(G)$ acting on $\Hom(E, V^\infty)$ on the target $V^\infty$, i.e. for $f\in \Sc(G)$ and $T\in
\Hom(E, V^\infty)$ we set $(f*T)(\eta):= f* (T(\eta))$.

\par  Suppose that $\Phi_E$ were not surjective.
Then $\operatorname{im}\Phi_E\subset  \Hom(E, V^\infty)$ would be a proper $\Sc(G)$-invariant subspace.
Upon the identification $ \Hom(E, V^\infty)= E^* \otimes V^\infty$ we then derive from the fact that
$V^\infty$ is an algebraically simple module for $\Sc(G)$  (a consequence of \eqref{cas-wa}) that
$\operatorname{im}\Phi_E=  F^\perp \otimes V^\infty$ for a subspace $0\neq F\subset E$. This then means
$$\operatorname{im}\Phi_E=\{ T\in \Hom(E, V^\infty)\mid   T|_F=0\}$$
which contradicts the fact that $\Sc(G)*F=V^\infty\neq \{0\}$ as $F\neq 0$.
\par Finally from $\Sc(G)/\ker \Phi_E \simeq E^* \otimes V^\infty$ we obtain the asserted uniqueness.  Indeed,
suppose you have $\ker \Phi_{E_1} = \ker \Phi_{E_2}$.
Then $\ker\Phi_{E_i}= \ker \Phi_{E_1+E_2}$ for $i=1,2$ and thus $\dim (E_1+E_2)^* =\dim E_i^*$ for $i=1,2$,  i.e. $E_1=E_2$.
\end{proof}

We apply Lemma \ref{lem molly} to the Hermitian forms $\sH_\pi^I$ of \eqref{I-Hermitian form}
as follows. Let $E=\M_{\oline \pi}^I$.

\begin{cor} Let $[\pi]\in \supp_{ \rm fin}^I(\mu)$.
There exists a unique Hermitian form $\sH$ on
$$\Hom(\M_{\oline \pi}^I, \Hc_\pi^\infty)\simeq \Hc_\pi^\infty \otimes \M_\pi^I$$
for which $\sH(\Phi_E(f))=\sH_\pi^I(f)$ for all $f\in \Sc(G/H)$.
This form is $G$-invariant and positive definite.
\end{cor}

\begin{proof} Clearly $ f\in\ker\Phi_E \Rightarrow \sH_\pi^I(f)=0$.
Moreover, since $[\pi]\in \supp_{ \rm fin}^I(\mu)$ we have
$$E=\M_{\oline \pi}^I=\Span\{ {\oline\eta}_j^{I,\lambda}\mid 1\leq j\leq m_\pi, \lambda \in \rho|_{\af_I}  +i\af_I^*\}$$
from which we deduce the converse implication.
\end{proof}

We use the symbol $\sH_\pi^I$ also for the form $\sH$ introduced in the corollary. Now a variant of Schur's Lemma implies
that $\sH_\pi^I$ viewed as a form on $\Hc_\pi^\infty \otimes \M_\pi^I$ is given by

\begin{equation} \label{FORM11} \sH_\pi^I(v\otimes \xi)
=\la v,v\ra_{\Hc_\pi} \la \xi, \xi\ra_{\M_\pi^I}\end{equation}
for a unique Hilbert inner product $\la \cdot, \cdot\ra_{\M_\pi^I}$ on $\M_\pi^I$.

We conclude this intermediate subsection with a simple observation of later use.

\begin{lemma}\label{molly injective} Keep the assumptions of Lemma \ref{lem molly} and let $(f_n)_{n\in \N}$ be a Dirac-sequence
in $C_c^\infty(G/H)$.  Then there exists
an $N=N(E)$ such that  the map
$$ \Phi_E(f_n):\, E \to V^\infty , \ \ \eta \mapsto f_n *\eta$$
is injective for all $n\geq N$.
\end{lemma}

\begin{proof} This is a special case of a more general fact.  Let $X$ be a locally convex topological
vector space and $E\subset X$ a finite dimensional subspace.
Let $T_n: E\to X$ be a family of linear continuous maps with $\lim_{n\to \infty} T_n(x)=x$ for all
$x\in  E$. We claim that there exists $N\in\N$ such that
$T_n$ is injective for all $n\geq N$. To see that we choose a closed complement to $E$
and obtain a continuous projection $p_E: X\to E$.   With  $S_n:= p_E \circ T_n$ we then obtain a sequence
$S_n \in \End(E)$ such that $S_n\to \1$. This proves the claim.
The lemma follows with $X=\widetilde V^{-\infty}$ and $T_n(x)= f_n*x$.
\end{proof}

\subsection{Induced Plancherel measure}
The following theorem was largely motivated by \cite[Th. 11.3]{SV}.

\begin{theorem}\label{Plancherel induced} {\rm (Induced Plancherel measure)}
For all $f\in C_c^\infty(Z_I)$ one has
\begin{equation}  \label{PT I-norm}\| f\|_{L^2(Z_I)}^2= \int_{\supp^I(\mu)}  \sH_\pi^I (f) \ d\mu (\pi)\, .\end{equation}
In particular, the Plancherel measure $\mu_I$ of $L^2(Z_I)$ is equivalent to $\mu$ restricted to
$ \supp^I(\mu)$,
and $\M_\pi^I$ as defined in \eqref{def M-pi-I} and
equipped with the Hermitian form obtained from \eqref{FORM11} provides
a multiplicity space for $\mu_I$-almost all $\pi$.  In other words
\begin{equation} \label{Planch ZI} L^2(Z_I) \underset{G \times A_I}\simeq \int_{\supp^I(\mu)}  \Hc_\pi \otimes \M_\pi^I \ d\mu(\pi)\,,\end{equation}
with the just described inner product on $\M_\pi^I$, is a Plancherel decomposition for $Z_I$.
Finally, the complement of $\supp_{ \rm fin}^I(\mu)$ in $\supp^I(\mu)$ is a null set.
\end{theorem}

\begin{proof} It is sufficient to prove this identity for test functions $f$ with support in
$P \Wc_I \cdot z_{0,I}$
because $P \Wc_I \cdot z_{0,I}$
exhausts $Z_I$ up to measure zero. Let such a test function $f$ be given.
\par Fix $X\in \af_I^{--}$. It follows from the exponential decay of $R(t)$ in Corollary \ref{main cor} that

\begin{equation} \label{Fatou}
{1\over n} \sum_{t=n+1}^{2n}  \int_{\hat G} \sH_\pi^{I, \rm pre}(f_t) \ d\mu(\pi) \to \|f\|_{L^2(Z_I)}^2
\end{equation}
as $n\to \infty$.
Define
$$\sH_\pi^{I, X-{\rm inv}} (f):=\lim_{n\to \infty} {1\over n} \sum_{t=n+1}^{2n} \sH_\pi^{I, \rm pre}(f_t)\in[0,\infty
]\, .$$
Then \eqref{Fatou} and Fatou's lemma imply
\begin{equation} \label{Fatou1}
 \int_{\hat G} \sH_\pi^{I, X-{\rm inv}} (f)\, d\mu(\pi) \le \|f\|_{L^2(Z_I)}^2 < \infty.
\end{equation}
Next set
$$\hat G_X:=\{ [\pi] \in \hat G \mid \M_\pi^I\neq\{0\} \ \text{and $\M_\pi^I$ is $X$-semisimple} \}\, .$$
By choosing a Dirac sequence $f_1, f_2, \ldots$ of $C_c^\infty (Z_I)$ which is supported
in $P\cdot z_{0,I} $ we obtain from Lemma \ref{molly injective} for each $[\pi]\in\hat G$
that $\oline \pi(f_j)|_{\M_{\oline \pi}^I}$
is injective for some $j$.
Hence by countable additivity it follows
from \eqref{Fatou1} together with \eqref{LIMIT HI-PRE} and
the definition of $\supp^I(\mu)$ in \eqref{def supp mu I} that $\mu(\supp^I(\mu) \bs \hat G_X)=0$. Further
for $[\pi]\in \hat G_X$ we have
$\sH_\pi^{I, X-{\rm inv}} (f)<\infty$ and from \eqref{limit 1} we infer

\begin{eqnarray}\notag \sH_\pi^{I, X-{\rm inv}} (f) &=& \sum_{j=1}^{m_\pi}
\sum_{\lambda\in \rho+ i\af_I^*}  \|\oline\pi(f)\oline\eta_j^{I,\lambda}\|^2  +\\
\label{X inv} &+&\sum_{j=1}^{m_\pi} \, 2 \re \!\!\! \sum_{\lambda\neq \lambda' \in \E_{\xi}\atop
(\lambda - \lambda')(X)\in 2\pi i\Z} \!\!\!
\la \oline\pi(f)\oline\eta_j^{I,\lambda}, \oline\pi(f)\oline\eta_j^{I,\lambda'}\ra\, .
\end{eqnarray}
Next we define
$$\hat G_{X, \rm reg}:=\{ [\pi] \in \hat G_X\mid (\forall \lambda\neq \lambda'\in (\rho-\W_\jf\chi_\pi)|_{\af_I}):
\ (\lambda- \lambda')(X)\not \in 2\pi i\Z\}$$
and deduce from \eqref{Fatou1}, \eqref{X inv}, and \eqref{I-Hermitian form} that

\begin{equation} \label{limit 5} \|f\|_{L^2(Z_I)}^2 \geq  \int_{\hat G_{X, \rm reg}} \sH_\pi^I(f) \ d\mu(\pi)
+ \int_{\hat G_X \bs \hat G_{X, \rm reg}} \sH_\pi^{I, X-{\rm inv}} (f) \ d\mu(\pi) \, .\end{equation}
Now we start iterating \eqref{limit 5} with finitely many $X\in \af_I^{--}$.
In more precision, let $X_1:=X$ and set $X_2:=\sqrt{2} X_1$.
Now the iteration of \eqref{limit 5} starts with  $a_t:=\exp(tX_2)$ while observing
$ \|f\|_{L^2(Z_I)}^2= \|f_t\|_{L^2(Z_I)}^2$ and taking weighted averages as before. Another application of
Fatou's Lemma then yields
\begin{eqnarray*} \|f\|_{L^2(Z_I)}^2 &\geq & \int_{\bigcup_{j=1}^2\hat G_{X_j, \rm reg}} \sH_\pi^I(f) \ d\mu(\pi)+\\
& + & \int_{\left(\bigcap_{j=1}^2\hat G_{X_j}\right) \bs \left(\bigcup_{j=1}^2 \hat G_{X_j, \rm reg}\right)} \sH_\pi^{I, \{X_1,X_2\}-{\rm inv}} (f) \ d\mu(\pi) \end{eqnarray*}
with
\begin{eqnarray}\notag \sH_\pi^{I, \{X_1,X_2\}-{\rm inv}} (f) &=& \sum_{j=1}^{m_\pi}
\sum_{\lambda\in \rho+ i\af_I^*}  \|\oline\pi(f)\oline\eta_j^{I,\lambda}\|^2  +\\
\label{X12 inv} &+&\sum_{j=1}^{m_\pi} 2 \re \sum_{
\lambda, \lambda'\in  \E_{\oline \eta_j^I}\atop
\lambda\neq \lambda', (\lambda-\lambda')(X_1)=0}
\la \oline\pi(f)\oline\eta_j^{I,\lambda}, \oline\pi(f)\oline\eta_j^{I,\lambda'}\ra
\end{eqnarray}
as a result of making \eqref{X inv} also invariant under $X_2$. Here we used that
$(\lambda - \lambda')(X_i)\in 2\pi \Z$ for $i=1,2$ means
$(\lambda-\lambda')(X_i)=0$.

Next take $X_3\in \af_I^{--}$ linearly independent to $X_1$ and then $X_4:=\sqrt{2} X_3$. This we continue
until $X_1, X_3, \ldots, X_{2m-1}$ is a basis of $\af_I$ contained in $\af_I^{--}$.

Notice that iterating \eqref{X12 inv} yields that
$$\sH_\pi^{I, \{X_1,\ldots, X_{2m}\}-{\rm inv}} (f) = \sH_\pi^I (f)$$
and we finally arrive at
\begin{equation} \label{limit 6} \|f\|_{L^2(Z_I)}^2 \geq  \int_{\hat G}  \sH_\pi^I(f) \ d\mu(\pi)\end{equation}
together with the fact $\mu(\supp^I(\mu) \bs \supp_{\rm fin}^I(\mu))=0$  as
 $\supp_{\rm fin}^I=\bigcap_j\hat G_{X_j}$.
\par To conclude the proof we observe for $X=X_1$ and any $\pi\in\hat G$ that
$\|\oline\pi(f_t)\oline\eta^I\| \leq \sum_{\lambda\in \E_{\oline \eta^I}} \|\oline \pi(f)\oline\eta^{I,\lambda}\|$ and thus
$$\|\oline\pi(f_t)\oline\eta^I\|^2 \leq |\W_\jf| \sum_{\lambda\in \E_{\oline \eta^I}} \|\oline \pi(f)\oline\eta^{I,\lambda}\|^2$$ as $|\E_{\oline \eta}|\leq |\W_\jf|$.
Summing over $ t$ and the $\oline \eta_j^I$ this implies via
\eqref{X inv} for all $[\pi]\in \supp_{\rm fin}^I\mu$
that
$${1\over n} \sum_{t=n+1}^{2n}   \sH_\pi^{I, \rm pre}(f_t)\leq |\W_\jf| \sH_\pi^I (f) $$
for all $n>0$.
Thus by \eqref{Fatou1} and dominated convergence
we can interchange limit and integral in \eqref{Fatou} and obtain actual equality in \eqref{Fatou1}:
\begin{equation} \label{limit 4}  \int_{\hat G} \sH_\pi^{I, X-{\rm inv}} (f) \ d\mu(\pi) = \|f\|_{L^2(Z_I)}^2 \, . \end{equation}
The just described iteration applied to \eqref{limit 4} then yields
 $$\int_{\hat G} \sH_\pi^I (f) \ d\mu(\pi) = \|f\|_{L^2(Z_I)}^2 $$
and finishes  the proof of the theorem.
\par  The final statements follow from uniqueness
of the Plancherel measure together with \eqref{FORM11}.
\end{proof}

\subsection{Extension to $\widetilde Z_I$} In view of Section \ref{subsection full match}
we can extend Theorem
\ref{Plancherel induced} to all
$f\in C_c^\infty(\widetilde Z_I)$:
\begin{equation}  \label{PT I-norm2}\| f\|_{L^2(\widetilde Z_I)}^2= \sum_{\sc,\st}
\int_{\supp^{I,\sc,\st}(\mu)}  \sH_{\pi,\sc,\st}^I (f_{\sc,\st}) \ d\mu (\pi)\end{equation}
where we put an extra index $\sc,\st$ when we consider objects, initially defined for $Z_I$,   now for
$Z_{I,\sc,\st}$. Let us further denote by
$\M_{\pi,\sc,\st}^I \subset(\Hc_\pi^{-\infty})^{H_{I,\sc}} $ the Hilbert space
$\M_\pi^I$ (with the inner product obtained from \eqref{FORM11}),
but for $Z_I$ replaced by $Z_{I,\sc,\st}=Z_{I,\sc}$.
We then form the direct sum of Hilbert spaces
$$ \widetilde \M_\pi^I=\bigoplus_{\sc,\st} \M_{\pi,\sc,\st}^I \, ,$$
and equip this space with
the diagonal action of $A_I$, i.e.~for
$\xi=(\xi_{\sc,\st})_{\sc,\st}\in \widetilde \M_\pi^I$ we have
$a\cdot \xi= (a\cdot \xi_{\sc,\st})_{\sc,\st}$.
Then we obtain  the following extension of \eqref{Planch ZI} to
\begin{equation} \label{Planch ZI extended }
L^2(\widetilde Z_I) \underset{G \times A_I}\simeq
\int_{\hat G}  \Hc_\pi \otimes \widetilde \M_\pi^I \ d\mu(\pi)\end{equation}

\subsection{The Maass-Selberg relations}
The multiplicity space $\widetilde \M_\pi^I$ are $A_I$-semisimple for $\mu$-almost all $[\pi]$
and thus admits a direct sum decomposition
$\widetilde \M_\pi^I =  \bigoplus_{\lambda\in \rho+ i\af_I^*}  \widetilde \M_\pi^{I,\lambda}$
with
$$\widetilde \M_\pi^{I,\lambda}=\{\xi \in \widetilde \M_\pi^I\mid  (\forall a\in A_I) \  a\cdot \xi= a^\lambda \xi\}\, .$$
Since the normalized right action of $A_I$ on $L^2(Z_I)$ is unitary it follows
that the Hermitian structure on $\widetilde \M_\pi^I$ is such that this
decomposition of $\widetilde \M_\pi^I$
is orthogonal for $\mu$-almost all $[\pi]$.

 \begin{theorem} \label{eta-I continuous} {\rm (Maass-Selberg relations)}
Let  $\lambda\in\rho|_{\af_I} + i\af_I^*$. Then for almost all
$[\pi] \in \bigcup_{\sc,\st} \supp^{I,\sc,\st}_{\rm fin}(\mu)$
the map

 $$ \sI^\lambda:  \M_\pi \to \widetilde \M_\pi^{I,\lambda}, \ \ \eta\mapsto (\eta_{\sc,\st}^{I,\lambda})_{\sc,\st}$$
 is a surjective partial isometry, i.e.~its Hermitian adjoint is a unitary isometry. \end{theorem}

\begin{proof} Let us denote by $\la\cdot, \cdot\ra$ the scalar product on $\widetilde \M_\pi^I$.
By definition
it is given by \eqref{FORM11} (summed over all $\sc,\st$) for almost all $[\pi]$.
Now summation of \eqref{I-Hermitian form} over all $\sc,\st$  implies for all
 $x\in \widetilde\M_\pi^I$ that

\begin{equation} \label{M-equation}\|x\|^2_{\M_\pi^I}
= \sum_{\sc,\st}\sum_{j=1}^{m_\pi} \sum_{\lambda\in \E_{\eta_j^I}} |\la x, (\eta_j)_{\sc,\st}^{I,\lambda}\ra|^2
\, .\end{equation}
In particular, for $x\in \widetilde\M_\pi^{I,\lambda}$
this is condition (\ref{ONB1}) so that Lemma \ref{lemma ONB} applies.
\end{proof}

\begin{rmk} Of particular interest is the case of a multiplicity one space, i.e. where we have
$\dim \M_\pi \leq 1$ for almost all $\pi \in \supp\mu$.  This is for instance satisfied in the group case
$Z= G \times G/ \diag G\simeq G$, for complex symmetric spaces, and in the Riemannian situation $Z=G/K$.
\par  For a symmetric space the condition that $\dim \M_\pi \leq 1$ for almost all $\pi$ implies $\Wc=\{\1\}$. To see that we first observe
that there are $|\Wc|$-many open $H$-orbits $\Oc\subset G/Q$, each isomorphic to $H/ L_H$ as a unimodular
$H$-space.  Integration over these open $H$-orbits
yields at least $|\Wc|$-many tempered functionals for representations $\pi$
with generic parameters in the most-continuous spectrum of $Z$, say
$\eta_{\pi, w}$ for $w\in \Wc$.  Now there is a subtle point that
a priori we only have $\M_\pi \subset (V_\pi^{-\infty})^H_{\rm temp}$.  But forming wave packets
finally yields that these $\eta_{\pi,w}$ indeed contribute a.e. to the $L^2$-spectrum.
For this one needs an estimate of $\eta_{\pi,w}$ which is locally uniform with
respect to $\pi$. For the case of a symmetric space $Z$ such an estimate is given in
\cite{vdBII} Thm.~9.1.  The statement follows.

\par The statement  above implies that $\M_\pi^I=\widetilde \M_\pi^I$.  Our Maass-Selberg relations in Theorem
\ref{eta-I continuous} then assert for $\eta\in \M_\pi$ with $\|\eta\|=1$ that $(\eta^{I,\lambda})_{\lambda}$ is an
orthonormal basis
of $\M_\pi^I=\widetilde \M_\pi^I$ (where we only count those $\lambda$ for which $\M_\pi^{I,\lambda}\neq \{ 0\}$).
In particular, for the group case this leads to the Maass-Selberg relations of Harish-Chandra \cite{HC3}, p.~146.
 \end{rmk}

We finish this section with an elementary lemma about finite dimensional Hilbert spaces.
It was used for Theorem \ref{eta-I continuous} above.

\begin{lemma}\label{lemma ONB} Let
$J: \M\to \Nc$ a linear map between two finite dimensional Hilbert spaces. Assume that for some
orthonormal basis $\eta_1, \ldots, \eta_n$ for $\M$ one has

\begin{equation} \label{ONB1}  \la x, x\ra= \sum_{j=1}^n  |\la x, J\eta_j\ra|^2 ,
\qquad (x\in \Nc).\end{equation}
Then the adjoint of $J$ is an isometry.
\end{lemma}

\begin{proof} It follows from \eqref{ONB1} that
$ \|x\|^2 = \sum_{j=1}^n |\la J^*x, \eta_j\ra |^2 = \|J^*x\|^2$.
\end{proof}

\section{Spectral Radon transforms and twisted discrete spectrum}

The constant term assignments
$$\M_\pi\ni \eta\mapsto \eta^I\in \M_\pi^I$$ give rise
to spectral Radon transform ${\mathsf R}_I : L^2(Z) \to L^2(Z_I)$ which is the topic of this
section. With the help of this transform we can characterize the twisted discrete series
$L^2(Z)_{\rm td}$ of $L^2(Z)$ spectrally.
The section starts with a brief recall on the twisted discrete series, see also \cite{KKOS} and \cite[Sect. 9]{KKS}.

\subsection{Twisted discrete series}\label{Subsection twisted}
Let us denote by $L^2(Z)_{\rm d}$ the discrete spectrum of $L^2(Z)$, i.e. the direct sum of all irreducible
subspaces. Now in case $\af_{Z,E}\neq \{0\}$, it is easy to see that $L^2(Z)_{\rm d}=\emptyset$, see \cite[Lemma 3.3]{KKOS}. In particular, for  $I\subsetneq S$ we have $L^2(G/H_I)_{\rm d}=\emptyset$ as $\af_{Z_I,E}=\af_I\neq \{0\}$.

Recall that the subspace $\af_{Z,E}=\af_S\subset \af_Z$ normalizes
$\hf$ and gives rise to the subalgebra $\hat \hf= \hf +\af_{Z,E}$.   Hence $A_{Z,E}:=A_S\subset A$
normalizes $H$ and acts
unitarily on $L^2(G/H)$ via the normalized right regular action

$$({\mathcal R}(a)f)(gH)= a^{-\rho} f(gaH)  \qquad (g\in G, a \in A_{Z,E}, f\in L^2(Z)).$$
Disintegration of $L^2(G/H)$ with respect to the right action of $A_{Z,E}$ then yields the unitary equivalence of
$G$-modules

\begin{equation} \label{central decomposition} L^2(Z)= \int_{\hat A_{Z,E}}  L^2(G/\hat H, \chi) \ d\chi,\end{equation}
where  $\hat A_{Z,E}$ denotes the unitary dual of the abelian Lie group $A_{Z,E}$, and  for each
unitary character $\chi: A_{Z,E}\to {\mathbb S}^1$ the $G$-module $L^2(G/\hat H, \chi)$ is a certain Hilbert space of densities explained in
\cite[Sect. 8]{KKS2} or \cite[Sect. 3.2] {KKOS}.  A spherical pair $(V,\eta)$ which embeds into some
$L^2(G/\hat H, \chi)$ will be referred to as a representation of the {\it twisted discrete series of $Z$.}
Further we denote by $L^2(G/\hat H, \chi)_{\rm d}$ the discrete spectrum and define the {\it twisted discrete
series} by
\begin{equation}\label{defin td}
L^2(Z)_{\rm td} = \int_{\hat A_{Z,E}}  L^2(G/\hat H, \chi)_{\rm d} \ d\chi
\end{equation}
made more rigorous in Subsection \ref{char td} below.

\subsection{Spectral Radon transforms} \label{SRT}

For $w\in \Wc$ we set $Z_w=G/H_w$.
Note that
$$ L^2(Z)\to L^2(Z_w), \ \ f\mapsto \left(gH_w\to f(gwH)\right)$$
is a unitary equivalence of $G$-representations.
Hence the abstract  Plancherel formula for $L^2(Z)$ induces one
for $L^2(Z_w)$ with the same Plancherel measure and isometries

$$ \M_\pi \to \M_{\pi,w},\ \ \eta\mapsto \eta_w\,.  $$

For every $I\subset S$ and $w\in \Wc$ we set $Z_{I,w}:= G/ (H_w)_I$ and keep in mind that for fixed
$I$, the various $(H_w)_I $ need not be $G$-conjugate (cf. Example \ref{ex SL3}).

\par Now given $\eta\in \M_\pi$ and  $w\in \Wc$ we note that $\eta_w^I=(w\cdot \eta)^I$ is fixed by $(H_w)_I$ and we use notation
$\M_{\pi,w}^I$ for $\M_\pi^I$ with respect to $(H_w )_I$.
In the sequel we assume that $[\pi]\in \supp\mu \subset \hat G$ is {\it generic}, that is $\M_{\pi,w}^I$ is $\af_I$-semisimple for all $I\subset S$ and $w\in \Wc$.  By  Theorem \ref{Plancherel induced}
with $H$ replaced by $H_w$ we obtain that the complement of the generic elements is a null set with respect to
$\mu$.
We endow $\M_{\pi,w}^I$ with the Hilbert space structure induced
from $\M_\pi$ via Theorem \ref{Plancherel induced}.

\par Our concern is with the {\it spectral Radon transforms} induced from the constant term maps:
$${\mathsf r}_{\pi,I,w} : \M_\pi \to \M_{\pi,w}^I, \ \ {\mathsf r}_{\pi,I,w}(\eta)=\eta_w^I\, .$$
and for $J\subset I$ their  transitions:

\begin{equation} \label{transit1} {\mathsf r}_{\pi,J,w}^I: \M_{\pi,w}^I\to \M_{\pi,w}^J, \ \ {\mathsf r}_{\pi,J,w}^I(\xi)=\xi^J\, .\end{equation}

We recall the transitivity of the constant terms \cite[Prop. 6.1]{DKS}:

\begin{lemma}\label{lemma transitive} Let $\eta \in \M_\pi$ and $w\in \Wc$. Then for all $J\subset I$ one has
$$ (\eta_w^I)^J = \eta_w^J\, .$$
\end{lemma}

The transitivity of the constant term maps then reflects in

\begin{equation} \label{transfer1}
{\mathsf r}_{\pi,J,w}^I\circ {\mathsf r}_{\pi,I,w} = {\mathsf r}_{\pi,J,w} \qquad (J\subset I)\, .
\end{equation}

\par Recall that ${\mathsf r}_{\pi, I,w}$ is a sum of at most $|\W_\jf|$-many partial isometries  by the Maass-Selberg relations
in Theorem \ref{eta-I continuous}.  Hence we obtain

\begin{equation} \label{bb-bound} \|{\mathsf r}_{\pi, I,w}\|\leq |\W_\jf| \, . \end{equation}

\begin{defprop} \label{prop Radon} Let $I\subset S$ and $w\in\Wc$.
The operator field
$$(\id_{\Hc_\pi} \otimes {\mathsf r}_{\pi,I,w})_{\pi\in \hat G }:
\quad
 \Hc_\pi \otimes \M_\pi \to \Hc_\pi \otimes \M_{\pi,w}^I$$
is measurable and induces a $G$-equivariant continuous map
$$ {\mathsf R}_{I,w}:  L^2(Z)\simeq \int_{\hat G}^\oplus  \Hc_\pi \otimes \M_\pi\ d\mu(\pi) \to L^2(Z_{I,w})\simeq
\int_{\hat G}^\oplus  \Hc_\pi \otimes \M_{\pi,w}^I \ d\mu(\pi) $$
Moreover
\begin{equation} \label{r-bound} \|{\mathsf R}_{I,w}\|\leq |\W_\jf|\, .\end{equation}
We call ${\mathsf R}_{I,w}$ the {\rm spectral Radon transform at $(I,w)$.}
\end{defprop}
\begin{proof} Since  the ${\mathsf r}_{\pi, I,w}$ reflect the pointwise convergent asymptotics of matrix coefficients,
the operator field is measurable. With the upper bound in \eqref{bb-bound} we then obtain that ${\mathsf R}_{I,w}$
is defined and continuous with norm bound \eqref{r-bound}. By definition ${\mathsf R}_{I,w}$ is then
$G$-equivariant, completing the proof.
\end{proof}

With \eqref{transit1} we obtain spectrally defined   Radon transforms:
\begin{equation} \label{transit2}  {\mathsf R}_{J,w}^I:  L^2(Z_{I,w})\to L^2(Z_{J,w})\qquad (J\subset I) \end{equation}
which then by \eqref{transfer1} satisfy
\begin{equation} \label{transfer2} {\mathsf R}_{J,w}=  {\mathsf R}_{J,w}^I\circ {\mathsf R}_{I,w} \qquad (J\subset I) \end{equation}

Putting the data of the various $(I,w)$ together, we arrive  at the {\it (full) spectral Radon transform}
$$ {\mathsf R}= \oplus_{I,w} {\mathsf R}_{I,w}:  L^2(Z) \to \bigoplus_{I\subset S} \bigoplus_{w\in \Wc}L^2(Z_{I,w})\, .
$$

\subsection{Characterization of the twisted discrete spectrum} \label{char td}
Next we want to define $L^2(Z)_{\rm td}$ rigorously in terms of the spectral Radon transforms.
Set

\begin{equation} \label{multi twist} \M_{\pi, \rm td}=\{ \xi \in \M_\pi\mid  \exists \chi \in \widehat A_{Z,E}\ \forall v\in V_\pi^\infty:  \ m_{v,\xi} \in L^2(\hat Z,\chi)_{\rm d}\}\end{equation}
and likewise we define $\M_{\pi, w, {\rm td}}^I$ for $w\in \Wc$ and $I\subset S$.

\par Then
\begin{equation} \label{defini twist} L^2(Z)_{\rm td}:= \bigcap_{w\in \Wc\atop I\subsetneq S} \ker {\mathsf R}_{I,w}\, .\end{equation}
defines a closed  subspace $G$-invariant subspace of $L^2(Z)$.

Next we need a reformulation of the characterization of the twisted discrete series from \cite[Sect. 8]{KKS2}
in the more suitable language of constant terms \cite[Th. 5.12]{DKS}, namely:

\begin{lemma}\label{vanishing ct}  Let $\eta\in \M_\pi$. Then the following are equivalent:
\begin{enumerate}
\item $\eta\in \M_{\pi,{\rm td}}$.
\item $\eta_w^I=0$ for all $w\in \Wc$ and $I\subsetneq S$.
\end{enumerate}
\end{lemma}

With the characterization in Lemma \ref{vanishing ct} we arrive at:

\begin{prop}  We have
\begin{equation} \label{deco twist}L^2(Z)_{\rm td} \simeq \int_{\hat G}^\oplus  \Hc_\pi\otimes \M_{\pi, \rm td} \ d\mu(\pi)\, .\end{equation}
In particular $L^2(Z)_{\rm td}\subset L^2(Z)$ is invariant under the normalized
right regular representation $\Rc$ of $A_{Z,E}$.
\end{prop}

\begin{proof} Both assertions follow from Lemma \ref{vanishing ct} and the involved
definitions \eqref{multi twist} and \eqref{defini twist}
\end{proof}

Since $L^2(Z)_{\rm td}$ is $A_{Z,E}$-invariant we obtain from
\eqref{deco twist}
a rigorous definition of \eqref{defin td} with $L^2(\hat Z,\chi)_{\rm d}$ equal to the
$\chi$-spectral part of $L^2(Z)_{\rm td}$ under $\Rc$.

\subsection{Restriction to the twisted discrete spectrum}  Applying the preceding theory with $L^2(Z)$ replaced  by $L^2(Z_{I,w})$  we obtain orthogonal projections

$$\operatorname{pr}_{I,w,{\rm td}}: L^2(Z_{I,w})\to L^2(Z_{I,w})_{\rm td}$$
and define $R_{I,w}:= \pr_{I,w, \rm td} \circ \ {\mathsf R}_{I,w}$. Note that
$$R_{I,w}: L^2(Z)\to L^2(Z_{I,w})$$
is a continuous $G$-equivariant  map. The {\it restricted spectral Radon transform} is then defined to be
$$ R= \oplus_{I,w} R_{I,w}:  L^2(Z) \to \bigoplus_{I\subset S} \bigoplus_{w\in \Wc}L^2(Z_{I,w})_{\rm td}\, .
$$

\section{Bernstein morphisms}\label{B-morphisms}
We define the {\it Bernstein morphism} $B$ as the Hilbert space adjoint $R^*$ of the restricted spectral Radon transform $R$.
With $B_{I,w}:= R_{I,w}^*$ we then have

$$ B: \bigoplus_{I\subset S} \bigoplus_{w\in \Wc}L^2(Z_{I,w})_{\rm td}\to L^2(Z), \ \ (f_{I,w})_{I, w} \mapsto   \sum_{I,w}  B_{I,w}(f_{I,w})\, .$$
The main result of this section then is:

\begin{theorem} {\rm (Plancherel Theorem --  Bernstein decomposition)} \label{thm planch}
The Bernstein morphism is a continuous surjective $G$-equivariant linear map.
Moreover, $B$ is isospectral, that is,  image and source have Plancherel measure in the same measure class.
\end{theorem}

After some technical preparations we give the proof of
Theorem \ref{thm planch}. Then, after applying  the material
on open $P$-orbits developed in Section \ref{subsection WI}  we derive
in Theorem \ref{thm planch refined} a refined Bernstein decomposition,
which agrees with the partition
$\Wc=\coprod_{\sc \in \sC_I} \coprod_{\st \in \sF_{I,\sc}} {\bf m}_{\sc, \st}(\Wc_{I,\sc})$ from \eqref{full deco W}.
\par Finally, by adding up the refined  Bernstein decompositions for the various $G$-orbits
in $\uZ(\R)$   we obtain in Theorem \ref{thm planch refined real points} the statement for $L^2(\uZ(\R))$
which is in full analogy to the p-adic statement of Sakellaridis-Venkatesh \cite[Cor.~11.6.2]{SV}.

\subsection{Proof of Theorem \ref{thm planch}}

Denote by ${\mathcal P}(S)$ the power set of $S$. With regard to $\eta\in \M_\pi$ we call a pair
$(I,w)\in {\mathcal P}(S) \times \Wc$ admissible  provided that $\eta_w^I\neq 0$.
Finally we call an $\eta$-admissible pair $(I,w)$ {\it optimal} provided that  the cardinality $|I|$ is minimal, i.e.
we have $\eta_{w'}^J=0$ for all $w'\in \Wc$ and $J\subsetneq I$.
Notice that, by definition, for every $\eta\neq 0$ there exists an $\eta$-optimal pair $(I,w)$.

The embedding theory
of tempered representations into twisted discrete series from \cite[Sect. 9] {KKS2}
then comes down to:

\begin{theorem}\label{lead lemma 2}  Let $0\neq \eta\in \M_\pi$ and
$(I,w)$ be an $\eta$-optimal pair. Then $\eta_w^I\in \M_{\pi,w,\rm td}^I$.
\end{theorem}

\begin{proof}  Let $(I,w)$ be $\eta$-optimal. Applying a base point shift
we may assume  that $w=\1$. According to Lemma \ref{vanishing ct} applied to $Z_I$ we need to show that
$(w_I \cdot \eta^I)^J=0$ for all $w_I \in \Wc_I$ and $J\subsetneq I$.  Let ${\bf m}(w_I)=w\in\Wc$. By the consistency relations
 \eqref{consist} we have $w_I \cdot \eta_I = \eta_w^I$.  Thus, by the transitivity of the constant term we have
 $$(w_I \cdot \eta^I)^J=\eta_w^J=0$$
 by the minimality of $|I|$. The theorem follows.
\end{proof}

\par Let us denote for each $[\pi]\in\hat G$ and each $I\subset S$, $w\in \Wc$ by
$ \xi \mapsto\xi_{\rm td}$ the orthogonal projection $\M_{\pi, w}^I \to \M_{\pi,w, {\rm td}}^I$.

With that we define a linear map between finite dimensional Hilbert spaces by
$$ {\mathsf r}_\pi=\oplus {\mathsf r}_{\pi,I,w, \rm td}:\M_\pi\to  \bigoplus_{I\subset S}\bigoplus_{w\in \Wc}
 \M_{\pi, w,{\rm td}}^I,\ \  \eta\mapsto (\eta_{w, \rm td}^I)_{I, w} $$
with $\eta_{w, \rm td}^I:=(\eta_w^I)_{\rm td}$.

\begin{rmk} Since $\xi\to \xi_{\rm td}$ is $A_I$-equivariant, we have
the orthogonal decomposition
$\M_{\pi,w, {\rm td}}^I=\bigoplus_{\mu\in \rho+ i\af_I^*}\M_{\pi,w, {\rm td}}^{I,\mu}$.
Thus every $\xi \in \M_{\pi, w, {\rm td}}^I$ decomposes as $\xi=\sum \xi^\mu$
with $\xi^\mu\in \M_{\pi,w, {\rm td}}^{I,\mu}$ for  $\mu \in \rho|_{\af_I} + i\af_I^*$ by \eqref{exponents}.
\end{rmk}

For any $\lambda\in \E_{\eta^I} \subset \rho|_{\af_I} +i\af_I^*$ (cf. \eqref{exponents})   we denote by
${\mathsf r}_{\pi, I,w, {\rm td}, \lambda}$
the map ${\mathsf r}_{\pi,I,w, \rm td}$ followed by orthogonal
projection to the $\lambda $-coordinate  $\M_{\pi, w, {\rm td}}^{I, \lambda}$ of
$\M_{\pi, w, {\rm td}}^I$.

Then Theorem \ref{lead lemma 2} yields the  technical key Lemma:

\begin{lemma}  \label{properties bpi} The following assertions hold:
\begin{enumerate}
\item \label{eins 1}${\mathsf r}_\pi$ is injective.
\item \label{zwei 2}For all $I\subset S, w\in \Wc$ and $\lambda\in \E_\pi$ the map
$${\mathsf r}_{\pi, I,w, {\rm td}, \lambda}:\M_\pi\to \M_{\pi,w, {\rm td}}^{I,\lambda}, \ \ \eta \mapsto \eta_{w,\rm td}^{I, \lambda}$$
is a surjective partial isometry.
\item \label{drei 3}The assignment $\pi \mapsto {\mathsf r}_\pi$ is measurable.
\end{enumerate}
\end{lemma}
\begin{proof} Let $0\neq \eta \in \M_\pi$. According to Theorem \ref{lead lemma 2} we find an $\eta$-optimal
pair $(I, w)$ such that $\eta_{w, {\rm td}}^I\neq 0$, establishing \eqref{eins 1}.
Having shown \eqref{eins 1}, assertion \eqref{zwei 2} is obtained from the Maass-Selberg relations
in Theorem \ref{eta-I continuous}:  we replace $H$ by $H_w$
and observe that $\eta\mapsto \eta_w$ establishes an isomorphism of $\M_\pi\to \M_{\pi,w}$
with $\M_{\pi,w}$ referring to $\M_\pi$ with $H$ replaced by $H_w$.

Finally (\ref{drei 3}) is by the definition of the measurable structures involved (see Section \ref{Section AbsPlanch} and
Proposition \ref{prop Radon}):
The family of maps
$${\mathsf r}_{\pi,I,w}: \M_\pi\to \M_{\pi,w}^I,\ \  \eta \mapsto \eta_w^I,$$
as well as the projection to discrete parts ${\mathsf r}_{\pi,I,w, \rm td}$ are measurable.
\end{proof}

We now define
$$b_\pi: \bigoplus_{I\subset S}\bigoplus_{w\in \Wc}  \M_{\pi,w,{\rm td}}^I \to \M_{\pi}$$
to be the adjoint of ${\mathsf r}_{\oline \pi}$ and note that $b_\pi$, being the adjoint of an injective
morphism,  is surjective.  Notice that
the Bernstein morphism is
$$B:  \bigoplus_{I\subset S}\bigoplus_{w\in \Wc}   L^2(Z_{I,w})_{\rm td}   \to L^2(Z)$$
is defined spectrally  by the operator field  $(b_\pi)_{\pi \in \supp \mu}$.
\begin{rmk} \label{Remark isometric} (Decomposition of $B$ into isometries)  For $I\subset S$ and $w\in \Wc$
we denote by $B_{I,w}$ the restriction of $B$ to $L^2(Z_{I,w})_{\rm td}$.

We claim
that there is an orthogonal decomposition
$$L^2(Z_{I,w})_{\rm td}=\bigoplus_{u \in \W_\jf} L^2(Z_{I,w})_{{\rm td,u}}$$
such that every restriction $B_{I,w,u}:= B|_{L^2(Z_{I,w})_{{\rm td}, u}}$ is an isometry.
To construct such a decomposition we choose for every $[\pi]\in \supp(\mu)$ with
infinitesimal character $\chi_\pi\in \jf_\C^*/ \W_\jf$ a representative $\lambda_\pi\in \jf_\C$, i.e
$\chi_\pi = \W_\jf \cdot \lambda_\pi$.
Let us denote by
$$\sP_{u}([\pi]): \M_{\pi,w}^I \to \M_{\pi,w}^{I, (\rho - u \cdot \lambda_\pi)|_{\af_I}}$$
the orthogonal projection. Our request for the choice $\lambda_\pi\in\chi_\pi$ is then such that
the operator field
$$ \supp(\mu)\ni [\pi]\mapsto \sP_u([\pi])\in \End(\M_{\pi,w}^I)$$
is measurable.  With
$$L^2(Z_{I,w})_{{\rm td},u}:=\int_{\hat G}^\oplus \Hc_\pi \otimes \M_{\pi,w,\rm td}^{I,(\rho - u \cdot \lambda_\pi)|_{\af_I}}\ d\mu(\pi)$$
we then obtain an orthogonal decomposition  $L^2(Z_{I,w})_{\rm td}=\bigoplus_{u \in \W_\jf} L^2(Z_{I,w})_{{\rm td, u}}$ for which $B_{I,w,u}$ is an isometry by Lemma \ref{properties bpi}\eqref{zwei 2}.
\end{rmk}

The final piece of information we need for the proof of Theorem
\ref{thm planch} is the following elementary result of functional analysis whose proof we omit.

\begin{lemma}\label{Hilbert lemma} Let $\Hc=\int_X^\oplus \Hc_x \ d\mu(x)$ be a direct integral of Hilbert spaces. Let further
$\Kc=\int_X^\oplus \Kc_x \ d\mu(x)$ and $\Lc=\int_X^\oplus \Lc_x \ d\mu(x)$ be closed decomposable subspaces of $\Hc$.
Suppose that $\Kc_x +\Lc_x\subset \Hc_x$ is closed for every $x\in X$. Then
$\Kc +\Lc\subset \Hc$ is closed.
\end{lemma}
\begin{proof}[Proof of Theorem \ref{thm planch}]
The surjectivity of the
$b_\pi$ together with Theorem \ref{Plancherel induced} shows that $B$ is an isospectral $G$-morphism
with dense image.  To see that $B$ is surjective we note that $B$ is a sum of isometries
each one of which has closed range. Thus $B$ is surjective by Lemma \ref{Hilbert lemma}.
\end{proof}

\begin{rmk} In case $\Wc=\{\1\}$, i.e.~there is only one open $P$-orbit, the Bernstein decomposition
becomes a lot simpler as the summation over $\Wc$ disappears in the domain of $B$.  We recall that $\Wc=\{\1\}$ is
satisfied for reductive groups $G \simeq G\times G/ G$, for complex spherical spaces, and for Riemannian symmetric spaces.
\end{rmk}

\subsection{Refinement of the Bernstein morphisms}\label{Subsection refine B}
In the definition of the Bernstein morphism a certain over-parametrizing takes place in the domain.
This will now be remedied via the partition
$\Wc=\coprod_{\sc \in \sC_I} \coprod_{\st \in \sF_{I,\sc}} {\bf m}_{\sc, \st}(\Wc_{I,\sc})$ from \eqref{full deco W}.
We recall the corresponding terminology from Subsection \ref{subsection all base points}.

For $\eta\in \M_\pi$, $\sc\in \sC_I$, $\st \in \sF_{I,\sc}$ we recall the functional
$\eta_{\sc,\st}= w(\sc, \st)\cdot \eta$ from Subsection \ref{subsection all base points}. Further we set
$\eta_{\sc,\st}^I:= (\eta_{\sc,\st})^I$ and given $w_{I,\sc}\in \Wc_{I,\sc}$ we define the functional
$(\eta_{\sc,\st}^I)_{w_{I,\sc}}:= w_{I,\sc}\cdot \eta_{\sc,\st}^I $. Likewise
for $\mu \in \af_{I,\C}^*$ we set $(\eta_{\sc,\st}^{I,\mu})_{w_{I,\sc}}:=w_{I,\sc}\cdot \eta_{\sc,\st}^{I,\mu} $.
\par Every $w\in \Wc$
can be written uniquely as $w={\bf m}_{\sc,\st}(w_{I,\sc})$ for $\sc\in \sC_I$, $\st\in \sF_{I,\sc}$ and $w_{I,c}\in \Wc_{I,\sc}$. In this context we recall from \eqref{consist2}
the consistency relation

\begin{equation} \label{Iw2}\eta_w^{I,\lambda}=(\eta_{\sc,\st}^{I,\lambda})_{w_{I,c}} \qquad
(w={\bf m}_{\sc,\st}(w_{I,c})\in \Wc, \lambda\in \E_{\eta^I})\, .
\end{equation}

Recall $H_{I,\sc}=(H_{w(\sc)})_I$ equals $H_{I,\sc,\st}= (H_{w(\sc,\st)})_I$. Hence
$(\eta_{\sc,\st}^{I,\lambda})_{w_{I,c}}$ is fixed by
$(H_{I,\sc})_{w_{I,c}}$.
On the other hand $\eta_w^{I,\lambda}$ is fixed under $(H_w)_I$. We recall from  \eqref{WWI2 general}
that the two groups are in fact equal:

$$ (H_w)_I = (H_{I,\sc})_{w_{I,c}}\, . $$

In this context it is worth to record the following extension of \cite[Th. 5.12]{DKS}:

\begin{prop}The following assertions are equivalent for $\eta\in \M_\pi$:
\begin{enumerate}
\item \label{1S} $\eta\in \M_{\pi, \rm td}$.
\item \label{2S} For all $I\subsetneq S$ and $\sc\in \sC_I, \st\in \sF_{I,\sc}$ one has $\eta_{\sc,\st}^I =0$.
\end{enumerate}
\end{prop}

\begin{proof}  Let $w\in \Wc$ and write it as $w={\bf m}_{\sc, \st}(w_{I,c})$. We recall
 \eqref{consist} which asserts that  $\eta_w^I = w_{I,c}\cdot \eta_{\sc,\st}^I$. In particular $\eta_w^I=0$ if and only if
$\eta_{\sc,\st}^I=0$ and the proposition follows from Lemma \ref{vanishing ct}.
 \end{proof}

For $\sc\in \sC_I$ and $\st \in F_{I,\sc}$  we set $Z_{I,\sc,\st}= Z_{I, w(\sc,\st)}$ and note that
$Z_{I,\sc,\st}=G/H_{I,\sc}$ is independent of $\st\in \sF_{I,\sc}$ by Lemma \ref{lemma HI comp}.
The following is then a refined version of the Bernstein decomposition, taking the fine partition
\eqref{full deco W} of $\Wc$ into account.

\begin{theorem} {\rm (Plancherel Theorem --  Bernstein decomposition refined)} \label{thm planch refined}
The restricted Bernstein morphism

$$B_{\rm res}:  \bigoplus_{I\subset S}  \bigoplus_{\sc \in \sC_I}\bigoplus_{\st \in \sF_{I,\sc}} L^2(Z_{I, \sc,\st})_{\rm td}
\to L^2(Z)$$
is surjective.
\end{theorem}

\begin{proof} Given the proof of Theorem \ref{thm planch} this comes down to the fact that the map
$$\tilde {\mathsf r}_\pi: \M_\pi \to \bigoplus_{I\subset S} \bigoplus_{\sc \in \sC_I} \bigoplus_{\st \in \sF_{I,\sc}}
\M_{\pi, w(\sc,\st), \rm td}^I, \ \ \eta\mapsto  (\eta_{\sc, \st, \rm td }^I)_{I, \sc, \st}$$
obtained from ${\mathsf r}_\pi$ by restricting the target remains injective.
Now we recall the proof of Lemma \ref{properties bpi}  \eqref{eins 1} and let $0\neq \eta \in \M_\pi$
with  $\eta_{w, {\rm td}}^I\neq 0$ for an $\eta$-optimal pair $(I,w)$. In particular,
$\eta_{w, {\rm td}}^I\neq 0$.  Let $w={\bf m}_{\sc,\st}(w_{I,c})$ for $w_{I,c}\in \Wc_{I,\sc}$ and $\st\in \sF_{I,\sc}$. Then the consistency relation \eqref{Iw2} yields
$\eta_{w, {\rm td}}^I =(w_{I,c}\cdot \eta_{\sc,\st}^I)_{\rm td}$ and thus $\eta_{\sc,\st,\rm td}^I\neq 0$, establishing
the injectivity of $\tilde {\mathsf r}_\pi$. The theorem follows.
\end{proof}

\subsection{Bernstein decomposition for $L^2(\uZ(\R))$}

Recall that $Z=G/H$ is only one $G$-orbit of $\uZ(\R)$.  To obtain the Bernstein decomposition
of $L^2(\uZ(\R))$ we just need to add the data
of the various $G$-orbits in $\uZ(\R)$.  We recall
$W_\R=(P\bs \uZ(\R))_{\rm open}\simeq F_\R/ F_M$ and choose representatives $\Wc_\R\subset G$
for $W_\R$ as we did with  $\Wc$ for $W$.
For $w\in \Wc_\R$ we set $Z_{I,w}:= G/ (H_w)_I$ with $(H_w)_I$ the real points
of the $\R$-algebraic group $ (\uH_w)_I$.
Notice that the $G$-orbit decomposition of $\uZ(\R)$ yields a natural partition of $W_\R$ by selecting
for a given $G$-orbit in $\uZ(\R)$ the open $P$-orbits it contains. Summing up the Bernstein morphism of all
$G$-orbits then yields  a $G$-morphism:

$$ B_\R : \bigoplus_{I\subset S} \bigoplus_{w\in \Wc_\R}  L^2(Z_{I,w})_{\rm td} \to L^2(\uZ(\R)) .$$
We then obtain from Theorem \ref{thm planch}:

\begin{theorem} {\rm (Plancherel Theorem  for $L^2(Z(\R))$ --  Bernstein decomposition)} \label{thm planch real points}
The Bernstein morphism $B_\R$  is a continuous surjective isospectral $G$-equivariant linear map.
\end{theorem}

Recall from the beginning of Section \ref{subsection WI} that
$W_{I,\R}=(P\bs \uZ_I(\R))_{\rm open}$ and $W_\R$ are canonically isomorphic.
  % denote by $n(I)$ the number of
%$\uG$-orbits in $\hat \uZ$ isomorphic to $\hat \uZ_I$.  Note that
%the $n(I)$'s depend on the particular chosen smooth toroidal compactification.  In the wonderful case we recall  that the
%$n(I)=1$ for all $I\subset S$.
%Let us number
%the $\uG$-orbits in $\hat \uZ$ isomorphic to $\hat \uZ_I$ by
%$\hat \uZ_{I,1}, \ldots, \hat \uZ_{I, n(I)}$.  Accordingly we obtain
%$\uZ_{I,1}, \ldots, \uZ_{I,n(I)}$.
 %In the language of normal bundle geometry every $G$-orbit in
%$\uZ_{I,j}(\R)$ points to exactly one $G$-orbit in $\uZ(\R)$, thus obtaining
In particular we obtain a generalization of \eqref{full deco W} to

$$ \Wc_\R = \Wc_{I,\R}= \coprod_{\sc\in \sC_{I,\R}}  \coprod_{\st \in \sF_{I,\sc}}  {\bf m}_{\sc,\st} (\Wc_{I,\sc})$$
with $\sC_{I,\R}:=\{  G\cdot \hat z_{w,I}\mid w\in W_\R\}$ etc.

The finer results in Theorem \ref{thm planch refined} then yield the
refined restricted Bernstein morphism
\begin{equation} \label{Bernstein real forms}
B_{\rm \R, res} : \bigoplus_{I\subset S}    L^2(\uZ_I(\R))_{\rm td}\to L^2(\uZ(\R))\end{equation}
with the same properties as in Theorem \ref{thm planch refined}:

\begin{theorem} {\rm (Plancherel Theorem  for $L^2(Z(\R))$ --  Bernstein decomposition refined)} \label{thm planch refined  real points}
The restriction $B_{\R, {\rm res}}$ of the Bernstein morphism $B_\R$  is a continuous surjective isospectral $G$-equivariant linear map.
\end{theorem}

\section{Elliptic elements and discrete series}

As a consequence of the Bernstein decomposition in Theorem \ref{thm planch}
we obtain in Theorem \ref{thm discrete} a general criterion for the existence of a discrete
spectrum in $L^2(G/H)$ for a unimodular real spherical space $G/H$.
The main additional tool is a theorem of \cite{HW}, by which
the wave front set of the left regular representation of a unimodular homogeneous space
$G/H$ is determined as the closure of $\Ad(G)\hf^\perp$.

\subsection{Existence of discrete spectrum}

As usual, we call an element $X\in\gf$ semisimple provided $\ad X$ is a semisimple operator.
Equivalently, $X\in\gf$ is semisimple if and only if its centralizer $\zf_\gf(X)$ is a reductive subalgebra.

An element $X\in \gf_\C$ is called {\it elliptic} if $\ad X$ is semisimple with  purely imaginary eigenvalues.
If $E\subset \gf_\C$ we denote by $E_{\rm ell}$ the subset of $E$ consisting of elliptic elements.
More generally we call an element $X\in \gf_\C$ {\it weakly elliptic} if $\Spec (\ad X)\subset i\R$ and
denote by $E_{\rm w-ell}$ the corresponding subset of $E\subset \gf_\C$.

\begin{theorem} \label{thm discrete} Let $Z=G/H$ be a unimodular real spherical space.
Suppose that $\operatorname{int} \hf_{\rm w-ell}^\perp\neq \emptyset$. Then $H=\hat H$ is
reductive and $L^2(Z)_{\rm d}\neq  \{0\}$.
\end{theorem}

Here $\operatorname{int} \hf_{\rm w-ell}^\perp$ refers to the interior of
$\hf_{\rm w-ell}^\perp$, in the vector space topology of $\hf^\perp$.
The proof is given in the course of the next two subsections.

\begin{rmk} In case $Z=G$ is a reductive group or more generally $Z=G/H$ is a symmetric space, then Theorem
\ref{thm discrete} comes down to the existence theorems of Harish-Chandra \cite{HC} and Flensted-Jensen \cite{FJ} about discrete
series. It is due to Harish-Chandra that $L^2(G)_{\rm d} \neq \emptyset$ if $\gf$ admits a compact
Cartan subalgebra. Flensted-Jensen generalized that to symmetric spaces by showing $L^2(G/H)_{\rm d}\neq \emptyset$
if there exists a compact abelian subspace  $\tf \subset \gf \cap \hf^\perp$ with $\dim \tf = \rank \uG/\uH$.
\end{rmk}

\begin{rmk} For the twisted discrete series an appropriate generalization of
Theorem \ref{thm discrete} reads
\begin{equation} \operatorname{int} \hat \hf^\perp_{\rm w-ell}\neq \emptyset \quad\Rightarrow \quad (\forall \chi \in \hat A_{Z,E})\ L^2(G/\hat H, \chi)_{\rm d}\neq \emptyset\end{equation}
and will presumably follow from results on wavefront sets of induced representations more general than what is obtained
in \cite{HW}.
\end{rmk}

\subsection{The geometry of elliptic elements}

To prepare the way for the proof of Theorem \ref{thm discrete} we
establish some foundational material on elliptic elements in $\hf^\perp$, and show
that if the weakly elliptic elements in $\hf^\perp$ have non-empty interior, then
$\hf$ is reductive in $\gf$.

\bigskip \par We consider the $H$-module $\hf^\perp\subset \gf$ and recall the canonical isomorphism
$(\gf/\hf)^*\simeq \hf^\perp$.  In the sequel we view $\af_Z\simeq \af_H^{\perp_\af}$ as a subspace
of $\af$ and likewise we view $\mf_Z=\mf/\mf_H \simeq \mf_H^{\perp_\mf}$ as a subspace of $\mf$.

\begin{lemma} $(\lf\cap\hf)^\perp=\hf^\perp\oplus \uf$
\end{lemma}

\begin{proof} Clearly $\hf^\perp + \uf\subset(\lf\cap\hf)^\perp $.
Moreover $\hf^\perp\cap\uf=\{0\}$ because $\kappa(\uf, \qf)=\{0\}$ and $\gf=\hf+\qf$.
The lemma now follows from $\dim \hf=\dim(\lf\cap\hf)+\dim\uf$.
\end{proof}

Let $T_0: (\lf\cap\hf)^\perp\to \uf$ be minus the projection along $\hf^\perp$.
It follows that

\begin{equation} \label{T0 orthogonal} (\af_Z +\mf_Z)^0:=\{ X + T_0(X):  X \in \af_Z+\mf_Z\} \subset \hf^\perp\, .\end{equation}
Similarly we set $\bfrak^0:=\{ X + T_0(X):  X \in \bfrak\}$ for $\bfrak\subset \af_Z +\mf_Z$ a subspace.

The following lemma is motivated by \cite[Th. 5.4 and Cor. 7.2]{Knop}   and \cite[Th. 5 and Th. 6]{Panyushev}.

\begin{lemma}  \label{lemma H perp polar}Let $Z=G/H$ be a real spherical space for which
there exists an $X_0\in \af_Z\cap \hf^\perp$  such that
 $\alpha(X_0)<0$ for all $\alpha\in \Sigma_\uf$.
Then the  canonical map
$$\Phi: H \times  (\af_Z+\mf_Z)^0\to \hf^\perp, \ \ (h, X) \mapsto \Ad(h)X$$
is generically submersive.
\end{lemma}

\begin{proof} We first note that $(\af_Z +\mf_Z)^0 + [\hf, X]\subset \hf^\perp$ for all $X\in (\af_Z +\mf_Z)^0$,
and that $\Phi$ is generically submersive if and only if there is equality
for some $X\in (\af_Z +\mf_Z)^0$. We will show that
\begin{equation} \label{gen surjective} (\af_Z +\mf_Z)^0 + [\hf, X_0]=\hf^\perp .\end{equation}

For $t>0$ we set $a_t:=\exp(tX_0)$.
By conjugation \eqref{gen surjective} is then equivalent to
\begin{equation} \label{gen surjective with t}
(\af_Z +\mf_Z)^0_t+[\hf_t,X_0]=\hf_t^\perp\, .\end{equation}
where $ (\af_Z+\mf_Z)^0_t:= \Ad(a_t)(\af_Z +\mf_Z)^0$ and
$\hf_t:=\Ad(a_t)\hf$.
Now note that by \eqref{T0 orthogonal} we have for $t\to\infty$ that
$(\af_Z+\mf_Z)_t^0 \to \af_Z +\mf_Z$ in the Grassmannian of subspaces.
Moreover
$\hf_t\to \hf_\emptyset=\lf\cap \hf + \oline \uf$ by \eqref{I-compression}.

On the other hand $(\hf_\emptyset)^\perp = \af_Z+\mf_Z + \oline \uf$.
As $[X_0,\oline \uf]=\oline \uf$ we obtain
\begin{equation*} \label{gen surjective limit} \af_Z +\mf_Z + [\hf_\emptyset, X_0]=(\hf_\emptyset)^\perp \, ,\end{equation*}
that is, \eqref{gen surjective with t} holds in the limit.
Hence it holds for $t$ sufficiently large.
\end{proof}

In analogy to \cite[Sect.~3]{Knop2} we call  $Z=G/H$ {\it non-degenerate} provided that
an element $X_0$ as in Lemma  \ref{lemma H perp polar} exists, and {\it degenerate} otherwise.
Flag varieties $Z=G/P$ with $P$ a parabolic subgroup
of $G$ are degenerate.  But in many cases $Z$ is non-degenerate as the following example shows.

\begin{ex}\label{ex-quasiaffine} (cf. \cite[Lemma 3.1]{Knop2}) Every quasi-affine real spherical space is non-degenerate.
Indeed, the constructive proof of the local structure theorem, see \cite[Section 2.1]{KKS}, yields
an $X_0\in \af_Z\cap \hf^\perp$ such that $\lf=\zf_\gf(X_0)$. Moreover, this element can be chosen such that
$\alpha(X_0)<0$ for all roots $\alpha\in \Sigma_\uf$.
\end{ex}
The following lemma was communicated to us by B. Harris.
\begin{lemma} \label{lem Harris} Let $\uG$ be an algebraic group defined over $\R$ and $\uH\subset \uG$ be an algebraic
subgroup defined over $\R$ as well.  Suppose that $Z=G/H$ is unimodular.  Then
$Z$ is quasi-affine, i.e. $\uZ=\uG/\uH$
is a quasi-affine variety.  \end{lemma}

\begin{proof}  Clearly $Z$ is unimodular if and only if $\uZ$ is unimodular. We assume first that
$\uH$ is connected and treat the general case at the end.
We recall the following transitivity
result, see \cite[Lemma 1.1]{Suk} and \cite[Th.4]{BHM}:  If there is a tower $\uH \subset \uH_1\subset \uG$ of subgroups, such that
$\uH_1/\uH$ and $\uG/\uH_1$ are both quasi-affine, then $\uG/\uH$ is quasi-affine.
Now for $d:=\dim \uH$ and $X_1,\ldots, X_d$ a basis of $\hf$ consider
$v_1:=X_1\wedge\ldots\wedge X_d \in \bigwedge^d \gf_\C$.  As $\uH$ is supposed to be unimodular and connected, we see that $\uH$ fixes $v_1$.  Let $\uH_1$ be the stabilizer of $v_1$ in $\uG$. Then
$$ \uG/\uH_1 \to \bigwedge^d\gf_\C, \ \ g\uH_1\mapsto g\cdot v_1$$
is injective and exhibits $\uG/\uH_1$ as quasi-affine. Moreover, as $\uH\subset \uH_1$ is normal, $\uH_1/\uH$
is affine and the transitivity result  of above applies.
This shows the lemma for $\uH=\uH_0$ connected.  As $F:=\uH/ \uH_0$ is finite and acts freely on $\uZ_0=\uG/\uH_0$
the quotient $\uZ=\uG/\uH \simeq \uZ_0/F$ is  geometric and quasi-affine as well (average polynomial function over $F$).
\end{proof}

It is interesting to record the following (cf. \cite[Th.3.2]{Knop2}):

\begin{lemma} \label{lemma non-deg}Let $Z=G/H$ be a non-degenerate real spherical space.  Then the set $\hf^\perp_{\rm ss}$ of semisimple
elements in $\hf^\perp$ has non-empty Zariski-open interior in $\hf^\perp$.
\end{lemma}

\begin{proof} Since $\gf_{\rm ss}$ has Zariski-open interior in $\gf$, it suffices to check that there is a non-empty open set of
semisimple elements in $\hf^\perp$.  Now $X_0$ is semisimple and for all elements
$X_1\in \af_Z+\mf_Z$ sufficiently close to $X_0$ we have in addition that $X_1+ \uf=\Ad(U)X_1$ by \cite[Lemma 2.6]{KKS}.
In view of \eqref{T0 orthogonal}  this implies that all elements $X_1 + T_0(X_1)$
are semisimple and belong to $(\af_Z+\mf_Z)^0$.
With Lemma \ref{lemma H perp polar} we conclude the proof.
\end{proof}

\begin{cor}\label{cor non-deg} Let $Z=G/H$ be a non-degenerate real spherical space and $E\subset \hf^\perp$.
Then the following are equivalent:
\begin{enumerate}
\item $\operatorname{int} E_{\rm ell}\neq \emptyset$.
\item $\operatorname{int} E_{\rm w-ell}\neq \emptyset$.
\end{enumerate}
\end{cor}

\begin{lemma}  \label{lemma elliptic}
The  following assertions hold:
\begin{enumerate}
\item \label{eins I} $\left[\Ad(\uH) (\af_Z+ \mf_Z)_\C^0\right]_{\rm w-ell} =  \Ad(\uH)\big((\af_Z +\mf_Z)_\C^0 \cap \zf(\gf_\C)+
i\af_Z^0  +\mf_Z^0 \big)$.
\item \label{zwei II} Suppose that $Z$ is non-degenerate and assume that
$\operatorname{int} \hf_{\rm w-ell}^\perp\neq \emptyset$. Then
$$ \operatorname{int}  \big (\hf^\perp \cap \Ad(\uH)(\zf(\gf)+ i\af_Z^0 +\mf_Z^0)  \big)\neq \emptyset\, .$$
\end{enumerate}
\end{lemma}

\begin{proof} For (\ref{eins I}) we  first observe that it suffices to show

$$\left[(\af_Z+ \mf_Z)_\C^0\right]_{\rm w-ell} = (\af_Z +\mf_Z)_\C^0 \cap \zf(\gf_\C)+i\af_Z^0  +\mf_Z^0 $$
%recall \eqref{T0 orthogonal}.
Let $Y\in  (\af_Z+ \mf_Z)_\C$ and
$X= Y + T_0(Y) \in (\af_Z+ \mf_Z)_\C^0$ as in \eqref{T0 orthogonal}. Then

$$\Spec (\ad X) = \Spec (\ad Y)\, .$$
Hence $X$ is weakly elliptic if and only if $Y\in  (\af_Z+\mf_Z)_\C\cap \zf(\gf_\C)+
i\af_Z  +\mf_Z$, that is, if and only if $X\in  (\af_Z+\mf_Z)_\C^0\cap \zf(\gf_\C)+
i\af_Z^0  +\mf_Z^0$.

For (\ref{zwei II}) we note that $\Ad(\uH)(\af_Z+\mf_Z)_\C^0$ is defined over $\R$ and Zariski dense in $\hf_\C^\perp$
as a consequence of Lemma \ref{lemma H perp polar}. Now (\ref{zwei II}) follows from (\ref{eins I}).
\end{proof}

Recall the edge $\af_{Z,E}\subset \af_Z$ and $\af_{Z,E}\subset \nf_\gf(\hf)$ with
$\nf_\gf(\hf)$ the normalizer of $\hf$ in $\gf$.

\begin{lemma}  \label{lemma no-elliptic} Let $Z$ be a non-degenerate real spherical space.
If $\af_{Z,E}\neq \{0\}$, then $\operatorname{int} \hf_{\rm w-ell}^\perp =\emptyset$.
\end{lemma}

\begin{proof} Let $\af_Z= \af_{Z,E} \oplus \af_{Z,S}$ be the orthogonal decomposition.
Recall $\hat \hf=\hf+\af_{Z,E}$ with $[\af_{Z,E}, \hf]\subset \hf$.  Define $\af_{Z,E}^0\subset \hf^\perp$
as below \eqref{T0 orthogonal}.  Then since $\af_{Z,E}^0 \cap \hat \hf^\perp=\{0\}$ we obtain by dimension count

\begin{equation} \label{h hat perp}  \hf^\perp=\hat\hf^\perp \oplus \af_{Z,E}^0\, .\end{equation}
Next we claim
\begin{equation} \label{center perp} \Ad(h)X - X \in \hat \hf_\C^\perp
\qquad (h\in \uH, X\in \af_{Z,E}^0)\, .\end{equation}
In fact, as $\uH$ is connected it suffices to show that $\kappa(e^{\ad Y} X, U)
=\kappa(X,U)$ for all $Y\in \hf_\C$ and $U\in\af_{Z,E}$.
By the invariance of the form $\kappa$ this is then implied by
$e^{-\ad Y}U \in U+\hf_\C$ as  $[\af_{Z,E}, \hf_\C]\subset \hf_\C$.
\par Suppose
$\operatorname{int} \hf_{\rm w-ell}^\perp \neq \emptyset$. According to Lemma \ref{lemma elliptic}
we thus find some subset $\Oc\subset i\af_{Z,E}^0 + i \af_{Z,S}^0+ \mf_Z^0$ such that
$\Ad(\uH)\Oc \cap \hf^\perp$ is open and non-empty.

Let $X= i X_1 + iX_2 + Y\in \Oc$ with $X_1 \in \af_{Z,E}^0,
X_2\in \af_{Z,S}^0, Y\in \mf_Z^0$ and let $h\in \uH$ be such that $\Ad(h)X\in\hf^\perp$.
With \eqref{center perp} we
get
$$\Ad(h) X= iX_1 +
\underbrace{(\Ad(h)(iX_1) - iX_1)}_{\in \hat \hf_\C^\perp}+
\underbrace{\Ad(h) (iX_2 + Y)}_{\in \hat \hf_\C^\perp}\in (i\af_{Z,E}^0+\hat\hf_\C^\perp)\cap \hf^\perp\, .$$
From \eqref{h hat perp} we then deduce $X_1=0$.  Hence $\Oc\subset i\af_{Z,S}^0+\mf_Z$.
Now as $\af_{Z,E}^0\neq \{0\}$ we have
$$\dim \hf/ \lf \cap \hf + \dim \af_{Z,S} + \dim \mf_Z <\dim \hf^\perp= \dim \gf/\hf$$
and therefore  $\Ad(\uH)(\af_{Z,S}^0 + \mf_Z^0)_\C\subset \hf_\C^\perp$ has empty interior,
a contradiction.  This concludes the proof.
\end{proof}

\begin{prop}\label{prop elliptic}  Let $Z=G/H$ be a unimodular  real spherical space. Suppose that $\operatorname{int} \hf_{\rm w-ell}^\perp\neq \emptyset$ where the interior is taken in $\hf^\perp$. Then $\hf$ is reductive in $\gf$.
\end{prop}

\begin{proof}  First we note that $Z$ is non-degenerate as $Z$ is requested to be  unimodular (see
Lemma  \ref{lem Harris} and Example \ref{ex-quasiaffine}).
We argue by contradiction and assume that
$\hf$ is not reductive.   Then \cite[Cor.~9.10]{KK} implies that $\af_{Z,E}\neq \{0\}$.
Now the assertion follows from Lemma \ref{lemma no-elliptic}.
\end{proof}

\begin{cor}\label{cor elliptic}  Let $\hf$ be a real spherical unimodular subalgebra and $I\subsetneq S$. Then
$\operatorname{int} \big( (\hf_I^\perp)_{\rm w-ell} \big) = \emptyset$.
\end{cor}

\begin{proof}  Since $\hf_I$ is a proper deformation of $\hf$ it cannot be reductive in $\gf$. Hence the
assertion follows from Proposition \ref{prop elliptic}.
\end{proof}

\subsection{Proof of Theorem \ref{thm discrete}}

\begin{proof}  The first assertion, $H=\hat H$ reductive in $G$, repeats Proposition \ref{prop elliptic}.
In particular $L^2(Z)_{\rm td}=L^2(Z)_{\rm d}$.

We recall that  to every
unitary representation $(\pi, E)$ of $G$ one attaches a wave-front set $\WF(\pi)$ which is an
$\Ad(G)$-invariant closed cone in  $\gf^*\simeq \gf$.
If $Z=G/H$ is a unimodular homogeneous space, then the wavefront set of the
left regular representation of $G$ on $L^2(G/H)$ was determined in \cite[Thm 2.1]{HW} as

\begin{equation}\label{WF-HW}  \WF(L^2(G/H))= \cl(\Ad(G) \hf^\perp)\end{equation}
with $\cl$ referring to the closure.

For the second assertion we compare wavefront sets of unitary $G$-representations. Recall that unitary representations with disintegration in the same measure class have the same wavefront sets.
Hence we obtain from Theorem \ref{thm planch} that
\begin{equation} \label{WF1} \WF(L^2(Z)) \subset \WF(L^2(Z)_{\rm d}) \cup  \bigcup_{I\subsetneq S}\bigcup_{\sc \in \sc_I} \WF(L^2(Z_{I,\sc})) \, .\end{equation}
On the other hand, we obtain from \eqref{WF-HW} that

\begin{equation} \label{WF2} \WF(L^2(Z_{I,\sc}))=\cl(\Ad(G) \hf_{I,\sc}^\perp)\qquad (I \subset S,\sc\in \sC_I)\, .\end{equation}

Let $Y:=\Ad(G)\hf^\perp\subset \gf$ and observe that $Y$ is the image of the
algebraic map
$$\Phi: G \times \hf^\perp \to \gf, \ \ (g,X)\mapsto \Ad(g)X\, .$$
In particular, it follows that $\dim \cl(Y) \bs Y <\dim Y$. Likewise
we have for $Y_{I,\sc}=\Ad(G) \hf_{I,\sc}^\perp$ that
$\dim \cl(Y_{I,\sc}) \bs Y_{I,\sc} < \dim Y_{I,\sc}$. By assumption
and Cor.~\ref{cor non-deg} the
elliptic elements $Y_{\rm ell}$ have non-empty interior
 in $Y$. Since $\dim \cl(Y) \bs Y <\dim Y$ we also obtain that $Y_{\rm ell}$ has non empty interior
$\operatorname{int}_{\cl(Y)}( Y_{\rm ell} )$
 in $\cl(Y)$.
On the other hand it follows from Corollary \ref{cor elliptic} that
$Y_{I,\sc, {\rm ell}}$  has no
interior in $Y_{I,\sc}$ when $I\neq S$.
From $\dim \cl(Y_{I,\sc}) \bs Y_{I,\sc} < \dim Y_{I,\sc}$
we thus infer that $(\cl (Y_{I,\sc}) )_{\rm ell}$ has empty interior
in $\cl(Y_{I,\sc})$.

\par From \eqref{WF-HW} and \eqref{WF1} we obtain

$$\emptyset \neq \operatorname{int}_{\cl(Y)}( Y_{\rm ell} )
\subset
\WF(L^2(Z)_{\rm d}) \cup  \bigcup_{I\subsetneq S, \sc}
[ \operatorname{int}_{\cl(Y)}( Y_{\rm ell} ) \cap \WF(L^2(Z_{I,\sc})) ],
$$
and since $Y_{I,\sc}\subset \cl(Y)$ it follows from \eqref{WF2} that
$$
\operatorname{int}_{\cl(Y)}( Y_{\rm ell} ) \cap \WF(L^2(Z_{I,\sc}))
\subset
\operatorname{int}_{\cl(Y_{I,\sc})} ( \cl(Y_{I,\sc})_{\rm ell} )=\emptyset
$$
for all $I\neq S$ and $\sc$. Hence $L^2(Z)_{\rm d}\neq 0$.
\end{proof}

\subsection{An example}

\begin{ex} \label{ex non-symmetric ell} We now give two examples of series of non-symmetric real spherical spaces
$Z=G/H$ for which $\operatorname{int} \hf_{\rm ell}^\perp \neq \emptyset$.
\par (a) Let $Z=G/H=\SO(n, n+1)/\GL(n,\R)$ for $n\geq 2$. We realize
$\gf =\so(n,n+1)$ as matrices of the form
$$X=\begin{pmatrix}  A  & B & v \\ C & -A^T & w \\  -w^T & -v^T& 0\end{pmatrix}$$
with $v,w\in \R^n$,  $A, B, C\in \Mat_{n\times n} (\R)$ subject to $B^T, C^T = -B, -C$.
Then $\hf$ consists of the matrices $X\in\gf$ with $B,C, v, w=0$.
First we consider the case where $n=2m$ is even. For
${\bf t}=(t_1, \ldots, t_m)\in \R^m$ we let
$D_{\bf t} =\diag( D_{t_1}, \ldots ,D_{t_m})\in \Mat_{n\times n}(\R)$ with $D_{t_i}= \begin{pmatrix}  0 & t_i \\
-t_i & 0\end{pmatrix}$. Further for ${\bf s}\in \R^m$ we set $v_{\bf s}=(s_1, s_1, s_2, s_2, \ldots, s_m,s_m)^T\in\R^n$.
Now consider the $n$-dimensional non-abelian subspace

$$\tf^0:= \left\{ \begin{pmatrix}  0 & D_{\bf t} & v_{\bf s}\\ -D_{\bf t}& 0 & v_{\bf s}  \\ -v_{\bf s}^T& - v_{\bf s}^T & 0\end{pmatrix} \mid  {\bf s}, {\bf t} \in \R^m, \right\}\subset \hf^\perp_{\rm ell}\, .$$
It is then easy to see that the $H$-stabilizer of a generic element $X\in\tf^0$ is trivial
with $[\hf, X] + \tf^0 =\hf^\perp$. Thus the polar map $H\times \tf^0\to \hf^\perp$ is
generically dominant and therefore $\operatorname{int}\hf_{\rm ell}^\perp \neq \emptyset$.
For $n=2m+1$ odd we modify $\tf^0$ as follows. We consider $D_{\bf t}$ now as
$n\times n$-matrix via the left upper corner embedding.  For ${\bf s}\in \R^{m+1}$ we further set
$v_{\bf s}=(s_1, s_1, \ldots, s_m, s_m, s_{m+1})\in \R^n$  and define

$$\tf^0:= \left\{ \begin{pmatrix}  0 & D_{\bf t} & v_{\bf s}\\ -D_{\bf t}& 0 & v_{\bf s}  \\ -v_{\bf s}^T& - v_{\bf s}^T & 0\end{pmatrix} \mid  {\bf s}\in \R^{m+1}, {\bf t} \in \R^m \right\}\subset \hf^\perp_{\rm ell}\, .$$
We now complete the arguments as in the even case.
\par (b)  Next we consider the cases $Z=G/H=\SU(n,n+1)/\Sp(2n,\R)$ for $n\geq 2$. Here
$\gf=\su(n,n+1)$ is realized as
the trace-free matrices of the form
$$X=\begin{pmatrix}  A  & B & v \\ C & -A^* & w \\  -w^* & -v^*& d\end{pmatrix}$$
with $v,w\in \C^n$, $A, B, C\in \Mat_{n\times n} (\C)$ subject to $B^*, C^* = -B, -C$, and $d\in i\R$.
Further we realize $\hf\simeq \sp(2n,\R)$ as the subalgebra
$$\hf=\{ X\in \gf\mid A\in \Mat_{n\times n}(\R), B,C\in i \Mat_{n\times n} (\R), v=w=0, d=0\}$$

For ${\bf t}=(t_1, \ldots, t_n)\in \C^n$ we let
$E_{\bf t}=\diag(t_1, \ldots, t_m)\in \Mat_{n\times n}(\C)$
and consider
$$\tf^0:= \left\{ X=
\begin{pmatrix}
E_{i\bf t}  & 0  & {\bf s} \\
0& E_{i\bf t}  & {\bf s} \\
-{\bf s}^T& -{\bf s}^T & d\end{pmatrix}
\mid {\bf s}, {\bf t}\in \R^n, \tr(X)=0\right\}\subset \hf^\perp_{\rm ell}\, .$$
Now proceed as in (a) and obtain that
$\operatorname{int} \hf^\perp_{\rm ell}\neq \emptyset$.
\end{ex}

\begin{cor}  For $Z=\SU(n,n+1),\R)/ \Sp(2n,\R)$ and $Z=\SO(n, n+1)/\GL(n,\R)$, $n\geq 2$,
we have $L^2(Z)_{\rm d}\neq \emptyset$.
\end{cor}

\begin{proof} In Example \ref{ex non-symmetric ell} we have shown $\operatorname{int} \hf_{\rm ell}^\perp\neq \emptyset$.
Apply Theorem \ref{thm discrete}. \end{proof}

\section{Moment maps and elliptic geometry}\label{section moment}
We expect that Theorem \ref{thm discrete} gives in fact an equivalence:  $L^2(Z)_{\rm d}\neq \{0\}$
if and only if $\operatorname{int} \hf^\perp_{\rm w-ell}\neq \emptyset$.
This section is devoted to the following theorem, which gives a geometric version of this expected equivalence.

\begin{theorem} \label{thm moment discrete}  Let $Z$ be a  non-degenerate real spherical space
with a strictly convex compression cone, i.e.~$\af_{Z,E}=\{0\}$. Then the
following statements are equivalent:
\begin{enumerate}

\item \label{EINS}$\cl(\Ad(G)\hf^\perp) =  \bigcup_{I\subsetneq  S}\bigcup_{\sc \in \sC_I}   \Ad(G)\hf_{I,\sc}^\perp$.
\item \label{ZWEI}$\operatorname{int} \hf^\perp_{\rm ell}=\emptyset$.
\end{enumerate}
\end{theorem}

\begin{rmk} \label{rmk thm moment discrete} (a) From Corollary \ref{cor elliptic}
we obtain that

$$\operatorname{int}_{\cl(\Ad(G)\hf^\perp)}  [\Ad(G)\hf_{I,\sc}^{\perp}]_{\rm ell} =\emptyset$$
for all $I\subsetneq S$ (see also the proof of Theorem \ref{thm discrete}). Hence we get
\eqref{EINS} $\Rightarrow$ \eqref{ZWEI}, which
is
the geometric equivalent of Theorem \ref{thm discrete}.

\par (b) Note that \eqref{ZWEI} is equivalent to $\operatorname{int} \hf^\perp_{\rm w-ell}=\emptyset$
by the assumption of non-degeneracy (see Cor.~\ref{cor non-deg}).
\par (c) For fixed $I\subset S$ we recall
$$\{ \hf_{I,\sc}: \sc \in \sC_I\} = \{ (\hf_w)_I:  w \in \Wc\}\, .$$
\end{rmk}

The goal of this section is to prove Theorem \ref{thm moment discrete}. The proof is
obtained via new insights on the geometry of the moment map of the Hamiltonian
$G$-action on the co-tangent bundle $T^*Z$,

\subsection{The moment map}

In this subsection $Z=G/H$ is a general algebraic homogeneous
space attached to a reductive group $G=\uG(\R)$ and an algebraic subgroup $H=\uH(\R)$.

In the sequel we identify  $\gf^*$ with $\gf$ via our non-degenerate
$\Ad(G)$-invariant form $\kappa$.  In this sense we also have $(\gf/\hf)^*\simeq \hf^\perp\subset \gf$
and we can view the co-tangent bundle $T^*Z$ of $Z$ as
$T^*Z=G\times_H \hf^\perp$. Recall that the $G$-action on $T^*Z$ is
Hamiltonian with corresponding $G$-equivariant moment map  given by

$$m:  T^*Z  \to \gf, \ \ [g,X]\mapsto \Ad(g)X\, .$$
Now for $X\in \hf^\perp$ the stabilizer in $G$ of $\xi:=[\1, X]\in T^*Z$ is
$G_\xi=  Z_H(X)$ whereas the stabilizer of $X=m(\xi)\in\gf$ is
$G_{m(\xi)}=Z_G(X)$.
It is then a general fact about the geometry of moment maps (see \cite[p.190]{GS}), that for the Lie algebras of $Z_H(X)$ and $Z_G(X)$ one has
\begin{equation} \label{centralizer normal} \zf_\hf(X)\triangleleft \zf_\gf(X)  \qquad (X\in \hf^\perp)\, .\end{equation}
Let us call an element
$X\in \hf^\perp$ {\it generic}, provided that $\dim \zf_\hf(X)$ is minimal. Then it follows from
 \cite[Th. 26.5]{GS} that

\begin{equation} \label{centralizer abelian}  \zf_\gf(X)/\zf_\hf(X) \ \text {is abelian for $X\in \hf^\perp$ generic}\, .\end{equation}
A somewhat sharper version of \eqref{centralizer abelian}  is:

\begin{lemma} \label{lemma moment torus}
\cite[Satz 8.1]{Knop} Assume that $\uZ=\uG/\uH$ is an algebraic homogeneous space defined over $\R$ attached to
a connected reductive group $\uG$. Then for $X$ in a dense open subset of $\hf^\perp$ one has
\begin{enumerate}
\item $Z_\uH(X) \triangleleft Z_\uG(X)$.
\item $Z_\uG(X)/Z_\uH(X)$ is a torus.
\end{enumerate}
In particular, $Z_G(X)/Z_H(X)$ is an abelian reductive Lie group.\end{lemma}

\subsection{Ellipticity relative to $Z$}

Moment map geometry suggests notions of ellipticity and weak ellipticity of elements $X\in \hf^\perp$
which are more intrinsic to $Z$.
\par Let us call an element  $X\in \hf^\perp$  {\it weakly $Z$-elliptic} provided that  $Z_G(X)/Z_H(X)$ is compact.
A weakly $Z$-elliptic element $X\in \hf^\perp$  will be called {\it $Z$-elliptic} if in addition
$X$ is semisimple.

\begin{lemma} \label{lemma w-Z-ell} Let $X$ be a generic weakly $Z$-elliptic element and let $(Z_G(X))_0= L_X \ltimes U_X$  be
a Levi-decomposition with $L_X$ reductive. Let $L_{H,X}:=L_X \cap H$.
Then $(Z_H(X))_0= L_{H,X}\ltimes U_X$ and there exists a compact torus
$T_X$ in the center $Z(L_X)$ of $L_X$ such that $L_X= L_{H,X} T_X$ and
$\lf_X =\lf_{H,X} \oplus \tf_X$ orthogonal.  Moreover,
\begin{equation}  \label{X in T_X}   X \in \tf_X +\uf_X  \end{equation}
and $X\in \tf_X$ if $X$ is semisimple.   In particular,
\begin{enumerate}
\item\label{11} Every generic weakly $Z$-elliptic element is weakly elliptic.
\item\label{12} Every generic $Z$-elliptic element is elliptic.
\end{enumerate}
\end{lemma}

\begin{proof}  Let $G_1:=(Z_G(X))_0$ and $G_2:=(Z_H(X))_0$.  Then by \eqref{centralizer normal} and \eqref{centralizer abelian},
$G_2\triangleleft G_1$ is a normal subgroup such that $G_1/G_2$ is compact, connected and abelian, i.e.~a compact torus.
Furthermore $G_3:= G_2 U_X$ is a closed normal subgroup such that $G_1/G_3 = L_X/ L_X \cap G_3$
is a compact torus.  This implies that $L_X =  (G_3\cap L_X) T_X$ with $T_X$ an (infinitesimally)
complementing compact torus in the center of $L_X$.  It then follows that
$G_2\cap L_X = G_3\cap L_X$, as there are no algebraic morphisms of a reductive group to a unipotent group.
Now the compactness of $G_1/G_2$ implies that $U_X\subset G_2$ as well.
Furthermore,  since $\lf_X$ and $\lf_{X,H}$ are both algebraic Lie algebras we see that
$\tf_X=\lf_{H,X}^{\perp_{\lf_X}}$ is the orthogonal complement.
\par Finally we decompose $X\in \hf^\perp\cap \zf_\gf(X)$ as $X=X_0 +X_1$ with  $X_0\in \lf_X$ and
$X_1\in \uf_X$.  Then $X\in \hf^\perp$ implies that $X_0\in\tf_X$, that is \eqref{X in T_X}.
If we further observe $X$ is semisimple if and only if $\zf_\gf(X)=\lf_X$ is reductive, then we see that the  remaining
statements of the lemma are consequences of \eqref{X in T_X}.
\end{proof}

\begin{rmk} Notice that $X=0$ is semisimple and elliptic but not weakly $Z$-elliptic unless
$Z=G/H$ is compact.  To see an example of  a generic weakly $Z$-elliptic element
which is not $Z$-elliptic, i.e. not semisimple, consider
$H=N$ for an $\R$-split group $G$.  Then $\hf^\perp =\af +\nf$.
Now for a regular nilpotent element $X\in \nf$ we have $Z_G(X)=Z_N(X)$ and thus
$X$ is generic and weakly $Z$-elliptic.
\end{rmk}

Our notion of non-degeneracy for real spherical spaces now generalizes
to all algebraic homogeneous spaces $Z=G/H$ as follows.
We call $Z=G/H$ {\it non-degenerate} provided that
$m(T^*Z)$ contains a  Zariski dense
open set of semisimple elements.  We recall from
\cite[Sect.~3] {Knop2} that all quasi-affine homogeneous spaces are non-degenerate.

\begin{prop} Let $Z=G/H$ be a non-degenerate homogeneous space.
Then the following assertions are equivalent:
\begin{enumerate}
\item \label{eins-aa}$\operatorname{int} \hf_{\rm ell}^\perp\neq \emptyset$.
\item \label{zwei-bb}$\operatorname{int} \hf_{\rm Z-ell}^\perp\neq \emptyset$.
\end{enumerate}
\end{prop}

\begin{proof} Here we prove \eqref{eins-aa}$\Rightarrow$\eqref{zwei-bb},
as the converse implication follows
immediately from Lemma~\ref{lemma w-Z-ell}\eqref{12} (in fact without assuming non-degeneracy).

Since $Z$ is non-degenerate, the image
$m(\alpha)$ is semisimple for $\alpha=[g,X]$ in a dense open subset of
$T^*Z=G\times_H\hf^\perp$. For those $\alpha$, the centralizer $\uL(\alpha):=Z_\uG(m(\alpha))$
of $m(\alpha)$ is a Levi subgroup of $\uG$ which is defined over $\R$. Since there are only finitely
many conjugacy classes of such subgroups, there is a dense open subset $\Tc$ of
$T^*Z$ such that for each $\alpha\in\Tc$,
$\uL(\alpha)$ is a Levi subgroup of
$\uG$ and the $G$-conjugacy class of its real points $L(\alpha)$
is locally constant on $\Tc$.
\par Let $\Tc_0$ be a connected component of $\Tc$,
and let  $\alpha_0\in\Tc_0$ and $\uL=\uL(\alpha_0)$.
Then $L(\alpha)$ is
$G$-conjugate to $L$ for all $\alpha\in\Tc_0$. Moreover, $\Tc_0$ is a
Hamiltonian $G$-manifold with moment map $m|_{\Tc_0}:\Tc_{0}\to\gf$.

\par Set $\Tc_{00}:=m_{\Tc_0}^{-1}(\lf)$. Then it follows from the Cross Section Theorem (cf. \cite[Th. 2.4.1]{GS2})
that $\Tc_{00}$ is a Hamiltonian $L$-manifold with moment map
$m|_{\Tc_{00}}:\Tc_{00}\to\lf$ the restriction of $m$ to $\Tc_{00}$.

\par Note that
$Z_G(m(\alpha_0))=L(\alpha_0)=L$. In particular $m(\alpha_0)\in \zf(\lf)$
is regular.  As $m(\Tc_{00})\subset \lf$ we thus find an
open neighborhood $U_0$ of $\alpha_0$ in $\Tc_{00}$  such that $L(\alpha)=Z_G(m(\alpha))\subset L$ for all $\alpha\in U_0$. On the other hand we know that $L(\alpha)$ is conjugate to $L$. Thus in fact $L(\alpha)=L$ for $\alpha\in U_0$. Hence by passing to
a dense open subset of $\Tc_{00}$ we may assume that
$L(\alpha)=L$ for all $\alpha\in \Tc_{00}$.
Since $G_\alpha\subset G_{m(\alpha)}=L$ we then have $G_\alpha=L_\alpha$ with
$L_\alpha$ the stabilizer of $\alpha\in \Tc_{00}$ in $L$.

\par Let $\cf\subset\lf$ be the $\R$-span of
$m(\Tc_{00})$.
We claim that
\begin{equation} \label{center moment} \cf\subset \zf(\lf)\end{equation}
with $\zf(\lf)$ the center of $\lf$. In fact, we have just seen that $G_{m(\alpha)}=L$ for all $\alpha\in \Tc_{00}$. Thus
$m(\alpha)\in \zf(\lf)$ for all $\alpha\in \Tc_{00}$.

\par  Next we recall the basic equivariant property for the derivative of the moment map \cite[eq. (26.2)] {GS}:

\begin{equation} \label{moment equivariant}  \kappa (d m(\alpha)(v), X) = \Omega_\alpha (\widetilde X_\alpha, v) \qquad (\alpha \in \Tc_{00}, v \in T_\alpha \Tc_{00}, X\in \lf)\end{equation}
where $\Omega$ is the symplectic form on $\Tc_{00}$,  $\widetilde X$ is the vector field on $\Tc_{00}$
associated to $X$ and $T_\alpha \Tc_{00}$ is the tangent space at $\alpha$.
Let $\gf_\alpha=\lf_\alpha$ be the Lie algebra of the stabilizer $G_\alpha=L_\alpha$  of $\alpha\in \Tc_{00}$.
We claim that $\cf^{\perp_\lf} \subset \gf_\alpha$.  To see that we first note that $dm(\alpha)(v)\in \cf $ by the definition of $\cf$. Hence we derive  from \eqref{moment equivariant} that $\Omega_\alpha(\widetilde X_\alpha, v)=0$ for all
$v$ if $X\perp \cf$. Since $\Omega$ is non-degenerate one obtains
$\widetilde X_\alpha=0$. Hence $\cf^{\perp_\lf}$ acts with vanishing vector fields on $\Tc_{00}$ and thus
$\cf^{\perp_\lf}\subset \gf_\alpha=\lf_\alpha$.
\par Notice that the claim implies in particular that the $L$-action on $\Tc_{00}$ factors through
the group $C:=L/ \la \exp(\cf^{\perp_\lf})\ra$ with Lie algebra $\cf$.

On the other hand, by passing to
a further open dense subset of $\Tc_{00}$  we may assume that $C_\alpha= G_{m(\alpha)}/ G_\alpha=L/G_\alpha$  is
a real form of a complex torus for all $\alpha\in \Tc_{00}$, see Lemma \ref{lemma moment torus}. Notice that
$C_\alpha$ is a quotient of $\cf$ and likewise  the Lie algebra $\cf_\alpha$ of $C_\alpha$ is a quotient of $\cf$.

Via
the non-degenerate form $\kappa$ we realize $\cf_\alpha$ as a subalgebra of $\cf\subset \lf$ and note that
$C_\alpha$ is compact if and only if $\cf_\alpha$ consists of elliptic elements. Further  $m(\alpha)\in \cf_\alpha$.

\par  From $\Tc_0=G\cdot \Tc_{00}$ we obtain that
for all $\xi$ in a dense open subset of $\Tc_0$ it holds true that
$m(\Tc_0)$ consists of elliptic elements  if $G_{m(\xi)}/ G_\xi$ is a compact torus.

Finally, every $\alpha\in T^*Z$ is in the $G$-orbit of an element $\xi=  [\1,X]$ with
$X\in\hf^\perp$ for which we recall $G_{m(\xi)}/G_\xi= Z_G(X)/Z_H(X)$.
Now the implication \eqref{eins-aa}$\Rightarrow$\eqref{zwei-bb}
follows from $m([\1,X])=X$.
\end{proof}

\subsection{The logarithmic tangent bundle}
Let $Z\hookrightarrow\hat Z$  be a compactification corresponding to a complete fan
$\Fc$ as in Section \ref {section compact}. In particular we recall
that $\hat Z$  was constructed as the closure of $Z$ in the smooth toroidal compactification
$\hat \uZ(\R)$ of $\uZ(\R)$ attached to $\Fc$.

\par According to \cite[Cor.~12.3]{KK}, there is a unique $G$-equivariant morphism
$\phi:\hat \uZ(\R)\to \Gr(\gf)$ into the Grassmannian of $\gf$ with
$\phi(z_0)=\hf^\perp$. Let $\mathcal E\to\Gr(\gf)$ be the
tautological vector bundle. Then the \emph{logarithmic cotangent bundle
  of $\hat \uZ(\R)$} is defined by $T^{\log} \hat \uZ(\R):=\phi^*\mathcal E$. Concretely
\[
  T^{\log}\hat \uZ(\R)=\{(z,X)\in\hat \uZ(\R)\times\gf\mid
  X\in\phi(z)\}.
\]
Then $T^{\log}\hat \uZ(\R)$ is a smooth $G$-manifold containing $T^*\uZ(\R)$ as
a dense open subset. It comes with a projection to the first factor
\[
  p:T^{\log}\hat \uZ(\R)\to\hat \uZ(\R), \ \ (z,X)\mapsto z
\]
making it into a vector bundle. On the other hand, the second
projection
\[
  m:T^{\log}\hat \uZ(\R)\to\gf, \ \ (z,X)\mapsto X
\]
is called the \emph{logarithmic moment map} since it restricts to
the moment map on $T^*Z$. Since $\hat \uZ(\R)$ is compact, the
logarithmic moment map is proper in the Hausdorff topology.

Next we recall from Section \ref{section compact} that each cone $\Cc\in\Fc$ corresponds to a $\uG$-orbit
$\hat \uZ_\Cc=\uG\cdot \hat z_{\Cc}\subset\hat \uZ$.  We have defined  $A_\Cc\subset A_Z$ to be the subtorus
with Lie algebra $\af_\Cc= \Span_\R \Cc$. Moreover for $I=I(\Cc)$ the
set of spherical roots vanishing on $\Cc$, we have $\af_\Cc\subset \af_I$ and
$\hat \hf_\Cc= \hf_I +\af_\Cc$.  Also recall
$\hat \uZ_\Cc \simeq \uG/\uA_\Cc \uH_I=\uZ_I/\uA_\Cc$.
Next we recall from Remark \ref{remark rel open}(c) that
\begin{equation} \label{real ZC}\hat Z \cap \hat \uZ_\Cc(\R)=\bigcup_{w\in \Wc} G \cdot
\hat z_{w,\Cc}\, .\end{equation}

\par Set $T^{\log}:= p^{-1}(\hat Z)$. For all $\Cc\in \Fc$ we
define  $T^{\log}_\Cc:=p^{-1}(\hat\uZ_\Cc\cap \hat Z)$ and note that $T^{\log} =\coprod_{\Cc\in \Fc} T_\Cc^{\log}$. Furthermore, for
 $I\subset S$ we put $T_I^{\log}:= \bigcup_{\Cc\in \Fc\atop I=I(\Cc)} T_\Cc^{\log}$.
 Since $m(\hat z_{w,\Cc})=\hf_{w,\Cc}=(\hf_w)_I$ for all $\Cc\in \Fc$ with $I(\Cc)=I$ we obtain
with Remark \ref{rmk thm moment discrete}(c)  and \eqref{real ZC} that
 \begin{equation} \label{log moment}
m(T^{\log}_I)=\bigcup_{\sc\in \sC_I}\Ad(G)\hf_{I,\sc}^\perp\, .
\end{equation}

 \subsection{Proof of Theorem \ref{thm moment discrete}}
 As mentioned in Remark \ref{rmk thm moment discrete}(a) we only need to show
\eqref{ZWEI} $\Rightarrow$ \eqref{EINS}.  Let $\alpha\in T^*Z$ be
generic. Then $m(\alpha)$ is not elliptic by assumption. Hence the
torus $A_\alpha:=G_{m(\alpha)}/ G_\alpha$ is not compact and
therefore contains a 1-parameter subgroup
  $\mu:\mathbb \R^\times \hookrightarrow \uA_\alpha(\R)$. Consider the orbit
  $A_\alpha\cdot \alpha\subset T^*Z$. Since its projection into $Z$ is closed (being a
  flat) also $A_\alpha\cdot \alpha$ is closed in $T^*Z$. The limit
  $\alpha_0:=\lim_{t\to0^+}\mu(t)\alpha$ exists in $\hat  Z$ since $m$
  is proper. Since $\alpha_0\not\in T^*Z$ we have
  $\alpha_0\in T^{\log}_I $ for some $I\neq S$ (here we used that the compression cone is strictly convex
  which implies that $T^{\log}_S= T^*Z$.)
  Hence

 $$  m(\alpha)=\Ad(\mu(t)) (m(\alpha))= m\left(\lim_{t\to 0^+}  \mu(t) \alpha\right)= m(\alpha_0)
  \in m(T^{\log}_I)=\bigcup_{\sc\in \sC_I}\Ad(G)\hf_{I,\sc}^\perp
 $$
by \eqref{log moment}.  Thus we obtain for $\alpha\in T^*Z$ generic that
\begin{equation}\label{dense generic} m(\alpha)\in \bigcup_{I\subsetneq  S}\bigcup_{\sc\in \sC_I} \Ad(G)\hf_{I,c}^\perp\, .
\end{equation}
Since the right hand side in \eqref{dense generic} consists of all proper deformations of $\Ad(G)\hf^\perp$,
hence is closed in $\cl(\Ad(G) \hf^\perp)$,  we obtain \eqref{EINS} from \eqref{dense generic} and the density of the generic elements.
\qed

\section{Harish-Chandra's group case}\label{group case}

In this section we apply the results of this paper to derive
Harish-Chandra's  formula for the Plancherel measure for a real reductive group \cite{HC3}. The Plancherel
measure contains naturally the formal degrees of discrete series representations of various inducing data.
The formal degrees were computed by Harish-Chandra in \cite{HC}.
The explicit knowledge of the formal degree is
treated as a black box in what follows.

\par We are considering a real reductive group $G'$ together with its both-sided
symmetries $G=G'\times G'$, by which $G'$ gets identified  with $Z=G/H$ where
$H=\diag(G')\subset G$ is the diagonal subgroup.  Let us recall that the topological assumption on $G'$ is that
$G'=\uG'(\R)$ for a reductive algebraic group $\uG'$ which is assumed to be connected.
If $P'=M'A'N'\subset G'$ is a minimal
parabolic subgroup of $G'$ and $\oline P'$ is  its opposite, then we obtain with $P=P' \times \oline{P'}\subset G$
a minimal parabolic subgroup of $G$ with $PH\subset G$ open and dense  as consequence of the Bruhat decomposition.   In particular $\Wc=\{\1\}$.

Next note that $\af= \af'\times \af'$,  $\af_H= \diag(\af')$ and $\af_Z=\af_H^{\perp_\af}$ is
the anti-diagonal
$$\af_Z=\{ (X,-X)\mid X\in \af'\}\, .$$
The assignment
$$\af'\to \af_Z, \ \ X\mapsto  \frac12(X, -X)$$
gives a natural identification.  If we denote by $\Sigma'=\Sigma(\af',\gf')\subset (\af')^*\bs \{0\}$
the (possibly reduced) root system for the pair $(\af',\gf')$, and further by $\Phi'\subset \Sigma'$ the set of simple roots
determined by the positive roots  $\Sigma'(\af',\nf')$, then the set of spherical roots $S\subset \af_Z^*$
naturally identifies with $\Phi'$.

\subsection{The abstract Plancherel Theorem for $L^2(Z)$}
Here we specialize the abstract Plancherel theory of Section \ref{subs APt} to the case at hand.
Recall that

$$(L, L^2(Z))\simeq  \left(\int_{\hat G} \pi \otimes \id \ d\mu(\pi), \int_{\hat G} \Hc_\pi\otimes \M_\pi
\ d\mu(\pi)\right)$$
with $\M_\pi\subset (\Hc_\pi^{-\infty})^H$.

Now any $\pi \in \hat G$ has the form $\pi=\pi_1\otimes \pi_2$ with $\pi_i\in \hat G'$.
Further, since $\Hc_{\pi_i}^\infty$ is a nuclear  Fr\'echet space (as a consequence of
Harish-Chandra's admissibility theorem) we have $\Hc_\pi^\infty=\Hc_{\pi_1}^\infty \hat\otimes \Hc_{\pi_2}^\infty
\simeq \Hom (\Hc_{\pi_1}^{-\infty}, \Hc_{\pi_2}^\infty)$ together with $\Hc_\pi^{-\infty}=\Hc_{\pi_1}^{-\infty} \hat\otimes \Hc_{\pi_2}^{-\infty}\simeq \Hom (\Hc_{\pi_1}^\infty, \Hc_{\pi_2}^{-\infty})$. Thus
$$(\Hc_\pi^{-\infty})^H \simeq \Hom_{G'}(\Hc_{\pi_1}^\infty, \Hc_{\pi_2}^{-\infty})\, .$$
We then claim
\begin{equation} \label{HC-claim1}  \dim (\Hc_\pi^{-\infty})^H\leq 1\end{equation}
and
\begin{equation} \label{HC-claim2} (\Hc_\pi^{-\infty})^H\neq \{0\} \iff \pi_2\simeq \oline \pi_1\end{equation}
with $\oline \pi_1$ the dual representation of $\pi_1$.
\par We first show "$\Rightarrow$" of \eqref{HC-claim2} and assume
 that $(\Hc_\pi^{-\infty})^H\neq \{0\}$.
 This means that  $\Hom_{G'}(\Hc_{\pi_1}^\infty, \Hc_{\pi_2}^{-\infty})\neq \{0\}$.
On the level of Harish-Chandra modules this yields $\Hom_{\gf'}(V_{\pi_1}, V_{\oline \pi_2} )\neq \{0\}$
and thus $\pi_2\simeq\oline \pi_1$.  The same reasoning also shows \eqref{HC-claim1}.

\par To see the converse in \eqref{HC-claim2}, we first supply some useful notation.  Given a Hilbert space $\Hc$ we denote by ${\mathcal B}_2(\Hc)$ the Hilbert space of Hilbert-Schmidt operators
and note that  ${\mathcal B}_2(\Hc)\simeq \Hc\hat \otimes \oline{\Hc}$ with $\hat \otimes$ the tensor product in the category
of Hilbert space and $\oline \Hc$ the dual to $\Hc$. Further we denote by ${\mathcal B}_1(\Hc)\subset {\mathcal B}_2(\Hc)$ the space of trace-class operators.

\par   Given a unitary representation $(\pi, \Hc_\pi)$ of $G'$, we set $\Hc_\Pi={\mathcal B}_2(\Hc_\pi)$ and obtain
a unitary representation $(\Pi, \Hc_\Pi)$
of $G=G'\times G'$  by

$$\Pi(g_1', g_2')T = \pi(g_1')\circ T \circ \pi(g_2')^{-1}\qquad (g_1',g_2'\in G', T\in \Hc_\Pi={\mathcal B}_2(\Hc_\pi))\, .$$
Notice that $\Pi\simeq \pi\otimes\oline{\pi}$ under the isomorphism ${\mathcal B}_2(\Hc_\pi)\simeq \Hc_\pi\hat \otimes \oline{\Hc_\pi}$,
and that the HS-norm on  ${\mathcal B}_2(\Hc_\pi)$ does not depend on the positive scaling class of the Hilbert norm
which defines the Hilbertian structure of $\Hc_\pi$.

Let us assume from now on
that $(\pi, \Hc_\pi)$ is irreducible.
We remind  that Harish-Chandra's basic admissibility theorem implies
$$\Hc_\Pi^\infty \subset {\mathcal B}_1(\Hc_\pi)\, .$$
Together with \eqref{HC-claim1} we thus obtain that

$$ (\Hc_\Pi^{-\infty})^H= \C  \tr_\pi$$
with $\tr_\pi$ denoting the restriction of the trace on ${\mathcal B }_1(\Hc_\pi)$ to $\Hc_\Pi^\infty$.  In particular, this
completes the proof of \eqref{HC-claim2}.

\par  From \eqref{HC-claim2} we then deduce
$$\supp \mu\subset \{ [\Pi]\mid [\pi]\in \hat{G'}\}\simeq \hat G'$$
and
$$\M_\Pi= \C  \tr_\pi\,\qquad ([\pi]\in\supp\mu)\, .$$
\par As the Hilbert-Schmidt norm on $\Hc_\Pi={\mathcal B}_2(\Hc_\pi)$ is
independent of the particular $G'$-invariant Hilbert norm on $\Hc_\pi$
we obtain a  natural Hilbert space structure on the one-dimensional space $\M_\Pi$
by the request $\|\tr_\pi\|=1$.
Then the natural left right representation $L=L'\otimes R'$ of $G=G'\times G'$ on $L^2(Z)$ decomposes as

$$(L'\otimes R', L^2(Z)) \underset{G'\times G'}\simeq\left(\int_{\hat {G'}}^\oplus \Pi \ d\mu(\pi),   \int_{\hat {G'}}^\oplus \Hc_\Pi \ d\mu(\pi)\right)\, .$$

\subsection{The Plancherel Theorem for $L^2(Z_I)_{\rm td}$}

\par We recall from Theorem \ref{thm planch}
the Bernstein decomposition
$$L^2(Z)= \sum_{I\subset S} B_I(L^2(Z_I)_{\rm td})\, .$$

\par For $I\subset S\simeq \Phi'$  we obtain a standard parabolic $P_I'= M_I' A_I' N_I'\supset P'$ and the
deformation $H_I$ of $H$ as
$$H_I =  \diag (M_I' A_I') (\oline{N_I'} \times N_I')\, .$$
with
$$\hat H_I = \diag (M_I') ( A_I'\oline{N_I'} \times A_I'N_I')\, .$$
Next we describe $L^2(Z_I)_{\rm td}$.  As in Subsection \ref{Subsection twisted} we decompose
every $f\in L^2(Z_I)$ as an  $A_I$-Fourier integral
$$f= \int_{i\af_I^*} f_\lambda \ d \lambda$$
where $f_\lambda \in L^2(\hat Z_I, \lambda)$ is given by
$$ f_\lambda(g)= \int_{A_I}  a^{-\rho - \lambda} f(gaH_I) \ da \qquad (g\in G)$$
If we denote by
$$\xi_\lambda: L^2(\hat Z_I, \lambda)^\infty\to \C,  \ \ f\mapsto f(\1)$$
the evaluation at $\1$, and write $L_\lambda$ for the left regular representation of $G$
on $L^2(\hat Z_I,\lambda)$, then we can rewrite the Fourier-inversion in terms of spherical characters
(as in Remark \ref{F-inverse})
\begin{equation} \label{Four1} f(z_{0,I}) =  \int_{i\af_I^*} \xi_\lambda (L_{-\lambda} (f) \xi_{-\lambda}) \ d \lambda\end{equation}
with $L_{-\lambda}= \oline {L_\lambda}$ the dual representation and $\xi_{-\lambda}= \oline{\xi_\lambda}$.
Next note that  we have by induction in stages

$$ L^2(\hat Z_I,  -\lambda)=\Ind_{\hat H_I}^G (\lambda)\simeq
\Ind_{\oline {P_I'}\times P_I'}^{G'\times G'} (L^2(M_I')\otimes \lambda)\, .$$
Thus $L^2(\hat Z_I, -\lambda)_{\rm d}$ is induced from the discrete series of $M_I'$.
\par In more detail, let  $(\sigma, \Hc_\sigma)$ be a discrete series representation of
$M_I'$ and $\lambda\in i(\af')^*$.   Then we denote by
$\Ind_{\oline {P_I'}}^{G'}(\sigma\otimes\lambda)$ the Hilbert space of measurable functions $f:  G'\to \Hc_\sigma$ with the transformation
property
$$ f(g' m_I' a_I' \oline{n_I'}) = \sigma(m_I')^{-1}  (a_I')^{-\lambda +\rho'} f(g') \qquad (g'\in G', m_I' a_I' \oline{n_I'}\in \oline{P_I'})$$
and endowed with the inner product
(of which the convergence is an extra assumption)

$$ \la f_1, f_2\ra =\int_{K'} \la  f_1(k'), f_2(k')\ra_\sigma \ dk'$$
where $K'\subset G'$ is a maximal compact subgroup of $G'$ with $\kf'\perp \af'$. The left regular representation
of $G'$ on $\Hc_{\sigma,\lambda}:=\Ind_{\oline {P_I'}}^{G'}(\sigma\otimes\lambda)$ is then unitary and denoted by
$\pi_{\sigma,\lambda}=\ind_{\oline {P_I'}}^{G'}(\sigma\otimes\lambda)$. Let us denote by $d(\sigma)$ the formal degree of the discrete
series representation of $M_I'$ (with respect to a chosen Haar measure $dm_I'$), i.e. the positive number for which we have

\begin{equation} \label{formal deg}
d(\sigma) \int_{M_I'}
\la \sigma(m_I') u,u'\ra \oline{\la \sigma(m_I') v,v'\ra} \ dm_I' =
\la u,v\ra \oline{ \la u',v' \ra }\end{equation}
for all $v,v',u,u'\in \Hc_\sigma$.

We now define a $G'\times G'$-equivariant linear map

$$\Phi_{\sigma,\lambda}:  \Ind_{\oline {P_I'}}^{G'}(\sigma\otimes\lambda)\hat \otimes \Ind_{P_I'}^{G'}(\oline \sigma\otimes(-\lambda)\to
L^2(\hat Z_I, - \lambda)_{\rm d}$$
by
$$ \Phi_{\sigma,\lambda}(f_1 \otimes f_2)(g_1', g_2'):= ( f_1(g_1'), f_2(g_2'))_\sigma \,,$$
with $( \cdot,\cdot)_\sigma$ referring here to the natural bilinear  pairing of $\sigma$ with its dual representation $\oline \sigma$.
The square integrability of the image follows from  the fact
that the norm for
$f\in L^2(\hat Z_I, - \lambda)$ can be computed
by means of the Haar measures on $K'$ and $M_I'$ (with the latter properly normalized)
as

$$ \|f\|_{L^2(\hat Z_I,-\lambda)}^2=  \int_{K'}\int_{K'}\int_{M_I'}  |f(k_1'm_I', k_2')|^2 \ dm_I' \ dk_1' \ dk_2'\, .$$
 In fact, with \eqref{formal deg} this calculation shows that
$d(\sigma)^{1/2} \Phi_{\sigma,\lambda}$ is isometric.

With the operator
$\sum_{\sigma} \int \Phi_{\sigma,\lambda}\,  d_\sigma\lambda$
we thus obtain a unitary $G$-equivalence
 \begin{equation} \label{planch Z_I} L^2(Z_I)_{\rm td} \underset{G'\times G'}\simeq
\underset{\sigma\in \hat {M_I'}_{\rm disc}}{\hat\bigoplus} \int_{i(\af_I')^*}^\oplus \Ind_{\oline {P_I'}}^{G'}(\sigma\otimes\lambda)\hat \otimes \Ind_{P_I'}^{G'}(\oline \sigma\otimes(-\lambda))\ d_\sigma \lambda
\end{equation}
where
$$ d_\sigma \lambda = d(\sigma)\,d\lambda $$
with $d\lambda$ the Lebesgue-measure on the Euclidean space  $i(\af_I')^*$, suitably normalized.

\par For any $I\subset S$ we now denote by $\mu^{I,\rm td}$ the restriction of the Plancherel measure $\mu$  to the closed
subspace $\operatorname{im} B_I\subset L^2(Z)$.

From Theorem \ref{Plancherel induced} we obtain from the uniqueness
of the measure class of the Plancherel measure for $L^2(Z_I)$  that:
\begin{itemize}
\item $\supp \mu^{I,\rm td}=\{[\Pi_{\sigma,\lambda}]\mid [\pi_{\sigma,\lambda}]\in \hat G', \sigma\in \hat {M_I'}_{\rm disc}, \lambda\in
i(\af_I')^*\}$,
\item $\ind_{P_I'}^{G'}(\oline \sigma\otimes(-\lambda))$ is isomorphic to
$\pi_{\sigma,\lambda}^*=\ind_{\oline P_I'}^{G'}(\sigma\otimes \lambda)^*$
for $\mu^{I,\rm td}$-almost all parameters $(\sigma,\lambda)$.
\end{itemize}

Next we move to the subtle point on  how to identify  $\Ind_{P_I'}^{G'}(\oline \sigma\otimes(-\lambda))$
with the dual representation of $\Ind_{\oline {P_I'}}^{G'}(\sigma\otimes\lambda)$.
For that we first remark that the pairing

\begin{equation}\label{natural dual}
\Ind_{\oline {P_I'}}^{G'}(\sigma\otimes\lambda) \times \Ind_{\oline {P_I'}}^{G'}(\oline \sigma\otimes(-\lambda))
            \to \C, \ \ (f_1,f_2)\mapsto \int_{K'}  (f_1(k'), f_2(k'))_\sigma \ dk'\end{equation}
 is $G'$-equivariant.  Thus the dual representation of $\pi_{\sigma,\lambda}=\ind_{\oline {P_I'}}^{G'}(\sigma\otimes\lambda)$ is unitarily equivalent to
$\pi_{\oline \sigma, -\lambda}=\ind_{\oline {P_I'}}^{G'}(\oline \sigma\otimes(-\lambda))$.

\par Next, we consider the long intertwining operator
\begin{equation} \label{def A} \A_{\sigma,\lambda}: \Ind_{\oline {P_I'}}^{G'}(\oline \sigma\otimes(-\lambda))\to \Ind_{P_I'}^{G'}(\oline \sigma\otimes(-\lambda))\end{equation}
\begin{equation} \label{R abs} \A_{\sigma,\lambda}(f)(g') =  \int_{N_I'} f(g' n_I')  \ dn_I' \qquad (g'\in G')\, .\end{equation}
Clearly, $\A_{\sigma,\lambda}(f)$ is defined near $g'=\1$ for functions $f$ with compact support in the non-compact picture, i.e.
$\supp f \subset \Omega \oline{P_I'}$ for $\Omega\subset N_I'$ compact.  By standard techniques of
meromorphic continuation in the $\lambda$-variable,
summarized in the following remark, we obtain that $\A_{\sigma,\lambda}$ is defined for generic $\lambda\in i(\af_I')^*$.

 \begin{rmk} \label{analyt cont}Let us briefly recall  the basic constructions leading to the definition of $\A_{\sigma,\lambda}$
in terms of meromorphic continuation (originally obtained in \cite{K-SII}).  In the first step
one embeds the irreducible representation $\oline \sigma$ of $M_I'$ into a minimal principal series representation of $M_I'$ via the Casselman
subrepresentation theorem.  In formulae, we consider $\oline \sigma$ as a subrepresentation of
$\ind_{M_I'\cap \oline P'}^{M_I'} (\oline \sigma_{M'}\otimes \lambda_0)$, where
$\oline \sigma_{M'}\in \hat M'$ and $\lambda_0 \in (\af'\cap \mf_I')_\C^*$.
Via induction in stages we then obtain that $\pi_{\oline \sigma,-\lambda}$ is a
subrepresentation of the minimal principal series $\ind_{\oline P'}^{G'} (\oline \sigma_{M'} \otimes
\mu )$ where $\mu =\lambda_0 -\lambda$.  It is important to note that
$\mu|_{\af'_I}=-\lambda$ for this initial parameter $\mu$. In the sequel $\oline \sigma_{M'}\in \hat M'$ will
be fixed, but we will allow the parameter $\mu\in (\af')_\C^*$ to vary. For $\re \mu$ in a certain open cone
this then leads to an intertwining operator

$$ \A(\mu): \Ind_{\oline {P'}}^{G'}(\oline \sigma_{M'}\otimes \mu)\to
\Ind_{(\oline P'\cap M_I')A'_IN_I' }^{G'}(\oline \sigma_{M'}\otimes \mu)$$
given by absolutely convergent integrals as in \eqref{R abs}.

\par In the second step, via Gindikin-Karpelevic change of variable
(i.e. by using a minimal string of parabolics  in the terminology of  \cite[Sect. 4]{K-SII}),
one obtains that the intertwining operator is a product of rank one intertwiners $\A_\alpha(\mu)$
attached to indivisible roots $\alpha\in \Sigma(\af', \nf_I')$. For these rank one operators one has well known explicit formulae which show that they
admit a meromorphic continuation via Bernstein's  $p^\lambda$.
In this regard it is important to note that the $\mu$-dependence of
$\cR_\alpha(\mu)$ is in fact only a dependence on $\mu_\alpha=\mu(\alpha^\vee)\in \C$.
Moreover, regardless of $\sigma_{M'}$, the operator $\A_\alpha(\mu)$ is defined and  invertible
provided that $\mu_\alpha\not \in {1\over N} \Z$ for an $N\in \N$ only depending on
$G'$, see \cite[Prop. B.1]{KKOS} which was based on \cite[Th. 1.1]{SpVo}.

\par If we now use that the roots $\alpha$ do not vanish identically on $\af_I'$, we obtain
$\A_{\sigma,\lambda}$,  as in \eqref{def A},  is defined and invertible for generic $\lambda \in i (\af_I')^*$.  In more precision
we define $\A_{\sigma,\lambda}$ as the restriction of $\A(\mu)$ to
the subrepresentation $\Ind_{\oline {P_I'}}^{G'}(\oline \sigma\otimes(-\lambda))$.
\end{rmk}

 The operator $\A_{\sigma,\lambda}$ is
$G'$-equivariant and continuous, and hence we obtain from Schur's Lemma that
\begin{equation} \label{defi tau}
\A_{\sigma,\lambda}^*\circ \A_{\sigma,\lambda} = \tau(\sigma,\lambda) \id
\end{equation}
for a number $\tau(\sigma,\lambda)\in [0,\infty]$ which is positive for generic
$\lambda\in i(\af_I')^*$. Here $\A_{\sigma,\lambda}^*$ is the Hilbert adjoint to $\A_{\sigma,\lambda}$.
This implies in particular
for all  $f\in \Hc_{\oline \sigma, -\lambda}= \Ind_{\oline {P_I'}}^{G'}(\oline \sigma\otimes(-\lambda))$
the following norm identity

\begin{equation} \label{inter norm}  \|\A_{\sigma,\lambda} f\|^2 = \tau(\sigma,\lambda) \|f\|^2\, .\end{equation}

\begin{rmk} The numbers  $\tau(\sigma,\lambda)$ are computable  via  rank one reduction, see Remark \ref{analyt cont} above.
\end{rmk}

Recall that ${\mathcal B}_2(\Hc_{\sigma,\lambda})\simeq \Hc_{\sigma,\lambda}\otimes
\oline{\Hc_{\sigma,\lambda}}$ and
from \eqref{natural dual} that $\Hc_{\oline \sigma, -\lambda}=\oline{\Hc_{\sigma,\lambda}}$  is the natural  (isometric) dual of  $\Hc_{\sigma,\lambda}$.
By combining \eqref{planch Z_I} and \eqref{inter norm} we thus obtain that the operator
$$\sum_\sigma\int \Phi_{\sigma,\lambda}\circ(\id_{\Hc_{\sigma,\lambda}}\otimes
\A_{\sigma,\lambda}) \,\mu(\sigma,\lambda)\, d\lambda$$ provides a unitary $G$-equivalence

\begin{equation} \label{planch Z_I re} L^2(Z_I)_{\rm td} \underset{G'\times G'}\simeq
\underset{\sigma\in \hat {M_I'}_{\rm disc}}{\hat\bigoplus} \int_{i(\af_I')^*}^\oplus {\mathcal B}_2(\Hc_{\sigma,\lambda}) \ \mu(\sigma,\lambda) d\lambda\,
\end{equation}
where

\begin{equation} \label{def mu} \mu(\sigma,\lambda):= \frac{d(\sigma)}{\tau(\sigma,\lambda)}.\end{equation}

Next we want to keep track of the implied isomorphism in \eqref{planch Z_I re} with more suitable language.
For that we define a one-dimensional Hilbert space
$\C_{\sigma,\lambda} = \C  \xi_{\sigma,\lambda}\subset({\mathcal B}_2(\Hc_{\sigma,\lambda}) ^{-\infty})^{H_I}$ with $\|\xi_{\sigma,\lambda}\|=1$  and where $\xi_{\sigma,\lambda}$ is defined by

\begin{equation} \label{def xi sigma lambda}
\xi_{\sigma,\lambda}(f_1\otimes f_2)= \left( f_1(e) ,\A_{\sigma,\lambda}(f_2)(e)\right)_\sigma \qquad
(f_1 \in \Hc_{\sigma,\lambda}^ \infty, f_2 \in \Hc_{\oline \sigma, -\lambda}^\infty)\, .\end{equation}

In this regard we note for $g=(g_1', g_2')\in G$ that

\begin{equation} \label{definition xi sigma} \Phi_{\sigma,\lambda}(f_1\otimes \A_{\sigma,\lambda}(f_2))(g) =  \xi_{\sigma,\lambda}( \Pi_{\sigma, \lambda}(g^{-1})(f_1\otimes f_2))\end{equation}
so that with the extended notation
\begin{equation} \label{planch Z_I Re} L^2(Z_I)_{\rm td} \underset{G'\times G'}\simeq
\underset{\sigma\in \hat {M_I'}_{\rm disc}}{\hat\bigoplus} \int_{i(\af_I')^*}^\oplus {\mathcal B}_2(\Hc_{\sigma,\lambda})\otimes \C_{\sigma,\lambda} \ \mu(\sigma,\lambda) d \lambda\, .
\end{equation}
we keep track also of the isomorphism from right to left. In view of \eqref{Four1} and the orthogonality relations
for the discrete series, this isomorphism is the inverse of the Fourier transform

$$F_{\rm td}\mapsto \oline \Pi_{\sigma,\lambda}(F)\oline \xi_{\sigma,\lambda}\in {\mathcal B}_2(\Hc_{\sigma,\lambda})^\infty$$
for $F\in C_c^\infty(Z_I)$, see \eqref{Four1} and Remark \ref{F-inverse}.  Here $F_{\rm td}$ refers to the
orthogonal projection of $F\in C_c^\infty(Z_I)\subset L^2(Z_I)$ to $L^2(Z_I)_{\rm td}$.

\subsubsection{Grouping into irreducibles}

The $G=G'\times G'$-representation in \eqref{planch Z_I Re}
is not multiplicity free as different $\Pi_{\sigma,\lambda}$ can yield  equivalent representations.
These equivalences are induced  by Weyl group orbits.
In more precision let $\W'$ be the Weyl group of $\Sigma'$.  Then
$$ \W_I':=\{ w|_{\af_I}\mid  w \in \W',\ w(\af_I')=\af_I'\}$$
gives rise to a subquotient of $\W'$ and finite subgroup of the orthogonal group
of $\af_I'$.

\begin{rmk}\label{str WI}  (Structure of $\W_I'$) In general we are not aware of a criterion for subsets $I\subset \Phi'=S$  which characterizes those for which $\W_I'$ is a reflection group. Nevertheless we can describe a fundamental domain
for the action of $\W_I'$ as
a union of simplicial cones as follows.
\par For $\alpha\in \Phi'$ we denote by $s_\alpha\in \W'$ the corresponding simple reflection and recall that
$$\W'(I):=\la s_\alpha\mid \alpha\in I\ra$$
is naturally a reflection group on
$$\af'(I):=\Span_\R\{ \alpha^\vee \mid \alpha\in I\}$$
with simple roots given by $I$.
Note that $\af'=\af'(I) \oplus \af_I'$ is an orthogonal decomposition. Next we recall the set $D_I'\subset \W'$
of distinguished representatives for $\W'/ \W'(I)$, namely with
$$D'_I:=\{w \in \W'\mid  w(I)\subset (\Sigma')^+\}$$
we obtain a bijection
$$D_I' \to \W'/\W'(I),\ \ w\mapsto [w]=w\W'(I)$$
with $w$ the unique minimal length representative of $[w]$.

\par For $I, J\subset S$ set
$$\W'(I,J)=\{ w \in \W'\mid   w (J)=I\}\, .$$
We claim that the map
$$R: \W'(I,I)\to \W_I', \ \ w\mapsto w|_{\af_I'}$$
is an isomorphism of groups. Let us first show that $R$ is defined. In fact, if $w\in \W'(I,I)$, then $w(I)=I$ implies that $w$ preserves $\af'(I)$ and hence its
orthogonal complement $\af_I'$.  Hence $R$ is defined.  Let us show now that $R$ is injective and assume
$w|_{\af_I'}= \id$. In particular $w$ fixes the face
$$(\af_I')^-:=\{ X\in \af_I'\mid (\forall \alpha\in S\bs I) \ \alpha(X)\leq 0\}$$
of $(\af')^-$, the closure of the Weyl chamber $(\af')^{--}$. Hence Chevalley's Lemma implies
that $w\in \W'(I)$, thus $w=\1$ as $w(I)=I$.
Finally we show that $R$ is surjective.  Let $w\in \W'$ such that $w(\af_I')=\af_I'$. Hence $w(\af'(I))=\af'(I)$.
From the description of $\W'/\W'(I)\simeq D_I'$ we find $w_1 \in \W'(I)$ such that $ww_1(I)\subset (\Sigma')^+$.
Note that $ww_1(\af_I')=ww_1(\af_I')$ holds as well. Hence $ww_1(I)\subset (\Sigma')^+ \cap \Span_\R I
= \W'(I)\cdot I $.  In particular we find $w_2 \in \W'(I)$ such that $w_2 w w_1(I)=I$. Hence $w_2ww_1 \in
\W'(I,I)$. Since $w_1|_{\af_I'}= w_2|_{\af_I'}=\id$, the surjectivity of $R$ follows.

\par Recall that subsets $I, J$ are called associated provided that $\W'(I,J)\neq \emptyset$.
This defines  an equivalence relation $I\sim J$ among subsets of $S$. Next we record
the tiling

\begin{equation} \label{tiling} \af_I' = \bigcup_{J\sim I} \bigcup_{w \in \W'(I,J)} w (\af_J')^-\end{equation}
meaning that the union of interiors $  \coprod_{ J\sim I} \coprod_{w \in \W'(I,J)} w (\af_J')^{--}$ is disjoint.
Now note that $\W_I'\simeq \W'(I,I)$ acts on each $\W'(I,J)$ from the left. We pick for each
orbit $[w] = \W_I' w \subset \W'(I,J)$ a representative $w$ (of minimal length).  Then the cone
\begin{equation}\label{fund1} C_I':= \bigcup_{J\sim I} \bigcup_{[w]\in \W_I'\bs \W'(I,J)} w (\af_J')^-\end{equation}
is a fundamental domain for the action of $\W_I'$ on $\af_I'$.
\end{rmk}

Let $C_I^*$ be a fundamental domain
for the dual action  of  $\W_I'$ on $(\af_I')^*$, constructed as in  \eqref{fund1}.
Let $w\in \W_I'$. Then for $\lambda\in i (\af_I')^*$ we define
 $\lambda_w= w \cdot \lambda= \lambda(w^{-1}\cdot )$. Likewise one defines
 $\sigma_w$ by $\sigma_w(m_I') = \sigma(w^{-1} m_I' w)$ where we tacitly allowed ourselves to identify
 $w\in \W_I'\simeq  (N_K(\af_I') \cap N_K(\af'))/ M' $ with a lift to $K$ which normalizes $M_I'$.

\begin{lemma}\label{Lemma Ldt1} Let $w\in \W_I'$. Then for generic $\lambda\in i (\af_I')^*$, the representation
$\pi_{\sigma,\lambda}$ is equivalent to $\pi_{\sigma_w, \lambda_w}$.
\end{lemma}

 \begin{proof} We use the intertwining operator
 \begin{equation} \label{def A2} \A_w: \Ind_{\oline{P_I}'}^{G'}(\sigma\otimes \lambda)\to \Ind_{w^{-1}\oline{P_I}'w}^{G'}(\sigma\otimes \lambda)\end{equation}

\begin{equation} \label{R abs2} \A_w(f)(g') =  \int_{w^{-1} \oline{N_I}'w/ w^{-1}\oline{N_I}'w \cap \oline{N_I}'} f(g' x)  \ dx \qquad (g'\in G')\end{equation}
which is, as a product of rank one intertwiners, generically defined by Remark  \ref{analyt cont}. The desired equivalence
of $\pi_{\sigma,\lambda}$ and $\pi_{\sigma_w, \lambda_w}$ is then obtained by composing
$\A_w$ with the right shift by $w$, i.e.
\begin{equation}\label{R inter} \Rc_w(f)(g'):= (\A_w f)(g'w)\end{equation}
yields the desired equivalence between
$\pi_{\sigma,\lambda}$ and $\pi_{\sigma_w, \lambda_w}$.
\end{proof}

The next lemma is a  generic form of the Langlands Disjointness Theorem (see \cite[Th. 14.90]{Knapp}) for which we
provide an elementary proof.

\begin{lemma} \label{Lemma Ldt2} Let $\sigma, \sigma'\in  \hat {M_I'}_{\rm disc}$ and $\lambda, \lambda' \in iC_I^*$ be such that the unitary representations $\pi_{\sigma,\lambda}$ and $\pi_{\sigma',\lambda'}$ are equivalent. Then for generic $\lambda$ one has
$\lambda=\lambda'$ and $\sigma=\sigma'$.\end{lemma}

\begin{proof} We first show that $\lambda=\lambda'$ for generic $\lambda$.
For that we consider the infinitesimal characters $\pi_{\sigma, \lambda}$.
For the  discrete series $\sigma$ a fairly elementary and short proof that
their infinitesimal characters are real is given in \cite{KKOS}. Now, from the
standard formulae for infinitesimal characters of induced representations, see \cite[Prop. 8.22]{Knapp},  we deduce from
$\pi_{\sigma, \lambda}\simeq \pi_{\sigma', \lambda'}$
that
\begin{equation} \label{orbit Wj}  \W'_{\jf'}\cdot (\mu_\sigma  + \lambda) = \W'_{\jf'}\cdot (\mu_{\sigma'} +\lambda')\, . \end{equation}
Here $\jf' = \af'+\tf'$ is a Cartan subalgebra of $\gf'$ which inflates the maximal split torus $\af'\subset \gf'$
by a maximal torus $\tf'\subset \mf'$. Further,
$\mu_{\sigma}, \mu_{\sigma'}\in (i \tf' + \af'\cap \mf_I')^*$ are representatives of the infinitesimal character
for $\sigma$, resp. $\sigma'$.  Note that $\W'_{\jf'}$ leaves the real form $\jf_\R':= \af'+i\tf'$ of $\jf_\C'$ invariant.
Hence comparing the imaginary parts (with respect to $\jf_\R'$) in \eqref{orbit Wj} yields for generic
$\lambda, \lambda'\in iC_I^*$ that $\lambda=\lambda'$.

\par  Finally we show that $\sigma$ is equivalent to $\sigma'$.
Let $F$ be a finite dimensional representation  of $G'$ with strictly dominant highest weight $\Lambda$   and highest weight vector
fixed by $M_I'$.  Hence $\Lambda \in (\af_I')^*$.
The translation functor moves for $\lambda$ generic the representations
$\pi_{\sigma, \lambda} $ to $\pi_{\sigma, \lambda +\Lambda} $ and $\pi_{\sigma', \lambda} $ to  $\pi_{\sigma', \lambda +\Lambda}$, see \cite[proof of Lemma 10.2.7] {Wal2}.

\par  We conclude that
$\pi_{\sigma, \lambda }$ is equivalent to $\pi_{\sigma', \lambda }$
also for a parameter $\lambda$ with $ \re \lambda$ sufficiently dominant.
This allows us to apply Langlands' Lemma \cite[Lemma 3.12]{L} for the asymptotics of $\la \pi_{\sigma,\lambda}(m_I' a'_t)f_1, f_2\ra$ for $f_1, f_2
\in \Hc_{\sigma,\lambda}^\infty$, $m_I'\in M_I'$ and $a_t'=\exp(tX')$ for $X'\in (a_I')^{--}$:

$$\lim_{t\to \infty}  (a_t')^{\lambda -\rho'}\la \pi_{\sigma, \lambda} (a_t' m_I') f_1, f_2\ra =
\la \sigma(m_I')[f_1(\1)], \A_{\sigma,\lambda}(f_2)(\1)\ra_\sigma,$$
see also \eqref{Lang Lemma} below.   Notice that with $f_1, f_2\in \Hc_{\sigma,\lambda}^\infty$ the vectors
$f_1(\1), \A(f_2)(\1)$ run over all pairs of smooth vectors in $V_\sigma^\infty$.
Likewise holds for $\sigma'$ and we obtain that the unitary representations
$\sigma$ and $\sigma'$ feature the same (smooth) matrix coefficients.
\par Now we recall the Gelfand-Naimark-Segal construction which  asserts for an irreducible unitary representation $\pi$
of a locally compact group $G$ on a Hilbert space $\Hc$ that one can recover $\pi$ by one matrix
coefficient $g \mapsto \la \pi(g)v, v\ra$ for $v\in \Hc$, $v\neq 0$. Consequently $\sigma$ and $\sigma'$ are equivalent,
concluding the proof of the lemma.
\end{proof}

\begin{rmk} \label{Remark only real} Let us stress that the only property of discrete series used in the
preceding proof of Lemma \ref{Lemma Ldt2} was that infinitesimal characters are real.\end{rmk}

\par By applying Lemma \ref{Lemma Ldt1} and Lemma \ref{Lemma Ldt2} to the disintegration formula \eqref{planch Z_I Re}
we obtain the grouping in inequivalent irreducibles, i.e. the Plancherel formula for $L^2(Z_I)_{\rm td}$:

\begin{equation} \label{planch Z_I irred}
L^2(Z_I)_{\rm td} \underset{G'\times G'}\simeq
\sum_{\sigma\in \hat {M_I'}_{\rm disc}}\int_{iC_I^*}  {\mathcal B}_2(\Hc_{\sigma,\lambda})   \otimes
\M^I_{\sigma,\lambda} \,\, \mu(\sigma,\lambda) d\lambda\,,\end{equation}
where
$$\M^I_{\sigma,\lambda}:=\M_{\Pi_{\sigma,\lambda}}^I =  ({\mathcal B}_2(\Hc_{\sigma,\lambda})^{-\infty})_{\rm temp}^{H_I}$$
is the multiplicity space.   Moreover, for generic $\lambda$ we have also seen that
\begin{equation}  \label{m-sigma}\M_{\sigma,\lambda}^I \simeq \bigoplus_{w\in \W_I'}  \C_{\sigma_w, \lambda_w} \end{equation}
as $\af_I$-module. In particular, we obtain that
\begin{equation} \label{Spec M-sigma} \Spec_{\af_I'}  \M_{\sigma,\lambda}^I= \rho'|_{\af_I'}  - \W_I'\cdot \lambda\, . \end{equation}

\subsection{The Maass-Selberg relations} From Theorem \ref{Plancherel induced} we obtain
that the multiplicity space $\M^I_{\sigma,\lambda}$ is endowed with the Hilbert space structure induced from
the one dimensional space
$\M_{\sigma,\lambda}:=\M_{\Pi_{\sigma,\lambda}}= \C\tr_{\pi_{\sigma,\lambda}}$.

Set $\eta_{\sigma,\lambda}:= \tr_{\pi_{\sigma,\lambda}} \in \M_{\sigma,\lambda}$
and recall from \eqref{decomp etaI} the orthogonal decomposition

$$\eta_{\sigma,\lambda}^I= \sum_{\xi \in (\rho-\\W_\jf \chi)|_{\af_I}}  \eta_{\sigma,\lambda}^{I, \xi} $$
with $\chi$ the infinitesimal character of $\Pi_{\sigma,\lambda}$.

Upon our identification of $\af_I$ with $\af_I'$ we obtain for $\lambda$ generic from \eqref{Spec M-sigma}
that $\eta_{\sigma,\lambda}^{I,\xi}\neq 0$ if and only if
$\xi \in \rho' + \W_I'\cdot \lambda$ and accordingly
$$\eta_{\sigma,\lambda}^I= \sum_{w\in \W_I'} \eta_{\sigma,\lambda}^{I, \rho' - w\lambda}  \, .$$

Further, our  Maass-Selberg relations in Theorem \ref{eta-I continuous} give
\begin{equation} \label{MS group} 1= \|\eta_{\sigma,\lambda}\|= \|\eta_{\sigma,\lambda}^{I,\xi}\|_{\M_{\sigma,\lambda}^I}
\end{equation}
for any $\xi$ with $\eta_{\sigma,\lambda}^{I,\xi}\neq 0$.

In order to proceed we need an elementary result on the asymptotics of the matrix coefficient
$$ \eta( \Pi_{\sigma,\lambda}(g) (f_1 \otimes \la \cdot, f_2\ra)) = \la \pi_{\sigma,\lambda}(g_1^{-1})f_1,
\pi_{\sigma,\lambda}(g_2^{-1})f_2\ra  \qquad ( f_1, f_2 \in \Hc_{\sigma,\lambda})$$
for $g = a=( \sqrt{a'}, \sqrt{a'}^{-1})\in A_I^{--}$ with $a'\in (A_I')^{--}$. In other words
we are interested in the asymptotics of

$$ a' \mapsto   \la \pi_{\sigma, \lambda} ((a')^{-1})f_1, f_2\ra $$
for $a'=a'_t=\exp(tX')$ with $X'\in (\af_I')^{--}$ and $t\to\infty$.
Then we have the following variant, observed
in \cite{KKOS}, of \cite[Lemma 3.12]{L}.

\begin{lemma}\label{Lemma KKOS}  Let $\lambda \in i\af_I^*$ and suppose that $f_1, f_2 \in \Hc_{\sigma,\lambda}^\infty$ are such that
$\supp f_i \subset \Omega \oline{P_I'}$ for some $\Omega\subset N_I'$ compact.
Then
\begin{equation}\label{Lang Lemma} \lim_{t\to \infty}a_t^{\lambda -\rho}  \la \pi_{\sigma, \lambda} ((a_t')^{-1})f_1, f_2\ra
= \la f_1(\1), \A_{\sigma,\lambda} (f_2)(\1)\ra_\sigma\end{equation}
 \end{lemma}

 \begin{proof} We use the non-compact model for $\pi_{\sigma,\lambda}$ and realize $f_1, f_2$ as $\sigma$-valued functions
on  $N_I$:

$$ \la \pi_{\sigma, \lambda} ((a'_t)^{-1})f_1, f_2\ra =(a_t')^{-\lambda+\rho'} \int_{N_I'}  \la f_1( a_t'n_I' (a_t')^{-1}), f_2(n_I')\ra_\sigma \ dn_I'\, .$$
Observe that
$$a'_t \Omega (a'_t)^{-1} \underset{t\to \infty}{\to} \{\1\}$$
 for all $\Omega\subset N_I'$ compact.  By the compactness of supports we are allowed to interchange limit and integral and the asserted formula follows. \
 \end{proof}

The Maass-Selberg relations \eqref{MS group} then yield the following key-identity:

\begin{lemma}\label{functionals identical}  For generic $\lambda \in i C_I^*$ we have
$\xi_{\sigma,\lambda}= \eta_{\sigma,\lambda}^{I,\rho'-\lambda}$ together
with $\C_{\sigma,\lambda} \subset \M_{\sigma,\lambda}^I$ as Hilbert spaces.
\end{lemma}

\begin{proof} First note that
$\xi_{\sigma,\lambda}$ and $\eta_{\sigma,\lambda}^{I,\rho'-\lambda}$ have to be
multiples of each other as they have the same $\af_I$-weight. Let us show
that this multiple is indeed $1$ by computing the asymptotics of the matrix coefficient:
Recall that for $a=( \sqrt{a'}, \sqrt{a'}^{-1})\in A_I^{--}$ with $a'\in (A_I')^{--}$ we have

$$ \eta( \Pi_{\sigma,\lambda}(a) (f_1 \otimes \la \cdot, f_2\ra)) = \la \pi_{\sigma, \lambda} ((a')^{-1})f_1, f_2\ra\, . $$
Now for $f_1, f_2$ as  in Lemma \ref{Lemma KKOS} we obtained in \eqref{Lang Lemma}
$$\la \pi_{\sigma, \lambda} ((a')^{-1})f_1, f_2\ra\sim(a')^{\rho'-\lambda}  \la f_1(e), \A(f_2)(e)\ra_\sigma\, .$$
Comparing with \eqref{def xi sigma lambda} we then get
indeed that
$\xi_{\sigma,\lambda}= \eta_{\sigma,\lambda}^{I,\rho'-\lambda}$.

\par Finally, as $\|\xi_{\sigma,\lambda}\|=1$
we obtain from the Maass-Selberg relations
\eqref{MS group} that $\C_{\sigma,\lambda} \subset \M_{\sigma.\lambda}^I$ as Hilbert spaces.
This completes the proof of the lemma. \end{proof}

\subsection{The Plancherel Theorem for $L^2(Z)$} From the fact that source and target of the Bernstein morphism have equivalent Plancherel measures
we obtain
\begin{equation} \label{supp union} \supp \mu =\bigcup_{I\subset S} \supp \mu^{I,\rm td}\end{equation}
with
$$\supp \mu^{I, \rm td}= \{ [\Pi_{\sigma, \lambda}]\in \hat G\mid \lambda\in iC_I^*, \sigma \in \hat {M_I'}_{\rm disc}\}$$

In the union \eqref{supp union} a certain overcounting takes place, which will be taken care of
in the next lemma:

\begin{lemma}\label{lemma disjoint supports}  Let $I, J\subset S$. Then the following assertions hold:
\begin{enumerate}
\item\label{111} If $I$ and $J$ are associated, i.e.~there exists a $w\in \W'$ such that $w(I)=J$,
then $\supp \mu^{I,\rm td}=\supp \mu^{J,\rm td}$.
\item\label{222} Otherwise $\supp \mu^{I,\rm td}\cap \supp \mu^{J,\rm td}$
has $\mu$-measure zero.
\end{enumerate}
\end{lemma}
\begin{proof} \eqref{111} Basic intertwining theory (assuming no particular knowledge on the discrete spectrum) as used above
implies that
$$\Spec L^2(Z_I)_{\rm td}=\Spec L^2(Z_J)_{\rm td}\subset \hat G$$ if $I$ and $J$ are associated.

\eqref{222} As the infinitesimal
characters for the discrete series of $M_I'$ and $M_J'$  are real (see \cite{KKOS}), we obtain
that the infinitesimal characters of the induced representations in $L^2(\hat Z_I, \lambda_I)_{\rm d}$ and
$L^2(\hat Z_J, \lambda_J)_{\rm d}$ for generic $\lambda_I ,\lambda_J \in i\af_J^*$
are different if $I$ and $J$ are not associated, see \eqref{orbit Wj} and the text following it.
\end{proof}

\bigskip
We are now ready to phrase the Plancherel theorem of Harish-Chandra in terms of the
Bernstein morphism. For this let
$$\Hc_I: = \sum_{\sigma\in \hat {M_I'}_{\rm disc}}\int_{iC_I^*}  {\mathcal B}_2(\Hc_{\sigma,\lambda})   \otimes \C{\xi_{\sigma,\lambda}} \ \mu(\sigma,\lambda) d\lambda, $$
viewed as a subspace  of  $L^2(Z_I)_{\rm td}$ as in \eqref{planch Z_I irred}.

Let $B'_I$ be the restriction of $B_I$ to $\Hc_I$. Select   a family $\mathcal{I}$ of  representatives of subsets of $S$ modulo association and set
$$B':=\bigoplus _{I \in \mathcal{I}}B'_I\, .$$

\begin{theorem} \label{planch HC} The map
$$B': \bigoplus_{I\in \Ic} \Hc_I \to L^2( Z)$$ is a bijective isometry, hence  the inverse of a Plancherel isomorphism.
In particular  we obtain the explicit Parseval-formula:
\begin{equation} \label{Parseval} \Vert f\Vert_{L^2(Z)}^2 =
\sum_{I\in \Ic} \sum_{\sigma\in \hat {M_I'}_{\rm disc}} \int_{iC_I^*} \Vert \pi_{\sigma, \lambda}(f)\Vert_{\rm HS}^2\
\mu(\sigma,\lambda) d\lambda\end{equation}
 for all $f\in C_c^\infty(Z)$.
\end{theorem}
\begin{proof}  By Lemma \ref{lemma disjoint supports} both sides have the same support in $\hat G$ and moreover have
multiplicity one. Next $B_I'$ is isometric by Lemma \ref{functionals identical}  and the spectral definition of the
Bernstein morphism (compare also to Remark \ref{Remark isometric}). Since for different $I\neq J\in \Ic$ the spectral
supports are disjoint by Lemma \ref{lemma disjoint supports}, the images of the various $B_I'$ are orthogonal.
The theorem follows.
\end{proof}

To obtain the original Parseval formula of Harish-Chandra in its standard form
we unwind  \eqref{Parseval} via $i\af_I'^*= \W_I' \cdot iC_I^*$ and average over association classes

\begin{equation} \label{HC Parseval} \Vert f\Vert_{L^2(Z)}^2 = \sum_{I\subset S} { 1 \over | [I]| \cdot | \W'_I|}
\sum_{\sigma\in \hat {M_I'}_{\rm disc}}\int_{i(\af_I')^*} \Vert \pi_{\sigma, \lambda}(f)\Vert_{\rm HS}^2\ \mu(\sigma,\lambda) d\lambda \end{equation}
where
$[I]$ is the equivalence class of $I\subset S$ under association.

\begin{rmk} Regarding the knowledge about representations of the discrete series,  let us stress that in the above derivation of the Plancherel formula for  a real reductive group we only used
the results of \cite{KKOS} on the infinitesimal characters of discrete series. These are valid for  general
real spherical spaces and when specialized to the group case comparably
soft and elementary opposed to the usage of the difficult classification of the discrete series by Harish-Chandra.
\par As byproduct of his classification of the discrete series Harish-Chandra obtained the following beautiful
geometric characterization of the discrete spectrum
\begin{equation} \label{HC disc}L^2(G)_{\rm d}\neq \emptyset \iff \hbox{$\gf$ contains a compact Cartan subalgebra}\,.\end{equation}
Let us emphasize once more that we obtained  "$\Leftarrow$" in this paper in the full generality
of real spherical spaces, see Theorem \ref{thm discrete}.
\par For  a general real spherical space a description of the (twisted) discrete spectrum in terms of parameters
is currently out of reach. Therefore, regarding the discrete spectrum of a real spherical space,
the  emphasis is to obtain "$\Rightarrow$" of \eqref{HC disc}  in general. Now for the group case, there is an economic way to obtain that: one first characterizes the discrete spectrum as cusp forms and then relates cusp forms to
to orbital integrals, see the account of Wallach \cite[Ch. 7]{Wal1}.  This idea, as well as all other known methods for the group,
fails to generalize to a real spherical space.

\par Finally, Harish-Chandra determined with the parameters of the discrete series also their formal degrees. In the group
case we saw that there is a canonical normalization of the one dimensional space of
$H$-invariant functionals $\M_\pi=\C\tr_\pi$, namely
by the trace.  Now for a general real spherical space the space of $H$-invariant functionals $\M_{\pi, {\rm td}}$ for a (twisted) discrete
series is no longer one-dimensional nor is it clear whether there is a canonical normalization of the inner product
on $\M_{\pi, {\rm td}}$. The only known general result beyond the group case is the case of holomorphic
discrete series on a symmetric space \cite{Kr}.
\end{rmk}

\section{The Plancherel formula for symmetric spaces} \label{section DBS}

In this section we apply the Bernstein decomposition to symmetric spaces and
derive the Plancherel formula of Delorme \cite{Delorme}
and  van den Ban-Schlichtkrull \cite{vdBS}. The account is rather
parallel to the group case. The only needed extra tool is  the
description of a generic basis of $H$-invariant distribution vectors for  induced representations in terms
of open $H$-orbits on real flag varieties $G/P$,  see \cite{vdB}, \cite{CD}.

For this section $Z=G/H$ is symmetric and we use the notation and results from Subsection \ref{Subsection m symmetric}.

\subsection{Normalization of discrete series} \label{normal disc}

This small paragraph is valid for a general unimodular spherical space $Z=G/H$. Let $[\pi]\in \hat G$
and $(\pi, \Hc)$ be a unitary model of $[\pi]$. We write $\M_{\pi, {\rm d}}\subset (\Hc^{-\infty})^H$ for the subspace
of those $\eta$ for which $m_{v,\eta}\in L^2(Z)$ for all $v\in \Hc^\infty$.  We define an inner product on $\M_{\pi,{\rm d}}$ by the request
that the Schur-Weyl orthogonality relations hold true:

\begin{equation} \label{SW-ortho}\int_Z  m_{v,\eta} (z) \oline {m_{v', \eta'}(z)} \ dz =  \la v, v'\ra_{\Hc}  \la \eta, \eta'\ra_{\M_{\pi, {\rm d}}}\end{equation}
Notice that the norm on $\M_{\pi, {\rm d}}$ depends on the unitary norm of $\Hc$ which is only unique up to positive scalar.

\begin{rmk} Given a pair of normalizations of $\la\cdot, \cdot\ra _{\Hc}$
and $\la\cdot, \cdot\ra_{\M_{\pi, {\rm d}}}$ one obtains a notion of formal degree
$d(\pi)$ analogous to \eqref{formal deg} by requiring
$$
d(\pi)\int_Z  m_{v,\eta} (z) \oline {m_{v', \eta'}(z)} \ dz =
 \la v, v'\ra_{\Hc}  \la \eta, \eta'\ra_{\M_{\pi, {\rm d}}}\, .
$$
The normalization of $\la \cdot, \cdot\ra_{\M_{\pi, {\rm d}}}$
by \eqref{SW-ortho} therefore amounts to setting $d(\pi)=1$.
Without a canonical normalization of
$\la \cdot, \cdot\ra_{\M_{\pi, {\rm d}}}$
this is the best we can offer. \end{rmk}

\subsection{The Plancherel formula for $L^2(Z_I)_{\rm td}$} \label{ZI planch 15}
Recall that $H_I= (M_I \cap H)\oline {U_I}$ is contained in $\oline {P_I}=M_I \oline {U_I}$ with
$\oline {P_I}/ H_I\simeq M_I/M_I\cap H$.  Hence  $L^2(Z_I)$ is parabolically
induced from $L^2(M_I/ M_I\cap H)$, and hence we obtain
\begin{equation}  \label{Z_I symm} L^2(Z_I)_{\rm td} \simeq \sum_{\sigma \in \hat M_I}
\int_{i \af_I^*}^\oplus  \Hc_{\sigma,\lambda}\otimes \M_{\sigma, {\rm d}} \, d\lambda\,.\end{equation}
Here $\Hc_{\sigma, \lambda}= \Ind_{\oline P_I}^G (\sigma\otimes \lambda)$
and $\M_{\sigma, {\rm d}}$ is the space of  $M_I\cap H$-invariant
functionals on $\Hc_\sigma^\infty$, which are square integrable
for the symmetric space $M_I/ M_I\cap H$, as defined in Subsection \ref{normal disc}
for $G/H$ and $\pi=\sigma$.
\par The space
$\Hc_{\sigma,\lambda}\otimes \M_{\sigma, {\rm d}}$ embeds into
$L^2(\hat Z_I,{ - \lambda})_{\rm d}$ isometrically (by our normalization of the discrete series) by
$$ \Phi_{\sigma, \lambda}:
\Hc_{\sigma,\lambda}\otimes \M_{\sigma, {\rm d}} \to L^2(\hat Z_I,{ -\lambda})_{\rm d}$$
defined by linear extension and completion of
$$\Phi_{\sigma,\lambda}(f \otimes \zeta)(gH_I) := \zeta (f(g))\qquad (f\in \Hc_{\sigma,\lambda}^\infty, \zeta\in\M_{\sigma,\rm d}, g\in G)\, .$$
For $\zeta \in \M_{\sigma, {\rm d}}$ let us also define an $H_I$-invariant functional
$\xi_{\sigma,\lambda, \zeta}$ on $\Hc_{\sigma, \lambda}^\infty$ via
\begin{equation} \label{def xi sigma eta} \xi_{\sigma,\lambda, \zeta}(f):= \zeta(f(e))\qquad (f\in \Hc_{\sigma,\lambda}^\infty)\end{equation}
and record
$$ \Phi_{\sigma, \lambda}(f\otimes\zeta) = \xi_{\sigma,\lambda, \zeta}(\pi_{\sigma, \lambda}(g^{-1}) f) \, .$$

\par Recall the little Weyl group $\W=\W_Z$ of the restricted roots system $\Sigma(\gf, \af_Z)$
from Subsection \ref{Subsubsection adapted}.

The decomposition \eqref{Z_I symm} is not yet the Plancherel formula
for $L^2(Z_I)_{\rm td}$, since it is not a grouping into irreducibles as
different  $\pi_{\sigma,\lambda}$ may yield
equivalent representations.  Similar to the group case
this possibility is governed by the subquotient
$$\W_I=\{ w|_{\af_I}\mid  w\in \W, \ w(\af_I)=\af_I\}$$
of $\W=\W_Z$ and a cone $C_I^*\subset \af_I^*$ as  fundamental domain for the dual action of $\W_I$
(see Remark \ref{str WI}.)  As in the group case we identify elements $w\in \W_I$ with lifts
to $K$ which normalize $M_I$.

\par As in Lemma \ref{Lemma Ldt1} and \eqref{R inter}
we obtain for every $w\in \W_I$, $\sigma\in\hat M_I$ and generic $\lambda\in i\af_I^*$  a $G$-intertwiner
$$ \Rc_w: \Hc_{\sigma, \lambda} \to \Hc_{\sigma_w, \lambda_w}\, $$
with $\sigma_w$ and $\lambda_w$ defined as before.
Next in full analogy
to Lemma \ref{Lemma Ldt2} we obtain:

\begin{lemma} For $\lambda, \lambda'\in iC_I^*$ generic and $\sigma, \sigma'$ in the discrete series of $L^2(M_I/M_I\cap H)$ (i.e.
both $\M_{\sigma, \rm d}$ and $\M_{\sigma', \rm d}$ are non-zero) one has
$$ \pi_{\sigma, \lambda}\simeq \pi_{\sigma', \lambda'}\quad \iff \quad \lambda = \lambda' \ \text{and} \
\sigma\simeq \sigma'\, .$$
\end{lemma}
\begin{proof} We recall that the proof of Lemma \ref{Lemma Ldt2} only requires
that the infinitesimal character of the inducing data $\sigma$ and $\sigma'$ are real,
see Remark \ref{Remark only real}. By \cite{KKOS} this is the case in the current situation as well.
\par Next we recall  that the root system $\Sigma_Z= \Sigma(\gf,\af_Z)\subset \af_Z^*$ is
obtained from the root system $\Sigma(\gf_\C,\jf_\C)\subset \jf_\R^*$ as the non-vanishing
restrictions. In particular, the faces $\af_I^-$ of $\af_Z^-$ are contained in the faces
$\jf_\R^-$ with respect to our lined up positive systems. This allows us now
to argue as in \eqref{orbit Wj} and conclude that $\lambda=\lambda'$.
\par The rest of the argument is then fully analogous.
\end{proof}

By grouping equivalent representations  in \eqref{Z_I symm} we then obtain the Plancherel
formula
\begin{equation}  \label{Z_I symm Re} L^2(Z_I)_{\rm td} \simeq \sum_{\sigma \in \hat M_I}
\int_{i C_I^*}^\oplus  \Hc_{\sigma,\lambda}\otimes \M_{\sigma, \lambda}^I \ d\lambda\end{equation}
with generic multiplicity space $\M_{\sigma,\lambda}^I=\M_{\sigma,\lambda, {\rm td}}^I$ of dimension
\begin{equation} \label{dimc1}\dim \M_{\sigma,\lambda}^I= |\W_I| \cdot \dim \M_{\sigma, \rm d}\,.\end{equation}
For $w\in \W_I$ let us denote by $\M_{\sigma_w,\rm d}$ the space $\M_{\sigma, \rm d}$ with
$M_I \cap H$ replaced by $w(M_I\cap H)w^{-1}= M_I \cap H_w$ and $\sigma$ replaced by $\sigma_w$.
 Since $L^2(M_I/ M_I\cap H)_{\rm d}\simeq
L^2(M_I/M_I \cap H_w)_{\rm d}$ we infer that $\M_{\sigma, \rm d}$ and $\M_{\sigma_w, \rm d}$ are
canonically isomorphic.
Now for each $w\in \W_I$ and $\zeta \in\M_{\sigma_w,\rm d}$ we can define
an $H_I$-invariant  functional of $\af_I$-weight $\rho-\lambda_w$ via

$$ \xi_{\sigma_w,\lambda_w, \zeta}: \Hc_{\sigma,\lambda}^\infty \to \C, \ \ f\mapsto
\zeta( (\Rc_w f)(\1))\, .$$
This functional yields an  embedding  of $\Hc_{\sigma,\lambda}^\infty$
into $L^2(Z_I)_{\rm td}$, i.e.  $ \xi_{\sigma_w,\lambda_w, \zeta}\in
\M_{\sigma,\lambda, \rm td}^{I,\rho-\lambda_w}$.
Moreover, by varying $\zeta$ we obtain for each $w\in \W_I$
a linear injection

\begin{equation} \label{inclusion 15} \M_{\sigma_w,\rm d}\to \M_{\sigma,\lambda, \rm td}^{I,\rho-\lambda_w}, \ \ \zeta\mapsto
\xi_{\sigma_w,\lambda_w, \zeta}\, .\end{equation}
We now count dimensions. With $\M_{\sigma,\lambda}^I=
\bigoplus_{\mu \in \rho + i\af_I^*}\M_{\sigma,\lambda, \rm td}^{I, \mu}$ and  \eqref{dimc1}
we obtain for generic $\lambda$
that the inclusion \eqref{inclusion 15} is an isomorphism, that is

\begin{equation} \label{formula 1}\M_{\sigma,\lambda,{\rm td}}^{I,\rho-\lambda_w}=
\{  \xi_{\sigma_w,\lambda_w, \zeta} \mid \zeta\in \M_{\sigma_w, \rm d}\}\,, \quad (w\in \W_I).
\end{equation}

\subsection{Support of the Plancherel measure}\label{support 15}
Previously we defined for $\sigma\in \hat M_I$ and $w\in \W_I$ the multiplicity space
$\M_{\sigma_w, \rm d}$.  We now also need a notion for every $w\in \Wc$.  For $w\in \Wc$
we  write $\M_{\sigma, w,{\rm d}}$
for the space of $M_I\cap H_w$-invariant functionals on $\Hc_\sigma^\infty$,
which are  square integrable for the symmetric space
$M_I/ M_I \cap H_w \simeq w^{-1} M_I w/ w^{-1}M_I w \cap H$.

Then, by the isospectrality of the Bernstein morphism we obtain that

\begin{equation} \label{support symmetric} \supp \mu =\left\{[\pi_{\sigma, \lambda}]\in \hat G\, \Bigg| \begin{aligned}
& I\subset S, \  \lambda\in iC_I^*, \\
& \sigma\in \hat M_I \ \text{s.t.} \ \exists \ w\in \Wc:\  \M_{\sigma, w, {\rm d}}\neq \{0\}\end{aligned}\right\}\, .\end{equation}
 Let us introduce the notion that $\sigma\in \hat M_I$ is {\it cuspidal} provided
$\M_{\sigma,w, {\rm d}}\neq \{0\}$ for some $w\in \Wc$.

\subsection{Generic dimension of multiplicity spaces}
To abbreviate matters let us set $\M_{\sigma, \lambda}= \M_{\pi_{\sigma, \lambda}}$ for
$[\pi_{\sigma, \lambda}]\in \supp \mu$.  The next goal is to obtain a precise description of
$\M_{\sigma, \lambda}$ for generic $\lambda$.
This is related to the geometry of open $H\times \oline{P_I}$-double cosets in $G$ which we
we recall from Section \ref{open double}. From Lemma \ref{Lemma Mats1}
there is an action of $\W(I)$ on $\Wc$ with identifications
$(P_I\bs Z)_{\rm open}\simeq\W(I)\bs \Wc$ and
$(P\bs Z_I)_{\rm open}\simeq \W(I)/\W(I)\cap \W_H$.

For what to come we need to interpret the quotient $\W(I)\bs \Wc$ in terms of the geometric decomposition $\Wc=\coprod_{\sc\in \sC_I}  \coprod_{\st \in \sF_{I,\sc}}  {\bf m}_{\sc,\st}(\Wc_{I,\sc})$ from \eqref{full deco W}.

\begin{lemma} \label{lemma fine W} With regard to $\Wc=\coprod_{\sc\in \sC_I}  \coprod_{\st \in \sF_{I,\sc}}  {\bf m}_{\sc,\st} (\Wc_{I,\sc})$
the action of $\W(I)$ on $\Wc$ acts on each subset ${\bf m}_{\sc,\st} (\Wc_{I,\sc})\subset \Wc$ transitively and
induces a natural bijection
\begin{equation} \W(I)\bs \Wc  \simeq \coprod_{\sc\in\sC_I} \sF_{I,\sc}\end{equation}

\end{lemma}

\begin{proof} Let us fix $\sc,\st$, and to save notation,  assume first $\sc=\st=\1$. Then $Z_{I,\sc,\st}=Z_I$ and
$\Wc_I=\Wc_{I,\1}$. Lemma \ref{Lemma Mats1}\eqref{twomats} implies that
$\W(I)$ acts transitively on $\Wc_I \simeq W_I =(P\bs Z_I)_{\rm open}$.  We claim that this holds for every
$Z_{I,\sc}\simeq Z_{I,\sc,\st}$, i.e. $\sW(I)$ acts transitively on  $(P\bs Z_{I,\sc,\st})_{\rm open }$.
To see that we recall the identifications $\W \simeq W \simeq F_M\bs F_\R$ with $F_\R$ the $2$-torsion subgroup
of $\uA_Z(\R)$.  Further we need the splitting $F_\R = F_{I,\R}\times F_{I,\R}^\perp$ derived from \eqref{AI torus deco}.
Now the $\W(I)$-orbits
on $\Wc\simeq F_M \bs F_\R$ correspond exactly to the $F_{I,\R}^\perp$-orbits on $F_M\bs F_\R$. Now the claim follows from the definition of $Z_{I,\sc,\st}$ and the fact
that  $(P\bs Z_{I,\sc, \st})_{\rm open}\simeq \Wc_{I,\sc}$ is mapped under ${\bf m}_{\sc,\st}$ in a $\W(I)$-equivariant way into $\Wc$, see  Lemma \ref{lemma 58} applied to $Z_{\sc,\st}= G/H_{w(\sc,\st)}\simeq Z=G/H$.

 \par The reasoning above implies further that the $\W(I)$-action on $\Wc$ respects with regard to
the decomposition  $\Wc= \coprod_{\sc\in \sC_I}  \coprod_{\st \in \sF_{I,\sc}}  {\bf m}_{\sc,\st} (\Wc_{I,\sc})$
the disjoint union and is trivial on the fibers $\sF_{I,\sc}$.
The lemma follows.
\end{proof}

\subsubsection{The description of $\M_{\sigma,\lambda}$}
We wish to relate $H$-invariant functionals on the induced representation
$\Hc_{\sigma,\lambda} = \Ind_{\oline{P_I}}^G (\sigma\otimes \lambda)$ with regard to the
open $H\times \oline{P_I}$-double cosets in $G$. Recall from \eqref{Matsuki Pbar} the bijection
$$ \W(I)\bs \Wc \to (H\bs G/ \oline{P_I})_{\rm open}, \ \ \W(I) w \mapsto Hw^{-\theta} \oline{P_I}\, .$$
Now we define for each $[w]= \W(I)w\in \W(I)\bs \Wc$ a subspace

$$\Hc_{\sigma,\lambda}^\infty[w]=\{ f\in \Hc_{\sigma,\lambda}^\infty\mid  \supp f \subset H w^{-\theta} \oline{P_I}\}$$
and for each $\eta \in (\Hc_{\sigma,\lambda}^{-\infty})^H$ we define the restrictions
$$\eta[w]:= \eta|_{\Hc_{\sigma,\lambda}^\infty[w]}\, .$$
These functionals have now a straightforward description.
Notice that $\eta[w]$ only depends on the double coset $H w^{-\theta} \oline{P_I}$.
This allows us to replace $\W(I)\bs \Wc$ by $\W(I)\bs \W$ and since elements
$w\in \W$ have representatives in $K$ have   $w^{-\theta} = w^{-1}$ for
$w\in \W$.

 Let now $w\in \W$. Notice that the $H$-stabilizer
of the point $w^{-1}\oline{P_I}\in G/\oline{P_I}$ is given by the (symmetric) subgroup
$H\cap w^{-1} M_I w$ of $w^{-1} M_I w$. Allowing a slight conflict with previous notation we
let $\sigma_w$ be the representation of $w^{-1} M_I w$ induced from the group isomorphism $M_I \simeq w^{-1} M_I w$.

Frobenius reciprocity then associates to each $\eta$ and $[w]$ a unique distribution vector
$$\zeta_\eta[w]\in (\Hc_{\sigma_w}^{-\infty})^{H\cap w^{-1} M_I w}$$
such that
$$\eta[w](f) = \int_{H/ H\cap w^{-1} M_I w}  \zeta_\eta[w] (f(hw^{-1})) \ dh(
H\cap w^{-1} M_I w)\qquad (f\in \Hc_{\sigma,\lambda}^\infty[w])\, .$$

For each $[w]\in \W(I)\bs \W$ we pick with $w\in \W$ a representative.
As $\W(I)$ normalizes $M_I$  via inner automorphisms,  it follows that
$\sigma_w$ depends only on $[w]$, up to equivalence.
Set

\begin{equation}\label{def V sigma}  V(\sigma):=\bigoplus_{[w]\in \W(I)\bs \W}  (\Hc_{\sigma_w}^{-\infty})^{H\cap w^{-1}M_Iw}\end{equation}
and consider then the  evaluation map
\begin{equation}\label{j-map} {\rm ev}_{\sigma,\lambda}:  (\Hc_{\sigma,\lambda}^{-\infty})^H \to V(\sigma)
, \ \ \eta\mapsto   (\zeta_\eta[w])_{w\in \W(I)\bs \W}\end{equation}
This map is a bijection for generic $\lambda$ by \cite[Thm.~5.10]{vdB} for the case of $P_I=Q$ and \cite[Thm.~3]{CD}
in general.
Sometimes it is useful to indicate the choice of the parabolic $\oline P_I$ above $M_I A_I$ which was used in the definition
of the induced representation $\Hc_{\sigma,\lambda}= \Ind_{\oline P_I}^G (\sigma\otimes\lambda)$. Then we write
${\rm ev}_{\oline{P_I},\sigma, \lambda}$ instead of ${\rm ev}_{\sigma, \lambda}$.
Further, for $\lambda$ generic we recall the standard notation of \cite{vdB} and \cite{CD}:
\begin{equation} \label{def j-map}j(\oline{P_I}, \sigma, \lambda, \zeta):={\rm ev}_{\oline{P_I}, \sigma, \lambda}^{-1}(\zeta)\in
(\Hc_{\sigma, \lambda}^{-\infty})^H\qquad (\zeta\in V(\sigma))\, .\end{equation}

Next we define a subspace of $V(\sigma)$ by
$$ V(\sigma)_2:=\bigoplus_{[w]\in \W(I)\bs \W}  \M_{\sigma_w, \rm d} $$
with $\M_{\sigma_w, \rm d}\subset (\Hc_{\sigma_w}^{-\infty})^{H\cap w^{-1}M_Iw}$
referring to $\M_{\sigma, \rm d}\subset (\Hc_{\sigma}^{-\infty})^{H\cap M_I}$ for
$M_I$ replaced by $w^{-1}M_Iw$.

In the sequel we assume that $\lambda \in i\af_I^*$ is generic, i.e. $j(\oline{P_I}, \sigma, \lambda, \zeta)$
is defined (the obstruction is a countable set of hyperplanes) and the representation $\pi_{\sigma,\lambda}$ is a
generic member in $\supp^{I,\rm td}\mu \subset \supp  \mu$, see Subsection \ref{SRT}.
Recall that our request is that $\sigma$ is cuspidal, as defined in  Section \ref{support 15}.

 The main result of this subsection then is:

\begin{theorem}\label{CD tempered}   Let  $\sigma$ be cuspidal. Then for Lebesgue-almost all $\lambda\in i\af_I^*$
the image of $\M_{\sigma,\lambda}$ by
${\rm ev}_{\sigma,\lambda}$ is
$V(\sigma)_2$, i.e.
$${\rm ev}_{\sigma,\lambda}:  \M_{\sigma,\lambda}\to V(\sigma)_2$$
is a bijection. In particular we have
\begin{equation} \label{vdB-CD} \dim \M_{\sigma, \lambda} = \sum_{ [w]\in \W(I)\bs \W}  \dim \M_{\sigma_w, {\rm d}}\, .\end{equation}
\end{theorem}

The proof of this theorem will be prepared by several lemmas.
The first lemma is valid for a general unimodular real spherical space $Z=G/H$ with Plancherel measure
$\mu$. In the sequel we consider
$L^2(Z)$ as a unitary module for $G\times A_{Z,E}$ and recall that the twisted discrete spectrum
$L^2(Z)_{\rm td} \subset L^2(Z)$ is a $G\times A_{Z,E}$-invariant subspace. Define
the {\it essentially continuous spectrum} by $L^2(Z)_{\rm ec} := L^2(Z)_{\rm td}^\perp$.  We write
$\mu^{\rm td}$ and $\mu^{\rm ec}$ for the Plancherel measures of $L^2(Z)_{\rm td}$ and
$L^2(Z)_{\rm ec}$.

\begin{lemma}\label{lemma 151} Let $Z=G/H$ be a unimodular real spherical space with Plancherel measure
$\mu$. Then
$$ \mu^{\rm ec}(\supp \mu^{\rm td})=0, $$
i.e. the Plancherel supports of $L^2(Z)_{\rm td}$ and $L^2(Z)_{\rm ec}$ do $\mu$-almost not interfere.
\end{lemma}
\begin{proof} The proof goes by comparing the infinitesimal characters of the representations
occurring in $\mu^{\rm td}$ and $\mu^{\rm ec}$. For that we recall that the map
$$ \Phi: \hat G \to \jf_\C^*/\W_\jf, \ \ \pi\mapsto \chi_\pi$$
is continuous.  Next the Bernstein decomposition of $L^2(Z)$ implies that
$\mu^{\rm ec} $ is equivalent to $\sum_{w\in \Wc}\sum_{I\subsetneq S} \mu^{I,w, \rm td}$ with
$\mu^{I,w,\rm td}$ the Plancherel measure of $L^2(Z_{I,w})_{\rm td}$.  In this regard
we note moreover that $\mu^{I,w, \rm td}$ is build up by the Lebesgue measure on $i\af_I^*$ and
counting measure over each fiber $\lambda\in i\af_I^*$.
Now the main result of \cite{KKOS} asserts that for each pair $I,w$
there is a $\W_\jf$-invariant lattice $\Lambda=\Lambda(I,w)\subset \jf_\R^*$
such that$$ \Phi(\supp \mu^{I,w,{\rm td}}) \subset   \left[\bigcup_{s\in \W_\jf} (\Lambda + i\Ad(s)\af_I^*)\right]/\W_\jf\, .$$
Now the continuity of $\Phi$ and the aforementioned structure of the various $\mu^{I,w,\rm td}$ with regard
to Lebesgue measures imply the lemma.
\end{proof}

A further important ingredient in the proof of Theorem \ref{CD tempered} is
the long intertwiner,  which we also used in treatment of the group case  \eqref{def A}:
$$\A_{\sigma,\lambda}: \Ind_{P_I}^G( \sigma\otimes \lambda)\to \Ind_{\oline{P_I}}^{G}(\sigma\otimes \lambda)$$
$$ \A_{\sigma,\lambda}(u)(g) =  \int_{\oline {N_I}} u(g \oline{n_I})  \ d\oline{n_I} \qquad (g\in G)$$
which is defined  near $g=\1$  for all $u$ with $\supp u \subset\Omega P_I$ for $\Omega\subset\oline{ N_I}$ compact,
and for general $u$ and $g$ by meromorphic continuation with respect to $\lambda$.
\par We wish to compute the asymptotics of $m_{v,\eta}$ for $\eta\in \M_{\pi, \sigma}$ for certain test vectors
$v\in \Hc_{\sigma,\lambda}^\infty$. In more precision, let  $u \in\Ind_{P_I}^G( \sigma\otimes \lambda)^\infty$ with
$\supp u \subset \Omega P_I$ with $\Omega$ as above.  Our test vectors
$v$ are then given by $v= \A_{\sigma,\lambda}(u)$.
Now with
\begin{equation}\label{display 152}  \tilde \eta= \eta\circ \A_{\sigma,\lambda}\, \end{equation}
we obtain the tautological identity

$$ m_{v, \eta}(g) = m_{u, \tilde \eta}(g)\, .$$
The advantage of using the opposite representative $\Ind_{P_I}^G(\sigma\otimes\lambda)$
for $[\pi_{\sigma,\lambda}]$
is that it allows us to compute the asymptotics of $m_{u, \tilde \eta}(a)$ for $a=a_t=\exp(tX)\in A_I^{--}$ on rays to infinity.
In more precision,  we have the following symmetric space
analogue of  Lemma \ref {Lemma KKOS}.
Let
\begin{equation}\label{defi zeta}
\tilde \zeta=\operatorname{ev} _{P_I,\sigma,\lambda}(\tilde \eta) \in V(\sigma).
\end{equation}

\begin{lemma} \label{lemma 15a}With the notation introduced above we have for all $m_I\in M_I$:

\begin{equation} \label{D-limit} \lim_{t\to\infty} a_t^{\lambda - \rho} m_{u,\tilde \eta}(m_I a_t)= \tilde\zeta_{[\1]}(\sigma(m_I^{-1}) (\A_{\sigma,\lambda}(u)(\1))) =
\tilde \zeta_{[\1]}(\sigma(m_I^{-1}) (v(\1))).
\end{equation}
for the $[\1]$-component $\tilde \zeta_{[\1]}\in (\Hc_\sigma^{-\infty})^{M_I\cap H}$ of
$\tilde\zeta$.
\end{lemma}

\begin{proof}  It is sufficient to prove the assertion for $m_I=\1$. Next, since
$a_t^{-1} \Omega  a_t \to \{\1\}$ we may assume in addition that
$\supp u \subset \Omega P_I \cap H P_I$.  Hence
$$m_{u,\tilde \eta}(a_t) =m_{u,\tilde \eta[\1]}(a_t)= ( \pi_{\sigma,\lambda}(a_t)(\tilde \eta[\1]))(u)$$
by the support condition of $u$.
It is then easy to verify (see the proof of \cite[Lemme 16]{Delorme2}) that
$$\lim_{t\to \infty}   a_t^{\lambda - \rho} \left(\pi_{\sigma,\lambda}(a_t)(\tilde \eta[\1])\right)=\tilde \zeta_{[\1]} \cdot d\oline n_I$$
as a distribution. The lemma follows.
\end{proof}

Note that for generic $\lambda$ the intertwiner
$\A_{\sigma,\lambda}$ induces a natural linear isomorphism
$$ b_{\sigma, \lambda}:  V(\sigma) \to V(\sigma),$$
defined by
$$ b_{\sigma,\lambda}(\xi)=
\operatorname{ev}_{P_I,\sigma, \lambda} \left( j(\oline {P_I}, \sigma,\lambda,\xi)\circ \A_{\sigma,\lambda}\right)\,.$$

In this regard we recall from \cite[Th. 2]{Delorme2}:
\begin{lemma}\label{petit B} For generic $\lambda$ one has
\begin{equation} \label{claim15} b_{\sigma,\lambda}(V(\sigma)_2) = V(\sigma)_2\, .\end{equation}
\end{lemma}

\begin{proof}[Proof of Theorem \ref{CD tempered}] Let $\eta\in (\Hc_{\sigma,\lambda}^{-\infty})^H$ and $\zeta=\zeta_\eta={\rm ev}_{\sigma,\lambda}(\eta)\in V(\sigma)$.
The task is to show that $\zeta\in V(\sigma)_2$ if and only if $\eta\in \M_{\sigma,\lambda}$.  Recall that
$\zeta= (\zeta_{[w]})_{[w]\in \W(I)\bs \W}$ is a tupel in accordance with the definition of $V(\sigma)$ in \eqref{def V sigma}.

\par Assume first that $\eta\in \M_{\sigma,\lambda}$.
The proof goes by comparing two different expressions
for the constant term $m_{v,\eta^I}$ for certain test vectors $v\in \Hc_{\sigma,\lambda}^\infty $.
According to Lemma \ref{lemma 151} applied to $Z=Z_I$ we may assume that
$\M_{\sigma,\lambda}^{I} = \M_{\sigma,\lambda, {\rm td}}^{I}$.

Hence $\eta^{I,\rho-\lambda}\in \M_{\sigma,\lambda}^{I,\rho-\lambda} = \M_{\sigma,\lambda, {\rm td}}^{I,\rho-\lambda}$. By  \eqref{formula 1} we then have
$\eta^{I,\rho-\lambda}=\xi_{\sigma,\lambda,\zeta'}$
for some $\zeta'\in \M_{\sigma, {\rm d}}$,
that is,
\begin{equation} \label{EQ1} \eta^{I,\rho-\lambda}(\pi(m_I)v)= \zeta'(\sigma(m_I^{-1}) (v(\1))) \qquad (v\in \Hc_{\sigma,\lambda}^\infty, m_I \in M_I)\, .\end{equation}

\par On the other hand we can compute the asymptotics via Lemma \ref{lemma 15a}.
Comparing \eqref{D-limit} with \eqref{EQ1} and using Theorem \ref{loc ct temp} yields
$$ \zeta'(\sigma(m_I^{-1})(v(\1))= \tilde \zeta_{[\1]} (\sigma(m_I^{-1})(v(\1))) \qquad (m_I\in M_I)$$
for our test vectors $v=\A(u)$.
Thus we have
\begin{equation} \label{identity 15} m_{\zeta', v(\1)}= m_{\tilde \zeta_{[\1]}, v(\1)}\end{equation}
as functions on $M_I/ M_I\cap H$.
We claim that $\tilde \zeta_{[\1]}\in \M_{\sigma,\rm d}$ and $\tilde \zeta_{[\1]}=\zeta'$.
To see that we first observe that there exists at least one $v$ with $v(\1)\neq 0$.
This is because $v(\1)\neq 0$ translates into $\int_{\oline{N_I}} u( {\oline n_I}) \ d{\oline n_I}\neq 0$ which can obviously
be achieved for one of our test vectors $u$.
Now recall  that $\zeta'\in \M_{\sigma,\rm d}$.
Hence  \eqref{identity 15} implies  that  $\tilde \zeta_{[\1]}\in \M_{\sigma, {\rm d}}$, since for $\tilde \zeta_{[\1]}$ to yield an
embedding into $L^2(M_I/M_I\cap H)$ only one non-zero matrix coefficient
$m_{\tilde\zeta_{[\1]}, v(\1)}$ has to be square integrable.
With that $\zeta'=\tilde \zeta_{[\1]}$ follows from the orthogonality
relations \eqref{SW-ortho} and \eqref{identity 15}:  For $\zeta_0=\zeta'-\tilde\zeta_{[\1]}$ we have
$$0=\|m_{\zeta_0, v(\1)}\|_{L^2(M_I/ M_I\cap H_I)}^2=\|v(\1)\|_{\Hc_\sigma}^2  \|\zeta_0\|_{\M_{\sigma, \rm d}}^2\, .$$

\par Next we let $w\in \W\simeq \Wc$ vary. Analogous reasoning via transport of structure $Z\to Z_w$
yields that $\tilde \zeta_{[w]}\in \M_{\sigma_w,\rm d}$.
Thus we arrive at
\begin{equation} \label{display 151}\tilde\zeta:= (\tilde \zeta_{[w]})_{[w]\in \W(I)\bs \W} \in V(\sigma)_2\,.\end{equation}

\par Now observe that
$b_{\sigma,\lambda}(\zeta)=\tilde\zeta\in V(\sigma)_2$ in view of
\eqref{display 152},
\eqref{defi zeta}, and \eqref{display 151},
and  from  Lemma \ref{petit B}
it then follows that $\zeta\in V(\sigma)_2$, i.e. we have shown the implication
$\operatorname{ev}_{\oline P_I, \sigma, \lambda}(\M_{\sigma,\lambda})\subset V(\sigma)_2$ of the theorem.

\par To complete the proof of the theorem we remain with the converse inclusion
$\operatorname{ev}_{\oline P_I, \sigma, \lambda}(\M_{\sigma,\lambda})\supset V(\sigma)_2$.
For that let $\zeta=(\zeta_{[w]})_{[w]}\in V(\sigma)_2$.
Forming wave packets via $\eta=j(\oline P_I, \sigma, \lambda, \zeta)$
for varying $\lambda$, we finally deduce  with \cite[Thm.~ 4]{Delorme2} that
$j(\oline P_I, \sigma, \lambda, \zeta)$ contributes to the $L^2$-spectrum of $Z$. Hence
$\eta=j(\oline P_I, \sigma, \lambda, \zeta)\in \M_{\sigma,\lambda}$ for Lebesgue almost all
$\lambda$, completing the proof of the theorem.
\end{proof}

In the course of the proof of Theorem \ref{CD tempered} we have shown the following
identity:

 \begin{lemma} \label{lemma scatter} Let $\lambda$ be generic and  $\eta\in \M_{\sigma,\lambda}$ such that
$\eta=j(\oline{P_I}, \sigma, \lambda, \zeta)$ for some $\zeta\in V(\sigma)_2$.
Then $\tilde \eta = \eta\circ \A_{\sigma, \lambda}$ is of the form
$\tilde \eta=j(P_I, \sigma, \lambda, \tilde \zeta)$ for a unique  $\tilde \zeta\in V(\sigma)_2$ and
\begin{equation} \label{scatter id}  \eta^{I,\rho-\lambda}= \xi_{\sigma, \lambda, \tilde \zeta_{[\1]}}\end{equation}
with $\tilde \zeta_{[\1]}= \tilde \eta[\1][\1]\in \M_{\sigma, \rm d}$ and $\xi_{\sigma, \lambda, \tilde \zeta_{[\1]}}$ defined as in \eqref{def xi sigma eta}.
\end{lemma}

Let us now transport the structure from $Z=G/H$ to $Z_w=G/H_w$ for $w\in \Wc$ and
write $j_w$ and $V(\sigma)_w$ for the $j$-map \eqref{def j-map} and multiplicity
space for $Z_w$. Note that  $V(\sigma)_w\simeq V(\sigma)$ by permutation of coordinates.

Then $\eta= j(\oline  P_I,  \sigma, \lambda,\zeta)$ for $\zeta=(\zeta_{[u]})_{[u]\in \W(I)\bs \Wc}\in V(\sigma)$
will be moved to $\eta_w$ which then can be written as
$\eta_w=j_w(\oline P_I, \sigma, \lambda, \zeta^w)$ for some $\zeta^w= (\zeta^w_{[u]})_{[u]\in \W(I)\bs \Wc}\in V(\sigma)_w$.
By the construction of $j$-maps
which relates invariant functionals to open $H$-orbits we then obtain from $\eta_w= \eta\circ w^{-1}$
the transition relations
\begin{equation}\label{transit zeta}  \zeta^w_{[\1]}= \zeta_{[w]} \,. \end{equation}

\begin{theorem} \label{Th. 15.8}For generic $\lambda\in iC_I^*$ and $\pi=\pi_{\sigma,\lambda}$ for $\sigma\in \hat M_I$ cuspidal
the map

\begin{equation} \label{D1} \M_\pi \to  \bigoplus_{ [w]\in \W(I)\bs\Wc }  \M_{\pi, w, {\rm td}}^{I,\rho-\lambda}, \ \
\eta\mapsto (\eta_{w}^{I,\rho-\lambda})_{[w]\in \W(I)\bs \Wc}\end{equation}
is a bijective isometry.
\end{theorem}
\begin{proof}  First note that both target and source have the same dimension by
Theorem \ref{CD tempered} and  equation \eqref{formula 1} applied to all spaces $Z_w$ via transport of structure.
Now
Lemma \ref{lemma scatter} together with \eqref{transit zeta} imply that the map is bijective.
Finally that the map is an isometry follows from the Maass-Selberg relations from Theorem \ref{eta-I continuous} -- for  that we use
$\W(I)\bs \Wc\simeq \bigcup_{\sc\in \sc_I} \sF_{I,\sc}$ from Lemma \ref{lemma fine W}.
\end{proof}

 \subsection{The Plancherel formula}
 As in the group case we select now with $\Ic\subset S$ a subset of representatives for the
 association classes.  Let us describe in terms of \eqref{D1} the
 inner  product on the multiplicity space $\M_\pi$ for $[\pi]=[\pi_{\sigma, \lambda}]$,
 where $\sigma$ is cuspidal with respect to $M_I$ and $I\in \Ic$.

  For that observe that the map
 $$\M_{\sigma,\rm d} \to \M_{\pi}^{I, \rho - \lambda}, \ \ \zeta\mapsto \xi_{\sigma,\lambda, \zeta}$$
 is a linear isometry  by \eqref{formula 1} if we request that the Plancherel measure for $L^2(Z_I)_{\rm td}$ is
 the Lebesgue measure $d\lambda$ times the counting measure of the discrete series, i.e.  we request
 the normalization \eqref{Z_I symm Re}.
 \par Via Theorem \ref{Th. 15.8} we can now
 normalize the Plancherel measure $\mu$ such that
 we have an isometric
 isomorphism:
 $$\bigoplus_{ [w]\in \W(I)\bs\Wc }  \M_{\pi, w, {\rm td}}^{I,\rho-\lambda}\simeq
 \bigoplus_{ [w]\in \W(I)\bs\Wc }  \M_{\sigma, w, \rm d}$$
 where $\M_{\sigma, w, \rm d}$ refers to the $M_I \cap H_w$-invariant
 square integrable functionals of the symmetric space $M_I/ M_I \cap H_w \simeq w^{-1}M_I w/ w^{-1}M_I w \cap H$.
  We now define
  $$\Hc_I = \bigoplus_{[w]\in \W(I)\bs \Wc}  \sum_{\sigma \in \hat M_I}
\int_{i C_I^*}^\oplus  \Hc_{\sigma,\lambda}\otimes \M_{\sigma,w, {\rm d}}  \ d\lambda$$
considered as a subspace of $L^2(Z_I)_{\rm td}$. Let $B_I'$ be the restriction of $B$ to $\Hc_I$.
Then with Theorem \ref{Th. 15.8} we obtain with the same reasoning as in the group case
the Plancherel formula for symmetric spaces:

\begin{theorem} {\rm (Plancherel formula for symmetric spaces)}
Let $Z=G/H$ be a symmetric space and let its Plancherel measure be normalized by unit asymptotics.
Then
$$B'=\bigoplus_{I\in \Ic} B_I': \bigoplus_{I\in \Ic} \Hc_I \to L^2(Z)$$
is a bijective $G$-equivariant isometry and is the inverse of the Fourier transform.
\end{theorem}

\end{document}